\documentclass[final]{siamltex}

\usepackage{times}                
\usepackage{bm,amsbsy}
\usepackage{amsmath, amssymb,verbatim,graphicx, amsfonts}
\usepackage{mathrsfs}
\usepackage{verbatim}
\usepackage{mfd}
\usepackage{multirow}
\usepackage{subfigure}
\usepackage{epsfig}
\usepackage{blkarray}
\usepackage{color}

\usepackage{pdfpages,setspace}
\usepackage[left=1in,top=1in,right=1in,bottom=1.0in]{geometry}
\usepackage{color}
\allowdisplaybreaks
\expandafter\def\expandafter\normalsize\expandafter{%
    \normalsize
    \setlength\abovedisplayskip{5pt}
    \setlength\belowdisplayskip{5pt}
    \setlength\abovedisplayshortskip{0pt}
    \setlength\belowdisplayshortskip{5pt}
}

\title{Energy Stable Discontinuous Galerkin Methods for Maxwell's Equations in Nonlinear Optical Media}

\author{Vrushali A. Bokil
\thanks{Department of Mathematics, Oregon State University,
Corvallis, OR 97331 U.S.A.
{\tt bokilv@math.oregonstate.edu}.}%
\and
 Yingda Cheng
\thanks{Department of Mathematics, Michigan State University,
East Lansing, MI 48824 U.S.A.
 {\tt ycheng@math.msu.edu}. Research is supported by NSF grant  DMS-1453661.}
\and
 Yan Jiang
\thanks{Department of Mathematics, Michigan State University,
East Lansing, MI 48824 U.S.A.
 {\tt jiangyan@math.msu.edu}. }
 \and
 Fengyan Li
\thanks{Department of Mathematical Sciences, Rensselaer Polytechnic Institute, Troy, NY 12180 U.S.A.
 {\tt  lif@rpi.edu}. Research is supported by NSF grant  DMS-1318409.}
}

\date{\today}

\begin{document}

\maketitle

\begin{abstract}
The propagation of electromagnetic waves in general media is modeled by the time-dependent Maxwell's partial differential equations (PDEs), coupled with constitutive laws that describe the response of the media. In this work, we focus on nonlinear optical media whose response is modeled by a system of first order nonlinear ordinary differential equations (ODEs), which include a single resonance linear Lorentz dispersion, and the nonlinearity comes from the instantaneous electronic Kerr response and the residual Raman molecular vibrational response. To design efficient, accurate, and stable computational methods, we apply high order discontinuous Galerkin discretizations in space to the hybrid PDE-ODE Maxwell system with several choices of numerical fluxes, and the resulting semi-discrete methods are shown to be energy stable. Under some restrictions on the strength of the nonlinearity, error estimates are also established. When we turn to fully discrete methods, the challenge to achieve provable stability lies in the temporal discretizations of the nonlinear terms. To overcome this, novel strategies are proposed to treat the nonlinearity in our model within the framework of the second-order leap-frog and implicit trapezoidal time integrators. The performance of the overall algorithms are  demonstrated through numerical simulations of kink and antikink waves, and third-harmonic generation in soliton propagation.

\end{abstract}



\section{Introduction}


Nonlinear optics is the study of the behavior of light in nonlinear media. This field has developed into a significant branch of physics since the introduction of intense lasers with high peak powers. 
In  nonlinear media, the material response depends nonlinearly on
the optical field, and many interesting physical phenomena, such as  frequency mixing and second/third-harmonic generation 
have been observed and  harnessed for practical applications.
We refer to classical textbooks
\cite{bloembergen1996nonlinear,boyd2003nonlinear, new2011introduction}
for a more detailed review of the field of nonlinear optics.

Our interest here is in the development of novel numerical schemes for the Maxwell's equations in nonlinear optical media. 
Relative to the widely used  asymptotic and paraxial wave   models derived from Maxwell's equations, such as  nonlinear Schr\"{o}dinger equation (NLS) and beam propagation method (BPM) \cite{bloembergen1996nonlinear,boyd2003nonlinear}, simulations of the nonlinear Maxwell's system in the time domain are
more computationally intensive. However, these simulations have the advantage of being substantially more robust
because they directly solve for fundamental quantities, the electromagnetic fields in space and time. These simulations also avoid the simplifying assumptions that lead to conventional asymptotic and paraxial propagation analyses, and are able to treat interacting waves at different frequencies directly \cite{joseph1997fdtd}.
 Recent optics and photonics research has focused on phenomena at smaller and smaller length scales or multiple spatial scales. For such phenomena simulating the full Maxwell PDE models is important in adequately capturing useful optical effects \cite{goorjian1992direct,joseph1997fdtd,taflove2013advances}. 


When Maxwell's equations are considered to model the electromagnetic (EM) waves propagating through a nonlinear optical medium, the medium response is described by constitutive laws that relate the electric field $\mb{E}$ and the electric flux density $\mb{D}$ through the polarization $\mb{P}$ of the medium.
In this work, we focus on a macroscopic phenomenological description of the polarization, which comprises both linear and nonlinear responses. Specifically, the linear response is modeled by a single resonance Lorentz dispersion, while the nonlinear response is cubic and incorporates the instantaneous {\it Kerr} effect and the delayed nonlinear Lorentz dispersion called {\it Raman scattering}.
Within this description, we will follow the {\it auxiliary differential equation} (ADE) approach, where the linear and nonlinear Lorentz dispersion is represented through a set of ODEs, describing the time evolution of $\mb{P}$ (hence of $\mb{D}$) forced by $\mb{E}$, appended to Maxwell's equations. 
An alternative representation is via a {\it recursive convolution method}, where $\mb{D}$ is computed from $\mb{E}$ through a time convolution integral
\cite{Taflove:2005}.

In the literature, finite difference time domain (FDTD) based, finite element (FEM) based, pseudospectral based methods, finite volume (FV) based, among others, are available for the integration of the full Maxwell's equations in nonlinear media, along with appended ODEs for the material response. The Yee scheme \cite{yee1966numerical} is an FDTD method for Maxwell's equations that has long been one of the gold standards for numerical simulation of Maxwell's equations in the time domain, especially for linear problems \cite{Taflove:2005}.
Maxwell's equations in a linear Lorentz medium with a nonlinear Kerr response are investigated in \cite{hile1996numerical, sorensen2005kink},  
while in \cite{gilles2000comparison}, additional effects due to Raman scattering are studied through a 1D FDTD analysis.
More references for
linear and nonlinear Lorentz dispersion, can be found in
\cite{
bourgeade2005numerical,greene2006general,ramadan2015systematic} for the 1D case, and in \cite{fujii2004high,joseph1993direct,ziolkowski1994nonlinear} for 2D and 3D cases. 
Yee based FDTD approaches result in second order schemes which accumulate significant errors over long time modeling of  wave propagation \cite{cohen2013higher,cohen2016finite}. While higher order FDTD methods can alleviate this issue, they can be cumbersome in modeling complex geometries. On the other hand, though the FEMs are well suited for modeling complex geometries, they have not been well developed for nonlinear Maxwell models. FEM analysis for some nonlinear models can be found in \cite{fisher2007efficient}. 
In \cite{tyrrell2005pseudospectral} a pseudospectral spatial domain (PSSD) approach is presented for linear Lorentz dispersion and nonlinear Kerr response, and in \cite{kinsler2007optical} optical carrier wave shock is studied using the PSSD technique. 
FV based methods for nonlinear Kerr media are
addressed in \cite{aregba2015godunov,de2000high} in which the Maxwell-Kerr model is
approached   as a hyperbolic system and approximated by a Godunov scheme, and a third order Roe solver, respectively, in one and two spatial dimensions.

In this work, we use high order discontinuous Galerkin (DG) methods for the spatial discretization of our nonlinear Maxwell models. This is motivated by various properties of DG methods, including high order accuracy,  excellent dispersive and dissipative properties in standard wave simulations,  flexibility in adaptive implementation and high parallelization, and suitability for complicated geometry (e.g., \cite{cockburn2000development,hesthaven2007nodal}). DG methods differ from classical finite element methods in their use of piecewise smooth approximate functions, while inter-element communication is achieved through the use of numerical fluxes, which are consistent with the physical fluxes and play a vital role in accuracy, stability, energy conservation, and computational efficiency. For the nonlinear Maxwell models that we consider, with the numerical fluxes chosen to be either central or alternating, the solutions to the semi-discrete DG methods satisfy an energy equation just as the exact solutions do, hence the methods are energy stable,  even in the presence of both the Kerr and Raman nonlinear effects. Another dissipative flux, inspired by the upwind flux for Maxwell's equations in a linear nondispersive dielectric \cite{hesthaven2002nodal}, called ``upwind flux'' in this paper, is also considered with the respective energy stability established. For the semi-discrete methods with all three types of numerical fluxes, error estimates are carried out under some additional assumptions on the strength of the nonlinearity in the underlying model.

In addition to the error estimates, the nonlinearity in the model poses challenges  to the design of fully discrete schemes with provable energy stability. As one major contribution, we propose in this work a novel strategy to discretize the nonlinear terms within the commonly used second-order leap-frog and implicit trapezoidal temporal discretizations. The resulting fully discrete methods are proved to be stable. More specifically, the method with the modified leap-frog time discretization is conditionally stable under a CFL condition, which is the {\it same} as the one for Maxwell's equations without Kerr, linear and nonlinear Lorentz dispersion; while the fully implicit method with the modified trapezoidal temporal discretization is unconditionally stable. In both cases, we find it important, at least from the theoretical point of view, to discretize the ODE part of the system {\it implicitly}. To our best knowledge, the temporal discretizations that are adapted to nonlinear models and with  {\it provable stability} are not yet available. In the present work, the methods  and numerical verification are presented for the model in one dimension, and their extension to higher dimension will be explored in a separate paper.

DG methods have grown to be broadly adopted for EM simulations in the past two decades. They have been developed and analyzed for time dependent linear models, including  
Maxwell's equations in free space (e.g.,
\cite{chung2013convergence, cockburn2004locally, hesthaven2002nodal}), dispersive media (e.g., \cite{gedney2012discontinuous, huang2011interior, lanteri2013convergence,lu2004discontinuous}), as well as metamaterials (e.g., \cite{chung2013staggered,li2011development,li2014analysis, li2015optimal}).
 However, there exists only limited study for DG methods for nonlinear Maxwell models. For example, in \cite{blank2013discontinuous, fezoui2015discontinuous}, Kerr nonlinearity is investigated, where the entire Maxwell PDE-ODE system is cast as a nonlinear hyperbolic conservation law, for which DG methods have long been known for their success. A relaxed version of the Kerr model, called the Kerr-Debye model, was examined in \cite{huangsecond}, where a second-order asymptotic-preserving and positivity-preserving DG scheme is designed and analyzed.

The rest of this paper is organized as follows. In Section \ref{sec:2}, Maxwell's equations in an optical medium with a nonlinear dispersive response are introduced. In Section \ref{sec:semi}, DG spatial discretizations are formulated, where energy stability is established for the resulting semi-discrete schemes. Error estimates are further carried out in Section \ref{sec:error}. In Section \ref{sec:fully}, temporal discretizations are presented within the framework of the second-order leap-frog and trapezoidal method, with a novel treatment of the nonlinear terms in the models aimed at obtaining energy stability for the fully discrete schemes. The performance of the overall algorithms are  demonstrated in Section \ref{sec:numerical} through numerical simulations of the propagating kink and antikink waves, and third harmonic generation in soliton propagation. Finally, concluding remarks are made in Section \ref{sec:conclusion}.

\section{Physical Model: Maxwell's Equations and Polarization}
\label{sec:2}

We begin with the Maxwell's equations, that govern the time evolution of the electric field $\mathbf{E}$ and magnetic field $\mathbf{H}$ in a non-magnetic nonlinear optical medium,
\begin{subequations}\label{eq:max}
\begin{align}\label{eq:max1}
&\ds\dd{t}{\mathbf{B}}+{\bf{\nabla}}\times \mathbf{E} = 0, \ \text{in} \  (0,T)\times \Omega,  \\[1.5ex]
\label{eq:max2}
&\ds\dd{t}{\mathbf{D}}+\mathbf{J_s} -{\bf{\nabla}}\times \mathbf{H} = 0, \  \text{in} \ (0,T)\times \Omega, \\[1.5ex]
\label{eq:max3}
& {\bf{\nabla}}\cdot \mathbf{B}  = 0, \ {\bf{\nabla}}\cdot \mathbf{D} = \rho, \ \text{in} \  (0,T)\times \Omega,
\end{align}
\end{subequations}
along with initial and boundary data in the domain $\Omega\subset\mathbb {R}^d, d=1, 2,3$. The variable $\mathbf{J}_s$ is the source current density, and $\rho$ is the charge density. The electric flux density $\mathbf{D}$ and the magnetic induction $\mathbf{B}$ are related to the electric and magnetic field, respectively, via the constitutive laws
\begin{equation}
\label{eq:constD}
\mathbf{D} = \epsilon_0(\epsilon_\infty\mathbf{E}+\mathbf{P}), \ \ \mathbf{B} = \mu_0\mathbf{H},
\end{equation}
where $\mathbf{P}$ is the polarization.
The dielectric parameters are $\epsilon_{0}$, the electric permittivity
of free space, $\epsilon_{\infty}$, the relative electric permittivity in the limit
of the infinite frequency, and $\mu_0$, the magnetic permeability of free space. We will assume here that all model parameters are constant, and the material is isotropic. The term $\epsilon_0\epsilon_\infty \mb{E}$ captures the linear instantaneous response of the material to the EM fields.

To model the linear and nonlinear dispersion in the material we use the auxiliary differential equation (ADE) approach as presented in \cite{gilles2000comparison, Taflove:2005}. The linear (L) delayed or retarded  response of the material to the EM field is captured in the polarization, $\mathbf{P}$, via a linear single resonance Lorentz response, which, in the form of a second order ODE, is given as,
\begin{equation}
\label{eq:polar:P}
\frac{\partial^2 \mathbf{P}^{\mathrm{L}}_{\mathrm{delay}}}{\partial t^2}+\frac{1}{\tau}\frac{\partial\mathbf{P}^{\mathrm{L}}_{\mathrm{delay}}}{\partial t} +\omega_0^2\mathbf{P}^{\mathrm{L}}_{\mathrm{delay}}=\omega_p^2\mb{E}.
\end{equation}
Here $\omega_0$ and $\omega_p$ are the resonance and plasma frequencies of the medium, respectively, and $\tau^{-1}$ is a damping constant. In addition, $\omega_p^2 = (\epsilon_s-\epsilon_\infty)\omega_0^2$, with $\epsilon_s$ as the relative permittivity at zero frequency.

For pulse widths that are sufficiently short (for e.g., shorter than 1 pico-second (ps) for Silica) \cite{hile1996numerical}, the nonlinear response has an instantaneous as well as a delayed component. For the nonlinear (NL) response of the medium, we will consider a cubic Kerr-type  instantaneous response, and a retarded Raman molecular vibrational response called Raman scattering. The Kerr effect is a phenomenon in which the refractive index of a material changes proportionally to the square of the applied electric field. Raman scattering arises from the electric field induced changes in the internal nuclear vibrations on time scales $\approx$ 1 to 100 femto-seconds (fs) \cite{goorjian1992computational}, and is modeled by a nonlinear single resonance Lorentz delayed response. 
The two nonlinear responses are given as
\[ \mb P^{\mathrm{NL}}=\mb P^{\mathrm{NL}}_{\mathrm{Kerr}}+\mb P^{\mathrm{NL}}_{\mathrm{delay}}=\underbrace{a(1-\theta)\mb E|\mb E|^2}_{\mathrm{Kerr}} + \underbrace{a\theta Q \mb{E}}_{\mathrm{Raman}}.\]
Here $a$ is a third order coupling constant, $\theta$ parameterizes the relative strength of the instantaneous electronic Kerr and retarded Raman molecular vibrational responses, and $Q$ describes the natural molecular vibrations within the dielectric material that has frequency many orders of magnitude less than the optical wave frequency, responding to the field intensity. 
The time evolution of $Q$ is given by the following ODE,
\begin{equation}
\label{eq:polar:Q}
\frac{\partial^2 Q }{\partial t^2} +\frac{1}{\tau_v}\frac{\partial Q}{\partial t} +\omega_v^2Q = \omega_v^2|\mb{E}|^2,
\end{equation}
where $\omega_v$ is the resonance frequency of the vibration, and $\tau_v^{-1}$ a damping constant. This is essentially a model for a simple linear oscillator, but coupled to the nonlinear field intensity $|\mb{E}|^2$.

Taking into account all the effects discussed above, the constitutive law for the electric flux density is given by
\begin{equation}
\mb D = \epsilon_0(\epsilon_\infty\mb E +\mb{P}^{\mathrm{L}}_{\mathrm{delay}}+a(1-\theta)\mb E|\mb E|^2 + a\theta Q \mb{E}).
\label{eq:Li:D2E}
\end{equation}
With this, the mathematical model for EM wave propagation in this nonlinear optical medium will be given as a PDE-ODE system \eqref{eq:max}-\eqref{eq:Li:D2E}.

In the present work, the first order form of the second order ODEs, \eqref{eq:polar:P} and \eqref{eq:polar:Q}, will be adopted, and we will focus our investigation on the model in one spatial dimension, as below,
%
%
\begin{subequations}
\label{eq:1d:sys}
\begin{align}
\mu_0\dd{t}{H}& = \dd{x}{E},\\
\dd{t}{D} &= \dd{x}{H},\\
\dd{t}{P} &= J,\\
\dd{t}{J} &= -\frac{1}{\tau}J -\omega_0^2P+\omega_p^2E,\\
\dd{t}{Q} &= \sigma,\\
\dd{t}{\sigma} &= -\frac{1}{\tau_v}\sigma -\omega_v^2Q +\omega_v^2E^2,
\end{align}
\end{subequations}
with the constitutive law
\begin{equation}
\label{eq:1d:sys:1}
D = \epsilon_0(\epsilon_\infty E +P+a(1-\theta)E^3+a\theta Q E),
\end{equation}
where $P=P^L_{\mathrm{delay}}$.
In this model, we assume uniformity of all the vector fields in the $y$ and $z$ directions. Thus, all derivatives with respect to $y$ and $z$ in the curl and divergence operators are set to zero. All field quantities are represented by a single scalar component. The scalar magnetic field $H$ (hence $B$) represents the 2nd (or the 3rd) component of the vector magnetic field $\mb{H}$, and the scalar electric flux density
 $D$ (hence $E$) represents the 3rd (or the 2nd) component of $\mb{D}$ (hence $\mb{E}$). Gauss's laws \eqref{eq:max3} only involve the $x$ derivatives of the 1st components of $\mb{B}$ and $\mb{D}$, and therefore they are decoupled from the one-dimensional model \eqref{eq:1d:sys}-\eqref{eq:1d:sys:1} and become irrelevant.
Under the assumption of periodic boundary conditions, the energy $\mathcal{E}=\mathcal{E}(t)$ of the system \eqref{eq:1d:sys}, defined as
\begin{equation}
\label{contenergy}
\mathcal{E}=\int_{\Omega}  \left (\frac{\mu_0}{2} H^2 + \frac{\epsilon_0 \epsilon_\infty}{2} E^2 +  \frac{\epsilon_0}{2\omega_p^2} J^2   + \frac{\epsilon_0\omega_0^2}{2 \omega_p^2}   P^2+ \frac{\epsilon_0  a\theta}{4\omega_v^2} \sigma^2 + \frac{\epsilon_0  a\theta}{2}  Q E^2  + \frac{3 \epsilon_0 a (1-\theta)}{4} E^4+\frac{\epsilon_0  a\theta}{4}Q^2\right) dx,
\end{equation}
satisfies the following relation,
\begin{equation}
\label{contdecay}
\frac{d}{dt}\mathcal{E}=-\frac{\epsilon_0}{\omega_p^2 \tau} \int_\Omega J^2 dx-\frac{\epsilon_0  a\theta}{2\omega_v^2 \tau_v}  \int_\Omega \sigma^2 dx \le 0.
\end{equation}
Note that $\mathcal{E}(t)$ is guaranteed non-negative only when $\theta\in[0,\frac{3}{4}]$.


\section{Semi-discrete Scheme: Discontinuous Galerkin Method}
\label{sec:semi}

In this section, we introduce a semi-discrete DG method in space for the one dimensional model problem \eqref{eq:1d:sys} - \eqref{eq:1d:sys:1}. For simplicity, periodic boundary conditions are considered in $x$ direction.  (See Sect. \ref{sec:num2} and the appendix for some more general boundary conditions.)
Let $\Omega=[x_L, x_R]$ be the computational domain, for which a mesh, $x_L=x_{1/2}<x_{3/2}<\cdots<x_{N+1/2}=x_R$, is introduced. Let $I_j=[x_{j-1/2},x_{j+1/2}]$ be a mesh element, with $x_j=\frac{1}{2}(x_{j-\frac{1}{2}}+x_{j+\frac{1}{2}})$ as its center, $h_j=x_{j+\frac{1}{2}}-x_{j-\frac{1}{2}}$ as its length, and $h=\max_{1\leq j\leq N} h_j$ as the largest meshsize.  We now define
a finite dimensional discrete space,
\begin{equation}\label{ldg:vhk}
V_h^k=\{v : v|_{I_j} \in P^k(I_j), \, j=1,2,\cdots,N \},
\end{equation}
 which consists of piecewise polynomials of degree up to $k$ with respect to the mesh.
For any $v\in V_h^k$, let $v^+_{j+\frac{1}{2}}$ (resp. $v^-_{j+\frac{1}{2}}$) denote the limit value of $v$ at $x_{j+ \frac{1}{2}}$ from the element $I_{j+1}$ (resp. $I_j$), $[v]_{j+\frac{1}{2}}=v^+_{j+\frac{1}{2}} - v^-_{j+\frac{1}{2}}$  denote its jump, and $\{v\}_{j+\frac{1}{2}}=\frac{1}{2}(v^+_{j+\frac{1}{2}}+v^-_{j+\frac{1}{2}})$ be its average, again at $x_{j+\frac{1}{2}}$. The mesh is assumed to be quasi-uniform, namely, there exists a positive constant $\delta$, such that $\frac{h}{\min_j h_j}<\delta$, as the mesh is refined.

The semi-discrete DG method for the system  \eqref{eq:1d:sys} - \eqref{eq:1d:sys:1}
is formulated as follows: find $H_h(t,\cdot)$, $D_h(t,\cdot)$, $E_h(t,\cdot)$,  $P_h(t,\cdot)$, $J_h(t,\cdot)$, $Q_h(t,\cdot)$, $\sigma_h(t,\cdot)\in V_h^k$, such that  $\forall j,$

\begin{subequations}
\label{eq:1d:sch}
\begin{align}
\mu_0&\int_{I_j}\dd{t}{H_h}\phi dx +\int_{I_j} E_h\dd{x}\phi dx- (\widehat{E_h}\phi^-)_{j+1/2}
+ (\widehat{E_h}\phi^+)_{j-1/2}=0,\quad \forall \phi\in V_h^k, \label{eq:sch1}\\
&\int_{I_j}\dd{t}{D_h} \phi dx +\int_{I_j} H_h\dd{x}\phi dx- (\widetilde{H_h}\phi^-)_{j+1/2}
+ (\widetilde{H_h}\phi^+)_{j-1/2}=0,\quad \forall \phi\in V_h^k, \label{eq:sch2}\\
& \dd{t}{P_h}=J_h,  \label{eq:sch3}\\
& \dd{t}{J_h}= -\left(\frac{1}{\tau}J_h +\omega_0^2P_h-\omega_p^2E_h\right), \label{eq:sch4}\\
&\dd{t}{Q_h}=\sigma_h, \label{eq:sch5}\\
&\int_{I_j}\dd{t}{\sigma_h} \phi dx= -\int_{I_j}\left(\frac{1}{\tau_v}\sigma_h +\omega_v^2Q_h -\omega_v^2E_h^2\right)\phi dx,\quad \forall \phi\in V_h^k. \label{eq:sch6}
\end{align}
\end{subequations}

The  constitutive law is imposed via the $L^2$ projection, namely,
\begin{equation}
\label{eq:1d:sch:1}
\int_{I_j} D_h  \phi dx= \int_{I_j} \epsilon_0\left(\epsilon_\infty E_h +a(1-\theta)E_h^3+P_h+a\theta Q_h E_h\right)\phi dx,\quad\forall\phi\in V_h^k.
\end{equation}

Both terms $\widehat{E_h}$ and  $\widetilde{H_h}$ are numerical fluxes. In this work, we take either central fluxes,
\begin{equation}
\label{eq:flux:c}
\widehat{E_h} = \{E_h\},\quad \widetilde{H_h} = \{H_h\},
\end{equation}
 one of the following alternating flux pair 
\begin{equation}
\label{eq:flux:a}
\widehat{E_h} = E_h^-, \;\; \widetilde{H_h} =H_h^+;  \qquad \widehat{E_h} = E_h^+, \;\;\widetilde{H_h} =H_h^-,
\end{equation}
or the dissipative flux inspired by the upwind flux for the Maxwell system without Kerr, linear Lorentz and Raman effects,
\begin{equation}
\label{eq:flux:u}
\widehat{E_h} = \{E_h\}+\frac{1}{2}\sqrt{\frac{\mu_0}{\epsilon_0 \epsilon_\infty}}[H_h],\quad \widetilde{H_h} = \{H_h\}+\frac{1}{2}\sqrt{\frac{\epsilon_0 \epsilon_\infty}{\mu_0}}[E_h].
\end{equation}
In the rest of the paper, we will call \eqref{eq:flux:u} as the upwind flux.
It is known that the choice of numerical fluxes is important for the properties of the schemes, such as in numerical stability, accuracy, and even computational efficiency (see Sections \ref{sec:fully} and \ref{sec:numerical} for more discussions). We emphasize that \eqref{eq:sch3}-\eqref{eq:sch5} hold in strong sense.
In the theorem below, we establish stability of the  semi-discrete DG scheme which is consistent with the energy stability \eqref{contenergy}-\eqref{contdecay} of the PDE-ODE system \eqref{eq:1d:sys}-\eqref{eq:1d:sys:1}. 

\begin{theorem} [Semi-discrete stability]
\label{thm:semi}
Under the assumption of periodic boundary conditions, the semi-discrete DG scheme \eqref{eq:1d:sch}-\eqref{eq:1d:sch:1} with central and alternating fluxes, \eqref{eq:flux:c} and \eqref{eq:flux:a}, satisfies
$$
\frac{d}{dt}\mathcal{E}_h=-\frac{\epsilon_0}{\omega_p^2 \tau} \int_\Omega J_h^2 dx-\frac{\epsilon_0  a\theta}{2\omega_v^2 \tau_v}  \int_\Omega \sigma_h^2 dx \le 0,
$$
and the DG scheme with the upwind flux  \eqref{eq:flux:u} satisfies
$$
\frac{d}{dt}\mathcal{E}_h=-\frac{\epsilon_0}{\omega_p^2 \tau} \int_\Omega J_h^2 dx-\frac{\epsilon_0  a\theta}{2\omega_v^2 \tau_v}  \int_\Omega \sigma_h^2 dx-\frac{1}{2}\sqrt{\frac{\mu_0}{\epsilon_0 \epsilon_\infty}}\sum_{j=1}^{N}[H_h]_{j+1/2}^2-\frac{1}{2}\sqrt{\frac{\epsilon_0 \epsilon_\infty}{\mu_0}}\sum_{j=1}^{N}[E_h]_{j+1/2}^2 \le 0,
$$
where
\small{
\begin{equation}
\mathcal{E}_h=\int_{\Omega}  \left (\frac{\mu_0}{2} H_h^2 + \frac{\epsilon_0 \epsilon_\infty}{2} E_h^2 +  \frac{\epsilon_0}{2\omega_p^2} J_h^2   + \frac{\epsilon_0\omega_0^2}{2 \omega_p^2}   P_h^2+ \frac{\epsilon_0  a\theta}{4\omega_v^2} \sigma_h^2 + \frac{\epsilon_0  a\theta}{2}  Q_h E_h^2  + \frac{3 \epsilon_0 a (1-\theta)}{4} E_h^4+\frac{\epsilon_0  a\theta}{4}Q_h^2\right) dx
\label{d-ene}
\end{equation}}
is the discrete energy. Moreover, $\mathcal{E}_h \ge 0$ when $\theta \in [0, \frac{3}{4}].$
\end{theorem}
\begin{proof}
Let $\phi=H_h$ in \eqref{eq:sch1}, $\phi=E_h$ in \eqref{eq:sch2} and sum up the two equalities over all elements, we obtain
\begin{equation}
\label{eq:d1}
\int_{\Omega}( \mu_0 H_h  \dd{t}{H_h}+ E_h \dd{t}{D_h}  ) dx +  \sum_{j=1}^{N} \int_{I_j} \dd{x}(E_h H_h)dx +\sum_{j=1}^{N} (\widehat{E_h}[H_h]
+\widetilde{H_h}[E_h])_{j-1/2} =0.
\end{equation}
 Note that with both central and alternating fluxes, \eqref{eq:flux:c} and \eqref{eq:flux:a}, we have
\begin{equation}
\label{eq:fluxidentity}
\widehat{E_h}[H_h]+\widetilde{H_h}[E_h]=[E_h H_h],
\end{equation}
and  with the  upwind flux \eqref{eq:flux:u}, we have
\begin{equation}
\label{eq:fluxidentityd}
\widehat{E_h}[H_h]+\widetilde{H_h}[E_h]=[E_h H_h]+\frac{1}{2}\sqrt{\frac{\mu_0}{\epsilon_0 \epsilon_\infty}}[H_h]^2+\frac{1}{2}\sqrt{\frac{\epsilon_0 \epsilon_\infty}{\mu_0}}[E_h]^2,
\end{equation}
while $\sum_{j=1}^{N} \int_{I_j} \dd{x}(E_h H_h)dx=-\sum_{j=1}^{N}[E_h H_h]_{j-1/2}$,
therefore \eqref{eq:d1} becomes
\begin{equation}
\label{eq:d2}
 \int_{\Omega}( \mu_0 H_h  \dd{t}{H_h}+ E_h \dd{t}{D_h}  ) dx=M(E_h, H_h),
\end{equation}
with
\begin{equation}
\label{Mflux}
M(E_h, H_h):=\left\{ \begin{array}{ll}
         0, & \mbox{for   flux choices \eqref{eq:flux:c} and \eqref{eq:flux:a}};\\
        -\frac{1}{2}\sqrt{\frac{\mu_0}{\epsilon_0 \epsilon_\infty}} \sum_{j=1}^{N} [H_h]_{j-1/2}^2-\frac{1}{2}\sqrt{\frac{\epsilon_0 \epsilon_\infty}{\mu_0}}\sum_{j=1}^{N} [E_h]_{j-1/2}^2, & \mbox{for   flux choice \eqref{eq:flux:u} },\end{array} \right.
        \end{equation}
which is non-positive.
Differentiating \eqref{eq:1d:sch:1} with respect to time, and substituting it into the equation \eqref{eq:d2}, we obtain
\begin{equation}
\label{eq:d3}
 \int_{\Omega}( \mu_0 H_h  \dd{t}{H_h}+ \epsilon_0 E_h  (\epsilon_\infty \dd{t}E_h +a(1-\theta) \dd{t} E_h^3+ \dd{t} P_h+a\theta \dd{t} (Q_h E_h) )   ) dx=M(E_h, H_h),
\end{equation}
which, with \eqref{eq:sch3}, is equivalent to
\begin{equation}
\label{eq:d4}
\frac{d}{dt}  \int_{\Omega}  \left (\frac{\mu_0}{2} H_h^2 + \frac{\epsilon_0 \epsilon_\infty}{2} E_h^2 + \frac{3 \epsilon_0 a (1-\theta)}{4} E_h^4\right) dx = - \int_{\Omega} \epsilon_0 E_h(J_h+a\theta \dd{t} (Q_h E_h)) dx+M(E_h, H_h).
\end{equation}

By \eqref{eq:sch4} and \eqref{eq:sch3},
\begin{eqnarray}
\int_{\Omega} J_h \dd{t}{J_h} dx = -\int_{\Omega}\left(\frac{1}{\tau}J_h +\omega_0^2P_h-\omega_p^2E_h\right)J_hdx \notag\\
=  -\int_{\Omega}\left(\frac{1}{\tau}J_h-\omega_p^2E_h\right)J_hdx -\omega_0^2 \int_{\Omega}  P_h \dd{t}{P_h} dx.
\end{eqnarray}
 This gives the relation
\begin{eqnarray}
\label{eq:d5}
\frac{d}{dt}  \int_{\Omega}  \left (   \frac{1}{2} J_h^2   + \frac{\omega_0^2}{2}   P_h^2 \right )  dx  
= -\int_{\Omega}\frac{1}{\tau}J^2_hdx+\omega_p^2 \int_\Omega E_h J_h dx.
\end{eqnarray}

Similarly, we take $\phi=\sigma_h$ in \eqref{eq:sch6},  sum up over all elements,  use \eqref{eq:sch5}, and obtain
\begin{eqnarray}
\int_{\Omega} \sigma_h \dd{t}{\sigma_h} dx = -\int_{\Omega}\left(\frac{1}{\tau_v}\sigma_h +\omega_v^2Q_h -\omega_v^2E_h^2\right) \sigma_hdx\notag \\
= -\int_{\Omega}\left(\frac{1}{\tau_v}\sigma_h  -\omega_v^2E_h^2\right) \sigma_hdx-\omega_v^2 \int_{\Omega}  Q_h \dd{t}{Q_h} dx,
\end{eqnarray}
which yields
\begin{eqnarray}
\label{eq:d6}
\frac{d}{dt}  \int_{\Omega}  \left (   \frac{1}{2} \sigma_h^2   + \frac{\omega_v^2}{2}  Q_h^2 \right )  dx
=-\int_{\Omega}\frac{1}{\tau_v}\sigma_h^2 dx+w_v^2  \int_{\Omega}E_h^2\sigma_h dx.
\end{eqnarray}
On the other hand,
\begin{equation}
\label{eq:d7}
\int_\Omega E_h \dd{t} (Q_h E_h)dx = \frac{1}{2}\int_\Omega (\dd{t} (Q_h E_h^2)+E_h^2 \dd{t}Q_h) dx = \frac{1}{2}\int_\Omega (\dd{t} (Q_h E_h^2)+E_h^2 \sigma_h) dx.
\end{equation}

Combining the results in  \eqref{eq:d4}, \eqref{eq:d5}, \eqref{eq:d6}, \eqref{eq:d7}, we now have
\begin{eqnarray*}
&&\frac{d}{dt}  \int_{\Omega}  \left (\frac{\mu_0}{2} H_h^2 + \frac{\epsilon_0 \epsilon_\infty}{2} E_h^2 + \frac{3 \epsilon_0 a (1-\theta)}{4} E_h^4\right) dx -M(E_h, H_h) \\
&&=
-\frac{\epsilon_0}{\omega_p^2} \left ( \frac{d}{dt}  \int_{\Omega}  \left (   \frac{1}{2} J_h^2   + \frac{\omega_0^2}{2}   P_h^2 \right )  dx  +\int_{\Omega}\frac{1}{\tau}J^2_hdx   \right )-\epsilon_0  a\theta \int_{\Omega}  E_h  \dd{t} (Q_h E_h) dx\\
&&=-\frac{\epsilon_0}{\omega_p^2} \left ( \frac{d}{dt}  \int_{\Omega}  \left (   \frac{1}{2} J_h^2   + \frac{\omega_0^2}{2}   P_h^2 \right )  dx  +\int_{\Omega}\frac{1}{\tau}J^2_hdx   \right )- \frac{\epsilon_0  a\theta}{2}\frac{d}{dt}  \int_{\Omega} Q_h E_h^2 dx\\
&&-\frac{\epsilon_0  a\theta}{2\omega_v^2} \left ( \frac{d}{dt}  \int_{\Omega}  \left (   \frac{1}{2} \sigma_h^2   + \frac{\omega_v^2}{2}  Q_h^2 \right )  dx + \int_{\Omega}\frac{1}{\tau_v}\sigma_h^2 dx  \right).
\end{eqnarray*}
This becomes
$$
\frac{d}{dt}\mathcal{E}_h=-\frac{\epsilon_0}{\omega_p^2 \tau} \int_\Omega J_h^2 dx-\frac{\epsilon_0  a\theta}{2\omega_v^2 \tau_v}  \int_\Omega \sigma_h^2 dx+M(E_h, H_h),
$$
with the discrete energy $\mathcal{E}_h$  defined in \eqref{d-ene}, which is guaranteed to be  nonnegative as long as $\theta \in [0, \frac{3}{4}]$.
 Since all model parameters are positive, clearly we  have
$
\frac{d}{dt}\mathcal{E}_h \le 0.
$
\end{proof}

We have demonstrated in the theorem above that the DG scheme with appropriate flux choices can successfully maintain the energy stability of the original system on the semi-discrete level.

\section{Semi-discrete Scheme: Error Estimates}
\label{sec:error}

In this section, we will establish the error estimates of the semi-discrete scheme, formulated in Section \ref{sec:semi}, up to a given time $T<\infty$.
The following projections, $\pi_h$ (defined from $L^2(\Omega)$ onto $V_h^k$) and $\pi_h^\pm$ (defined from $H^1(\Omega)$ onto $V_h^k$), will be used in the analysis.
\begin{enumerate}
\item $L^2$ projection $\pi_h$: $\pi_h w \in V_h^k$, such that $\forall j$
\begin{equation}
\int_{I_j} \pi_h w \, v \,dx =\int_{I_j} w \, v \, dx, \qquad \forall v \in P^k(I_j).
\end{equation}
\item Gauss-Radau projection $\pi_h^-$: $\pi_h^- w \in V_h^k$, such that $\forall j$
\begin{equation}
\int_{I_j} \pi_h^- w \, v \,dx =\int_{I_j} w \, v \, dx, \qquad \forall v \in P^{k-1}(I_j),
\label{eq:proj-}
\end{equation}
and $ (\pi_h^- w)^-_{j+\frac12}=w_{j+\frac12}^-$.
\item Gauss-Radau projection $\pi_h^+$: $\pi_h^+ w \in V_h^k$, such that  $\forall j$
\begin{equation}
\int_{I_j} \pi_h^+ w \, v \,dx =\int_{I_j} w \, v \, dx, \qquad \forall v \in P^{k-1}(I_j),
\label{eq:proj+}
\end{equation}
and $(\pi_h^+ w)^+_{j-\frac12}=w_{j-\frac12}^+.$
\end{enumerate}
These projections are commonly used in analyzing DG methods, and the following approximation property and estimate can be easily established \cite{ciarlet2002finite}:
\begin{align}
\|w-\Pi_h w\|^2+h\sum_j ((w-\Pi_hw)_{j+\frac12}^\pm)^2 \leq C_\star h^{2k+2}\|w\|^2_{H^{k+1}},\quad \forall w\in H^{k+1}(\Omega),\label{eq:approx1}
\end{align}
with $\Pi_h=\pi_h, \pi_h^-$ or $\pi_h^+$, and
\begin{align}
\|\Pi_h w\|_\infty\leq
\left\{
\begin{array}{lll}
C_k\|w\|_\infty, &\forall w\in W^{1,\infty}(\Omega), & \text{when}\; \Pi_h=\pi_h^\pm,\\
C_k\|w\|_\infty, &\forall w\in L^{\infty}(\Omega), &\text{when}\; \Pi_h=\pi_h.
\end{array}
\right.
\label{eq:approx3}
\end{align}
Here $w-\Pi_h w$ represents the projection error. In \eqref{eq:approx1}-\eqref{eq:approx3},
$\|\cdot\|$, $\|\cdot\|_\infty$, and $\|\cdot\|_{H^{k+1}}$ stand for the $L^2$-norm, $L^\infty$-norm, and $H^{k+1}$-norm on $\Omega$, respectively. And $\|w\|_{W^{1,\infty}}=(\|w\|_\infty^2+\|\frac{dw}{dx}\|_\infty^2)^{1/2}$ is the $W^{1,\infty}$-norm for $W^{1,\infty}(\Omega)$.
The constant $C_\star$ depends  on $k$ but not on $h$ or $w$.
Throughout the paper, $C_\star$ denotes a generic constant which may depend on $k$ and mesh parameter $\delta$. If we want to emphasize the sole dependence of $k$, this generic constant will be denoted by $C_k$, which usually is computable. $C$ is another generic constant, which is independent of $h$, but may depend on $k$, mesh parameter $\delta$, and some Sobolev norms of the exact solution of \eqref{eq:1d:sys:1} up to time $T$.   There is one more generic constant $C_\text{model}$, which depends on some or all model parameters.
 Each type of generic constants may take different values at different occurrences.
In the analysis, the following inverse inequality will also be needed,
\begin{equation}
\label{eq:inv}
h^2\int_{I_j}(v_x)^2 dx + h \left((v_{j-\frac12}^+)^2+(v_{j+\frac12}^-)^2\right) \leq C_\star \int_{I_j}v^2 dx,\qquad \forall v\in V_h^k.
\end{equation}
A direct consequence of \eqref{eq:approx3} is $\|w-\Pi_h w\|_\infty\leq C_k\|w\|_\infty, \forall w\in W^{1,\infty}(\Omega)$ when $\Pi_h=\pi_h^\pm$ (or $\forall w\in L^\infty(\Omega)$ when $\Pi_h=\pi_h$).

We start with decomposing the error in $E$ into two parts,
$E-E_h=\eta_E-\zeta_E$, with $\eta_E=E-\pi_h^EE$, $\zeta_E=E_h-\pi_h^EE$, where $\pi_h^E$ is a projection operator onto $V_h^k$. We later also use $E-\eta_E=\pi_h^EE$. Similarly, one can define the decomposition of errors in other quantities, namely $u-u_h=u-\pi_h^u-(u_h-\pi_h^u)=\eta_u-\zeta_u$, with $u$ being $H, P, Q, J, \sigma$. In the analysis, $\pi_h^u$ is taken to be $\pi_h$, the $L^2$-projection, for $u=E, H, P, Q, J, \sigma$, except for the following two cases: when the numerical fluxes are alternating, we take
\begin{equation}
(\pi_h^E, \pi_h^H)=\left\{
\begin{array}{ll}
(\pi_h^+, \pi_h^-) & \text{when}\; (\widehat{E_h}, \widetilde{H_h})=(E_h^+, H_h^-),\\
(\pi_h^-, \pi_h^+) & \text{when}\; (\widehat{E_h}, \widetilde{H_h})=(E_h^-, H_h^+),
\end{array}
\right.
\end{equation}
while with the upwind flux, we use
\begin{eqnarray}
\pi_h^E (E, H)&=\frac{1}{2}\pi_h^+\big(E+\sqrt{\frac{\mu_0}{\epsilon_0\epsilon_\infty}}H\big)+\frac{1}{2}\pi_h^-\big(E-\sqrt{\frac{\mu_0}{\epsilon_0\epsilon_\infty}}H\big),\\
\pi_h^H  (E, H)&=\frac{1}{2}\pi_h^+\big(H+\sqrt{\frac{\epsilon_0\epsilon_\infty}{\mu_0}}E\big)+\frac{1}{2}\pi_h^-\big(H-\sqrt{\frac{\epsilon_0\epsilon_\infty}{\mu_0}}E\big).
\end{eqnarray}
 See \cite{cheng2017L2} (such as  Lemma 2.4) for the properties of such vector-form projection operators. 
For the a priori error estimate in next theorem, we assume the following regularity for the exact solutions,
\begin{equation}
E, H, P, Q, J, \sigma\in W^{1,\infty}([0, T], H^{k+1}(\Omega)),
\label{eq:reg1}
\end{equation}
and
\begin{equation}
 E\in W^{1,\infty}([0, T], W^{1,\infty}(\Omega)), \quad Q\in W^{1,\infty}([0, T], L^{\infty}(\Omega)),
 \label{eq:reg2}
 \end{equation}
where the former  is standard for error analysis of linear models, and  the latter are needed to  treat nonlinearity.

\begin{theorem}[Error estimates of semi-discrete scheme]
\label{thm:err}
Assuming the periodic boundary condition and the exact solutions being as regular as \eqref{eq:reg1}-\eqref{eq:reg2}, under the conditions that $\theta\in [0, \frac{1}{4})$ and the strength of nonlinearity is sufficiently small,
the following error estimates hold for the semi-discrete DG scheme \eqref{eq:1d:sch}-\eqref{eq:1d:sch:1} with flux choices \eqref{eq:flux:c}, \eqref{eq:flux:a}, or \eqref{eq:flux:u}
\begin{equation}
\label{eq:err:20}
\|u-u_h\|\leq CC_\text{model} h^{r}, \qquad u=E, H, P, Q, J, \sigma,
\end{equation}
where
\begin{equation}
\label{eq:par:r}
r=\left\{
\begin{array}{ll}
k&\text{for central flux} \; \eqref{eq:flux:c}\\
k+1&\text{ for alternating flux } \; \eqref{eq:flux:a}\; \text{and upwind flux}\; \eqref{eq:flux:u}.
\end{array}
\right.
\end{equation}
\end{theorem}

\begin{proof}
With the numerical fluxes being consistent, the proposed semi-discrete scheme is consistent. That is, \eqref{eq:1d:sch} holds if the numerical solutions are replaced by the exact ones, while the test functions are still taken from $V_h^k$. From this, one can get
the error equations,

\begin{subequations}
\label{eq:1d:e}
\begin{align}
&\mu_0\int_\Omega\dd{t}{\zeta_H}\phi dx +\sum_{j=1}^{N} \int_{I_j} \zeta_E\dd{x}\phi dx+\sum_{j=1}^{N} (\widehat{\zeta_E}[\phi])_{j-1/2}\notag\\
&\hspace{0.5in} =\mu_0\int_\Omega\dd{t}{\eta_H}\phi dx +\sum_{j=1}^{N} \int_{I_j} \eta_E\dd{x}\phi dx+\sum_{j=1}^{N} (\widehat{\eta_E}[\phi])_{j-1/2} ,\quad \forall \phi\in V_h^k,\label{eq:e1}\\
&\int_\Omega\dd{t}{(D_h-D)} \phi dx +\sum_{j=1}^{N} \int_{I_j} \zeta_H\dd{x} \phi dx+\sum_{j=1}^{N}(\widetilde{\zeta_H}[\phi])_{j-1/2}\notag\\
&\hspace{0.5in}=\sum_{j=1}^{N} \int_{I_j} \eta_H\dd{x} \phi dx+\sum_{j=1}^{N}(\widetilde{\eta_H}[\phi])_{j-1/2},\; \forall \phi\in V_h^k, \label{eq:e2}\\
& \dd{t}{\zeta_P} -\zeta_J=\dd{t}{\eta_P} - \eta_J,  \qquad\dd{t}{\zeta_Q} -\zeta_\sigma=\dd{t}{\eta_Q} -\eta_\sigma, \label{eq:e3}\\
& \dd{t}{\zeta_J}+\frac{1}{\tau}\zeta_J +\omega_0^2\zeta_P-\omega_p^2\zeta_E
= \dd{t}{\eta_J}+\frac{1}{\tau}\eta_J +\omega_0^2\eta_P-\omega_p^2\eta_E, \label{eq:e4}\\
&\int_\Omega\dd{t}{\zeta_\sigma} \phi dx+\int_\Omega\left(\frac{1}{\tau_v}\zeta_\sigma +\omega_v^2\zeta_Q -\omega_v^2(E_h^2-E^2)\right)\phi dx\notag\\
&\hspace{0.5in}=
\int_\Omega\dd{t}{\eta_\sigma} \phi dx+\int_\Omega\left(\frac{1}{\tau_v}\eta_\sigma +\omega_v^2\eta_Q\right)\phi dx
,\; \forall \phi\in V_h^k, \label{eq:e6}
\end{align}
\end{subequations}
coupled with
\begin{equation}
\label{eq:e7}
\int_\Omega (D_h-D)  \phi dx= \int_\Omega \epsilon_0\Big(\epsilon_\infty (\zeta_E-\eta_E) +a(1-\theta)(E_h^3-E^3)+\zeta_P-\eta_P+a\theta (Q_h E_h-QE)\Big)\phi dx.\quad\forall\phi\in V_h^k.
\end{equation}

Now we take $\phi=\zeta_H$ in \eqref{eq:e1}, $\phi=\zeta_E$ in \eqref{eq:e2},
$\phi=\zeta_\sigma$ in \eqref{eq:e6}. We then differentiate \eqref{eq:e7} in time $t$, and take $\phi=\zeta_E$. Following similar steps as in the proof of Theorem \ref{thm:semi}, we get
\begin{align}
\label{eq:e8}
\frac{d}{dt}\int_{\Omega} & \left (\frac{\mu_0}{2} \zeta_H^2 + \frac{\epsilon_0 \epsilon_\infty}{2} \zeta_E^2 +  \frac{\epsilon_0}{2\omega_p^2} \zeta_J^2   + \frac{\epsilon_0\omega_0^2}{2 \omega_p^2}   \zeta_P^2+ \frac{\epsilon_0  a\theta}{4\omega_v^2} \zeta_\sigma^2 + \frac{\epsilon_0  a\theta}{2}  \zeta_Q \zeta_E^2  + \frac{3 \epsilon_0 a (1-\theta)}{4} \zeta_E^4+\frac{\epsilon_0  a\theta}{4}\zeta_Q^2\right) dx\notag\\
&+\epsilon_0a(1-\theta) \frac{d}{dt}\int_{\Omega}\left( 2(E-\eta_E)\zeta_E^3+\frac{3}{2}(E-\eta_E)^2\zeta_E^2\right)dx
+\frac{\epsilon_0a\theta}{2} \frac{d}{dt}\int_{\Omega}(Q-\eta_Q)\zeta_E^2dx\\
&+\frac{\epsilon_0}{\omega_p^2 \tau} \int_\Omega \zeta_J^2 dx+\frac{\epsilon_0  a\theta}{2\omega_v^2 \tau_v}  \int_\Omega \zeta_\sigma^2 dx-
M(\zeta_E, \zeta_H)
=\sum_{j=1}^4\Lambda_j.\notag
\end{align}
Here the non-positive term $M(\cdot, \cdot)$ is defined in \eqref{Mflux}.
  The four terms on the right are
\begin{equation}
\Lambda_1=\mu_0\int_\Omega(\dd{t}\eta_H) \zeta_Hdx,
\end{equation}
\begin{equation}
\Lambda_2=\sum_{j=1}^{N}\int_{I_j}\left(\eta_E(\dd{x}\zeta_H)+\eta_H(\dd{x}\zeta_E)\right)dx
+\sum_{j=1}^N(\widehat{\eta_E}[\zeta_H]+\widetilde{\eta_H}[\zeta_E])_{j-1/2},
\end{equation}
\begin{align}
\Lambda_3=&\frac{\epsilon_0}{\omega_p^2}\int_\Omega \left(\dd{t}\eta_J+\frac{1}{\tau}\eta_J+\omega_0^2\eta_P-\omega_p^2\eta_E\right)\zeta_J dx
+\frac{\epsilon_0\omega_0^2}{\omega_p^2}\int_\Omega \left(\dd{t}\eta_P-\eta_J\right)\zeta_P dx\notag\\
&+\frac{\epsilon_0 a\theta}{2\omega_v^2}\int_\Omega \left(\dd{t}\eta_\sigma+\frac{1}{\tau_v}\eta_\sigma\right)\zeta_\sigma dx
+\frac{\epsilon_0 a \theta}{2} \int_\Omega \left(\eta_Q-2E\eta_E+\eta_E^2\right)\zeta_\sigma dx,
\end{align}
\begin{align}
\Lambda_4=&\epsilon_0\epsilon_\infty\int_\Omega(\dd{t}\eta_E)\zeta_Edx
+\epsilon_0\int_\Omega \eta_J\zeta_Edx
+\epsilon_0 a (1-\theta)\int_\Omega \dd{t}(\eta_E^3-3E\eta_E^2+3E^2\eta_E)\zeta_E dx\\
&+\epsilon_0 a \theta\int_\Omega \dd{t}(Q\eta_E+(E-\eta_E)\eta_Q)\zeta_Edx
- \epsilon_0 a \theta\int_\Omega (E-\eta_E)(\dd{t}\eta_Q -\eta_\sigma)\zeta_Edx\notag\\
&-\epsilon_0 a (1-\theta)\int_\Omega \left(\dd{t}(E-\eta_E) \zeta_E^3
          +\frac{3}{2}\dd{t}(E-\eta_E)^2 \zeta_E^2 \right)dx\notag\\
&-\frac{\epsilon_0 a\theta}{2} \int_\Omega \dd{t}(Q-\eta_Q)\zeta_E^2dx+\frac{\epsilon_0 a \theta}{2} \int_\Omega (\dd{t}\eta_Q-\eta_\sigma)\zeta_Qdx-\epsilon_0 a\theta\int_\Omega \dd{t}(E-\eta_E)\zeta_Q\zeta_Edx.\notag
\end{align}

Next we will take two steps to estimate the left and the right hand side of \eqref{eq:e8}, respectively.

\medskip
\noindent {\sf Step 1:}
Compared with the discrete energy in the stability analysis, the terms in the second row of the left hand side of \eqref{eq:e8} are new, and they arise from the discretizations of nonlinear terms. With arbitrarily chosen constant parameters $\rho_\text{err}\in(0,1)$ and $\kappa_\text{err}\in(0,1)$, we have

 \begin{align}
\label{eq:e9}
\int_{\Omega} & \left (\frac{\mu_0}{2} \zeta_H^2 + \frac{\epsilon_0 \epsilon_\infty}{2} \zeta_E^2 +  \frac{\epsilon_0}{2\omega_p^2} \zeta_J^2   + \frac{\epsilon_0\omega_0^2}{2 \omega_p^2}   \zeta_P^2+ \frac{\epsilon_0  a\theta}{4\omega_v^2} \zeta_\sigma^2 + \frac{\epsilon_0  a\theta}{2}  \zeta_Q \zeta_E^2  + \frac{3 \epsilon_0 a (1-\theta)}{4} \zeta_E^4+\frac{\epsilon_0  a\theta}{4}\zeta_Q^2\right) dx\notag\\
&+\epsilon_0a(1-\theta) \int_{\Omega}\left( 2(E-\eta_E)\zeta_E^3+\frac{3}{2}(E-\eta_E)^2\zeta_E^2\right)dx
+\frac{\epsilon_0a\theta}{2} \int_{\Omega}(Q-\eta_Q)\zeta_E^2dx\notag\\
&=\int_{\Omega} \left (\frac{\mu_0}{2} \zeta_H^2 +  \frac{\epsilon_0}{2\omega_p^2} \zeta_J^2   + \frac{\epsilon_0\omega_0^2}{2 \omega_p^2}   \zeta_P^2+ \frac{\epsilon_0  a\theta}{4\omega_v^2} \zeta_\sigma^2 + \frac{ \epsilon_0 a (1-\theta)\rho_\text{err}}{12} \zeta_E^4+\frac{\epsilon_0  a\theta\rho_\text{err}}{4}\zeta_Q^2 + \frac{\epsilon_0 \epsilon_\infty\kappa_\text{err}}{2} \zeta_E^2
\right) dx\notag\\
&+\frac{3\epsilon_0a(1-\theta)}{2} \int_{\Omega}(E-\eta_E+\frac{2}{3}\zeta_E)^2\zeta_E^2dx+\frac{\epsilon_0}{2} \int_{\Omega}\left(\epsilon_\infty(1-\kappa_\text{err})+a\theta (Q-\eta_Q)\right)\zeta_E^2dx\notag\\
&+
\int_\Omega \left(\frac{ \epsilon_0 a (1-\theta)(1-\rho_\text{err})}{12} \zeta_E^4 + \frac{\epsilon_0  a\theta}{2}  \zeta_Q \zeta_E^2  +\frac{\epsilon_0  a\theta(1-\rho_\text{err})}{4}\zeta_Q^2\right) dx\notag\\
&\geq
\int_{\Omega} \left (\frac{\mu_0}{2} \zeta_H^2 +  \frac{\epsilon_0}{2\omega_p^2} \zeta_J^2   + \frac{\epsilon_0\omega_0^2}{2 \omega_p^2}   \zeta_P^2+ \frac{\epsilon_0  a\theta}{4\omega_v^2} \zeta_\sigma^2 + \frac{ \epsilon_0 a (1-\theta)\rho_\text{err}}{12} \zeta_E^4+\frac{\epsilon_0  a\theta\rho_\text{err}}{4}\zeta_Q^2 + \frac{\epsilon_0 \epsilon_\infty\kappa_\text{err}}{2} \zeta_E^2
\right) dx\notag\\
&:={\mathcal E}_h^\text{(mod)},
\end{align}
under the conditions:
\begin{align}
&\text{\bf Condition 1}:  \qquad \theta \in[0, \frac{1}{3(1-\rho_\text{err})^{-2}+1}]
\label{eq:cond1}\\
&\text{\bf Condition 2}:  \qquad
 a\theta C_k\|Q\|_\infty\leq \epsilon_\infty(1-\kappa_\text{err})
\label{eq:cond2}
\end{align}
with a computable constant $C_k$ from \eqref{eq:approx3}.
Indeed, under  Condition 1,
$$\frac{ \epsilon_0 a (1-\theta)(1-\rho_\text{err})}{12} \zeta_E^4 + \frac{\epsilon_0  a\theta}{2}  \zeta_Q \zeta_E^2  +\frac{\epsilon_0  a\theta(1-\rho_\text{err})}{4}\zeta_Q^2\geq 0$$
holds, while Condition 2  is to ensure
$$\epsilon_\infty(1-\kappa_\text{err})+a\theta (Q-\eta_Q)\geq \epsilon_\infty(1-\kappa_\text{err})-a\theta\|\pi_hQ\|_\infty
\geq
\epsilon_\infty(1-\kappa_\text{err})-a\theta C_k\|Q\|_\infty\geq 0.$$

\medskip
\noindent{\sf Step 2:} Next we will  estimate $\Lambda_j$, $j=1, \cdots, 4$. Cauchy-Schwartz inequality, Young's inequality, as well as approximation result and estimate in \eqref{eq:approx1}-\eqref{eq:approx3} will be used repeatedly. For $\Lambda_1$,
\begin{equation}
\label{eq:lam1}
|\Lambda_1|\leq \mu_0\|\dd{t}\eta_H\|^2+\frac{\mu_0}{4}\|\zeta_H\|^2\leq CC_\text{model} h^{2k+2}+\frac{\mu_0}{4}\|\zeta_H\|^2.
\end{equation}
We here have used $\dd{t}\eta_H=\dd{t}H-\pi_h^H\dd{t}H$. As for $\Lambda_2$, with the choice of the projection operators $\pi_h^E$ and $\pi_h^H$, one has
\begin{equation} \Lambda_2=0 \label{eq:lam2_0}\end{equation}
 for alternating and upwind flux; while for central flux, we have
\begin{equation}
\label{eq:lam2}
|\Lambda_2|=
\sum_{j=1}^{N}(\widehat{\eta_E}[\zeta_H]+\widetilde{\eta_H}[\zeta_E])_{j-1/2}
\leq CC_\text{model}C(\kappa_\text{err})h^{2k}+  \frac{\mu_0}{4} \|\zeta_H\|^2 + \frac{\epsilon_0 \epsilon_\infty\kappa_\text{err}}{8} \|\zeta_E\|^2.
\end{equation}

For $\Lambda_3$,
\begin{align}
|\Lambda_3|\leq & \frac{\epsilon_0}{2\omega_p^2}(\|\dd{t}\eta_J+\omega_0^2\eta_P-\omega_p^2\eta_E\|^2+   \|\zeta_J\|^2)
+\frac{\epsilon_0}{\omega_p^2\tau}(\frac{1}{4}\|\eta_J\|^2+\|\zeta_J\|^2)\notag\\
&
+\frac{\epsilon_0\omega_0^2}{2\omega_p^2}(\|\dd{t}\eta_P-\eta_J\|^2+\|\zeta_P\|^2)
+\frac{\epsilon_0 a\theta}{2\omega_v^2\tau_v}(\frac{1}{4}\|\eta_\sigma\|^2+\|\zeta_\sigma\|^2)\notag\\
&+\frac{\epsilon_0 a\theta}{4\omega_v^2}(\|\dd{t}\eta_\sigma+\omega_v^2 (\eta_Q-2E\eta_E+\eta_E^2)\|^2+\|\zeta_\sigma\|^2).
\end{align}
Using the approximation property and estimate in \eqref{eq:approx1}-\eqref{eq:approx3}, as well as the boundedness of $E$,  we have
$$\|E\eta_E\|\leq \|E\|_\infty\|\eta_E\|,\qquad \|\eta_E^2\| \leq \|\eta_E\|_\infty\|\eta_E\|\leq C_k \|E\|_\infty\|\eta_E\|, $$
  hence get
\begin{align}
\label{eq:lam3}
|\Lambda_3|\leq & CC_\text{model} h^{2k+2}
+
\frac{\epsilon_0}{2\omega_p^2} \|\zeta_J\|^2
+\frac{\epsilon_0}{\omega_p^2\tau}\|\zeta_J\|^2
+\frac{\epsilon_0\omega_0^2}{2\omega_p^2}\|\zeta_P\|^2
+\frac{\epsilon_0 a\theta}{2\omega_v^2\tau_v}\|\zeta_\sigma\|^2+\frac{\epsilon_0 a\theta}{4\omega_v^2}\|\zeta_\sigma\|^2.
\end{align}

Term  $\Lambda_4$ is relatively subtle, and we will proceed as follows.
\begin{align}
|\Lambda_4|\leq &\epsilon_0\epsilon_\infty (\frac{1}{4\alpha_1}\|\dd{t}\eta_E\|^2 +\alpha_1\|\zeta_E\|^2)
+\epsilon_0 (\frac{1}{4\alpha_2}\|\eta_J\|^2 +\alpha_2\|\zeta_E\|^2) \notag\\
&+\epsilon_0 a (1-\theta) (\frac{1}{4\alpha_3}\|\dd{t}(\eta_E^3-3E\eta_E^2+3E^2\eta_E)\|^2+\alpha_3\|\zeta_E\|^2)\notag\\
&+\epsilon_0 a \theta (\frac{1}{4\alpha_4}\|\dd{t}(Q\eta_E+(E-\eta_E)\eta_Q)\|^2+\alpha_4\|\zeta_E\|^2)\notag\\
&
+\epsilon_0 a \theta(\frac{1}{4\alpha_5}\|E-\eta_E \|_\infty^2\|\dd{t}\eta_Q -\eta_\sigma\|^2+\alpha_5\|\zeta_E\|^2)\notag\\
&+\epsilon_0 a (1-\theta) (\frac{1}{4\alpha_6}\|\zeta_E^2\|^2+\alpha_6\|\dd{t}(E-\eta_E)\|_\infty^2\|\zeta_E\|^2)+ \frac{3\epsilon_0 a (1-\theta)}{2}\|\dd{t}(E-\eta_E)^2\|_\infty\|\zeta_E\|^2\notag\\
&+\frac{\epsilon_0 a\theta}{2} \|\dd{t}(Q-\eta_Q)\|_\infty\|\zeta_E\|^2+\frac{\epsilon_0 a \theta}{2} (\frac{1}{4\alpha_7}\|\dd{t}\eta_Q-\eta_\sigma\|^2+\alpha_7\|\zeta_Q\|^2)\notag\\
&+\frac{\epsilon_0 a \theta}{2} (\frac{1}{\alpha_8}\|\dd{t}(E-\eta_E)\|_\infty^2 \|\zeta_E\|^2+\alpha_8\|\zeta_Q\|^2).
\end{align}

The constant parameters $\alpha_j$, $j=1, \cdots, 7$, are chosen so that
\begin{equation}
\label{eq:err:12}
\epsilon_0\epsilon_\infty\alpha_1=\epsilon_0\alpha_2=\epsilon_0a(1-\theta)\alpha_3=\epsilon_0a\theta\alpha_4=\epsilon_0a\theta\alpha_5=\frac{1}{5}\frac{\epsilon_0\epsilon_\infty\kappa_\text{err}}{8},
\end{equation}
\begin{equation}
\frac{\epsilon_0 a (1-\theta)}{4\alpha_6}=\frac{ \epsilon_0 a (1-\theta)\rho_\text{err}}{12},\quad \frac{\epsilon_0 a\theta}{2}\alpha_7=\frac{\epsilon_0 a\theta}{2}\alpha_8=\frac{1}{2}\frac{\epsilon_0 a\theta\rho_\text{err}}{4}.
\end{equation}
We then further restrict the strength of the nonlinearity such that
 \begin{align}
 \epsilon_0a\Big((\alpha_6(1-\theta)&+\frac{\theta}{2\alpha_8}) \|\dd{t}(E-\eta_E)\|_\infty^2+ \frac{3(1-\theta)}{2}\|\dd{t}(E-\eta_E)^2\|_\infty
 +\frac{\theta}{2} \|\dd{t}(Q-\eta_Q)\|_\infty\Big) \|\zeta_E\|^2\notag\\
 &\leq \frac{\epsilon_0\epsilon_\infty\kappa_\text{err}}{4}\|\zeta_E\|^2,
\end{align}
and this, with the estimate \eqref{eq:approx3}, can be ensured under the condition
 \begin{equation}
  \label{eq:cond22}
  \text{\bf Condition 3}: \quad a\left(\frac{3-\theta}{\rho_\text{err}}C_k^2\|\dd{t}E\|^2_\infty+3(1-\theta) C_k^2\|\dd{t}E\|_\infty\|E\|_\infty+\frac{\theta}{2}C_k\|\dd{t}Q\|_\infty
 \right) \leq \frac{\epsilon_\infty\kappa_\text{err}}{4}.
\end{equation}
Using \eqref{eq:err:12}-\eqref{eq:cond2} and applying \eqref{eq:approx1}-\eqref{eq:approx3}, we are able to bound $\Lambda_4$
\begin{align}
\label{eq:lam4}
|\Lambda_4|\leq &CC_\text{model} h^{2k+2} + \frac{3\epsilon_0\epsilon_\infty\kappa_\text{err}}{8}\|\zeta_E\|^2
+\frac{ \epsilon_0 a (1-\theta)\rho_\text{err}}{12}\|\zeta_E^2\|^2+\frac{\epsilon_0 a\theta\rho_\text{err}}{4}\|\zeta_Q\|^2.
\end{align}

Now we can combine  \eqref{eq:e8}, \eqref{eq:e9}, \eqref{eq:lam1}, \eqref{eq:lam2_0}, \eqref{eq:lam2}, \eqref{eq:lam3}, \eqref{eq:lam4},  and reach
\begin{align}
\frac{d}{dt}{\mathcal E}_h^{\text{(mod)}}\leq
\frac{d}{dt}{\mathcal E}_h^{\text{(mod)}}-\frac{1}{2} M(\zeta_E, \zeta_H)
\leq {\mathcal E}_h^{\text{(mod)}}+CC_\text{model}C(\kappa_\text{err}, \rho_\text{err}) h^{2r},
\end{align}
where $r$ is specified in \eqref{eq:par:r}.
Finally, we apply Gronwall inequality and the estimation of projection errors in \eqref{eq:approx1}, and conclude that
\begin{equation}
\|u-u_h\|\leq \|\eta_u\|+\|\zeta_u\|\leq CC_\text{model}C(\kappa_\text{err}, \rho_\text{err}) h^{r}, \qquad u=E, H, P, Q, J, \sigma,
\end{equation}
under the Conditions 1-3. Note that $\rho_\text{err}$ is arbitrary in $(0,1)$, Condition 1 essentially implies $\theta\in[0, \frac{1}{4})$, while  Conditions 2-3 require the smallness of the strength of the nonlinearity. In our final result \eqref{eq:err:20}, we no longer carry the two parameters $\kappa_\text{err}$ and $\rho_\text{err}$.
\end{proof}

\section{Fully Discrete Scheme and Energy Analysis}
\label{sec:fully}

In this section, we focus on fully discrete schemes for the nonlinear PDE-ODE system \eqref{eq:1d:sys}. A particular focus will be on designing  temporal discretizations, with which the fully discrete methods have {\em provable} energy stability. This turns out to be a nontrivial task for the nonlinear model examined in this work.
Common choices, such as  the second order leap-frog or implicit trapezoidal method, may not yield provable stability results as for the linear models. The main difficulties arise from the nonlinear Kerr and Raman terms. What we will develop in this section can be understood as novel modifications of leap-frog or implicit trapezoidal method, in the presence of these nonlinear effects. The proposed temporal discretizations are still of formal second order accuracy. We will establish the energy stability for the resulting fully discrete methods. The time discretizations developed here can be used not only in conjunction with DG spatial discretizations, but also with other type discretizations, and this will be addressed in our future work.

We design two second-order time schemes, both implicit in the ODE parts. The first scheme uses the leap-frog staggered in time for the PDE parts.  Given $u^n_h(\cdot)\in V_h^k$ at $t^n$, with $u=H, D, E, P, Q, J, \sigma$, we look for  $u^{n+1}_h(\cdot)\in V_h^k$ at $t^{n+1}=t^n+\Delta t$, with $u=H, D, E, P, Q, J, \sigma$, satisfying $\forall j$
\begin{subequations}
\label{eq:1d:schst}
\begin{align}
\mu_0&\int_{I_j} \frac{H_h^{n+1/2}-H_h^n}{\Delta t/2}\phi dx +\int_{I_j} E_h^n \dd{x}\phi dx- (\widehat{E_h^n}\phi^-)_{j+1/2}
+ (\widehat{E_h^n}\phi^+)_{j-1/2}=0,\quad \forall \phi\in V_h^k, \label{eq:schst1}\\
&\int_{I_j} \frac{D_h^{n+1}-D_h^n}{\Delta t} \phi dx +\int_{I_j} H_h^{n+1/2}\dd{x}\phi dx- (\widetilde{H_h^{n+1/2}}\phi^-)_{j+1/2}
+ (\widetilde{H_h^{n+1/2}}\phi^+)_{j-1/2}=0,\quad \forall \phi\in V_h^k, \label{eq:schst2}\\
&\int_{I_j} D_h^{n+1}  \phi dx= \int_{I_j} \epsilon_0\left(\epsilon_\infty E_h^{n+1} +a(1-\theta)Y_h^{n+1}+P_h^{n+1}+a\theta Q_h^{n+1} E_h^{n+1}\right)\phi dx,\quad\forall\phi\in V_h^k,  \label{eq:schst3}\\
&\int_{I_j} Y_h^{n+1}  \phi dx= \int_{I_j}  \left(Y_h^n+\frac{3}{2}((E_h^{n+1})^2+(E_h^n)^2)(E_h^{n+1}-E_h^n)\right) \phi dx,\quad\forall\phi\in V_h^k,  \label{eq:schst4}\\
&  \frac{Q_h^{n+1}-Q_h^n}{\Delta t}=\frac{1}{2}(\sigma_h^n+\sigma_h^{n+1}), \label{eq:schst5}\\
&\int_{I_j}  \frac{\sigma_h^{n+1}-\sigma_h^n}{\Delta t} \phi dx= -\frac{1}{2} \int_{I_j}\left(\frac{1}{\tau_v}(\sigma_h^n +\sigma_h^{n+1})+\omega_v^2 (Q_h^n   + Q_h^{n+1}) -2\omega_v^2E^n_h E^{n+1}_h\right)\phi dx,\quad \forall \phi\in V_h^k, \label{eq:schst6}\\
&  \frac{P_h^{n+1}-P_h^n}{\Delta t}=\frac{1}{2}(J_h^n+J_h^{n+1}),  \label{eq:schst7}\\
&  \frac{J_h^{n+1}-J_h^n}{\Delta t}= - \frac{1}{2}\left(\frac{1}{\tau}(J_h^n+J_h^{n+1}) +\omega_0^2(P_h^n+P_h^{n+1})-\omega_p^2(E_h^n+E_h^{n+1})\right), \label{eq:schst8} \\
\mu_0&\int_{I_j} \frac{H_h^{n+1}-H_h^{n+1/2}}{\Delta t/2}\phi dx +\int_{I_j} E_h^{n+1} \dd{x} \phi dx- (\widehat{\widehat{E_h^{n+1}}}\phi^-)_{j+1/2}
+ (\widehat{\widehat{E_h^{n+1}}}\phi^+)_{j-1/2}=0,\quad \forall \phi\in V_h^k. \label{eq:schst9}
\end{align}
\end{subequations}
 The flux terms in the scheme have no ambiguity for the central and alternating fluxes \eqref{eq:flux:c}-\eqref{eq:flux:a} with $\widehat{\widehat{E^n_h}}=\widehat{E_h^n}$, and their expressions are omitted for brevity. For the upwind flux \eqref{eq:flux:u}, the flux terms should be defined as
\begin{subequations}
\label{eq:flux:uf0}
\begin{align}
&\widehat{E_h^n} = \{E_h^n\}+\frac{1}{2}\sqrt{\frac{\mu_0}{\epsilon_0 \epsilon_\infty}}\left[H_h^{n+1/2}\right],\quad \widehat{\widehat{E_h^{n}}} = \{E_h^{n}\}+\frac{1}{2}\sqrt{\frac{\mu_0}{\epsilon_0 \epsilon_\infty}}\left[H_h^{n-1/2}\right], \\
&\widetilde{H_h^{n+1/2}} = \{H_h^{n+1/2}\}+\frac{1}{2}\sqrt{\frac{\epsilon_0 \epsilon_\infty}{\mu_0}}\left [\frac{E_h^n+E_h^{n+1}}{2} \right]
\end{align}
\end{subequations}
as in the standard leap-frog formulations. Notice that   the scheme is implicit for the upwind flux, but for the alternating and central fluxes, the implicit part is only on the ODEs which can be locally solved in each element. In practice, we
use a Newton's method to obtain $E_h^{n+1}, Q_h^{n+1}, \sigma_h^{n+1}, P_h^{n+1}, J_h^{n+1}$ from  \eqref{eq:schst2}-\eqref{eq:schst8}.
The main novelty of the formulation is the introduction of $Y_h^n$ in \eqref{eq:schst4} as an auxiliary variable  to approximate $Y=E^3.$
 This is motivated by the fact that $dY=3E^2 dE$ and is defined to achieve energy stability of the fully discrete scheme as shown in Theorem \ref{thm:full}. One does not need to store $Y_h^n$, instead only
its temporal difference is needed to be substituted into \eqref{eq:schst2}. Another change in the scheme for stability consideration is the discretization of $E^2$ term in \eqref{eq:schst6} as $E_h^n E_h^{n+1}$. This is motivated by theoretical analysis as shown in  the proof of Theorem \ref{thm:full}.

Similarly, our second formulation, which is a fully implicit scheme writes
\begin{subequations}
\label{eq:1d:schim}
\begin{align}
\mu_0&\int_{I_j} \frac{H_h^{n+1}-H_h^n}{\Delta t}\phi dx +\int_{I_j} \frac{E_h^{n+1}+E_h^n}{2} \dd{x}\phi dx\notag\\
&\hspace{0.6in}- (\widehat{\frac{E_h^{n+1}+E_h^n}{2}}\phi^-)_{j+1/2}
+ (\widehat{\frac{E_h^{n+1}+E_h^n}{2}}\phi^+)_{j-1/2}=0,\quad \forall \phi\in V_h^k, \label{eq:schim1}\\
&\int_{I_j} \frac{D_h^{n+1}-D_h^n}{\Delta t} \phi dx +\int_{I_j} \frac{H_h^{n+1}+H_h^n}{2}\dd{x} \phi dx\notag\\
&\hspace{0.6in}- (\widetilde{\frac{H_h^{n+1}+H_h^n}{2}}\phi^-)_{j+1/2}
+ (\widetilde{\frac{H_h^{n+1}+H_h^n}{2}}\phi^+)_{j-1/2}=0,\quad \forall \phi\in V_h^k, \label{eq:schim2}\\
&\int_{I_j} D_h^{n+1}  \phi dx= \int_{I_j} \epsilon_0\left(\epsilon_\infty E_h^{n+1} +a(1-\theta)Y_h^{n+1}+P_h^{n+1}+a\theta Q_h^{n+1} E_h^{n+1}\right)\phi dx,\quad\forall\phi\in V_h^k,  \label{eq:schim3}\\
&\int_{I_j} Y_h^{n+1}  \phi dx= \int_{I_j}  \left(Y_h^n+\frac{3}{2}((E_h^{n+1})^2+(E_h^n)^2)(E_h^{n+1}-E_h^n)\right) \phi dx,\quad\forall\phi\in V_h^k,  \label{eq:schim4}\\
&  \frac{Q_h^{n+1}-Q_h^n}{\Delta t}=\frac{1}{2}(\sigma_h^n+\sigma_h^{n+1}), \label{eq:schim5}\\
&\int_{I_j}  \frac{\sigma_h^{n+1}-\sigma_h^n}{\Delta t} \phi dx= -\frac{1}{2} \int_{I_j}\left(\frac{1}{\tau_v}(\sigma_h^n +\sigma_h^{n+1})+\omega_v^2 (Q_h^n   + Q_h^{n+1}) -2\omega_v^2E^n_h E^{n+1}_h\right)\phi dx,\quad \forall \phi\in V_h^k,\label{eq:schim6}\\
&  \frac{P_h^{n+1}-P_h^n}{\Delta t}=\frac{1}{2}(J_h^n+J_h^{n+1}),  \label{eq:schim7}\\
&  \frac{J_h^{n+1}-J_h^n}{\Delta t}= - \frac{1}{2}\left(\frac{1}{\tau}(J_h^n+J_h^{n+1}) +\omega_0^2(P_h^n+P_h^{n+1})-\omega_p^2(E_h^n+E_h^{n+1})\right).  \label{eq:schim8}
\end{align}
\end{subequations}
The scheme is of second order accuracy in time. The flux terms are defined according to their semi-discrete counterparts. For example, with
 the upwind flux \eqref{eq:flux:u}, the flux terms are
\begin{subequations}
\label{eq:flux:uf}
\begin{align}
&\widehat{\frac{E_h^{n+1}+E_h^n}{2}}=\left\{\frac{E_h^{n+1}+E_h^n}{2}\right \}+\frac{1}{2}\sqrt{\frac{\mu_0}{\epsilon_0 \epsilon_\infty}}\left [\frac{H_h^{n+1}+H_h^n}{2} \right], \\
&\widetilde{\frac{H_h^{n+1}+H_h^n}{2}}= \left\{\frac{H_h^{n+1}+H_h^n}{2}\right \}+\frac{1}{2}\sqrt{\frac{\epsilon_0 \epsilon_\infty}{\mu_0}}\left [\frac{E_h^{n+1}+E_h^n}{2} \right].
\end{align}
\end{subequations}

\begin{theorem} [Fully discrete stability]
\label{thm:full}
Assuming the periodic boundary condition, then the fully discrete scheme \eqref{eq:1d:schst} with central and alternating fluxes,  \eqref{eq:flux:c} and \eqref{eq:flux:a}, satisfies
\begin{equation}
\label{eq:full:stab1}
 \mathcal{E}_h^{n+1}- \mathcal{E}_h^{n}=-\frac{\epsilon_0 \Delta t}{4 \omega_p^2 \tau} \int_\Omega (J_h^{n+1} +J_h^{n})^2 dx-\frac{\epsilon_0  a\theta \Delta t}{8\omega_v^2 \tau_v}  \int_\Omega  (\sigma_h^{n+1} +\sigma_h^{n})^2  dx \le 0,
\end{equation}
where
\begin{eqnarray}
\mathcal{E}_h^n&=&\int_{\Omega}  \frac{\mu_0}{2} H_h^{n+1/2}H_h^{n-1/2} + \frac{\epsilon_0 \epsilon_\infty}{2} (E_h^n)^2 +  \frac{\epsilon_0}{2\omega_p^2} (J_h^n)^2 + \frac{\epsilon_0\omega_0^2}{2 \omega_p^2}   (P_h^n)^2 \label{eq:full:ene1} \\
&&+ \frac{\epsilon_0  a\theta}{4\omega_v^2} (\sigma_h^n)^2 + \frac{\epsilon_0  a\theta}{2}  Q_h^n (E_h^n)^2  + \frac{3 \epsilon_0 a (1-\theta)}{4} (E_h^n)^4+\frac{\epsilon_0  a\theta}{4}(Q_h^n)^2  dx\notag
\end{eqnarray}
is the discrete energy. In addition, $\mathcal{E}_h \ge 0$ if $\theta \in [0, \frac{3}{4}]$ and the CFL condition $\frac{\Delta t}{h} \le 
 C_\star\sqrt{ \mu_0 \epsilon_0 \epsilon_\infty}$ is satisfied.

The  fully discrete scheme \eqref{eq:1d:schst} with the upwind flux \eqref{eq:flux:uf0}  satisfies
\begin{eqnarray}
\label{eq:full:stab2}
&& \mathcal{E}_h^{n+1}- \mathcal{E}_h^{n} =
-\frac{\epsilon_0 \Delta t}{4 \omega_p^2 \tau} \int_\Omega (J_h^{n+1} +J_h^{n})^2 dx-\frac{\epsilon_0  a\theta \Delta t}{8\omega_v^2 \tau_v}  \int_\Omega  (\sigma_h^{n+1}+\sigma_h^{n})^2  dx \\
&&-\frac{\Delta t}{8}\sqrt{\frac{\mu_0}{\epsilon_0 \epsilon_\infty}}\sum_{j=1}^N[H_h^{n-1/2}+H_h^{n+1/2}]_{j+1/2}^2-\frac{\Delta t}{8}\sqrt{\frac{\epsilon_0 \epsilon_\infty}{\mu_0}}\sum_{j=1}^N[E_h^n+E_h^{n+1}]_{j+1/2}^2 \le 0,\notag
\end{eqnarray}
where
\begin{eqnarray}
\mathcal{E}_h^n&=&\int_{\Omega}  \frac{\mu_0}{2} H_h^{n+1/2} H_h^{n-1/2} + \frac{\epsilon_0 \epsilon_\infty}{2} (E_h^n)^2 +  \frac{\epsilon_0}{2\omega_p^2} (J_h^n)^2 + \frac{\epsilon_0\omega_0^2}{2 \omega_p^2}   (P_h^n)^2+ \frac{\epsilon_0  a\theta}{4\omega_v^2} (\sigma_h^n)^2 \notag \\
&& + \frac{\epsilon_0  a\theta}{2}  Q_h^n (E_h^n)^2  + \frac{3 \epsilon_0 a (1-\theta)}{4} (E_h^n)^4+\frac{\epsilon_0  a\theta}{4}(Q_h^n)^2  dx\notag\\
&&+\frac{\Delta t}{8}\sqrt{\frac{\mu_0}{\epsilon_0 \epsilon_\infty}}\sum_{j=1}^N ([H_h^{n-1/2}] [H_h^{n-1/2}+H_h^{n+1/2}])_{j+1/2} \label{eq:full:ene2}
\end{eqnarray}
is the discrete energy. In addition, $\mathcal{E}_h \ge 0$ if $\theta \in [0, \frac{3}{4}]$ and the CFL condition $\frac{\Delta t}{h} \le \frac{C_\star\mu_0 \min(1,  \sqrt{\frac{\epsilon_0 \epsilon_\infty}{\mu_0}}) }{(\sqrt{2}+\min(1,  \sqrt{\frac{\epsilon_0 \epsilon_\infty}{\mu_0}}))} $
is satisfied.

Similarly, the  fully discrete scheme \eqref{eq:1d:schim} with central and alternating fluxes, \eqref{eq:flux:c} and \eqref{eq:flux:a},  satisfies
\begin{equation}
\label{eq:full:stab3}
 \mathcal{E}_h^{n+1}- \mathcal{E}_h^{n}=-\frac{\epsilon_0 \Delta t}{4 \omega_p^2 \tau} \int_\Omega (J_h^{n+1} +J_h^{n})^2 dx-\frac{\epsilon_0  a\theta \Delta t}{8\omega_v^2 \tau_v}  \int_\Omega  (\sigma_h^{n+1} +\sigma_h^{n})^2  dx \le 0,
\end{equation}
and that with the upwind flux \eqref{eq:flux:uf}  satisfies
\begin{eqnarray}
&& \mathcal{E}_h^{n+1}- \mathcal{E}_h^{n} =
-\frac{\epsilon_0 \Delta t}{4 \omega_p^2 \tau} \int_\Omega (J_h^{n+1} +J_h^{n})^2 dx-\frac{\epsilon_0  a\theta \Delta t}{8\omega_v^2 \tau_v}  \int_\Omega  (\sigma_h^{n+1}+\sigma_h^{n})^2  dx \label{eq:full:stab4}\\
&&-\frac{\Delta t}{8}\sqrt{\frac{\mu_0}{\epsilon_0 \epsilon_\infty}}\sum_{j=1}^N[H_h^{n}+H_h^{n+1}]_{j+1/2}^2-\frac{\Delta t}{8}\sqrt{\frac{\epsilon_0 \epsilon_\infty}{\mu_0}}\sum_{j=1}^N[E_h^n+E_h^{n+1}]_{j+1/2}^2 \le 0,\notag
\end{eqnarray}
where
\begin{eqnarray}
\mathcal{E}_h^n&=&\int_{\Omega}  \frac{\mu_0}{2} (H_h^{n})^2 + \frac{\epsilon_0 \epsilon_\infty}{2} (E_h^n)^2 +  \frac{\epsilon_0}{2\omega_p^2} (J_h^n)^2 + \frac{\epsilon_0\omega_0^2}{2 \omega_p^2}   (P_h^n)^2 \label{eq:full:ene3} \\
&&+ \frac{\epsilon_0  a\theta}{4\omega_v^2} (\sigma_h^n)^2 + \frac{\epsilon_0  a\theta}{2}  Q_h^n (E_h^n)^2  + \frac{3 \epsilon_0 a (1-\theta)}{4} (E_h^n)^4+\frac{\epsilon_0  a\theta}{4}(Q_h^n)^2  dx.\notag
\end{eqnarray}
It is non-negative when $\theta \in [0, \frac{3}{4}]$. In other words, the scheme  \eqref{eq:1d:schim} is unconditionally stable for all three flux choices.

\end{theorem}

\begin{proof} We  will only prove the results for scheme \eqref{eq:1d:schst}, while the proof for the scheme \eqref{eq:1d:schim} shares great similarity and is omitted.

Apply two time steps of \eqref{eq:schst1} and \eqref{eq:schst9}, we have
\begin{align}
\label{eq:pff1}
\mu_0\int_{I_j} \frac{H_h^{n+3/2}-H_h^{n-1/2}}{\Delta t}\phi dx &+\int_{I_j} (E_h^{n+1}+E_h^n) \dd{x}\phi dx- (\mathcal{F}(E_h^{n+1}+E_h^n)\phi^-)_{j+1/2}
 \\
&+ (\mathcal{F}(E_h^{n+1}+E_h^n)\phi^+)_{j-1/2}=0, \quad \forall \phi\in V_h^k,\notag
\end{align}
where
\begin{equation}
\mathcal{F}(E_h^{n+1}+E_h^n): =\left\{
\begin{array}{ll}
\widehat{E_h^{n+1}+E_h^n},&\text{central or alternating fluxes}\\
 \{E_h^{n+1}+E_h^n\}+\frac{1}{4}\sqrt{\frac{\mu_0}{\epsilon_0 \epsilon_\infty}}\left[H_h^{n+3/2}+2H_h^{n+1/2}+H_h^{n-1/2}\right],&\text{upwind flux}.
\end{array}
\right.
\end{equation}

Let $\phi=E_h^{n+1}+E_h^n$ in \eqref{eq:schst2}, $\phi=H_h^{n+1/2}$ in \eqref{eq:pff1} and sum up the two equalities over all elements, we obtain
\begin{eqnarray}
\label{eq:pff2}
&&\int_\Omega  \mu_0 (H_h^{n+3/2}-H_h^{n-1/2}) H_h^{n+1/2} +(D_h^{n+1}-D_h^n) (E_h^{n+1}+E_h^n)dx \\
&&=\left\{ \begin{array}{ll}
         0 & \mbox{ central or alternating  fluxes};\\
        -\frac{\Delta t}{4}\sqrt{\frac{\mu_0}{\epsilon_0 \epsilon_\infty}}\sum_{j=1}^N([H_h^{n+3/2}+2H_h^{n+1/2}+H_h^{n-1/2}][H_h^{n+1/2}])_{j+1/2}-&\\
\;\;\;\;\;\;\frac{\Delta t}{4}\sqrt{\frac{\epsilon_0 \epsilon_\infty}{\mu_0}}\sum_{j=1}^N[E_h^n+E_h^{n+1}]^2_{j+1/2} & \mbox{ upwind flux }\end{array} \right. \notag
\end{eqnarray}
by the identity  \eqref{eq:fluxidentity} in the proof of Theorem \ref{thm:semi}.

Using \eqref{eq:schst3}, we obtain
\begin{align}
\label{eq:pff3}
&\int_\Omega   (D_h^{n+1}-D_h^n) (E_h^{n+1}+E_h^n)dx  \notag \\
= &\epsilon_0 \int_\Omega  \epsilon_\infty ((E_h^{n+1})^2- (E_h^{n})^2) +  \frac{3}{2} a(1-\theta)((E_h^{n+1})^4- (E_h^{n})^4)+( P_h^{n+1}- P_h^{n}) (E_h^{n+1}+E_h^n) \notag \\
&+a \theta ( Q_h^{n+1} E_h^{n+1}- Q_h^{n} E_h^{n}) (E_h^{n+1}+E_h^n) dx,
\end{align}
and here \eqref{eq:schst4} is used. By \eqref{eq:schst7} and \eqref{eq:schst8},
\begin{align}
\label{eq:pff4}
&\int_\Omega  ( P_h^{n+1}- P_h^{n}) (E_h^{n+1}+E_h^n) dx \notag  \\
=& \int_\Omega \frac{1}{\omega_p^2} ( P_h^{n+1}- P_h^{n}) \left ( 2 \frac{J_h^{n+1}-J_h^n}{\Delta t}+\frac{1}{\tau}(J_h^n+J_h^{n+1}) +\omega_0^2(P_h^n+P_h^{n+1}) \right)   dx   \notag \\
=& \int_\Omega \frac{\Delta t}{2 \omega_p^2} ( J_h^{n+1}+ J_h^{n}) \left ( 2 \frac{J_h^{n+1}-J_h^n}{\Delta t}+\frac{1}{\tau}(J_h^n+J_h^{n+1})  \right)   dx + \frac{\omega_0^2}{\omega_p^2}  \int_\Omega  (P_h^{n+1})^2- (P_h^{n})^2  dx  \notag  \\
=& \frac{1}{\omega_p^2}  \int_\Omega  (J_h^{n+1})^2- (J_h^{n})^2  dx + \frac{\omega_0^2}{\omega_p^2}  \int_\Omega  (P_h^{n+1})^2- (P_h^{n})^2  dx +\frac{\Delta t}{2  \tau \omega_p^2} \int_\Omega  (J_h^{n+1} +J_h^{n})^2  dx.
\end{align}
On the other hand,
$$
 \int_\Omega ( Q_h^{n+1} E_h^{n+1}- Q_h^{n} E_h^{n}) (E_h^{n+1}+E_h^n) dx =  \int_\Omega Q_h^{n+1} (E_h^{n+1})^2 - Q_h^{n} (E_h^{n})^2 + E_h^{n+1} E_h^{n} (Q_h^{n+1}-Q_h^n) dx.
$$
By \eqref{eq:schst6} and \eqref{eq:schst5}, we have
\begin{align}
\label{eq:pff5}
&\int_\Omega   E_h^{n+1} E_h^{n} (Q_h^{n+1}-Q_h^n) dx \notag  \\
=& \int_\Omega \frac{1}{\omega_v^2}   \left (   \frac{\sigma_h^{n+1}-\sigma_h^n}{\Delta t}+\frac{1}{2 \tau_v}(\sigma_h^n+\sigma_h^{n+1}) +\frac{\omega_v^2}{2}(Q_h^n+Q_h^{n+1}) \right)   (Q_h^{n+1}-Q_h^n)  dx\notag \\
=& \int_\Omega \frac{\Delta t}{2 \omega_v^2}   \left (   \frac{\sigma_h^{n+1}-\sigma_h^n}{\Delta t}+\frac{1}{2 \tau_v}(\sigma_h^n+\sigma_h^{n+1})  \right)   (\sigma_h^{n+1}+\sigma_h^n)  dx  + \frac{1}{2 }  \int_\Omega  (Q_h^{n+1})^2- (Q_h^{n})^2  dx  \notag \\
=&  \frac{1}{2 \omega_v^2}    \int_\Omega  (\sigma_h^{n+1})^2- (\sigma_h^{n})^2  dx   + \frac{1}{2}  \int_\Omega  (Q_h^{n+1})^2- (Q_h^{n})^2  dx +\frac{\Delta t}{4  \tau_v \omega_v^2} \int_\Omega  (\sigma_h^{n+1} +\sigma_h^{n})^2  dx.
\end{align}

Substituting \eqref{eq:pff3}-\eqref{eq:pff5} into  \eqref{eq:pff2}, we have shown the energy stability \eqref{eq:full:stab1} for the scheme \eqref{eq:1d:schst} with the alternating and central fluxes, where the discrete energy $\mathcal{E}_h^n$ is defined in \eqref{eq:full:ene1}. When the flux is upwind, we instead get the energy stability in \eqref{eq:full:stab2}, with the discrete  energy $\mathcal{E}_h^n$ defined in \eqref{eq:full:ene2}.

\medskip
The final step is to find the conditions to guarantee the discrete energy to be non-negative.  We define two operators
\begin{align}
\widehat{\mathcal{H}}(E_h^n, \phi)&=
\sum_{j=1}^N \int_{I_j} E_h^n \dd{x} \phi dx+\sum_{j=1}^N \widehat{E_h^n}[\phi]_{j+1/2},\\
\widehat{\widehat{\mathcal{H}}}(E_h^n, \phi)&=
\sum_{j=1}^N  \int_{I_j} E_h^n \dd{x} \phi dx+\sum_{j=1}^N \widehat{\widehat{E_h^n}}[\phi]_{j+1/2}.
\end{align}

From \eqref{eq:schst1}, we have
\begin{equation}
\label{eq:pff6:00}
\int_\Omega H_h^{n+1/2}  \phi dx = \int_\Omega H_h^n \phi dx -\frac{\Delta t}{2 \mu_0}  \widehat{\mathcal{H}}(E_h^n, \phi),\quad \forall \phi\in V_h^k,
\end{equation}
and similarly by \eqref{eq:schst9}
\begin{equation}
\label{eq:pff6}
\int_\Omega H_h^{n}  \phi dx = \int_\Omega H_h^{n-1/2} \phi dx-\frac{\Delta t}{2 \mu_0}  \widehat{\widehat{\mathcal{H}}}(E_h^n, \phi),\quad \forall \phi\in V_h^k.
\end{equation}

Therefore,
\begin{align}
&\int_{\Omega}  \frac{\mu_0}{2} H_h^{n+1/2} H_h^{n-1/2} dx  = \int_{\Omega}  \frac{\mu_0}{2} H_h^{n} H_h^{n-1/2} dx -\frac{\Delta t}{4}  \widehat{\mathcal{H}}(E_h^n, H_h^{n-1/2}) \notag \\
=& \int_{\Omega}  \frac{\mu_0}{2} (H_h^{n})^2  dx +\frac{\Delta t}{4} \widehat{\widehat{ \mathcal{H}}}(E_h^n, H_h^{n}) -\frac{\Delta t}{4} \widehat{\mathcal{H}}(E_h^n, H_h^{n-1/2})  \notag\\
=& \int_{\Omega}  \frac{\mu_0}{2} (H_h^{n})^2  dx -\frac{\mu_0}{2} \|H_h^{n}-H_h^{n-1/2}\|^2 +\frac{\Delta t}{4} \Big( \widehat{\widehat{ \mathcal{H}}}(E_h^n, H_h^{n-1/2})-\widehat{\mathcal{H}}(E_h^n, H_h^{n-1/2})\Big).\label{eq:full:stab:100}
\end{align}
In the last equality, \eqref{eq:pff6} has been used.

With the alternating or central fluxes, the last term on the right of \eqref{eq:full:stab:100} vanishes, hence
\begin{equation}
\label{eq:full:stab:101}
\int_{\Omega}  \frac{\mu_0}{2} H_h^{n+1/2} H_h^{n-1/2} dx=\int_{\Omega}  \frac{\mu_0}{2} (H_h^{n})^2  dx-\frac{\mu_0}{2} \|H_h^{n}-H_h^{n-1/2}\|^2.
\end{equation}
On the other hand, using inverse inequality \eqref{eq:inv},
 one gets
\begin{equation}
|\widehat{\widehat{\mathcal{H}}}(\varphi,\phi)| \le \frac{C_\star}{h} \|\varphi\|\, \|\phi\|, \quad \forall \varphi, \phi \in V_h^k.
 \label{eq:full:stab:103}
\end{equation}
In addition, we take
$\phi=H_h^{n}-H_h^{n-1/2}$ in \eqref{eq:pff6}, and reach
\begin{equation}
 \label{eq:full:stab:104}
\|H_h^{n}-H_h^{n-1/2}\| \le \frac{\Delta t C_\star}{2 \mu_0 h} \|E_h^n\|,
\end{equation}
and therefore
$$
\left |\int_{\Omega}  \frac{\mu_0}{2} H_h^{n+1/2} H_h^{n-1/2} dx -\int_{\Omega}  \frac{\mu_0}{2} (H_h^{n})^2  dx \right |  \le  \frac{\Delta t^2 C_\star^2}{8 h^2 \mu_0} \|E_h^n\|^2.
$$
This estimate guarantees 
 $\mathcal{E}_h^n$ being non-negative
if $\theta \in [0, \frac{3}{4}]$ and $ \frac{\Delta t^2 C_\star^2}{8 h^2 \mu_0} \le  \frac{\epsilon_0 \epsilon_\infty}{2} $, i.e. the CFL condition $\frac{\Delta t}{h} \le \frac{2}{C_\star} \sqrt{ \mu_0 \epsilon_0 \epsilon_\infty}$ is satisfied.  This condition can be also written as
$\frac{\Delta t}{h} \le C_\star \sqrt{ \mu_0 \epsilon_0 \epsilon_\infty}$.  Here the generic constant $C_\star$  depends on $k$ and mesh regularity parameter $\delta$ and is independent of $h$.

With the upwind flux in \eqref{eq:flux:uf0}, using the definitions of  $\widehat{\widehat{ \mathcal{H}}}$ and  $\widehat{ \mathcal{H}}$, the last term on the right of \eqref{eq:full:stab:100} becomes
\begin{align}
&\frac{\Delta t}{4} \Big( \widehat{\widehat{ \mathcal{H}}}(E_h^n, H_h^{n-1/2})-\widehat{\mathcal{H}}(E_h^n, H_h^{n-1/2})\Big)=\frac{\Delta t}{4} \sum_{j=1}^N ((\widehat{\widehat{E^n_h}}-\widehat{E^n_h})[H_h^{n-1/2}])_{j+1/2}\notag\\
=&
\frac{\Delta t}{8}\sqrt{\frac{\mu_0}{\epsilon_0\epsilon_\infty}}\sum_{j=1}^N
([H_h^{n-1/2}-H_h^{n+1/2}][H_h^{n-1/2}])_{j+1/2}.\label{eq:full:stab:102}
\end{align}

Now taking into account the jump terms in the discrete energy \eqref{eq:full:ene2}, and with \eqref{eq:full:stab:100}\eqref{eq:full:stab:102}, we have
\begin{eqnarray}
&&\int_{\Omega}  \frac{\mu_0}{2} H_h^{n+1/2} H_h^{n-1/2} dx +\frac{\Delta t}{8}\sqrt{\frac{\mu_0}{\epsilon_0 \epsilon_\infty}} \sum_{j=1}^N [H_h^{n-1/2}] [H_h^{n-1/2}+H_h^{n+1/2}]_{j+1/2} \notag \\
&&=  \int_{\Omega}  \frac{\mu_0}{2} (H_h^{n})^2  dx+\frac{\Delta t}{4}\sqrt{\frac{\mu_0}{\epsilon_0 \epsilon_\infty}}\sum_{j=1}^{N}    [H_h^{n-1/2}]^2_{j+1/2}-\frac{\mu_0}{2} \|H_h^{n}-H_h^{n-1/2}\|^2. \label{eqn:pfh}
\end{eqnarray}
An analogue of \eqref{eq:full:stab:104} for the upwind flux can be obtained
$$
\|H_h^{n}-H_h^{n-1/2}\| \le \frac{\Delta t C_\star}{2 \mu_0 h} (\|E_h^n\|+\|H_h^{n-1/2}\|).
$$

We can choose $\Delta t$ small enough so that $A=\frac{\Delta t C_\star}{2 \mu_0 h}$ satisfies $\frac{A}{1-A} \le \frac{1}{\sqrt{2}} \min(1,  \sqrt{\frac{\epsilon_0 \epsilon_\infty}{\mu_0}})$, then
$$
\|H_h^{n}-H_h^{n-1/2}\| \le A (\|E_h^n\|+\|H_h^{n}\|+\|H_h^{n}-H_h^{n-1/2}\| ).
$$
Hence
$$
\|H_h^{n}-H_h^{n-1/2}\| \le \frac{A}{1-A} (\|E_h^n\|+\|H_h^{n}\|),
$$
and
$$
\frac{\mu_0}{2} \|H_h^{n}-H_h^{n-1/2}\|^2 \le  \mu_0 \left(\frac{A}{1-A}\right)^2 (\|E_h^n\|^2+\|H_h^{n}\|^2)
\le \frac{\mu_0}{2}\|H_h^{n}\|^2+\frac{\epsilon_0 \epsilon_\infty}{2}\|E_h^n\|^2.
$$
Combined with \eqref{eqn:pfh} and \eqref{eq:full:ene2}, we have shown that $\mathcal{E}_h^n$ in \eqref{eq:full:ene2}  is nonnegative, provided $\theta \in [0, \frac{3}{4}]$ and the CFL condition $\frac{\Delta t}{h} \le \frac{2\mu_0 \min(1,  \sqrt{\frac{\epsilon_0 \epsilon_\infty}{\mu_0}}) }{C_\star(\sqrt{2}+\min(1,  \sqrt{\frac{\epsilon_0 \epsilon_\infty}{\mu_0}}))} $ is satisfied.  This condition can also be written as
$\frac{\Delta t}{h} \le \frac{C_\star\mu_0 \min(1,  \sqrt{\frac{\epsilon_0 \epsilon_\infty}{\mu_0}}) }{(\sqrt{2}+\min(1,  \sqrt{\frac{\epsilon_0 \epsilon_\infty}{\mu_0}}))} $.
\end{proof}

 From the proof, one can see that the fully discrete scheme with the leap-frog temporal discretization is conditionally stable, under a CFL condition that is the \emph{same} as the one for Maxwell's equations without Kerr, linear Lorentz and Raman effects; while the fully implicit scheme is unconditionally stable.

\section{Numerical Results}
\label{sec:numerical}

In this section, we demonstrate the behavior of the fully discrete schemes through two numerical examples. The simulations are performed on the rescaled equations with the scaling  chosen as follows: let the reference time scale be $t_0$, and reference space  scale be $x_0$ with $x_{0}=ct_{0}$ and $c=1/\sqrt{\mu_0\epsilon_0}$. Henceforth, the rescaled fields and constants are defined based on  a reference electric field $E_0$ as follows,
\begin{align*}
(H/E_{0})\sqrt{\mu_{0}/\epsilon_{0}}\rightarrow H, \ \ \
D/(\epsilon_{0}E_{0}) \rightarrow D, \ \ \
P/E_{0}\rightarrow P, \\
(J/E_{0})t_{0}\rightarrow J,\ \ \
E/E_{0}\rightarrow E, \ \ \
Q/E_{0}^{2} \rightarrow Q, \\
(\sigma/E_{0}^{2})t_{0}\rightarrow \sigma, \ \ \
a E_{0}^{2} \rightarrow a,\\
\omega_{0}t_{0}\rightarrow\omega_{0}, \ \ \
\omega_{p}t_{0}\rightarrow\omega_{p}, \ \ \
\omega_{v}t_{0}\rightarrow\omega_{v}, \\
(1/\tau)t_{0}\rightarrow 1/\tau, \ \ \
(1/\tau_{v})t_{0}\rightarrow 1/\tau_{v},
\end{align*}
where for simplicity, we have used the same notation to denote the scaled and original variables. In summary,  we arrive at the  dimensionless Maxwell's equations
\begin{subequations}
\label{eq:scale}
\begin{align}
& \partial_{t}H=\partial_{x}E,\qquad \partial_{t}D=\partial_{x}H,\\
& \partial_{t}P = J,\qquad  \partial_{t}J = -\frac{1}{\tau} J - \omega_{0}^2 P +\omega_{p}^2 E,\\
& \partial_{t}Q = \sigma,\qquad  \partial_{t}\sigma = -\frac{1}{\tau_{v}}\sigma - \omega_{v}^2Q + \omega_{v}^2 E^2,\\
& D=\epsilon_{\infty}E  + a(1-\theta)E^3 +P + a\theta QE.
\end{align}
\end{subequations}


Correspondingly, the energy $\mathcal{E}(t)$
$$
\mathcal{E}(t)=\int_{\Omega}  \left (\frac{1}{2} H^2 + \frac{\epsilon_\infty}{2} E^2 +  \frac{1}{2\omega_p^2} J^2   + \frac{\omega_0^2}{2 \omega_p^2}   P^2+ \frac{a\theta}{4\omega_v^2} \sigma^2 + \frac{a\theta}{2}  Q E^2  + \frac{3 a (1-\theta)}{4} E^4+\frac{a\theta}{4}Q^2\right) dx,
$$
satisfies the relation
$$
\frac{d}{dt}\mathcal{E}(t)=-\frac{1}{\omega_p^2 \tau} \int_\Omega J^2 dx-\frac{a\theta}{2\omega_v^2 \tau_v}  \int_\Omega \sigma^2 dx \le 0.
$$
For the rescaled system \eqref{eq:scale}, all fluxes retain their original definition except for the upwind flux, which is modified to
\begin{align}
\widehat{E}_{h}=\{E_{h}\}+\frac{1}{2\sqrt{\epsilon_{\infty}}}[H_{h}], \ \ \
\widetilde{H}_{h}=\{H_{h}\}+\frac{\sqrt{\epsilon_{\infty}}}{2}[E_{h}].
\end{align}
We further refer to one alternating flux
$$\widehat{E}_{h}=E_{h}^{+},  \ \ \ \widetilde{H}_{h}=H_{h}^{-}, $$
as alternating flux I, and
$$\widehat{E}_{h}=E_{h}^{-},  \ \ \ \widetilde{H}_{h}=H_{h}^{+}, $$
as alternating flux II.
For  numerical simulations in this section, we use a uniform mesh with size $h_{j}=h=(x_{R}-x_{L})/N$ for all $j$.  When solving the nonlinear system, we employ a Jacobian-free Newton-Krylov solver \cite{knoll2004jacobian} with absolute error threshold $\epsilon=10^{-10}$.

\subsection{Kink shape solutions}
The first numerical test we consider is originally discussed in \cite{sorensen2005kink}, where a traveling wave solution was constructed for the instantaneous intensity-dependent Kerr response neglecting the
 influence of damping, i.e., $\theta=0$, $\tau=\infty$ in \eqref{eq:scale}.  This yields a simplified system
\begin{subequations}
	\label{eq:scale1}
	\begin{align}
	& \partial_{t}H=\partial_{x}E, \qquad 
      \partial_{t}D=\partial_{x}H,\\
	& \partial_{t}P = J,\qquad
	 \partial_{t}J = - \omega_{0}^2 P +\omega_{p}^2 E,\\
	& D=\epsilon_{\infty}E  + a E^3 + P.
	\end{align}
\end{subequations}
We use this example  as an accuracy test.
As shown in \cite{sorensen2005kink},   we can find a traveling wave solution $E(x,t)=E(\xi)$, where $\xi=x-vt$, and similarly for other variables $H$, $D$, $P$ and $J$. Here, $E(\xi)$ is comprised of a  kink and antikink wave, and is solved based on the following ODE
\begin{subequations}
\label{eq:ODE}
\begin{align}
& \frac{dE}{d\xi}=\Phi,\\
& \frac{d\Phi}{d\xi}=\frac{6av^2E\Phi^2 + (\epsilon_\infty\omega_			0^2+\omega_p^2-\omega_0^2/v^2)E  + a \omega_0^2 E^3}{1 - \epsilon_{\infty} v^2 - 3av^2E^2}.
\end{align}
\end{subequations}
The parameters are
\begin{align}
\label{eq:para}
& \epsilon_{\infty} = 2.25, \ \ \
  \epsilon_{s} = 5.25, \ \ \
  \beta_{1} = \epsilon_{s} - \epsilon_{\infty}, \nonumber\\
& \omega_{0}=93.627179982222216,  \ \ \
   \omega_{p}=\omega_{0}\sqrt{\beta_{1}}, \nonumber\\
&   a = \epsilon_{\infty}/3, \ \ \
 v = 0.6545 / \sqrt{\epsilon_\infty},\nonumber\\
 & E(0)=0, \ \ \ \Phi(0)=  0.24919666777865812.
\end{align}
Here, $\omega_{0}$ and $\Phi(0)$ are carefully chosen such that $E$ and $\Phi$ are both $6$-periodic.  The approximate solution of \eqref{eq:ODE}, as shown in Figure \ref{Fig1.1}, is obtained with 160000 grid points by a third order Runge-Kutta method. This serves as   the initial condition for the electric field  in the Maxwell's system \eqref{eq:scale1}.
 Furthermore, with the help of \eqref{eq:scale1} and the property that all variables are traveling waves, we can obtain the initial conditions for other variables:
\begin{align*}
&  H(x,0) =-\frac{1}{v}E(\xi), \ \ \ D(x,0) = \frac{1}{v^2}E(\xi),\\
& P(x,0) =(\frac{1}{v^2}-\epsilon_{\infty})E(\xi) - a E(\xi)^3,\\
& J(x,0) =(\epsilon_{\infty}v-\frac{1}{v}) \Phi(\xi) + 3avE(\xi)^2 \Phi(\xi).
\end{align*}

\begin{figure}
	\centering
	\subfigure[Initial condition $E(x,0)$.]{
		\includegraphics[width=0.4\textwidth]{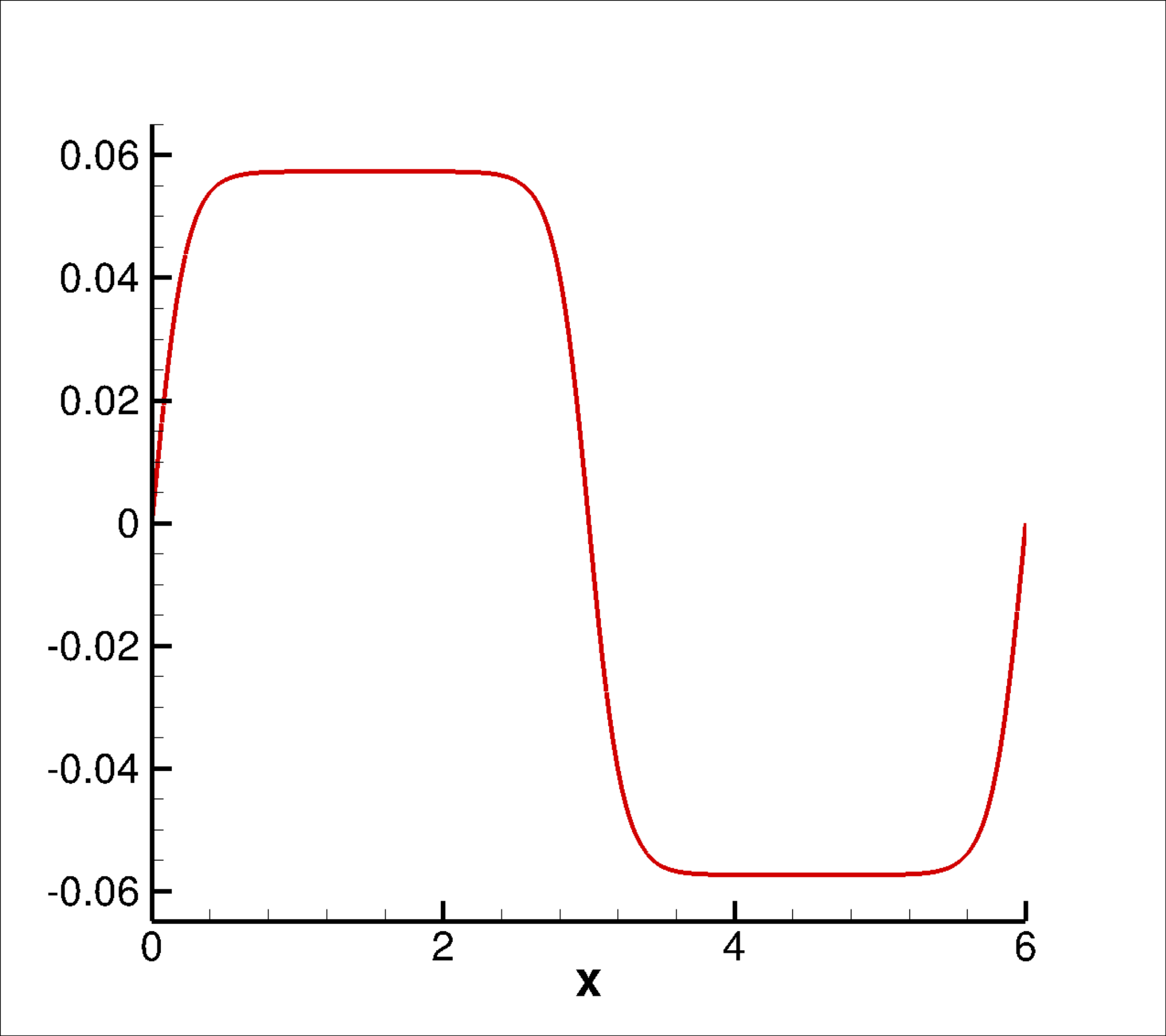}\label{Fig1.1}}
	\subfigure[Reference solution $E(x,t)$.]{
		\includegraphics[width=0.4\textwidth]{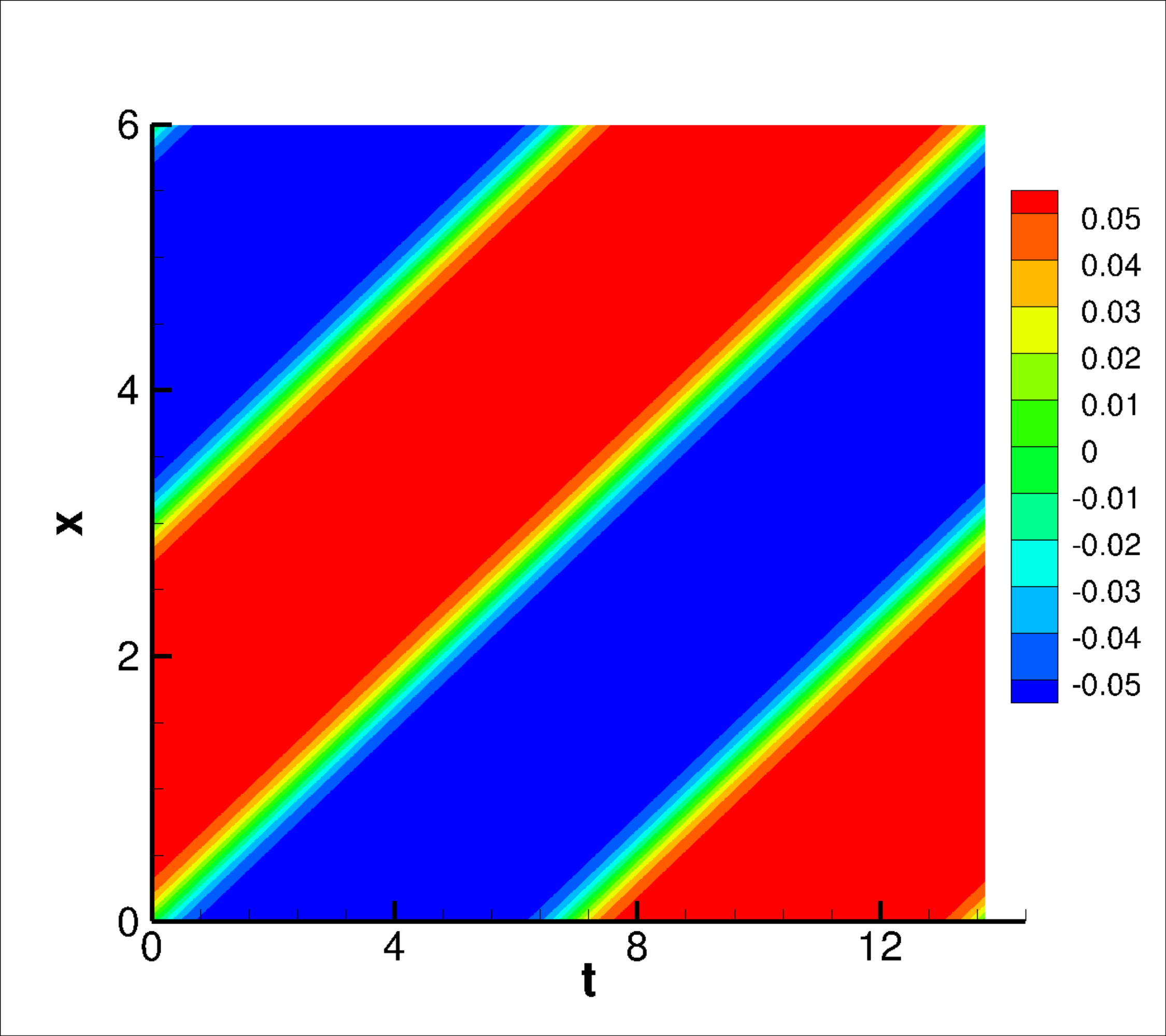}\label{Fig1.2}}
	\caption{\em A traveling kink and antikink wave: the electric field.}
	\label{Fig1}
\end{figure}

 Numerical results are provided at $t = 6/v$, when
the wave moves back to the same position as the initial condition. Time steps are chosen as
$$\Delta t=CFL\times h^{(k+1)/2}$$
to guarantee $(k+1)$-th order accuracy in time, and the CFL numbers are listed in Table \ref{tab0}. Since leap-frog scheme uses staggered time, we set CFL$=0.2/v$ for $k=1$, such that  the last time step has full length $\Delta t$.
This can help us avoid the influence on accuracy caused by time step changes. When $k=2$ or $3$, we do not need to do this because the time steps are already pretty small.  Note that the time steps of the fully implicit scheme are taken to be much larger.
\begin{table}[htb]
	\caption{\label{tab0}\em A traveling kink and antikink wave: $CFL$ number. }
	\centering
		\begin{tabular}{|c|c|c|}
			\hline
			$k$   &  Leap-frog scheme &   Fully implicit scheme   \\\hline
			1  &  $0.2/v$  &  5   \\
			2  &  1  &  10   \\
			3  &  2  &  20   \\\hline
		\end{tabular}
\end{table}

We list the errors and orders of accuracy of $E$ in Tables \ref{tab1}-\ref{tab3}. All calculations give  the optimal $(k+1)$-th order, except that  for the central flux, if we use the leap-frog scheme, the order of accuracy will be suboptimal when $k=1$. 

Next, we investigate the numerical energy behaviors with $N=400$ grid points.
The results are listed in Figure \ref{Fig2}.
Since $\theta=0$ and $\tau=\infty$, following the proof in Theorem \ref{thm:full}, we obtain that the discrete energy $\mathcal{E}_{h}^n$ with alternating and central fluxes satisfies
$$\mathcal{E}^{n+1}_{h}-\mathcal{E}^{n}_{h}=0,$$
where
\begin{eqnarray}
\mathcal{E}_h^n=\int_{\Omega}  \frac{1}{2} H_h^{n+1/2} H_h^{n-1/2} + \frac{\epsilon_\infty}{2} (E_h^n)^2 +  \frac{1}{2\omega_p^2} (J_h^n)^2 + \frac{\omega_0^2}{2 \omega_p^2}   (P_h^n)^2  + \frac{3 a}{4} (E_h^n)^4dx \notag
\end{eqnarray}
for leap-frog scheme, and
\begin{eqnarray}
\label{eq:fullenergy1}
\mathcal{E}_h^n=\int_{\Omega}  \frac{1}{2} (H_h^{n})^2 + \frac{\epsilon_\infty}{2} (E_h^n)^2 +  \frac{1}{2\omega_p^2} (J_h^n)^2 + \frac{\omega_0^2}{2 \omega_p^2}   (P_h^n)^2  + \frac{3 a}{4} (E_h^n)^4dx
\end{eqnarray}
for fully implicit scheme. Therefore, the schemes are energy-conserving. Figure \ref{Fig2} shows that the numerical results are consistent with our analysis:  the leap-frog scheme conserve discrete energy up to the machine error, while the fully implicit scheme has larger errors, which is caused by larger time steps and the error from the Newton solver (we set the tolerance as $\epsilon=10^{-10}$).

On the other hand, the upwind flux is dissipative. When employing the upwind flux and leap-frog scheme, we have
\begin{eqnarray}
\mathcal{E}_h^{n+1}- \mathcal{E}_h^{n} =
-\frac{\Delta t}{8\sqrt{\epsilon_\infty}}\sum_{j=1}^{N}[H_h^{n-1/2}+H_h^{n+1/2}]_{j+1/2}^2 - \frac{\Delta t\sqrt{ \epsilon_\infty}}{8}\sum_{j=1}^{N}[E_h^n+E_h^{n+1}]_{j+1/2}^2 \le 0, \notag
\end{eqnarray}
where the discrete energy is
\begin{eqnarray}
\mathcal{E}_h^n&=&\int_{\Omega}  \frac{1}{2} H_h^{n+1/2} H_h^{n-1/2} + \frac{\epsilon_\infty}{2} (E_h^n)^2 +  \frac{1}{2\omega_p^2} (J_h^n)^2 + \frac{\omega_0^2}{2 \omega_p^2}   (P_h^n)^2 + \frac{3 a}{4} (E_h^n)^4 dx \notag \\
&&
+\frac{\Delta t}{8\sqrt{\epsilon_\infty}}\sum_{j=1}^{N} ([H_h^{n-1/2}] [H_h^{n-1/2}+H_h^{n+1/2}])_{j+1/2}. \notag
\end{eqnarray}
For the fully implicit scheme with upwind flux, the discrete energy \eqref{eq:fullenergy1} will satisfy
\begin{eqnarray}
\mathcal{E}_h^{n+1}- \mathcal{E}_h^{n} =
-\frac{\Delta t}{8\sqrt{\epsilon_\infty}} \sum_{j=1}^{N}[H_h^{n}+H_h^{n+1}]_{j+1/2}^2-\frac{\Delta t\sqrt{\epsilon_\infty}}{8}\sum_{j=1}^{N}[E_h^n+E_h^{n+1}]_{j+1/2}^2 \le 0.\notag
\end{eqnarray}
 We observe the predicted behavior numerically (Figure \ref{Fig2}). Note that for $k=3$, the initial energy increase  is caused by error from the Newton solver. 

\begin{table}[htb]
	\caption{\label{tab1}\em A traveling kink and antikink wave: errors and orders of accuracy of $E$. $k=1$. }
	\centering
		\begin{tabular}{|c|c|cc|cc|cc|cc|}
			\hline
			\multirow{2}{*}{} &\multirow{2}{*}{$N$} &
			\multicolumn{4}{c|}{Leap-frog scheme} &
			\multicolumn{4}{c|}{Fully implicit scheme} \\
			\cline{3-10}
			&   &  $L_2$ errors &   order  &  $L_\infty$ error  & order
			&  $L_2$ errors &   order  &  $L_\infty$ error  & order  \\\hline
			\multirow{5}{2cm}{upwind flux}
			&  100   &  4.48E-04  &    --    &  1.80E-03  &     --
						&  9.24E-03  &    --    &  2.38E-02  &     --    \\
			&  200   &  7.68E-05  &  2.55  &  3.46E-04  &  2.38
						&  3.75E-03  &  1.30  &  1.16E-02  &  1.04  \\
			&  400   &  1.32E-05  &  2.54  &  5.92E-05  &  2.55
						&  1.15E-03  &  1.71  &  4.20E-03  &  1.46  \\
			&  800   &  2.75E-06  &  2.26  &  1.17E-05  &  2.33
						&  3.01E-04  &  1.93  &  1.20E-03  &  1.80   \\
			& 1600  &  6.57E-07  &  2.06  &  3.10E-06  &  1.92
						&  7.59E-05  &  1.99  &  3.11E-04  &  1.95  \\\hline
			\multirow{5}{2cm}{central flux}
			&  100   &  1.07E-03  &    --    &  4.53E-03  &     --
						&  9.13E-03  &    --    &  2.36E-02  &     --    \\
			&  200   &  2.77E-04  &  1.94  &  1.31E-03  &  1.79
						&  3.64E-03  &  1.33  &  1.18E-02  &  1.12   \\
			&  400   &  7.12E-05  &  1.96  &  3.59E-04  &  1.87
						&  1.10E-03  &  1.72  &  4.20E-03  &  1.49  \\
			&  800   &  1.97E-05  &  1.85  &  1.05E-04  &  1.78
						&  2.90E-04  &  1.93  &  1.23E-03  &  1.77  \\
			& 1600  &  6.91E-06  &  1.51  &  3.51E-05  &  1.57
						&  7.32E-05  &  1.98  &  3.24E-04  &  1.93  \\\hline
			\multirow{5}{2cm}{alternating flux I}
			&  100   &  1.15E-04  &    --    &  5.22E-03  &     --
						&  9.38E-03  &	 --    &  2.41E-02  &	    --    \\
			&  200   &  2.96E-05  &  1.96  &  1.13E-04  &  2.21
						&  3.77E-03  &  1.31  &  1.16E-02  &  1.06  \\
			&  400   &  8.91E-06  &  1.73  &  3.45E-05  &  1.71
						&  1.15E-03  &  1.71  &  4.21E-03  &  1.46  \\
			&  800   &  2.01E-06  &  2.15  &  9.48E-06  &  1.86
						&  3.01E-04  &  1.93  &  1.21E-03  &  1.80  \\
			& 1600  &  4.62E-07  &  2.12  &  2.03E-06  &  2.22
						&  7.59E-05  &  1.99  &  3.11E-04  &  1.95  \\\hline
			\multirow{5}{2cm}{alternating flux II}
			&  100   &  9.41E-05  &    --    &  4.79E-04  &     --
						&  9.39E-03  &	 --    &  2.41E-02  &     --    \\
			&  200   &  3.04E-05  &  1.63  &  1.49E-04  &  1.69
						&  3.78E-03  &  1.31  &  1.16E-02  &  1.05  \\
			&  400   &  8.36E-06  &  1.86  &  3.61E-05  &  2.04
						&  1.15E-03  &  1.71  &  4.20E-03  &  1.47  \\
			&  800   &  1.87E-06  &  2.16  &  8.79E-06  &  2.04
						&  3.01E-04  &  1.93  &  1.21E-03  &  1.80  \\
			& 1600  &  5.40E-07  &  1.79  &  2.65E-06  &  1.73
					     &  7.59E-05  &  1.99  &  3.11E-04  &  1.95  \\\hline
		\end{tabular}
\end{table}

\begin{table}[htb]
	\caption{\label{tab2}\em A traveling kink and antikink wave: errors and orders of accuracy of $E$. $k=2$. }
	\centering
		\begin{tabular}{|c|c|cc|cc|cc|cc|}
			\hline
			\multirow{2}{*}{} &\multirow{2}{*}{$N$} &
			\multicolumn{4}{c|}{Leap-frog scheme} &
			\multicolumn{4}{c|}{Fully implicit scheme} \\
			\cline{3-10}
			&   &  $L_2$ errors &   order  &  $L_\infty$ error  & order
			&  $L_2$ errors &   order  &  $L_\infty$ error  & order  \\\hline
			\multirow{4}{2cm}{upwind flux}
			&  100   &  2.49E-05  &    --     &  1.50E-04  &     --
						&  3.66E-03  &    --     &  1.13E-02  &     --    \\
			&  200   &  3.04E-06  &  3.04  &  1.88E-05  &  3.00 	
						&  5.71E-04  &  2.68  &  2.21E-03  &  2.35  \\
			&  400   &  3.69E-07  &  3.04  &  2.28E-06  &  3.04
						&  7.29E-05  &  2.97  &  2.99E-04  &  2.89  \\
			&  800   &  4.72E-08  &  2.96  &  3.04E-07  &  2.91
						&  9.13E-06  &  3.00  &  3.77E-05  &  2.99  \\
			\hline
			\multirow{4}{2cm}{central flux}
			&  100   &  2.32E-05  &    --     &  1.05E-04  &     --
						&  3.66E-03  &    --     &  1.13E-02  &     --    \\
			&  200   &  2.97E-06  &  2.96  &  1.38E-05  &  2.92  	
						&  5.72E-04  &  2.68  &  2.21E-03  &  2.35  \\
			&  400   &  3.71E-07  &  3.00  &  1.75E-06  &  2.99
						&  7.29E-05  &  2.97  &  2.99E-04  &  2.89  \\
			&  800   &  4.62E-08  &  3.00  &  2.19E-07  &  2.99
						&  9.13E-06  &  3.00  &  3.77E-05  &  2.99  \\
			\hline
			\multirow{4}{2cm}{alternating flux I}
			&  100   &  2.38E-05  &    --    &  1.04E-04  &     --
						&  3.66E-03  &    --    &  1.13E-02  &     --    \\
			&  200   &  2.98E-06  &  3.00  &  1.33E-05  &  2.96
						&  5.71E-04  &  2.68  &  2.21E-03  &  2.35 \\
			&  400   &  3.72E-07  &  3.00  &  1.70E-06  &  2.96
						&  7.29E-05  &  2.97  &  2.99E-04  &  2.89 \\
			&  800   &  4.62E-08  &  3.01  &  2.08E-07  &  3.04
						&  9.13E-06  &  3.00  &  3.77E-05  &  2.99  \\
			\hline
			\multirow{4}{2cm}{alternating flux II}
			&  100   &  2.41E-05  &    --    &  1.25E-04  &     --
						&  3.66E-03  &    --    &  1.13E-02  &     --     \\
			&  200   &  3.03E-06  &  2.99  &  1.67E-05  &  2.91
						&  5.72E-04  &  2.68  &  2.21E-03  &  2.35  \\
			&  400   &  3.73E-07  &  3.02  &  1.98E-06  &  3.07
						&  7.29E-05  &  2.97  &  2.99E-04  &  2.89  \\
			&  800   &  4.66E-08  &  3.00  &  2.67E-07  &  2.89
						&  9.13E-06  &  3.00  &  3.77E-05  &  2.99  \\
			\hline
		\end{tabular}
\end{table}

\begin{table}[htb]
	\caption{\label{tab3}\em A traveling kink and antikink wave: errors and orders of accuracy of $E$. $k=3$. }
	\centering
		\begin{tabular}{|c|c|cc|cc|cc|cc|}
			\hline
			\multirow{2}{*}{} &\multirow{2}{*}{$N$} &
			\multicolumn{4}{c|}{Leap-frog scheme} &
			\multicolumn{4}{c|}{Fully implicit scheme} \\
			\cline{3-10}
			&   &  $L_2$ errors &   order  &  $L_\infty$ error  & order
			&  $L_2$ errors &   order  &  $L_\infty$ error  & order  \\\hline
			\multirow{3}{2cm}{upwind flux}
			&  100   &  5.58E-06  &    --    &  2.41E-05  &     --
						&  1.07E-03  &    --    &  3.93E-03  &    --      \\
			&  200   &  3.50E-07  &  3.99  &  1.48E-06  &  4.03
						&  7.00E-05  &  3.93  &  2.87E-04  &  3.77  \\
			&  400   &  2.18E-08  &  4.00  &  1.05E-07  &  3.82
						&  4.38E-06  &  4.00  &  1.81E-05  &  3.99  \\
			\hline
			\multirow{3}{2cm}{central flux}
			&  100   &  5.68E-06  &    --    &  2.59E-05  &     --
						&  1.07E-03  &    --    &  3.93E-03  &    --      \\
			&  200   &  3.62E-07  &  3.97  &  1.73E-06  &  3.91
						&  7.00E-05  &  3.93  &  2.87E-04  &  3.77  \\
			&  400   &  2.27E-08  &  4.00  &  1.09E-07  &  3.98
						&  4.37E-06  &  4.00  &  1.80E-05  &  3.99  \\
		    \hline
			\multirow{3}{2cm}{alternating flux I}
			&  100   &  5.60E-06  &    --     &  2.30E-05  &     --
						&  1.07E-03  &    --     &  3.93E-03  &    --     \\
			&  200   &  3.53E-07  &  3.99  &  1.44E-06  &  4.00
						&  7.00E-05  &  3.93  &  2.87E-04  &  3.77  \\
			&  400   &  2.19E-08  &  4.01  &  8.97E-08  &  4.00
						&  4.38E-06  &  4.00  &  1.80E-05  &  3.99  \\
			\hline
			\multirow{3}{2cm}{alternating flux II}
			&  100   &  5.60E-06  &    --     &  2.35E-05  &     --
						&  1.07E-03  &    --     &  3.93E-03  &    --    \\
			&  200   &  3.53E-07  &  3.99  &  1.45E-06  &  4.02
						&  7.00E-05  &  3.93  &  2.87E-04  &  3.77  \\
			&  400   &  2.19E-08  &  4.01  &  9.01E-08  &  4.01
						&  4.38E-06  &  4.00  &  1.81E-05  &  3.99  \\
			\hline
		\end{tabular}
\end{table}

\begin{figure}
	\centering
	\subfigure[Leap-frog scheme. $k=1$.]{
		\includegraphics[width=0.28\textwidth]{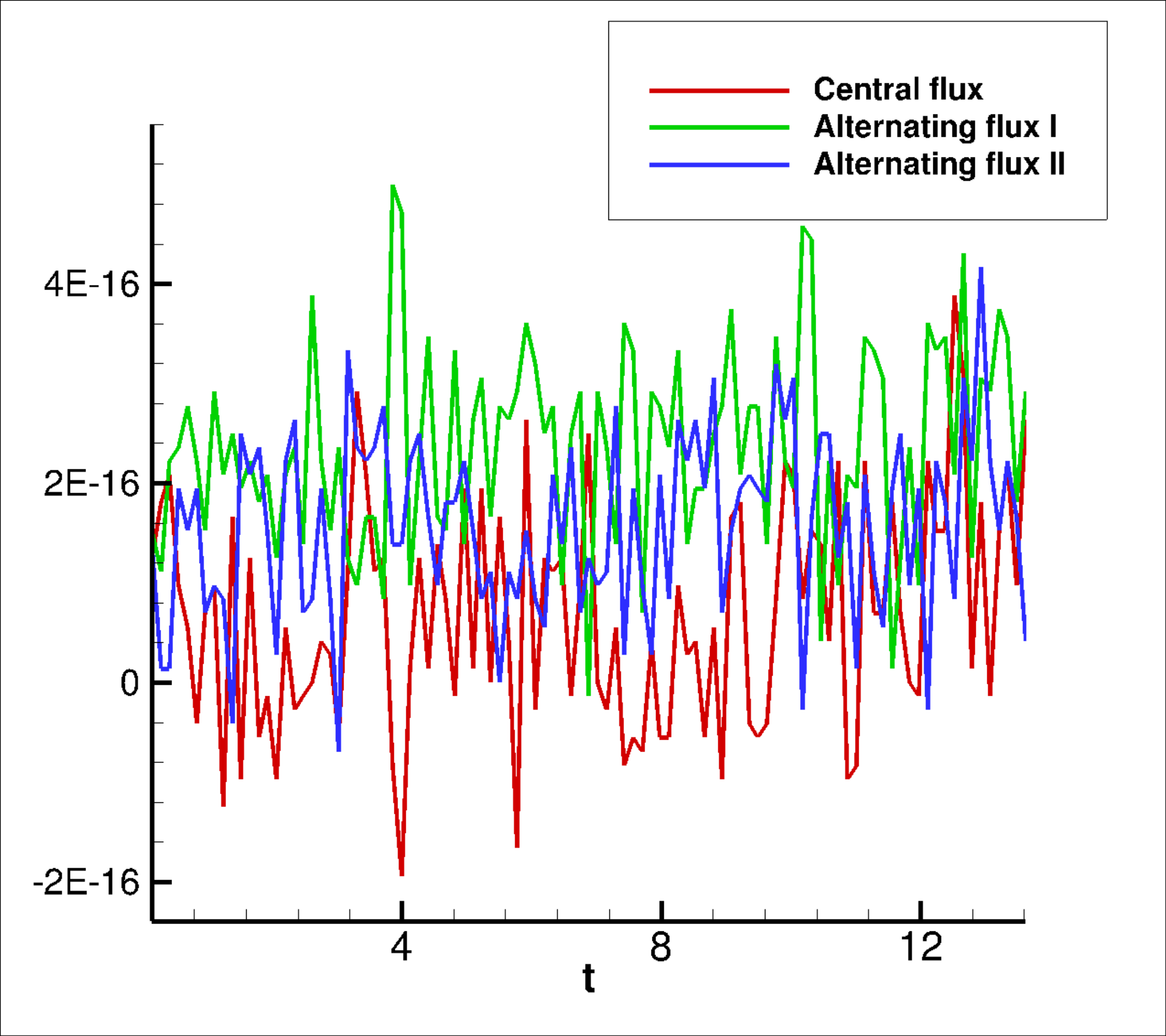}\label{Fig2.1}}
	\subfigure[Leap-frog scheme. $k=2$.]{
		\includegraphics[width=0.28\textwidth]{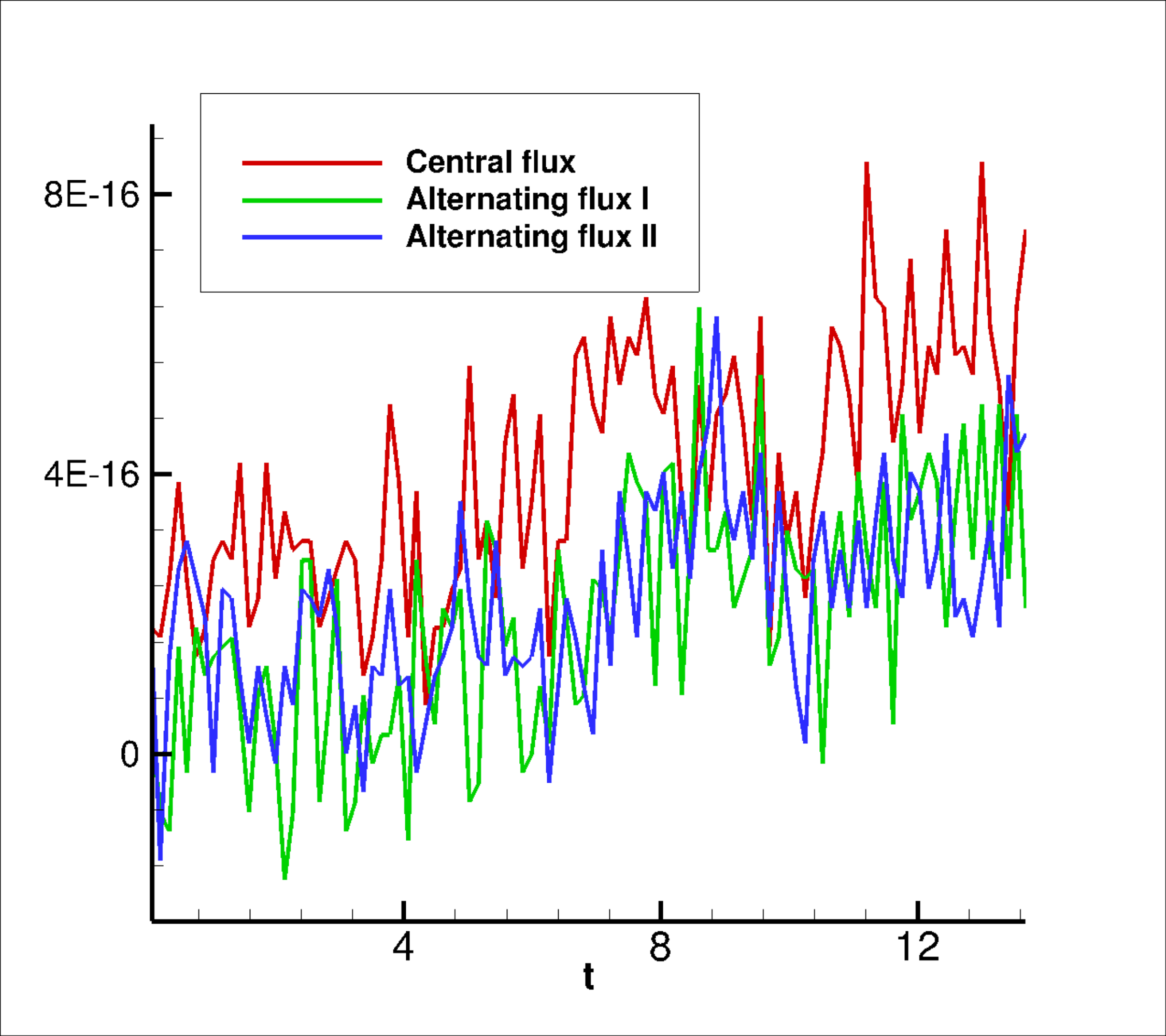}\label{Fig2.3}}
	\subfigure[Leap-frog scheme. $k=3$.]{
		\includegraphics[width=0.28\textwidth]{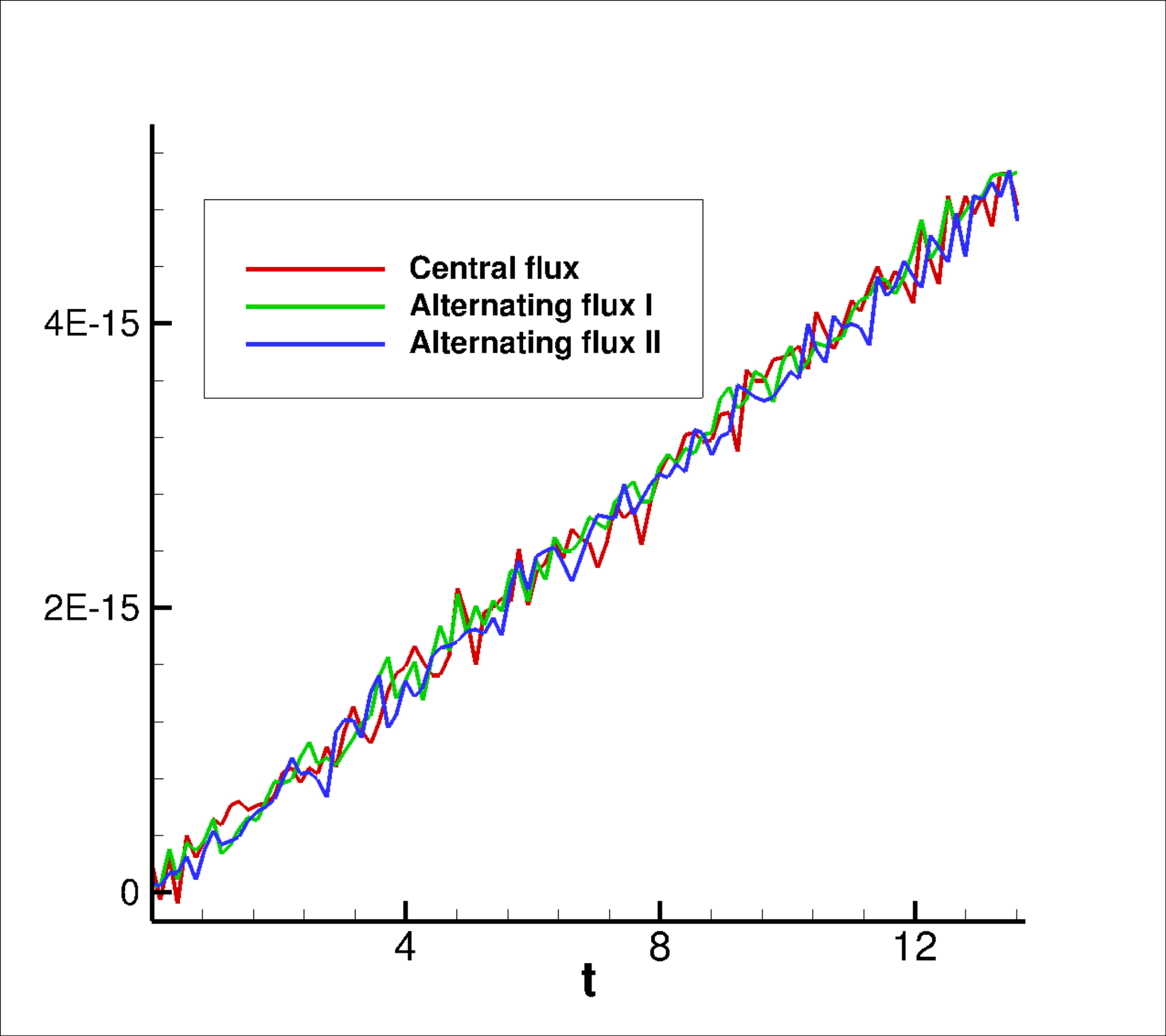}\label{Fig2.5}}
			\subfigure[Fully implicit scheme. $k=1$.]{
		\includegraphics[width=0.28\textwidth]{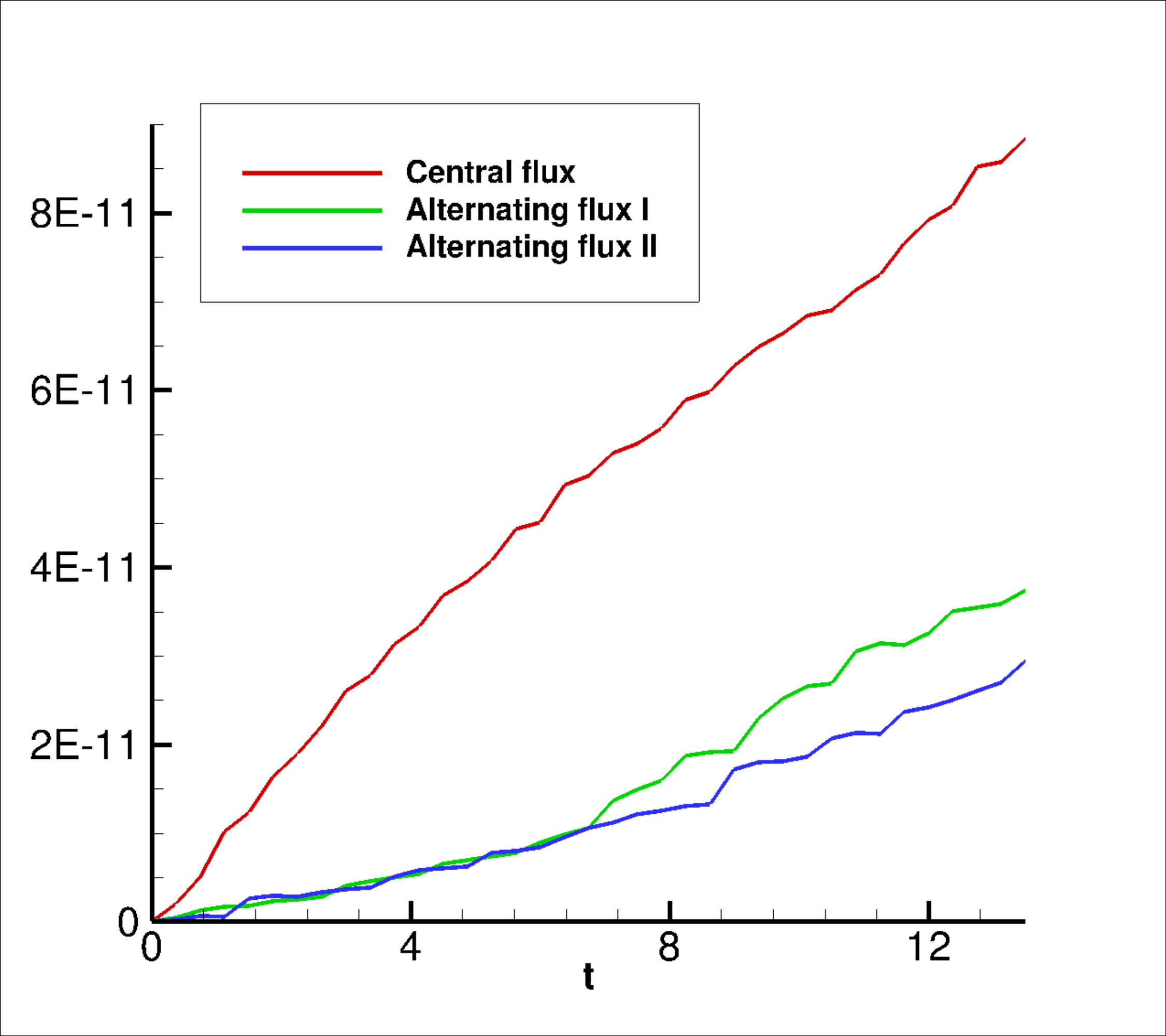}\label{Fig2.2}}
	\subfigure[Fully implicit scheme. $k=2$.]{
		\includegraphics[width=0.28\textwidth]{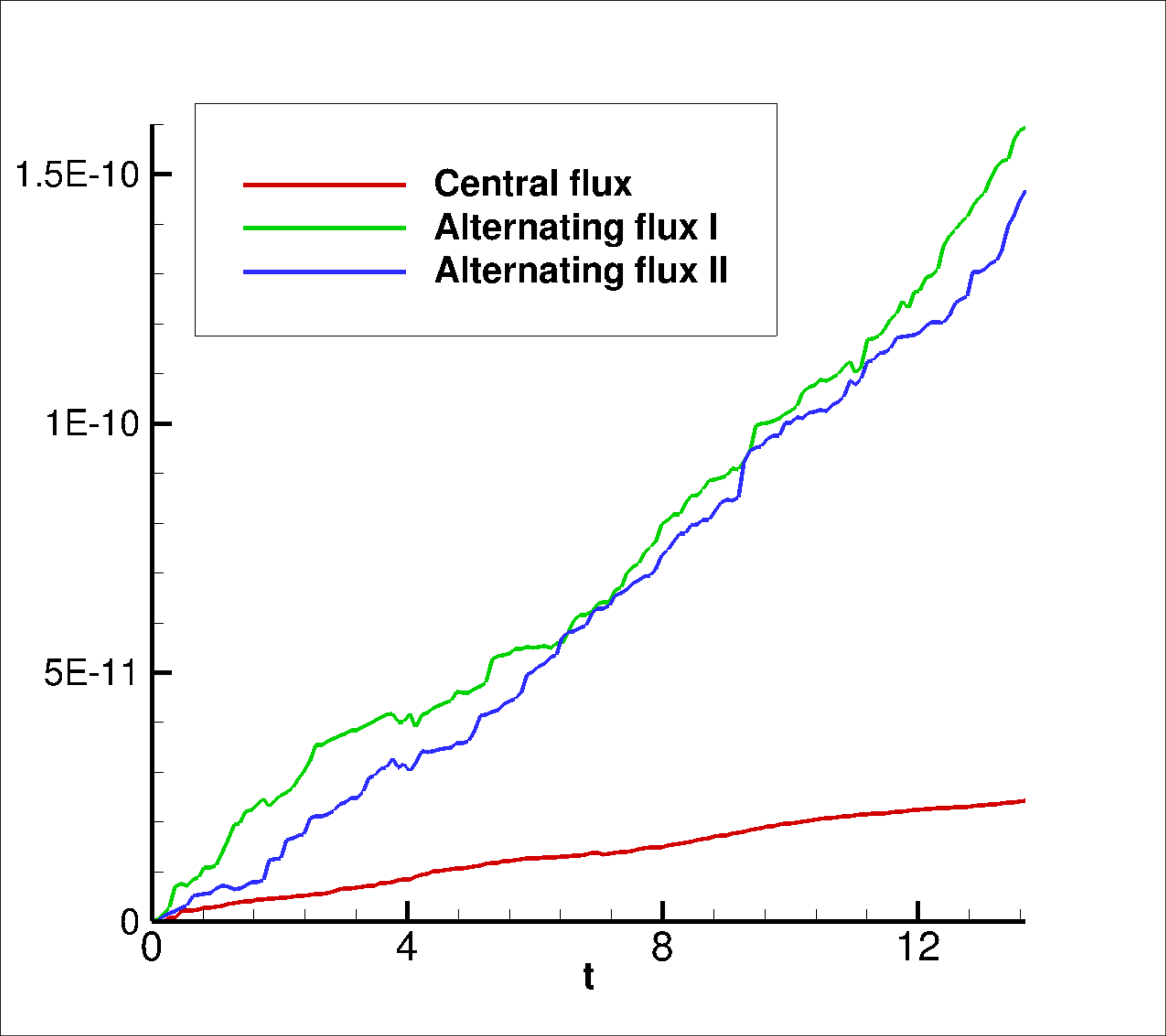}\label{Fig2.4}}
	\subfigure[Fully implicit scheme. $k=3$.]{
		\includegraphics[width=0.28\textwidth]{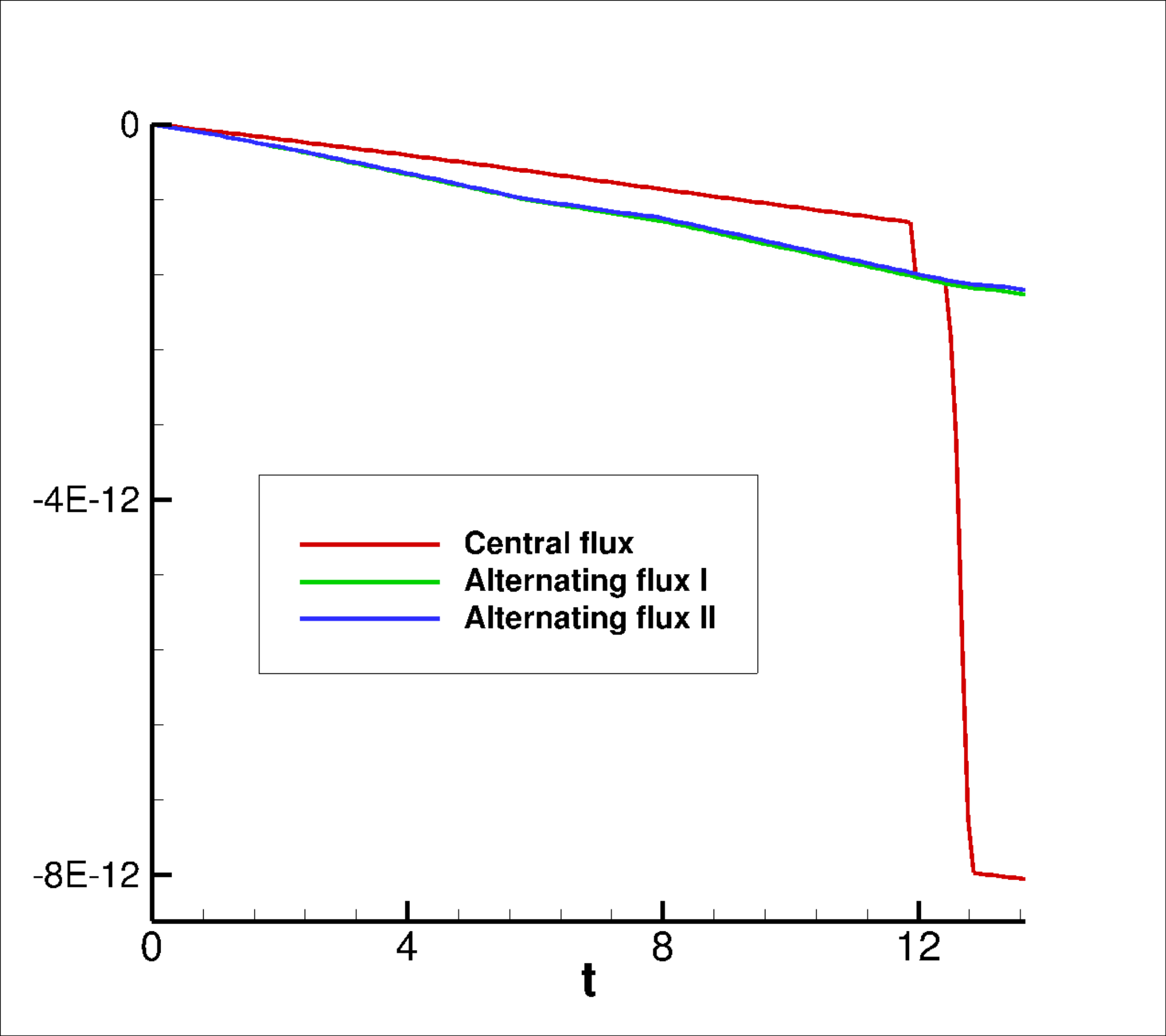}\label{Fig2.6}}
	\subfigure[Leap-frog scheme. $k=1$.]{
		\includegraphics[width=0.28\textwidth]{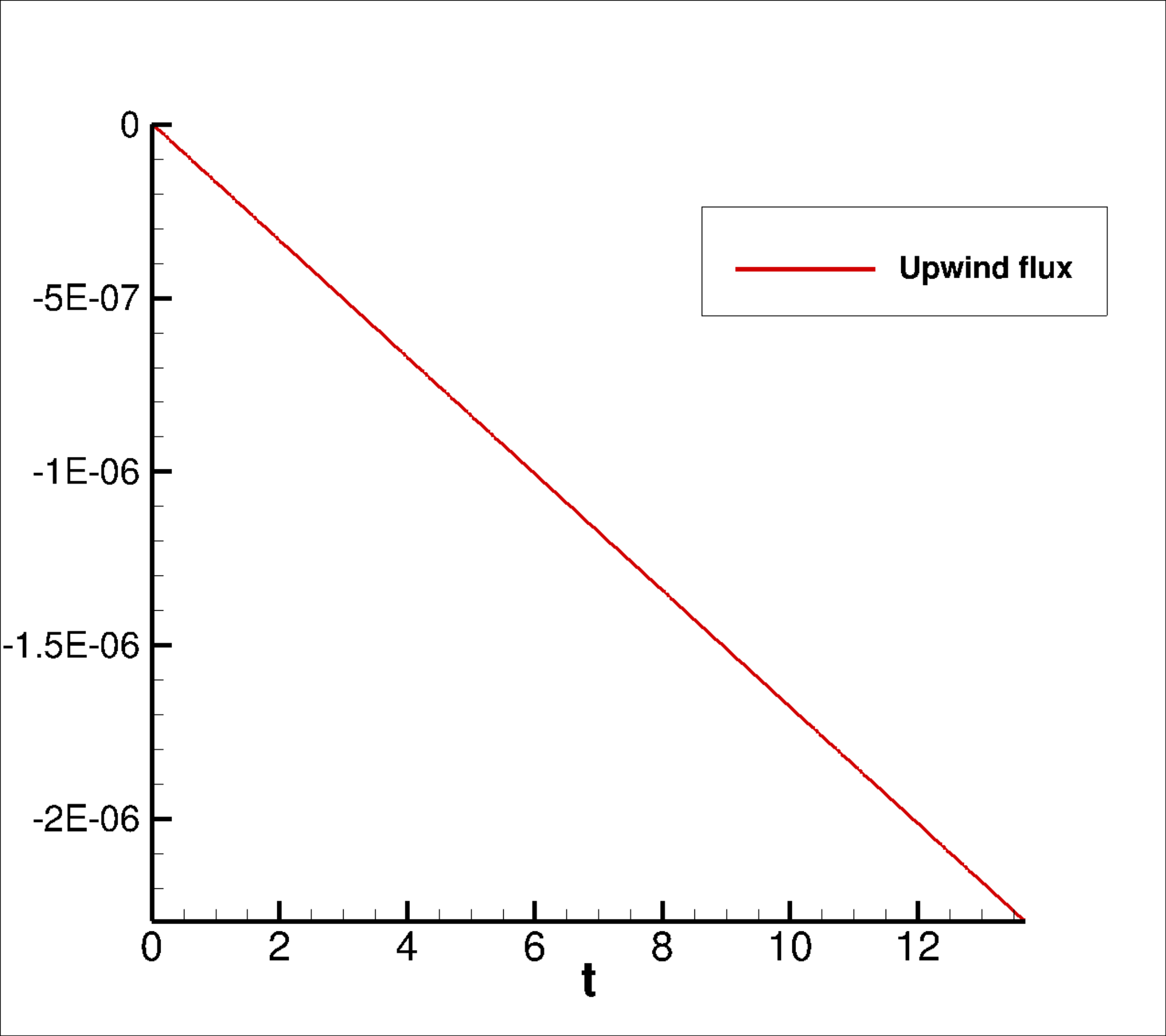}\label{Fig2.1u}}
	\subfigure[Leap-frog scheme. $k=2$.]{
		\includegraphics[width=0.28\textwidth]{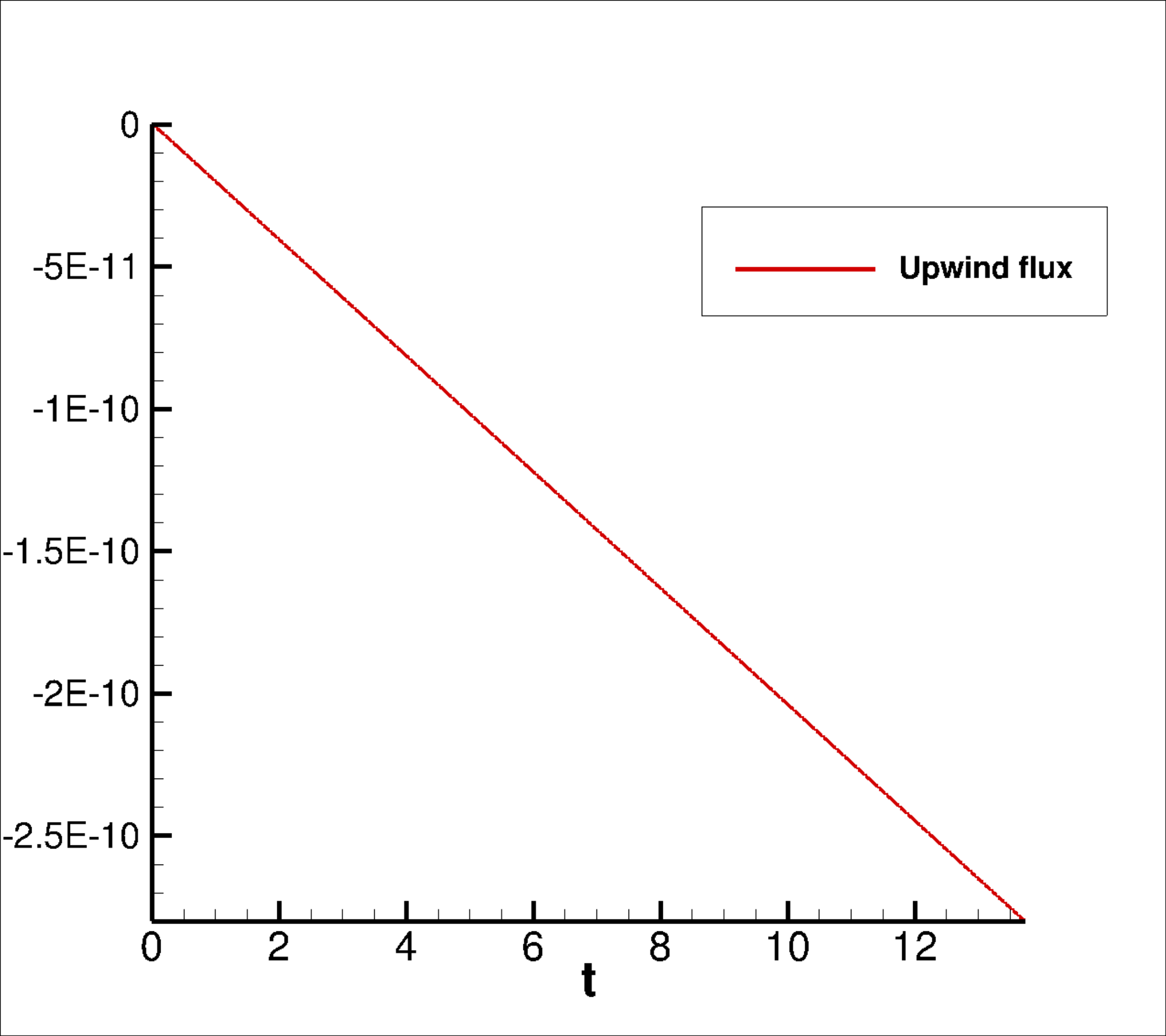}\label{Fig2.3u}}
	\subfigure[Leap-frog scheme. $k=3$.]{
		\includegraphics[width=0.28\textwidth]{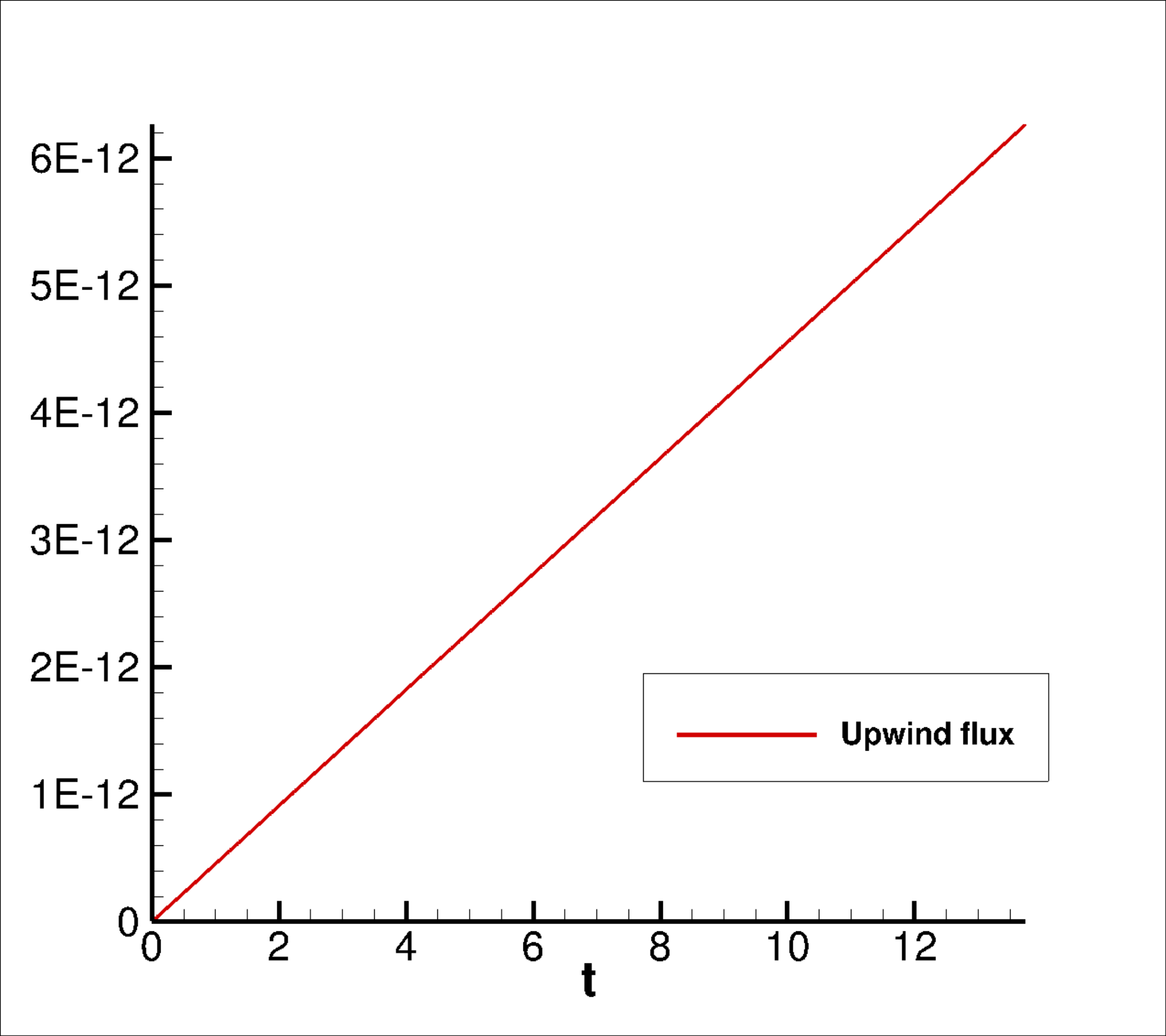}\label{Fig2.5u}}
	\subfigure[Fully implicit scheme. $k=1$.]{
		\includegraphics[width=0.28\textwidth]{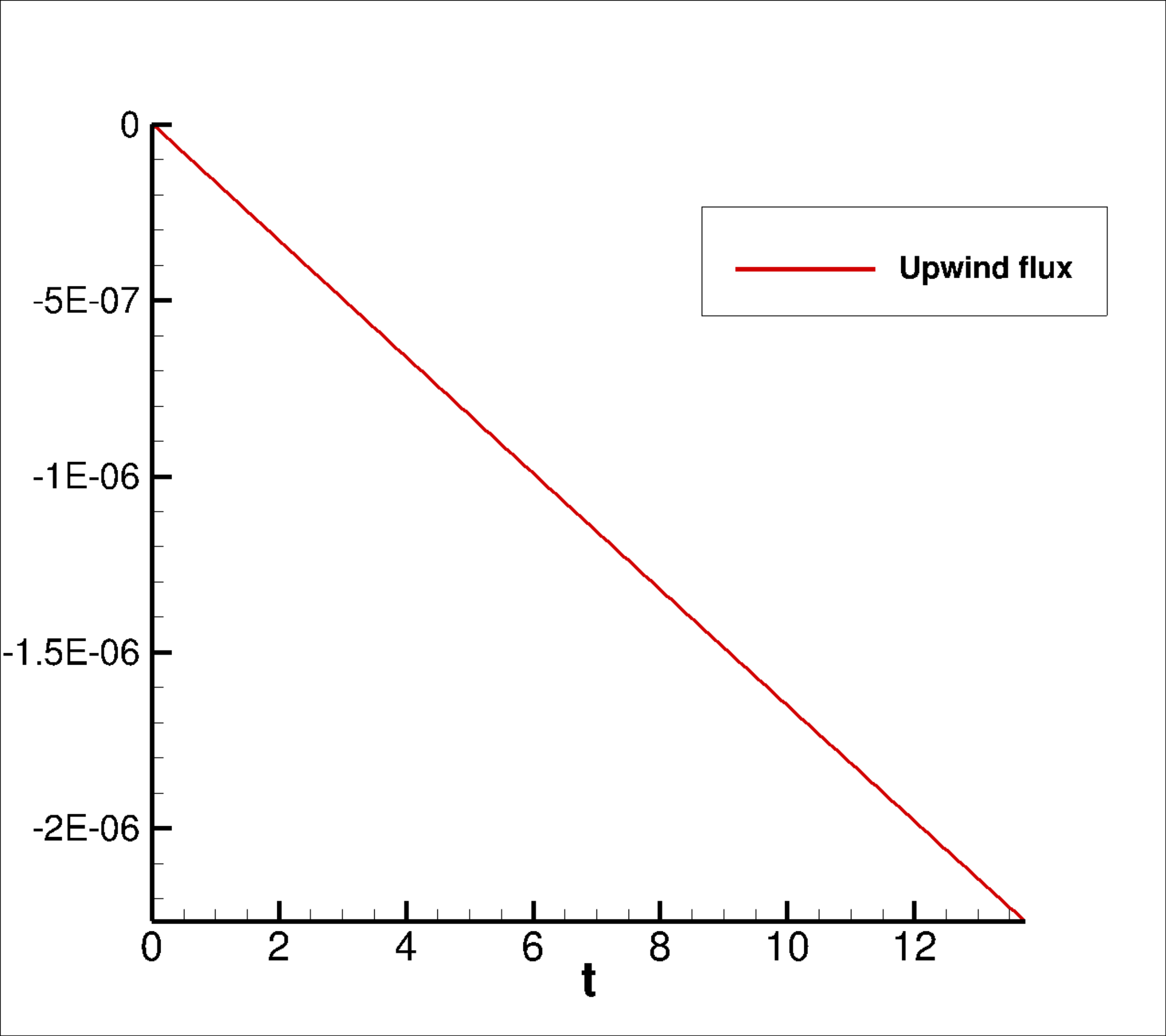}\label{Fig2.2u}}
	\subfigure[Fully implicit scheme. $k=2$.]{
		\includegraphics[width=0.28\textwidth]{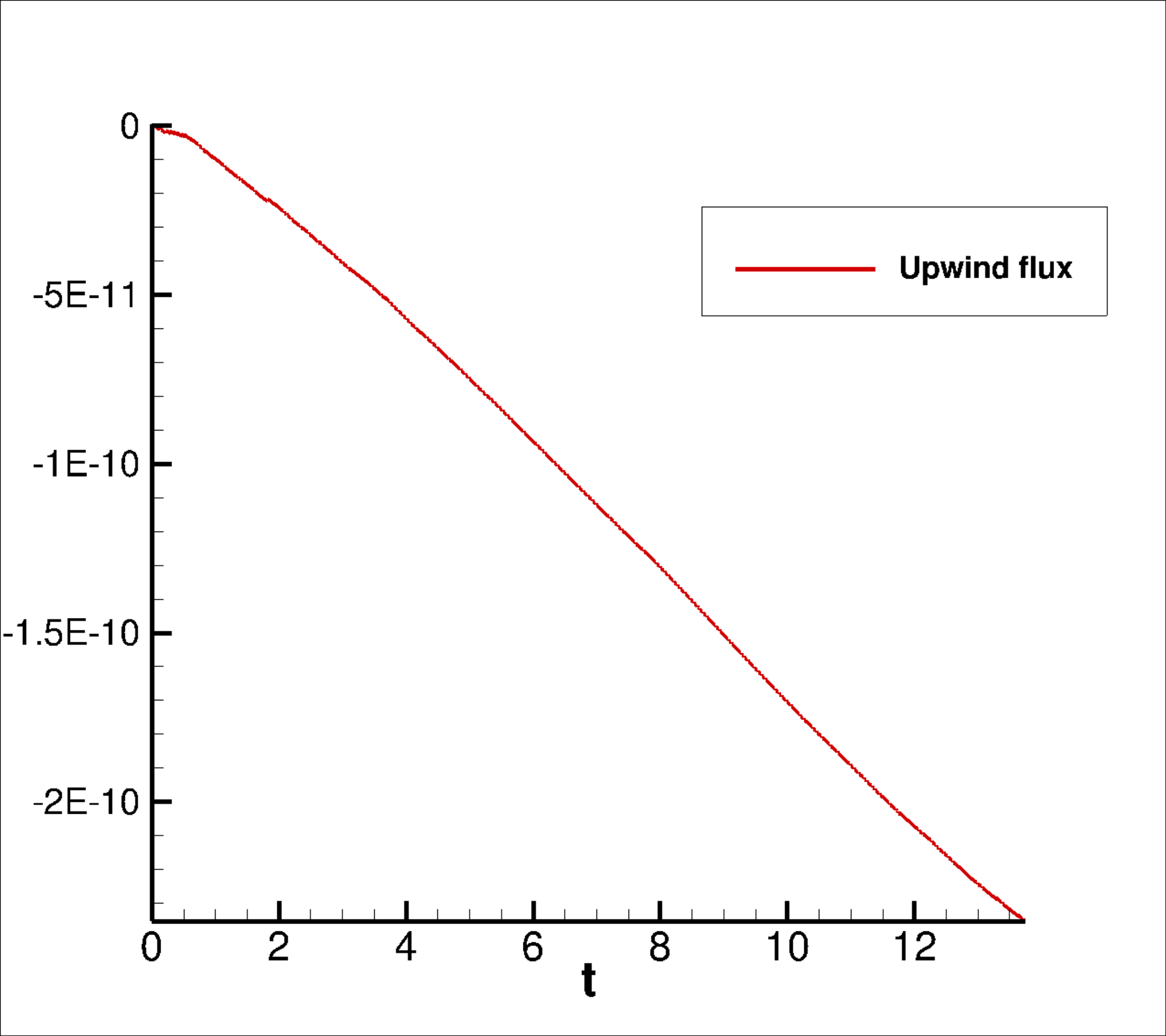}\label{Fig2.4u}}
	\subfigure[Fully implicit scheme. $k=3$.]{
		\includegraphics[width=0.28\textwidth]{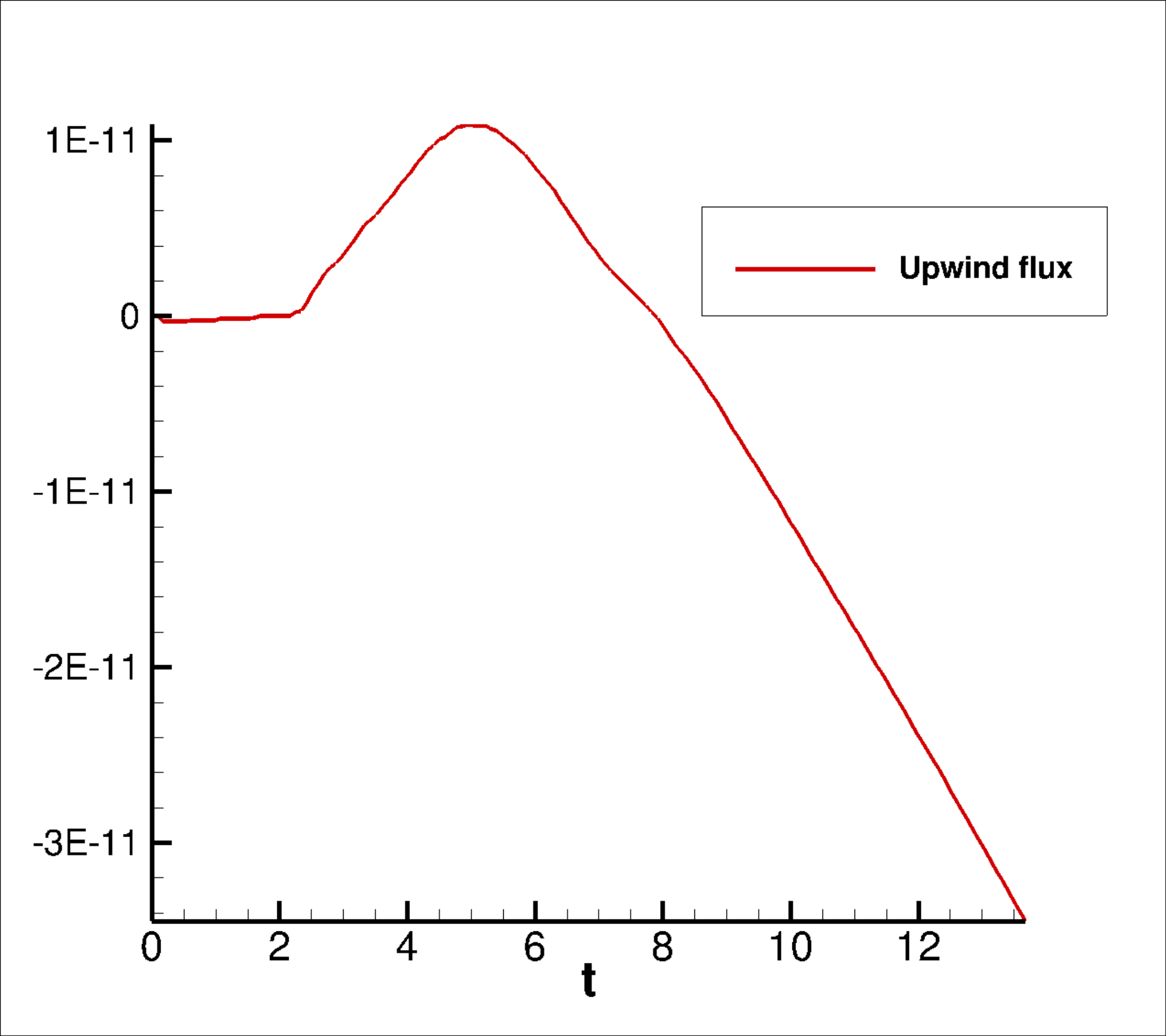}\label{Fig2.6u}}
	\caption{\em A traveling kink and antikink wave: the time evolution of the relative deviation in energy. $N=400$ grid points.  }
	\label{Fig2}
\end{figure}

\subsection{Soliton propagation}
\label{sec:num2}

In this example, we will consider the soliton propagation in the full Maxwell model \eqref{eq:scale}, similar to the setup in \cite{gilles2000comparison}. The computational domain is $x\in[0, 45].$
The coefficients in this example are chosen as
 \begin{align*}
 \epsilon_{\infty}=2.25, \ \ \
 \epsilon_{s} = 5.25, \ \ \
 \beta_{1} = \epsilon_{s} - \epsilon_{\infty}, \\
 1/\tau = 1.168\times 10^{-5}, \ \ \
 1/\tau_{v} = 29.2/32, \\
  a = 0.07,\ \ \
  \theta = 0.3, \ \ \
\Omega_{0}=12.57,\\
\omega_{0} = 5.84, \ \ \
\omega_{v} = 1.28, \ \ \
\omega_{p} =\omega_{0}\sqrt{\beta_{1}}.
 \end{align*}

Initially, all fields are zero.  The left  boundary is  injected with  an incoming solitary wave, for which the electric field is prescribed as
\begin{align}
E(x=0,t)=f(t)\cos(\Omega_{0}t),
\end{align}
where $f(t)=M\ \mathrm{sech}(t-20)$. $M$ is a constant to be specified later. Similar to \cite{gilles2000comparison}, the boundary condition of $H$ can be approximated from the linearized dispersion relation. Assuming a space-time harmonic variation $e^{i(\omega t-kx)}$ of all fields, the exact dispersion relation associated with the linear parts of the system \eqref{eq:scale} is
\begin{align}
\epsilon_{\infty} \omega^4 - i\frac{\epsilon_{\infty}}{\tau}\omega^3 - (\epsilon_{\infty}\omega_{0}^2+\omega_{p}^2+k^2)\omega^2 +  i \frac{1}{\tau} k^2 \omega +k^2\omega_{0}^2=0.
\end{align}
The solution corresponding to the wave propagating to the right is
\begin{align}
k=\omega \sqrt{\epsilon_{\infty}} \sqrt{1-\frac{\omega_{p}^2/\epsilon_{\infty}} {\omega^2-i\omega/\tau-\omega_{0}^2}}.
\end{align}
Then we take the approximate value of $H$ as
\begin{align}
H(x=0,t)=&\int_{-\infty}^{\infty}\hat{H}(\omega)e^{i\omega t}d\omega\nonumber\\
\simeq& \frac{1}{2}\big[ \sum_{m=0}^{8} \frac{(-i)^{m}}{m!} (\frac{1}{Z})^{(m)}|_{\omega=\Omega_0} f^{(m)}(t) \big] e^{i\Omega_{0}t} + c.c.,
\end{align}
where $c.c.$ denotes the complex conjugate of the first term, $f^{(m)}(t)$ is the $m$-th derivative of $f(t)$, and $ (\frac{1}{Z})^{(m)}$ is the $m$-th derivative of $Z=-\omega/k$ with respect to $\omega$.

We treat the right boundary as an absorbing wall corresponding to the linearized system, similar to the procedure performed in \cite{hile1996numerical}. Neglecting the nonlinear effects and the delayed response in \eqref{eq:scale}, we have
\begin{align*}
& \partial_{t}(H+\sqrt{\epsilon_{\infty}}E) = \frac{1}{\sqrt{\epsilon_{\infty}}}\partial_{x}(H+\sqrt{\epsilon_{\infty}}E)\\
& \partial_{t}(H-\sqrt{\epsilon_{\infty}}E) = -\frac{1}{\sqrt{\epsilon_{\infty}}}\partial_{x}(H-\sqrt{\epsilon_{\infty}}E).
\end{align*}
Because only waves that propagate to the right are allowed, the left going characteristic variable $H+\sqrt{\epsilon_{\infty}}E$ is set to be zero at the right boundary $x_{R}=x_{N+1/2}$. Therefore, for semi-discrete scheme, we require
\begin{align*}
& (H_{h}+\sqrt{\epsilon_{\infty}}E_{h})^{+}_{N+1/2} = 0,\\
& (H_{h}-\sqrt{\epsilon_{\infty}}E_{h})^{+}_{N+1/2} = (H_{h}-\sqrt{\epsilon_{\infty}}E_{h})^{-}_{N+1/2}.
\end{align*}
This corresponds to rewriting the central flux as
\begin{equation}
	\label{eq:bccn}
\widehat{E_{h}}|_{N+1/2} = \frac{3}{4}E_{h}|^{-}_{N+1/2} - \frac{1}{4\sqrt{\epsilon_{\infty}}}H_{h}|^{-}_{N+1/2},\quad
\widetilde{H_{h}}|_{N+1/2} = \frac{3}{4}H_{h}|^{-}_{N+1/2} - \frac{\sqrt{\epsilon_{\infty}}}{4}E_{h}|^{-}_{N+1/2},
\end{equation}
and rewriting the upwind flux as
\begin{equation}
\label{eq:bcup}
 \widehat{E_{h}}|_{N+1/2} = \frac{1}{2}E_{h}|^{-}_{N+1/2} - \frac{1}{2\sqrt{\epsilon_{\infty}}}H_{h}|^{-}_{N+1/2},\quad
\widetilde{H_{h}}|_{N+1/2} = \frac{1}{2}H_{h}|^{-}_{N+1/2} - \frac{\sqrt{\epsilon_{\infty}}}{2}E_{h}|^{-}_{N+1/2}.
\end{equation}
To guarantee   better stability results for the outflowing edge, when using alternating fluxes, we employ the central  flux \eqref{eq:bccn} at the right boundary instead. With this boundary condition,   the energy relation such as those in Theorem \ref{thm:semi} should be adjusted accordingly. For example, we can verify that the semi-discrete scheme with alternating and central fluxes satisfy
\begin{eqnarray}
\frac{d}{dt}\mathcal{E}_h&=&-\frac{1}{\omega_p^2 \tau} \int_\Omega J_h^2 dx-\frac{a\theta}{2\omega_v^2 \tau_v}  \int_\Omega \sigma_h^2 dx
 - \frac{1}{4\sqrt{\epsilon_{\infty}}} ( (H_{h}-\sqrt{\epsilon_\infty}E_{h}) ^{-}_{N+1/2} )^2 -\Theta_{in}
\le -\Theta_{in}
\end{eqnarray}
with energy
	\begin{equation}
	\label{eq:enedefc}
	\mathcal{E}_h=\int_{\Omega}  \left (\frac{1}{2} H_h^2 + \frac{\epsilon_\infty}{2} E_h^2 +  \frac{1}{2\omega_p^2} J_h^2   + \frac{\omega_0^2}{2 \omega_p^2}   P_h^2+ \frac{a\theta}{4\omega_v^2} \sigma_h^2 + \frac{a\theta}{2}  Q_h E_h^2  + \frac{3 \epsilon_0 a (1-\theta)}{4} E_h^4+\frac{a\theta}{4}Q_h^2\right) dx,
	\end{equation}
and the contribution from the inflow boundary
\begin{align}
\label{eq:theta1}
\Theta_{in}=\left\{\begin{array}{ll}
\frac{1}{2} E(0,t)H_{h}|^+_{1/2}+ \frac{1}{2} H(0,t)E_{h}|^+_{1/2}, & \mbox{ for central flux},\\
H(0,t)E_{h}|^+_{1/2}, & \mbox{ for alternating flux I},\\
E(0,t)H_{h}|^+_{1/2}, & \mbox{ for alternating flux II}.\\
\end{array}
\right.
\end{align}
The scheme with the upwind flux satisfies
\begin{eqnarray}
\frac{d}{dt}\mathcal{E}_h&=&-\frac{1}{\omega_p^2 \tau} \int_\Omega J_h^2 dx-\frac{a\theta}{2\omega_v^2 \tau_v}  \int_\Omega \sigma_h^2 dx
-\frac{1}{2\sqrt{\epsilon_\infty}}\sum_{j=1}^{N-1}[H_h]_{j+1/2}^2
-\frac{\sqrt{ \epsilon_\infty}}{2}\sum_{j=1}^{N-1}[E_h]_{j+1/2}^2  \\\notag
&&- \frac{1}{2\sqrt{\epsilon_\infty}}  (H_{h}|^{-}_{N+1/2})^2
- \frac{\sqrt{\epsilon_\infty}}{2} ( E_{h}|^{-}_{N+1/2} )^2 -\Theta_{in} 
\le -\Theta_{in},\notag
\end{eqnarray}
with the same energy definition as in \eqref{eq:enedefc} and
\begin{align*}
\Theta_{in}=&\frac{1}{2} (E(0,t)H_{h}|^+_{1/2}+ H(0,t)E_{h}|^+_{1/2})
+\frac{1}{2\sqrt{\epsilon_\infty}}H_{h}|^{+}_{1/2}[H_{h}]_{1/2}
+\frac{\sqrt{\epsilon_\infty}}{2}E_{h}|^{+}_{1/2}[E_{h}]_{1/2}  \\
=&\frac{1}{4\sqrt{\epsilon_\infty}}\left(H_{h}|^+_{1/2}+\sqrt{\epsilon_\infty}E(0,t)\right)^2+\frac{1}{4\sqrt{\epsilon_\infty}}\left(H(0,t)+\sqrt{\epsilon_\infty}E_{h}|^+_{1/2}\right)^2+\frac{1}{4\sqrt{\epsilon_\infty}}[H_{h}]_{1/2}^2+\frac{\sqrt{\epsilon_\infty}}{4}[E_{h}]_{1/2}^2\\
&-\frac{1}{2\sqrt{\epsilon_\infty}}H(0,t)^2-\frac{\sqrt{\epsilon_\infty}}{2}E(0,t)^2.
\end{align*}
Therefore,
$$
\frac{d}{dt}\mathcal{E}_h \le \frac{1}{2\sqrt{\epsilon_\infty}}H(0,t)^2+\frac{\sqrt{\epsilon_\infty}}{2}E(0,t)^2,
$$
implying energy stability.

For the fully discrete schemes, there is no ambiguity defining the fluxes \eqref{eq:bccn}, \eqref{eq:bcup} for implicit scheme.
While for the leap-frog formulations with the upwind flux \eqref{eq:bcup}, we use
\begin{align}
	&\widehat{E_h^n}|_{N+1/2} = \frac{1}{2}E_h^n|^{-}_{N+1/2}-\frac{1}{2\sqrt{ \epsilon_\infty}} H_h^{n+1/2}|^{-}_{N+1/2},\quad
	\widehat{\widehat{E_h^{n}}}|_{N+1/2} = \frac{1}{2}E_h^{n}|^{-}_{N+1/2}-\frac{1}{2\sqrt{ \epsilon_\infty}} H_h^{n-1/2}|^{-}_{N+1/2}, \notag \\
	&\widetilde{H_h^{n+1/2}}|_{N+1/2} = \frac{1}{2} H_h^{n+1/2}|^{-}_{N+1/2}-\frac{\sqrt{\epsilon_\infty}}{2}\frac{E_h^n+E_h^{n+1}}{2} |^{-}_{N+1/2}. \label{eq:ffup}
\end{align}
While for the other  fluxes  \eqref{eq:bccn},
\begin{align}
&\widehat{E_h^n}|_{N+1/2} = \frac{3}{4}E_h^n|^{-}_{N+1/2}-\frac{1}{4\sqrt{ \epsilon_\infty}} H_h^{n+1/2}|^{-}_{N+1/2},\quad
\widehat{\widehat{E_h^{n}}}|_{N+1/2} = \frac{3}{4}E_h^{n}|^{-}_{N+1/2}-\frac{1}{4\sqrt{ \epsilon_\infty}} H_h^{n-1/2}|^{-}_{N+1/2}, \notag \\
&\widetilde{H_h^{n+1/2}}|_{N+1/2} = \frac{3}{4} H_h^{n+1/2}|^{-}_{N+1/2}-\frac{\sqrt{\epsilon_\infty}}{4}\frac{E_h^n+E_h^{n+1}}{2} |^{-}_{N+1/2}. \label{eq:ffcn}
\end{align}
Implementation-wise, with \eqref{eq:ffcn}, at the rightmost cell $I_{N}=[x_{N-1/2},x_{N+1/2}]$, we need to solve the nonlinear system to obtain $H_{h}$ by Newton's method. The energy relation for the resulting fully discrete scheme is summarized in the appendix.

We take $N=6400$,   and the time step
 $\Delta t=CFL\times h.$
CFL numbers, listed in Table \ref{tabadd1}, are chosen to ensure the convergence of Newton's method   to the correct solution.

\begin{table}[htb]
	\caption{\label{tabadd1}\em Soliton propagation: $CFL$ number. }
	\centering
		\begin{tabular}{|c|c|c|c|}
		\hline
			\multicolumn{2}{|c|}{Leap-frog scheme} &
			\multicolumn{2}{c|}{Fully implicit scheme} \\\hline
			Central/upwind flux &   Alternating flux  &  Central/upwind flux &   Alternating flux \\\hline
			0.05  &  0.1  &  0.3  &  0.5   \\\hline
		\end{tabular}
\end{table}

We simulate the transient fundamental ($M=1$) and second-order ($M=2$) temporal soliton evolution using various schemes with different orders. The plots of the electric field at $t=40, 80$ are provided in Figures \ref{Fig3}-\ref{Fig6}. As shown in \cite{hile1996numerical, gilles2000comparison}, a daughter pulse travels ahead the soliton-like pulse, resulting from the third-harmonic generation. This daughter pulse is much smaller in amplitude than the soliton pulse, and the frequency is about 3 times as that of the soliton pulse. The daughter pulse is evident in all simulations except with the upwind flux and $k=1$, where the numerical dissipation damps its magnitude significantly. Some reflections from the right boundary is present for the central flux. This is also observed in \cite{hile1996numerical} for the finite difference scheme due to the approximate boundary conditions. As a consequence, there will be spurious oscillation near the right boundary, especially for higher order scheme.  On the other hand, such oscillations are not observed for   alternating fluxes   or the upwind flux.

In Figures \ref{Fig7} - \ref{Fig10}, we  plot the transient evolution of the total energy and pulse area. Here, the pulse area is obtained by the composite trapezoidal rule between two extrema points of $E$. To distinguish the soliton pulse and the daughter pulse effectively, we only consider the soliton pulse area when $|E|\geq0.01$.
Numerical results represent high  agreement between the leap-frog scheme and the fully implicit scheme. Notice that we employ the approximation boundary condition $H(x=0,t)$, and the two alternating fluxes need different inflow information, which means one uses $E(x=0,t)$ and the other one uses $H(x=0,t)$. Hence, there is a slight discrepancy between the total energy with those two fluxes, as well as the pulse area. When using the central flux, both $E(x=0,t)$ and $H(x=0,t)$ are required, therefore the energy and pulse area calculated by the central flux stay in between the two alternating fluxes,  which is consistent with our analysis in \eqref{eq:theta1}. In Figures \ref{Fig7} and \ref{Fig8}, it is observed that the total energy decreases after the entire wave entering the domain, demonstrating the energy stability of the schemes. In particular, the energy calculated from upwind flux displays slightly more damping especially when $t$ is large and $k=1$.

 \begin{figure}
 	\centering
 	\subfigure{
 		\includegraphics[width=0.3\textwidth]{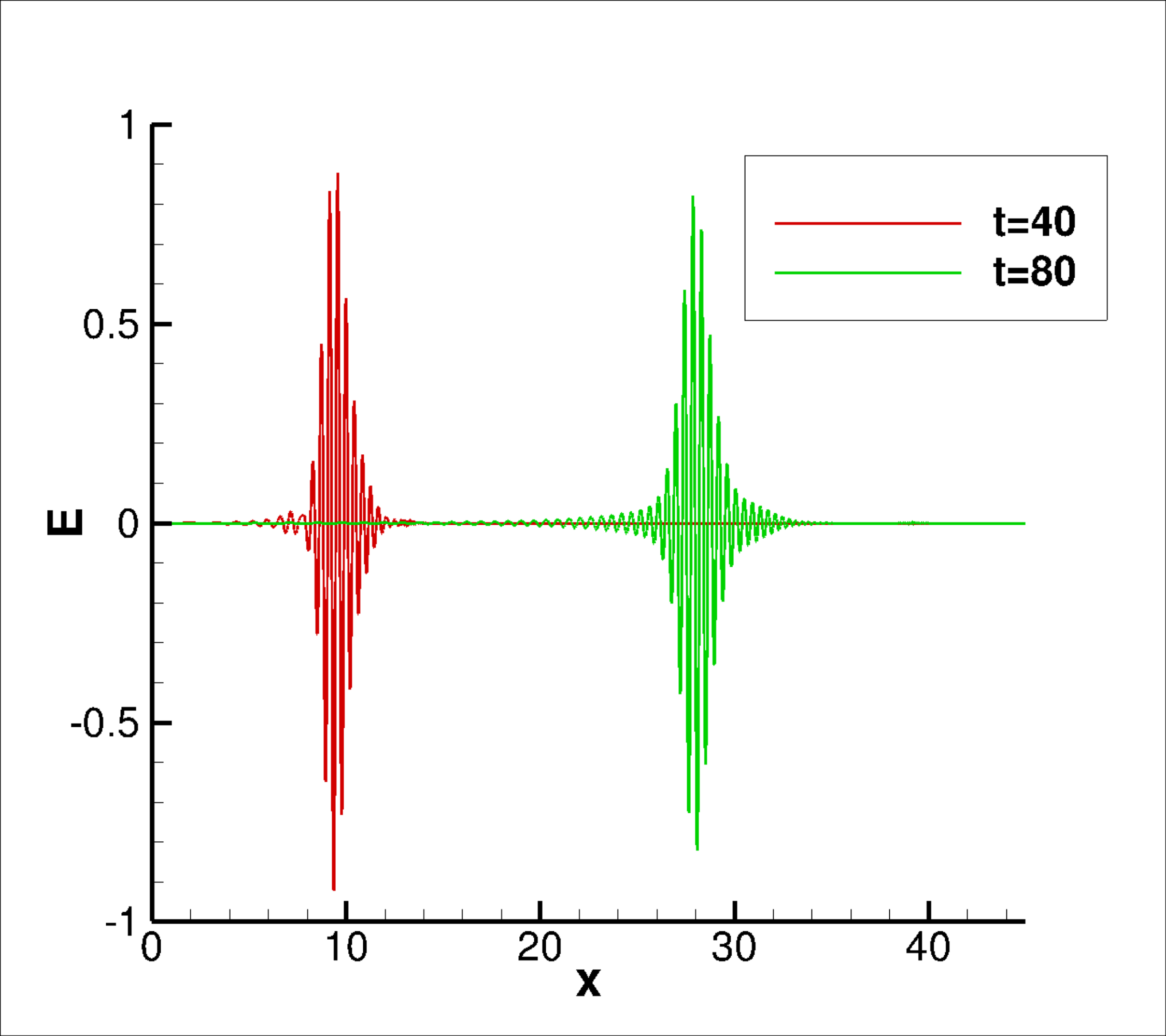}\label{Fig3.1}}
 	\subfigure{
 		\includegraphics[width=0.3\textwidth]{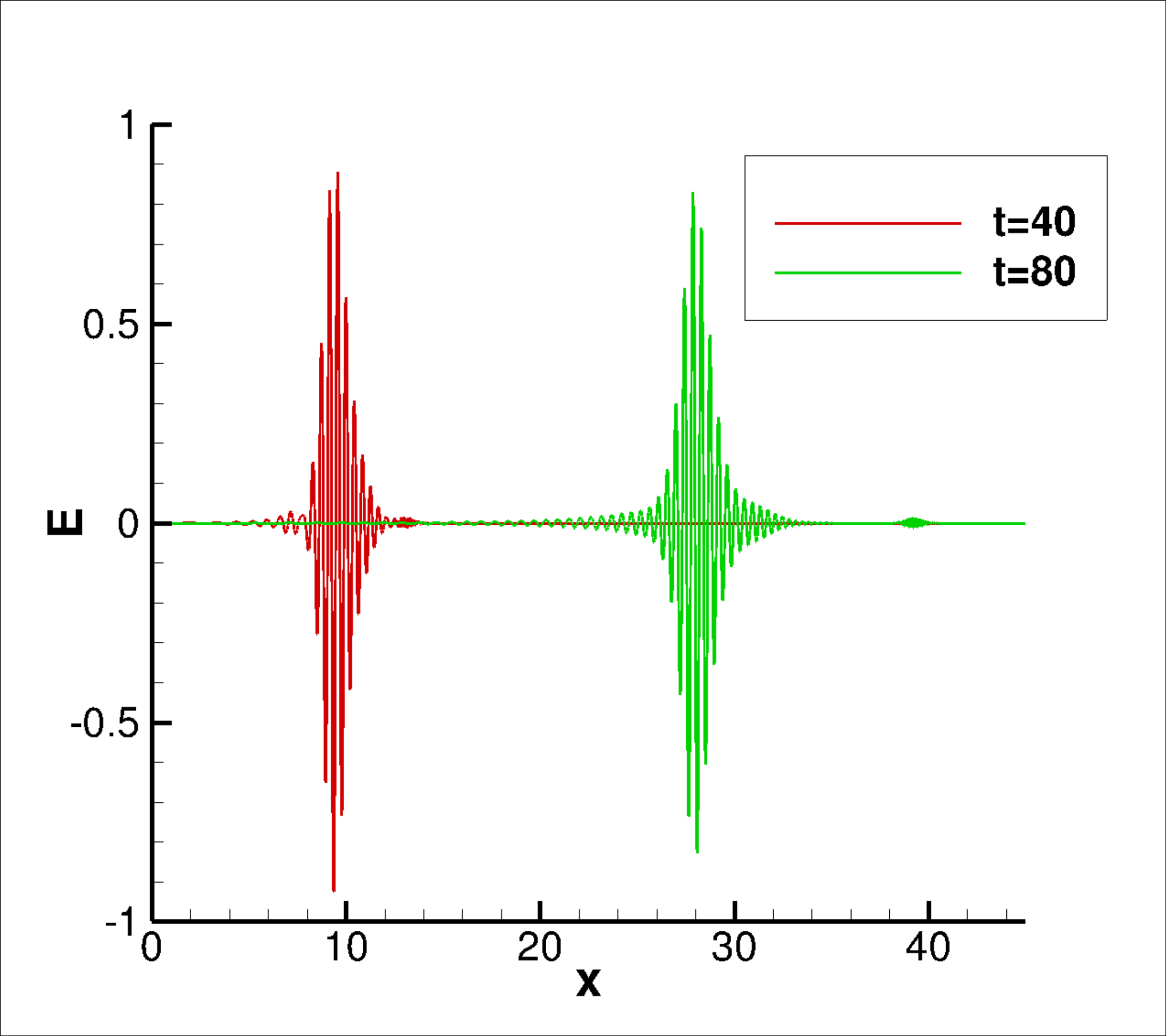}\label{Fig3.2}}
 	\subfigure{
 		\includegraphics[width=0.3\textwidth]{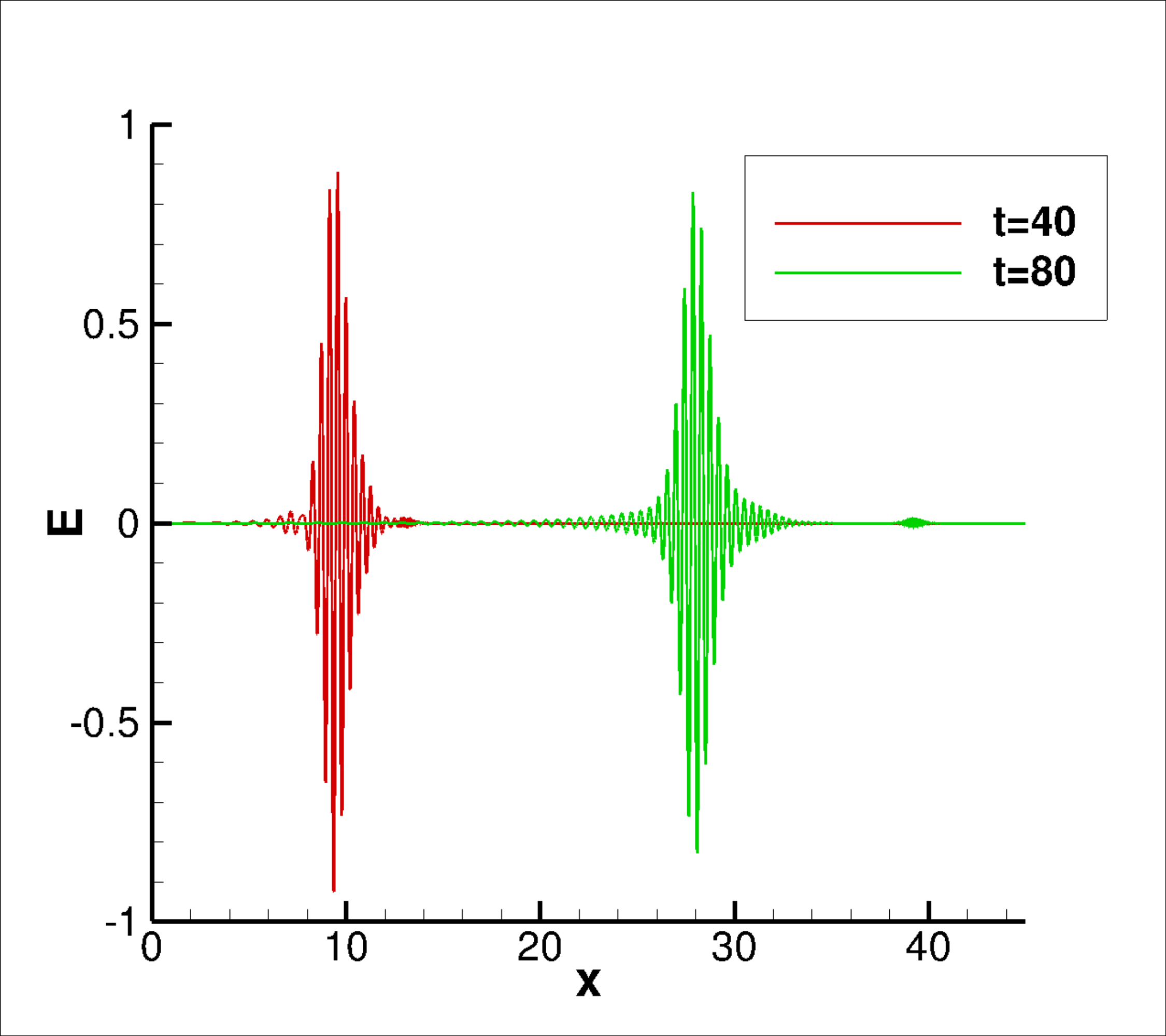}\label{Fig3.3}}
 	\subfigure{
 		\includegraphics[width=0.3\textwidth]{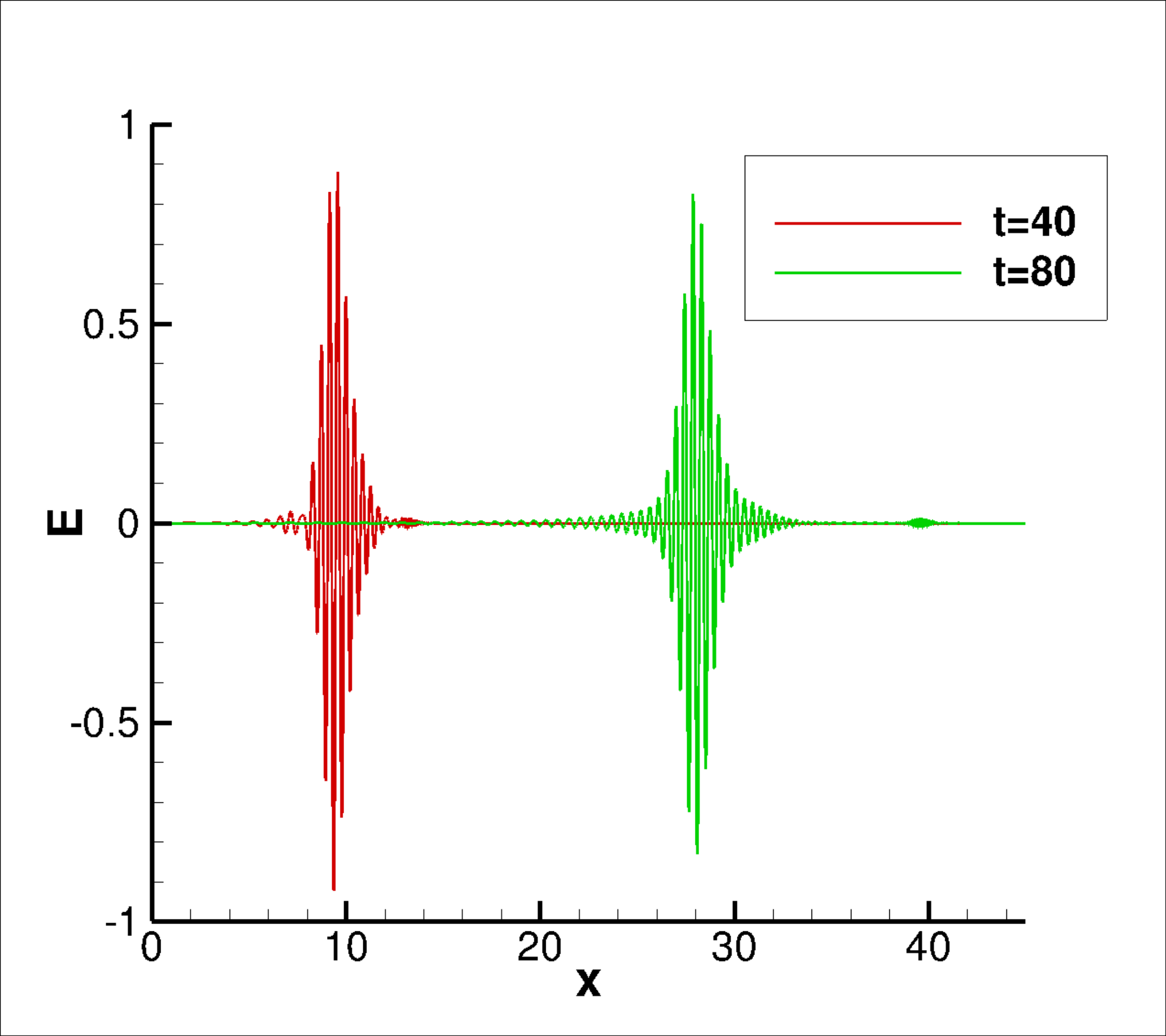}\label{Fig3.4}}
 	\subfigure{
 		\includegraphics[width=0.3\textwidth]{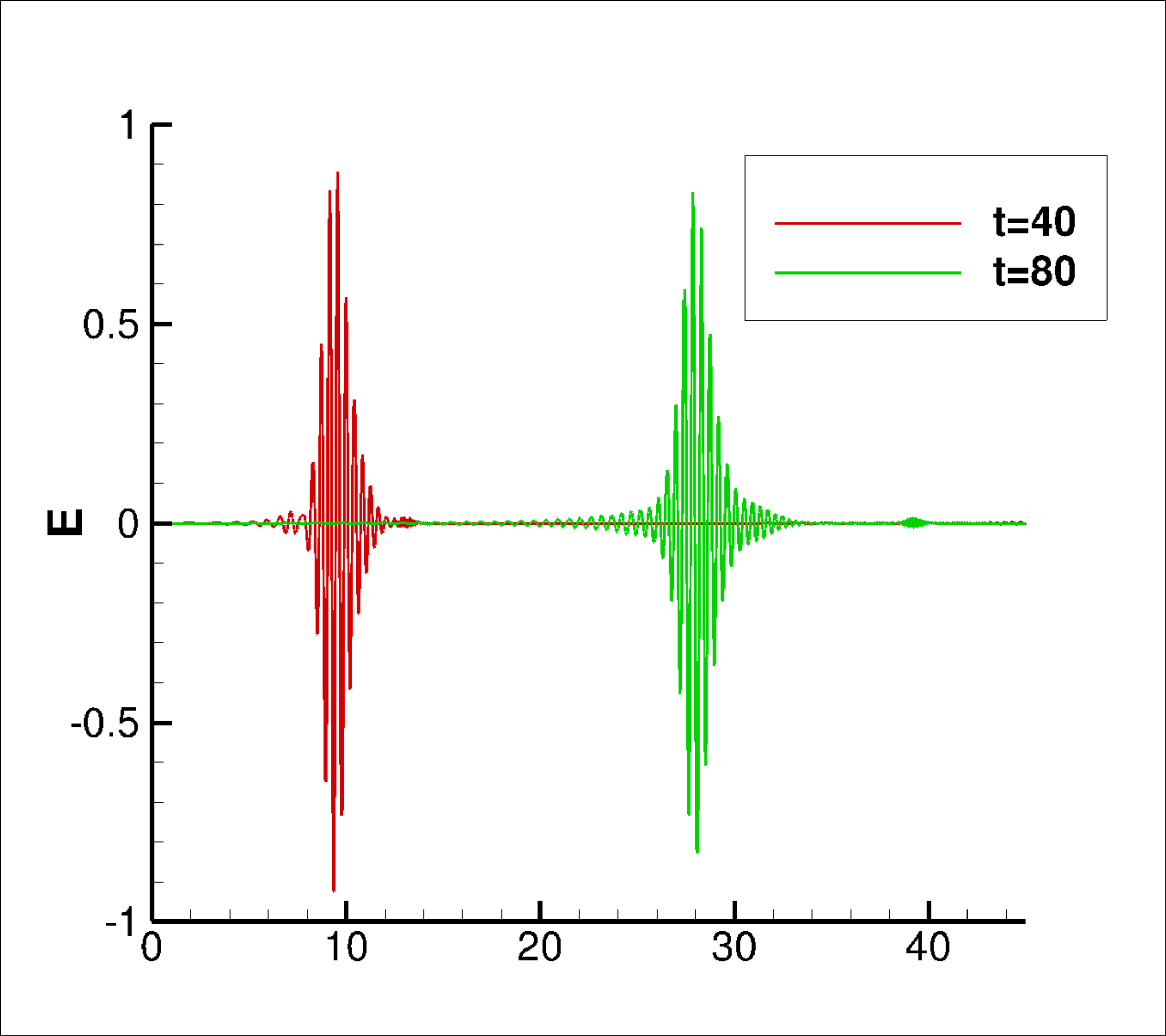}\label{Fig3.5}}
 	\subfigure{
 		\includegraphics[width=0.3\textwidth]{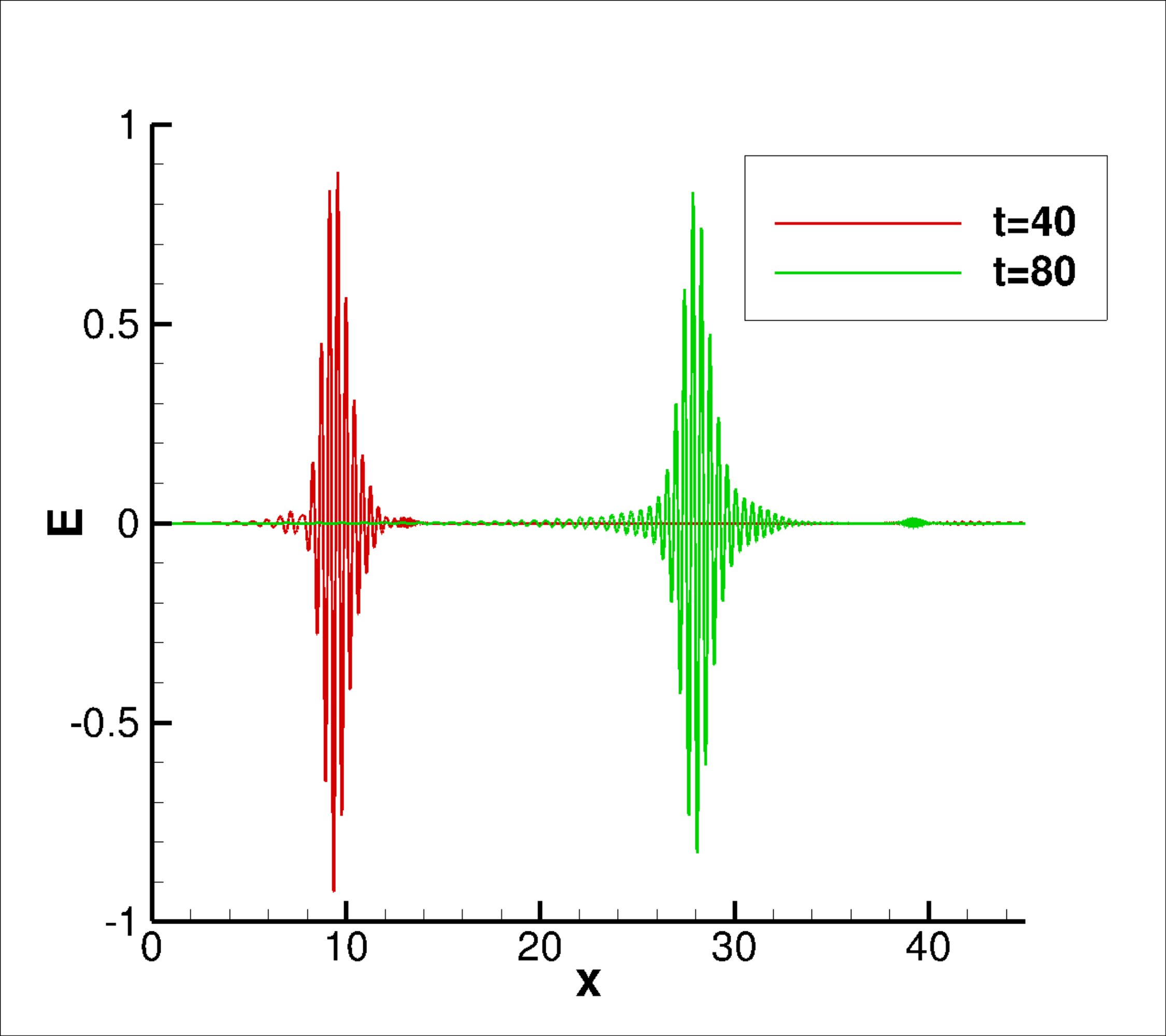}\label{Fig3.6}}
 	 \subfigure{
 		\includegraphics[width=0.3\textwidth]{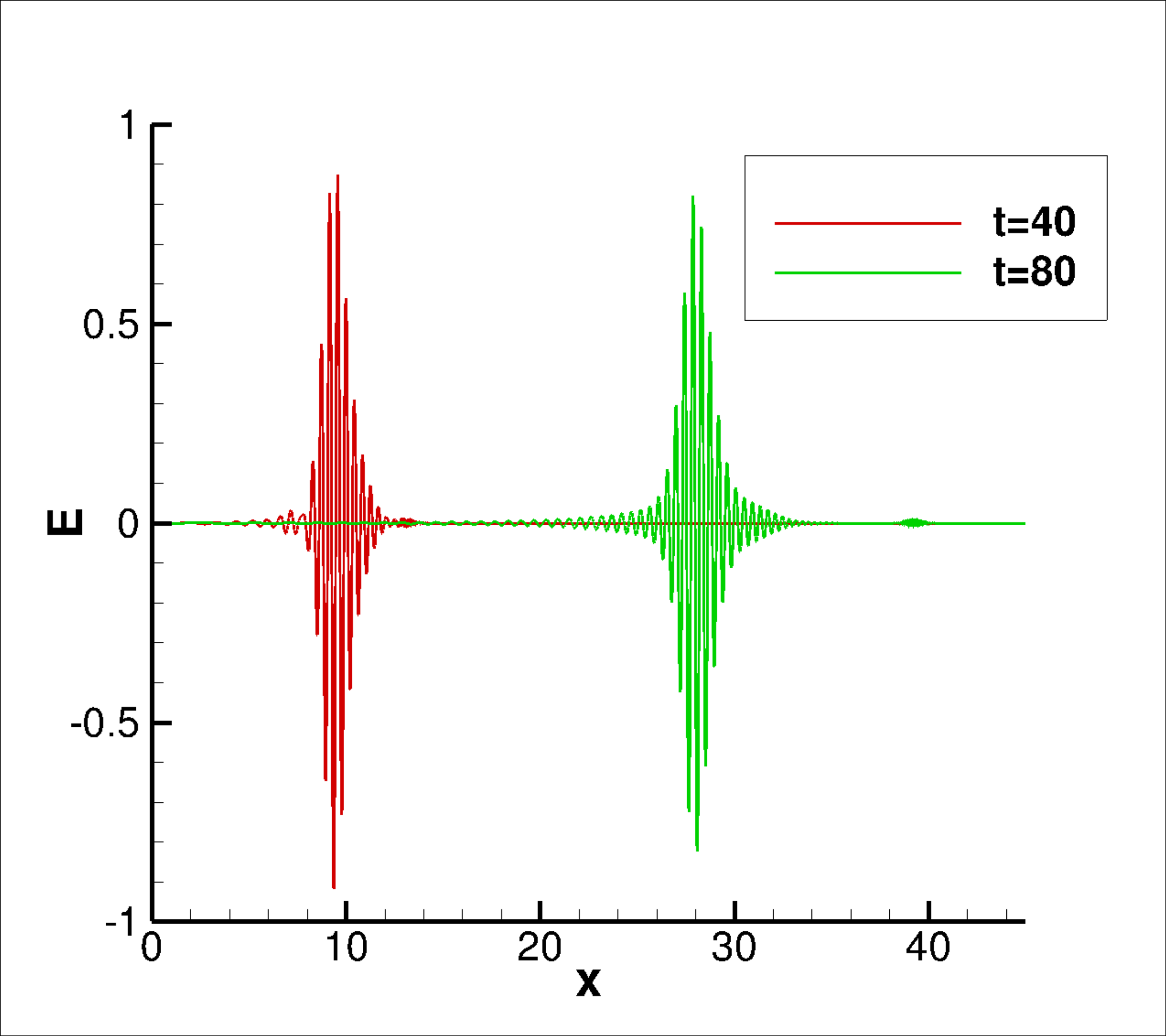}\label{Fig3.7}}
 	 \subfigure{
 	 	\includegraphics[width=0.3\textwidth]{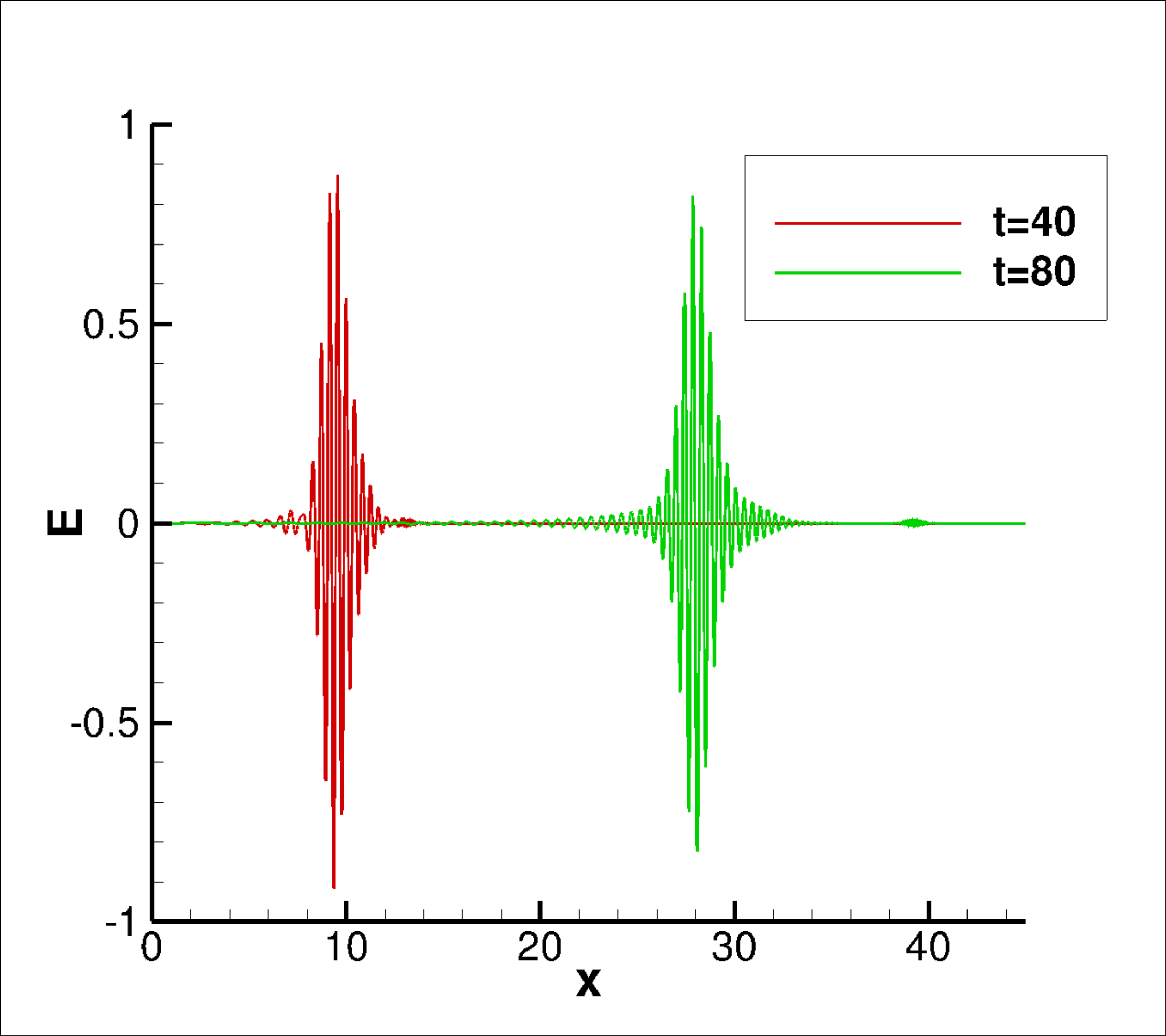}\label{Fig3.8}}
 	 \subfigure{
 	 	\includegraphics[width=0.3\textwidth]{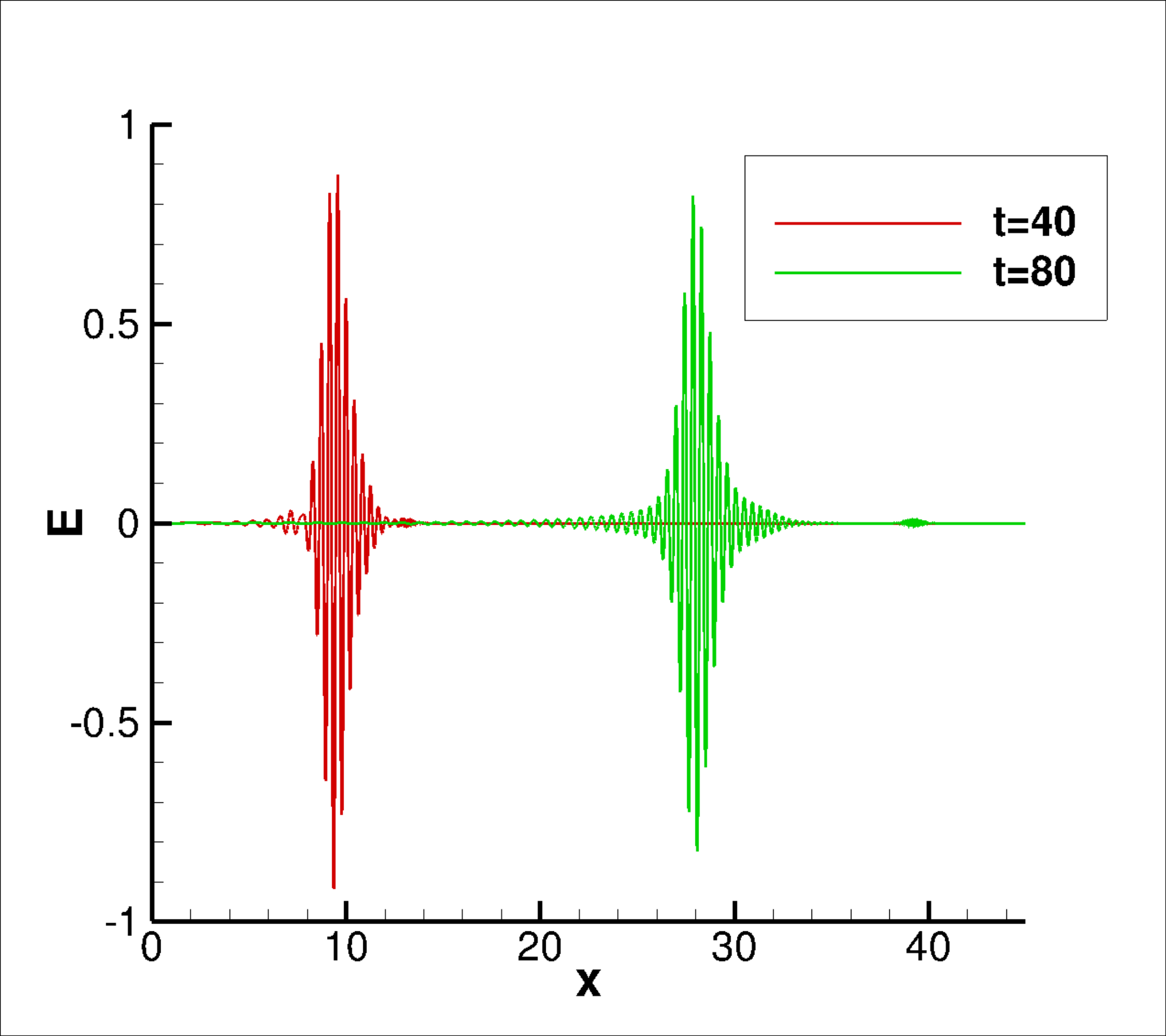}\label{Fig3.9}}
 	  \subfigure{
 	  	\includegraphics[width=0.3\textwidth]{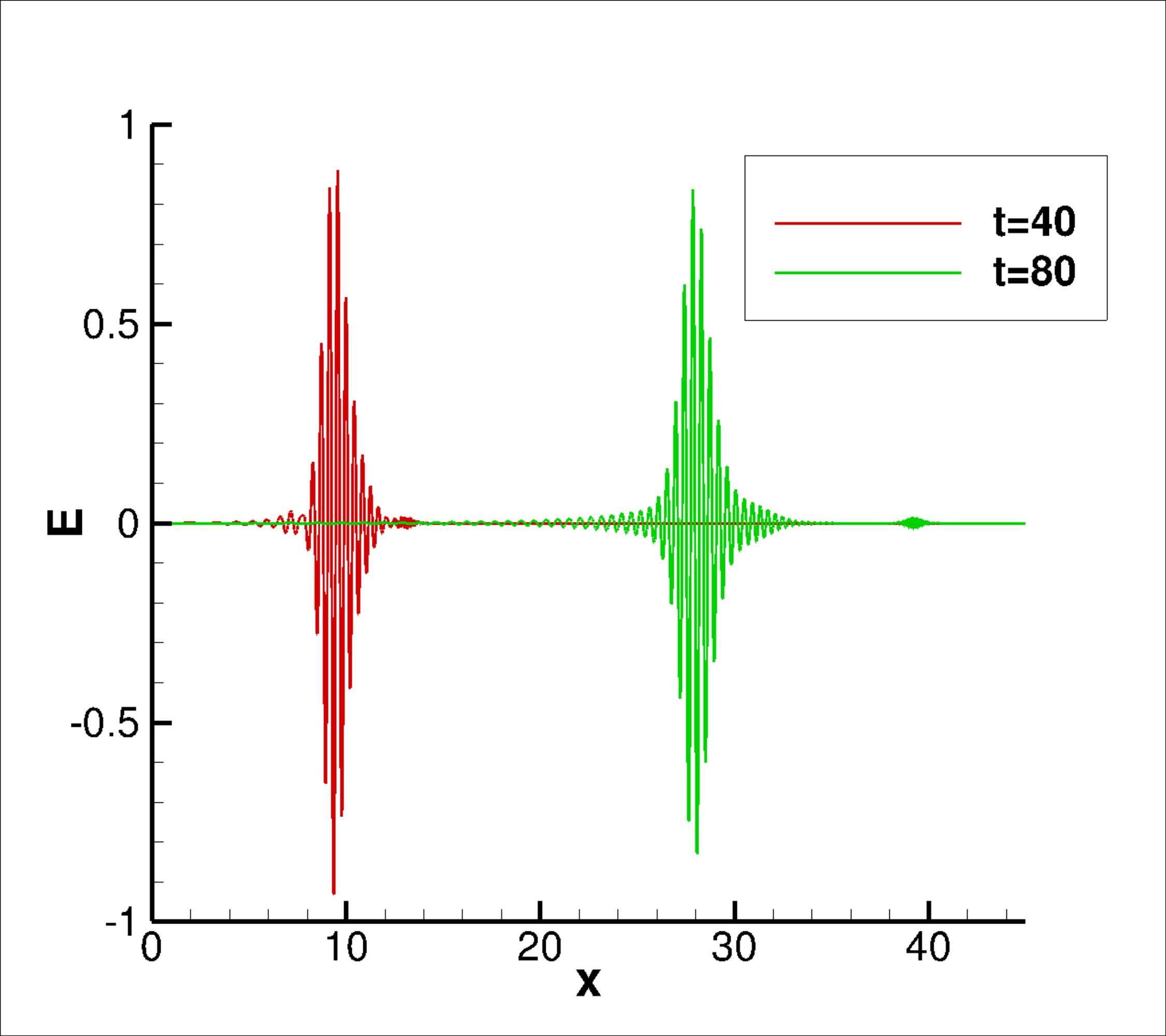}\label{Fig3.10}}
 	  \subfigure{
 	  	\includegraphics[width=0.3\textwidth]{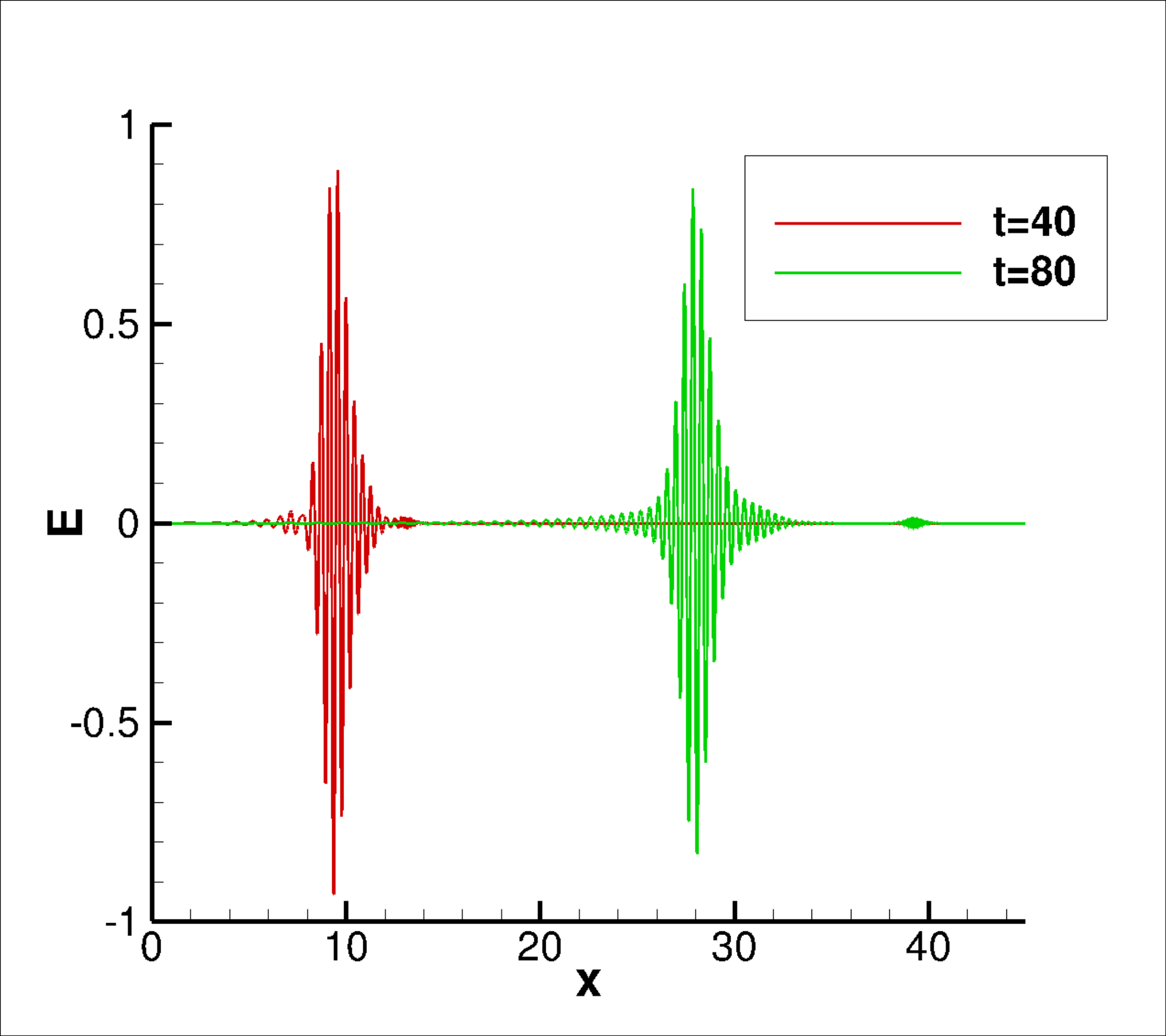}\label{Fig3.11}}
 	  \subfigure{
 	  	\includegraphics[width=0.3\textwidth]{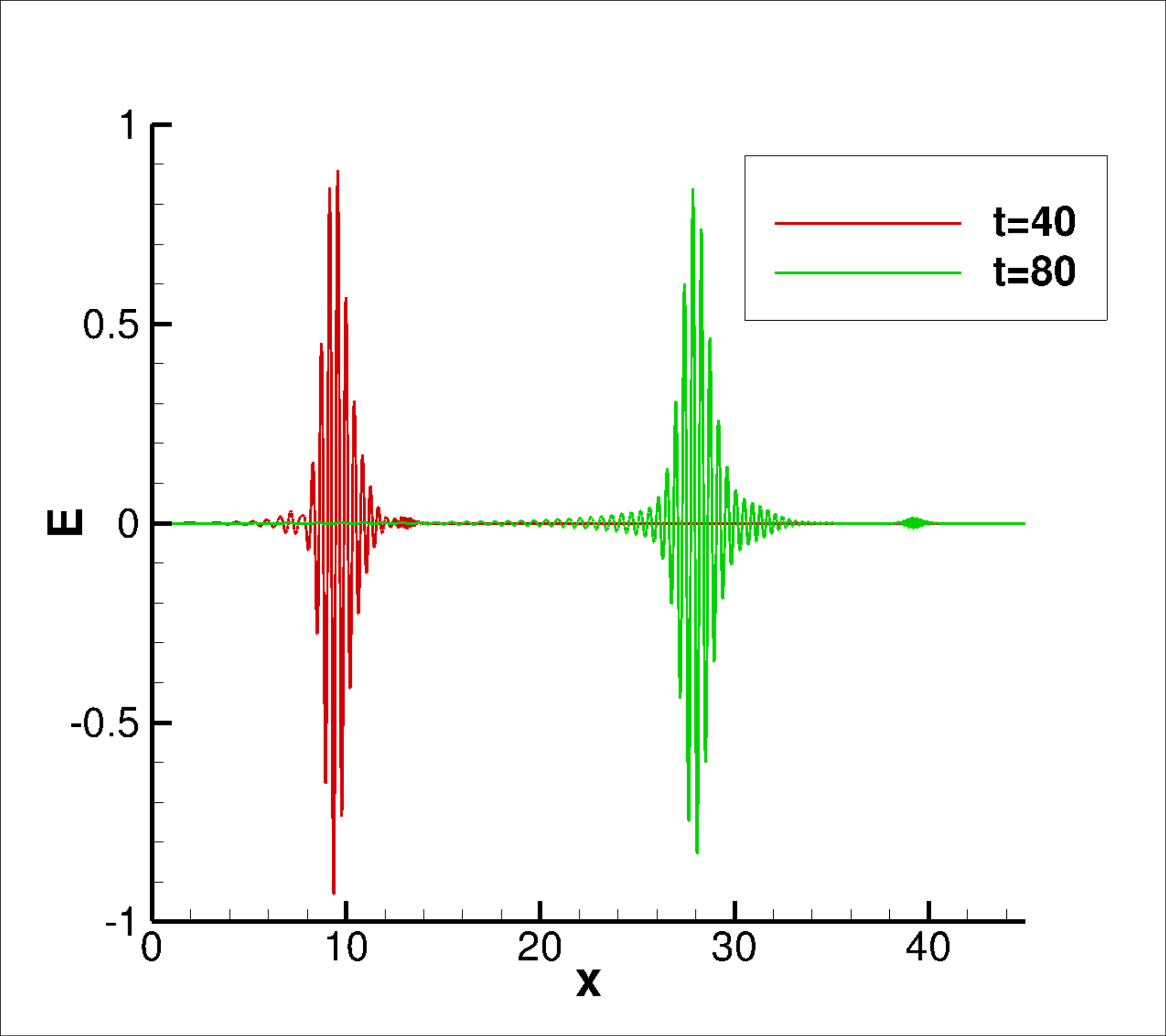}\label{Fig3.12}}
 	\caption{\em  Transient fundamental ($M=1$) temporal soliton propagation with the leap-frog scheme. $N=6400$ grid points. First column: $k=1$; second column: $k=2$; third column: $k=3$. First row: upwind flux;  second row: central flux; third row: alternating flux I; fourth row: alternating flux II.  }
 	\label{Fig3}
 \end{figure}

   \begin{figure}
 	\centering
 	\subfigure{
 		\includegraphics[width=0.3\textwidth]{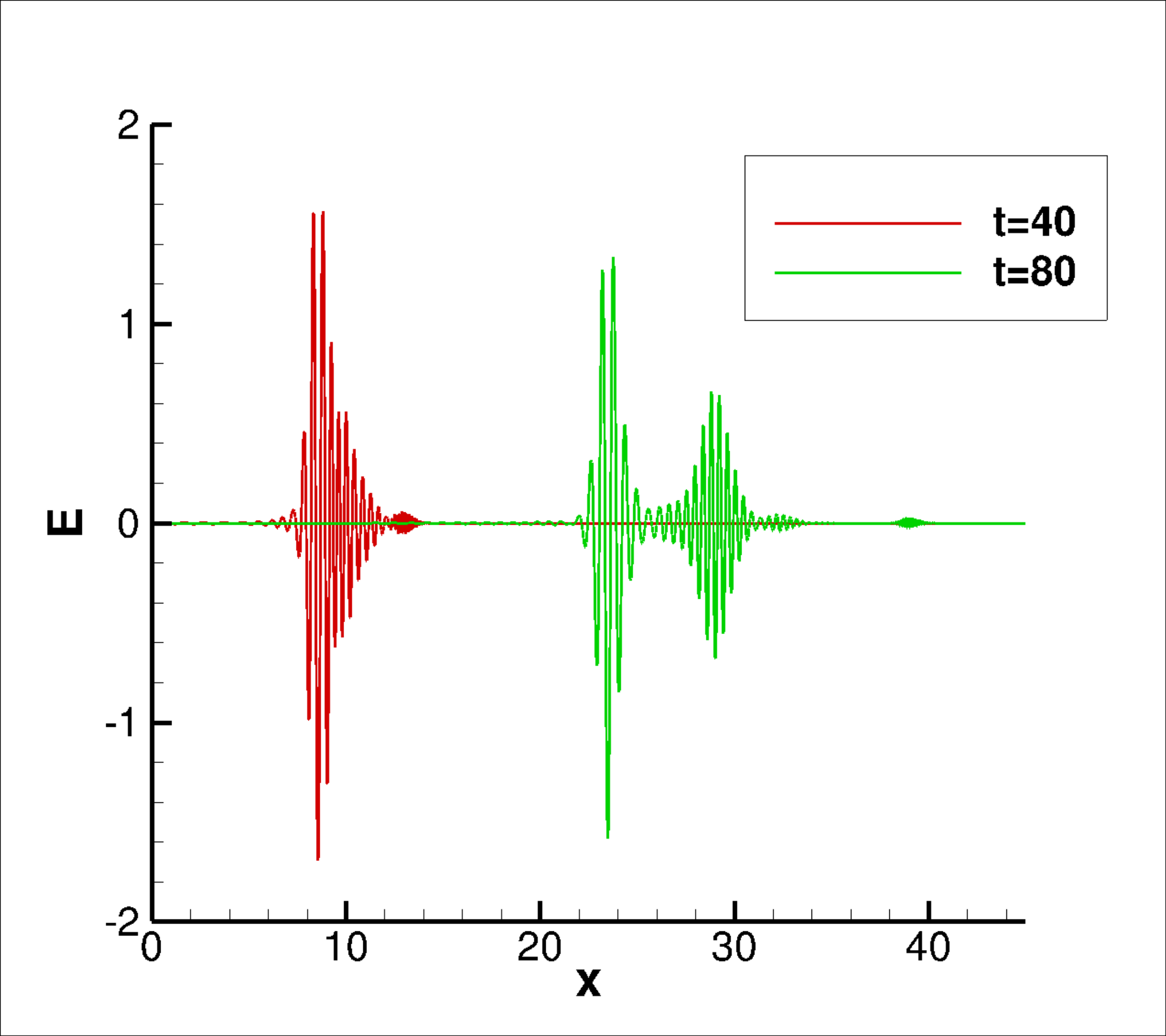}\label{Fig4.1}}
 	\subfigure{
 		\includegraphics[width=0.3\textwidth]{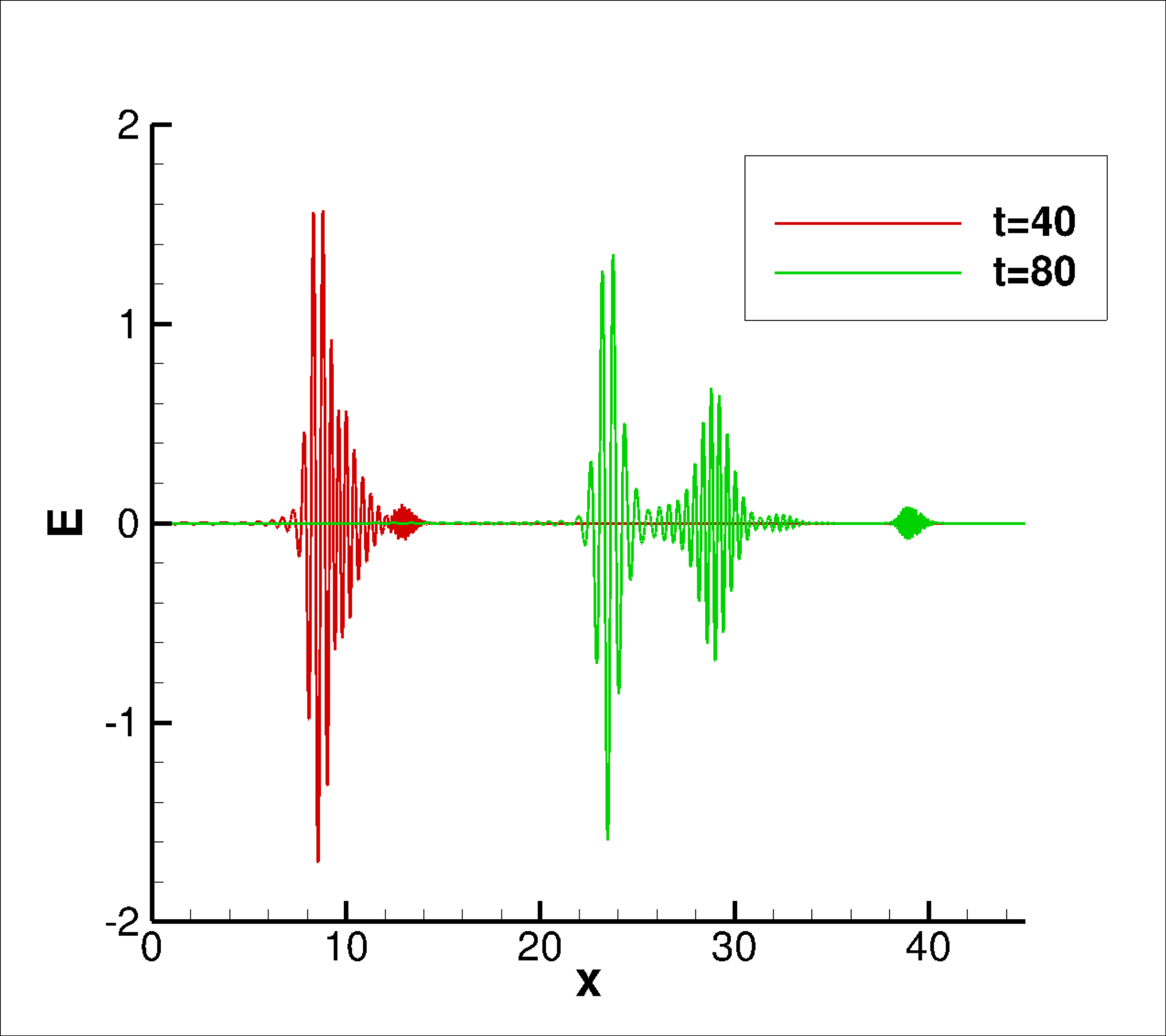}\label{Fig4.2}}
 	\subfigure{
 		\includegraphics[width=0.3\textwidth]{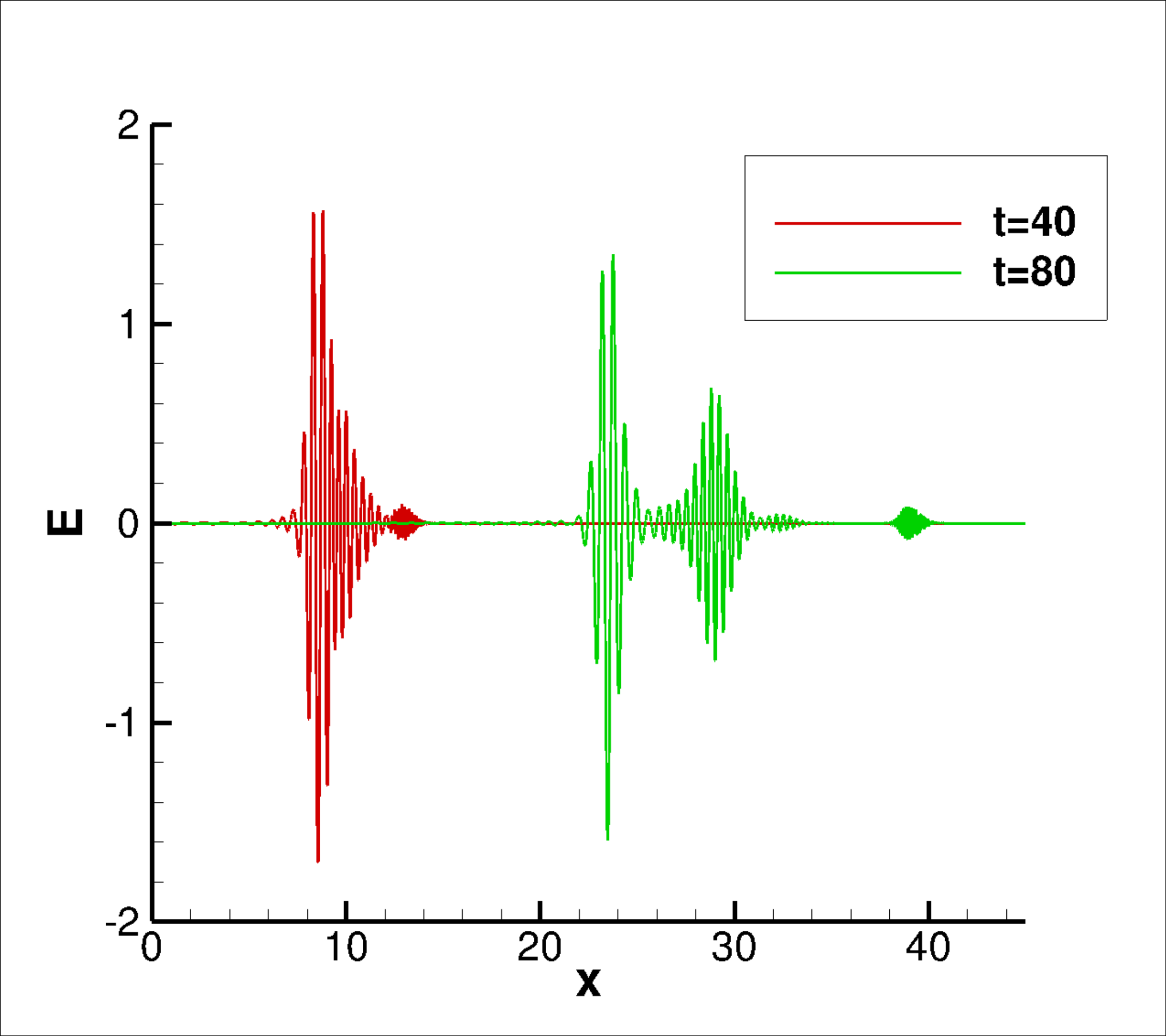}\label{Fig4.3}}\\
 	 \subfigure{
 		\includegraphics[width=0.3\textwidth]{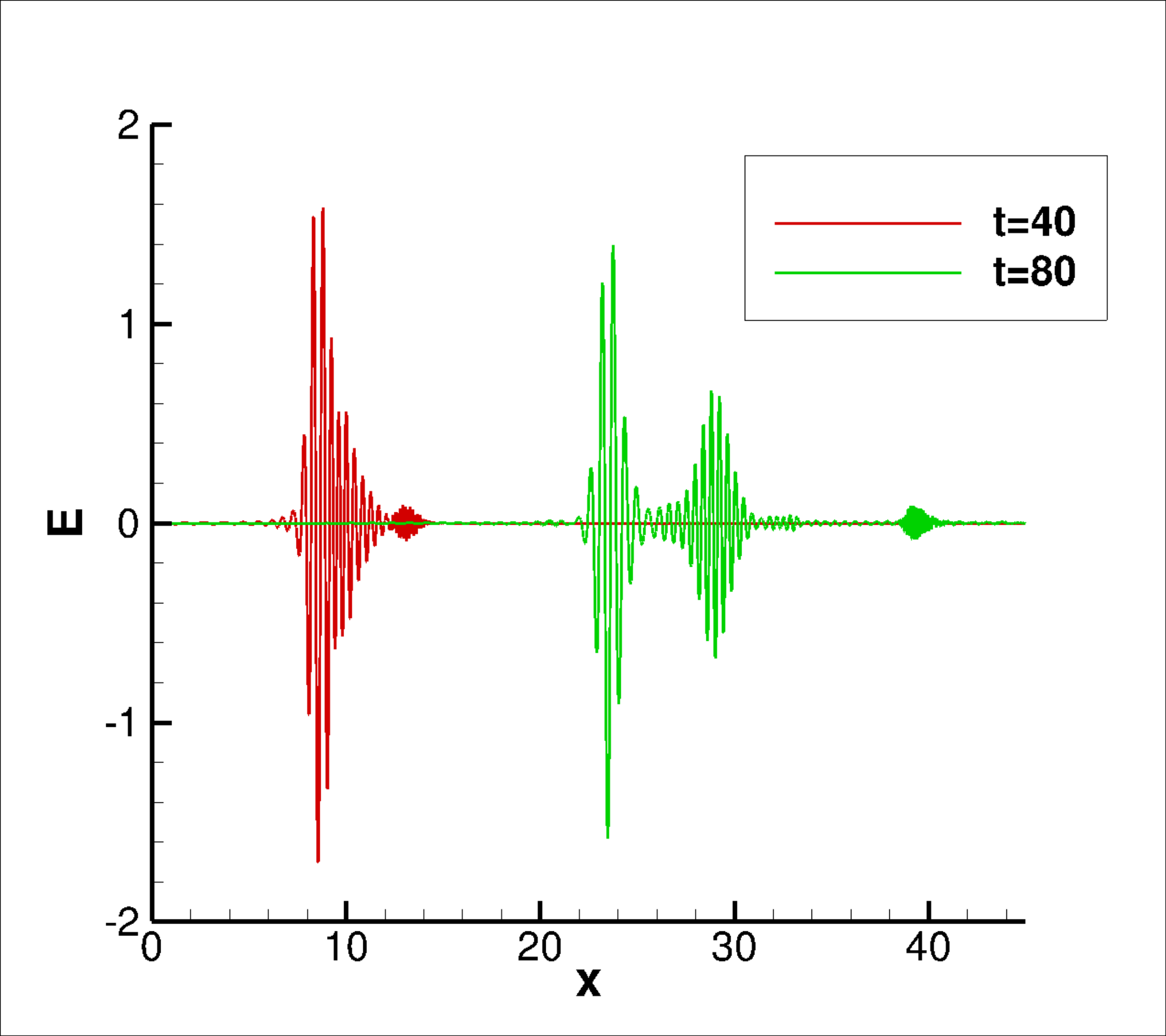}\label{Fig4.4}}
 	 \subfigure{
 	 	\includegraphics[width=0.3\textwidth]{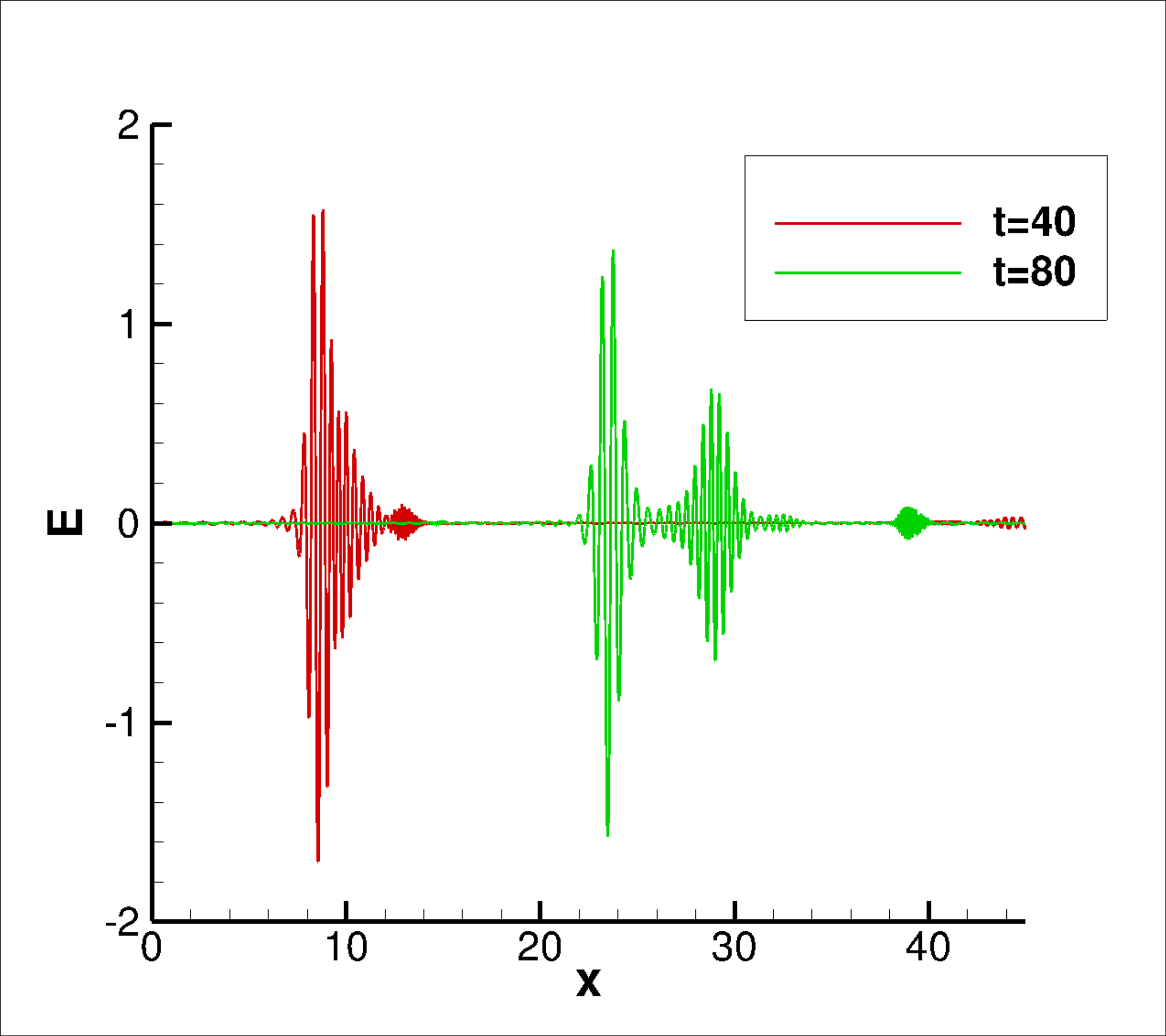}\label{Fig4.5}}
 	 \subfigure{
 	 	\includegraphics[width=0.3\textwidth]{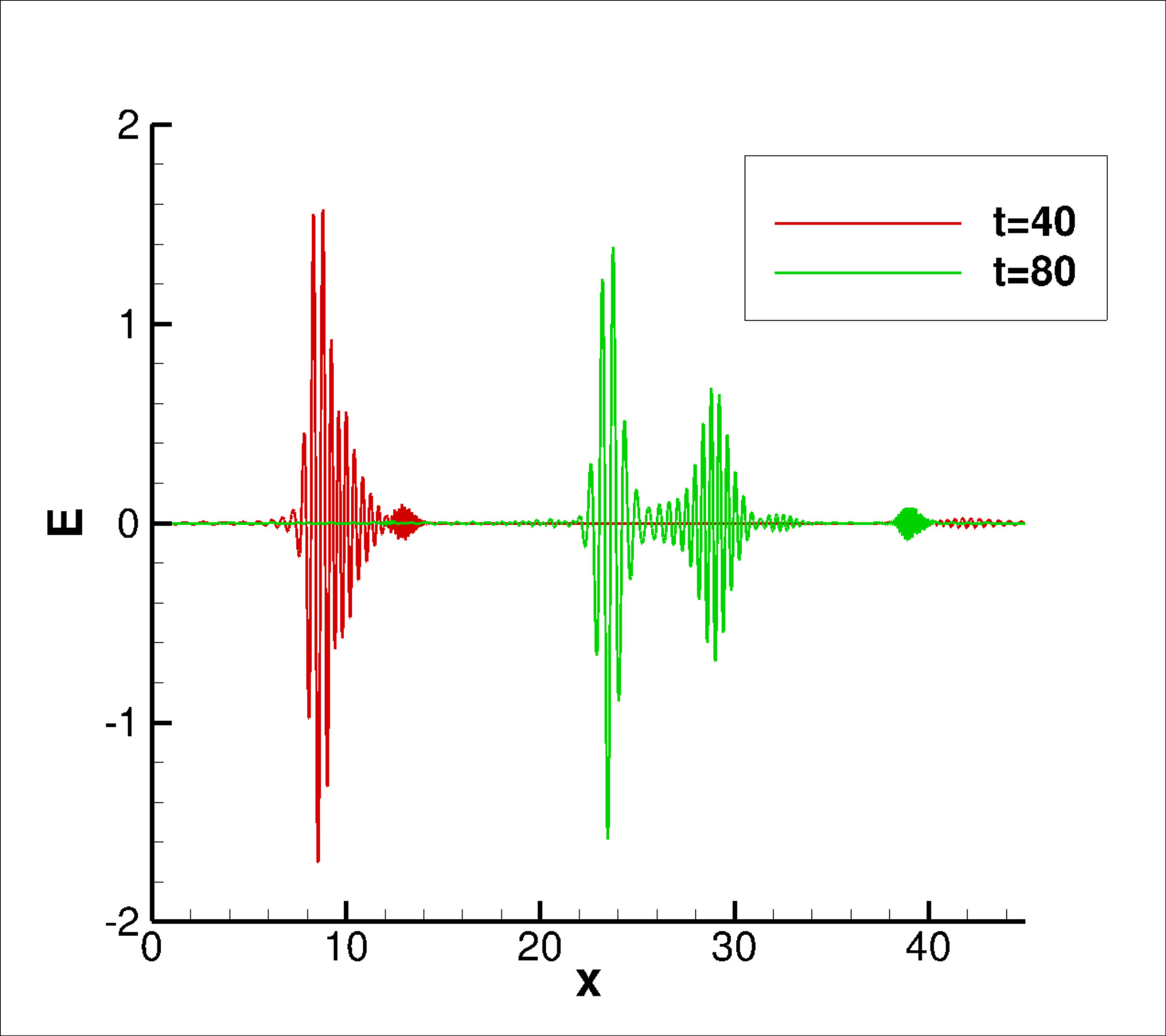}\label{Fig4.6}}\\
 	  \subfigure{
 	  	\includegraphics[width=0.3\textwidth]{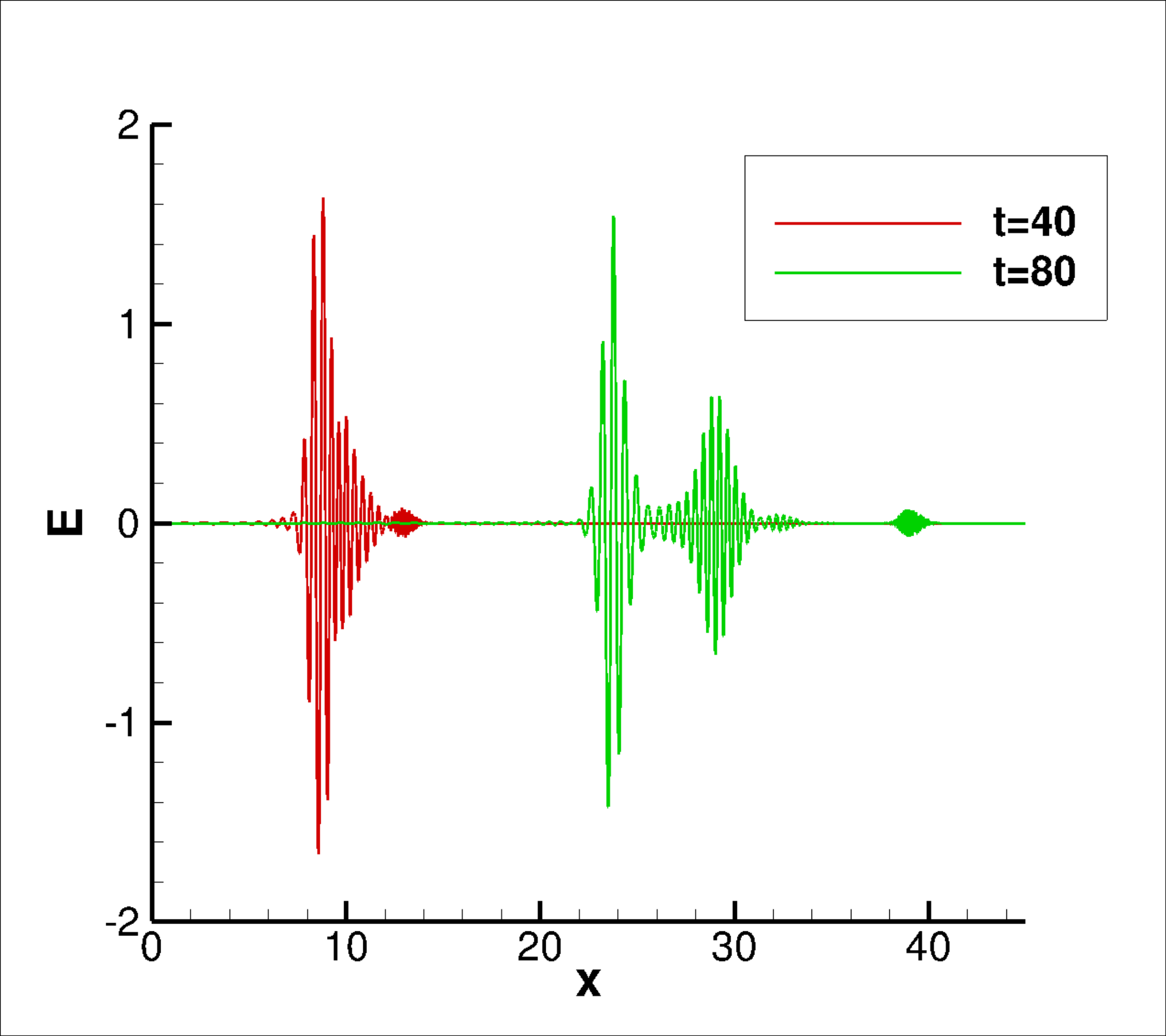}\label{Fig4.7}}
 	  \subfigure{
 	  	\includegraphics[width=0.3\textwidth]{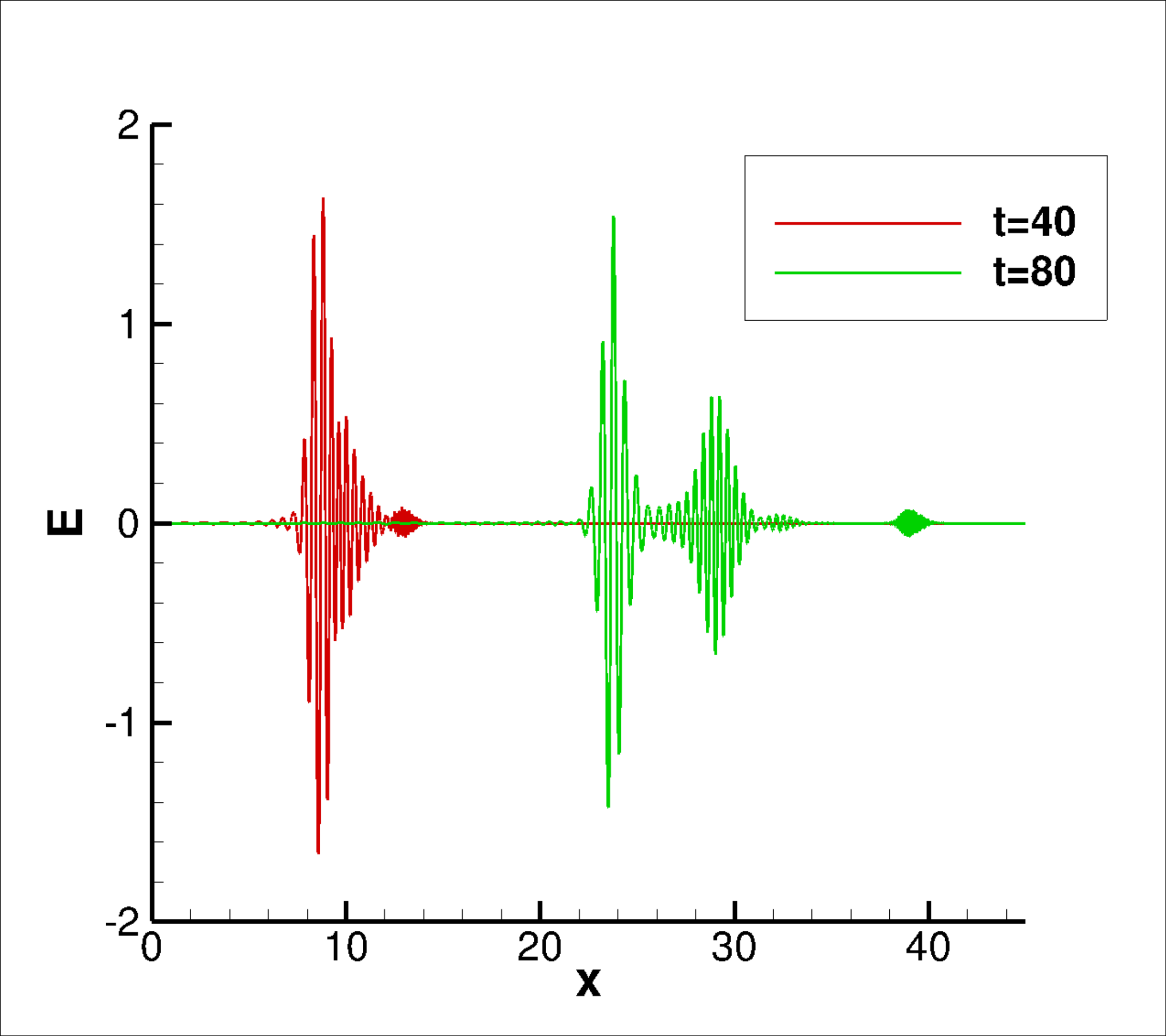}\label{Fig4.8}}
 	  \subfigure{
 	  	\includegraphics[width=0.3\textwidth]{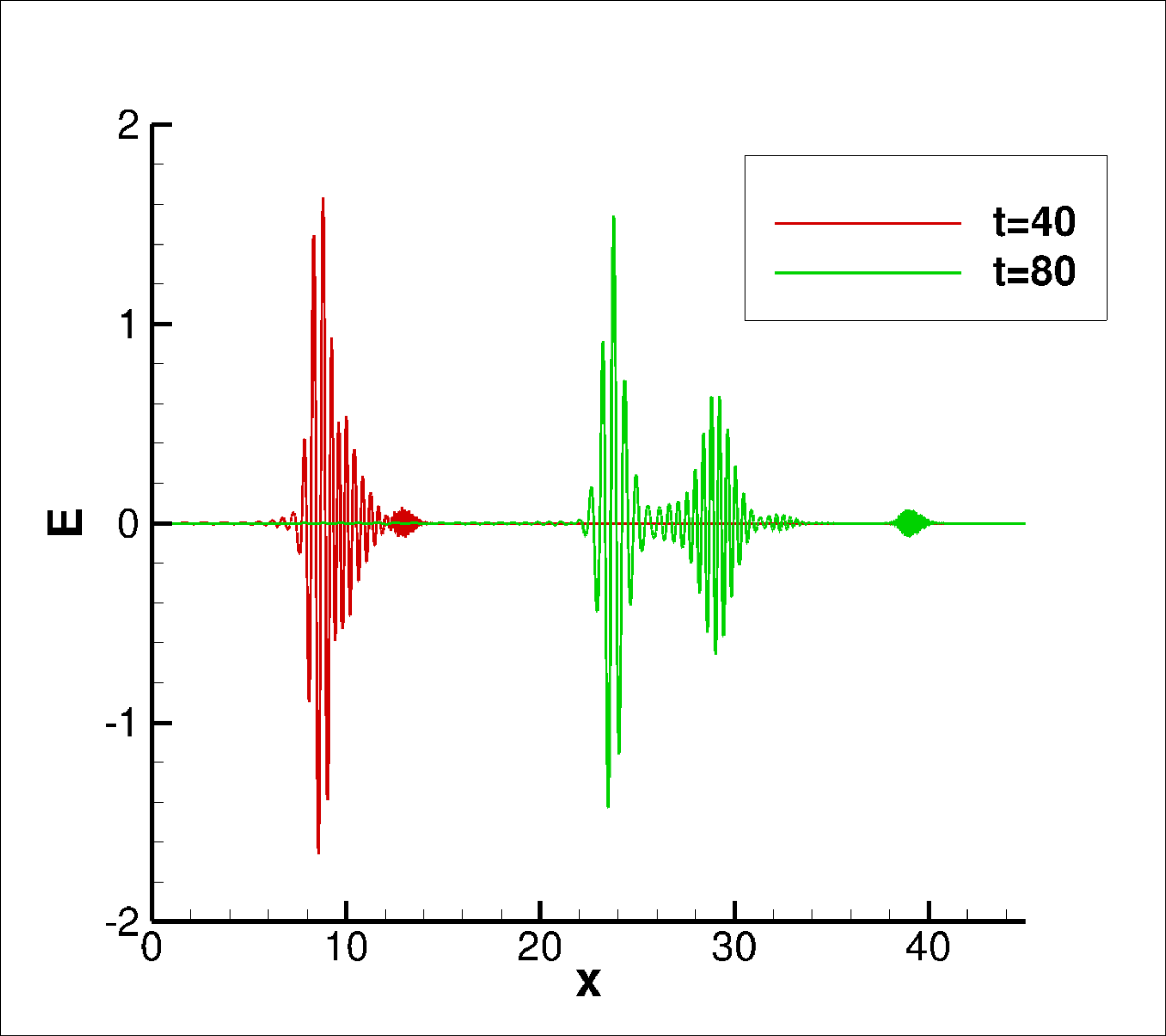}\label{Fig4.9}}\\
 	  \subfigure{
 	  	\includegraphics[width=0.3\textwidth]{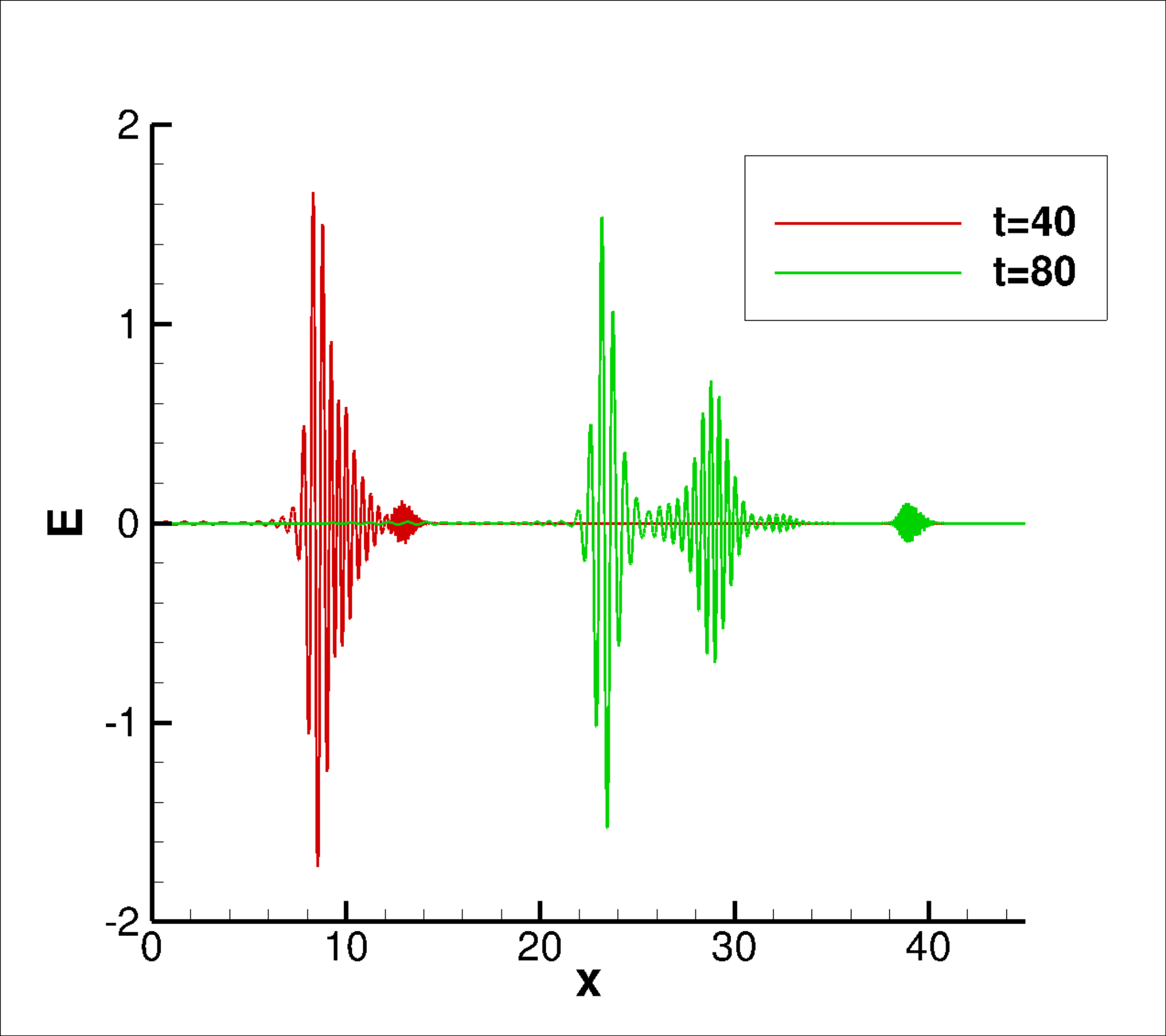}\label{Fig4.10}}
 	  \subfigure{
 	  	\includegraphics[width=0.3\textwidth]{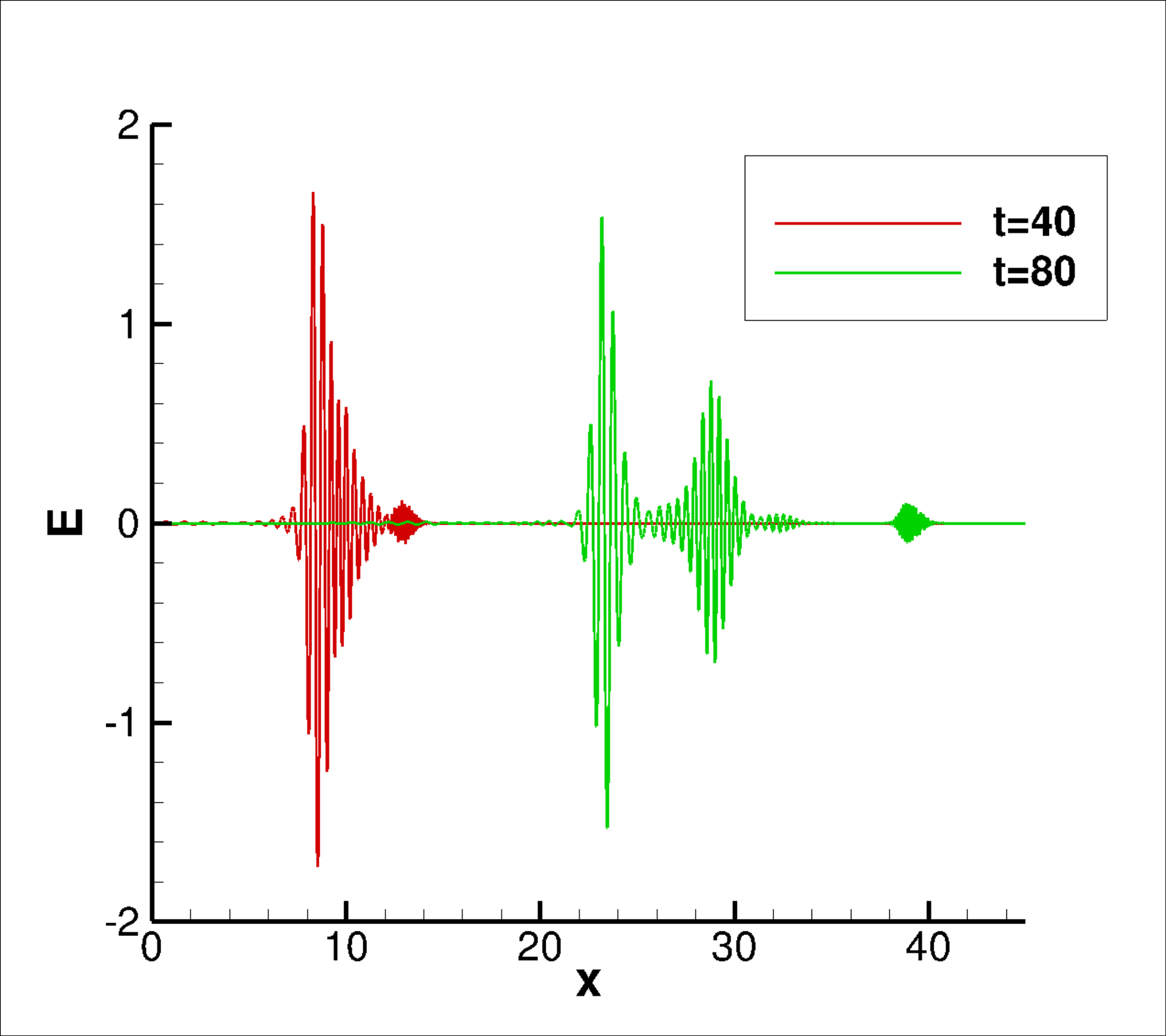}\label{Fig4.11}}
 	  \subfigure{
 	  	\includegraphics[width=0.3\textwidth]{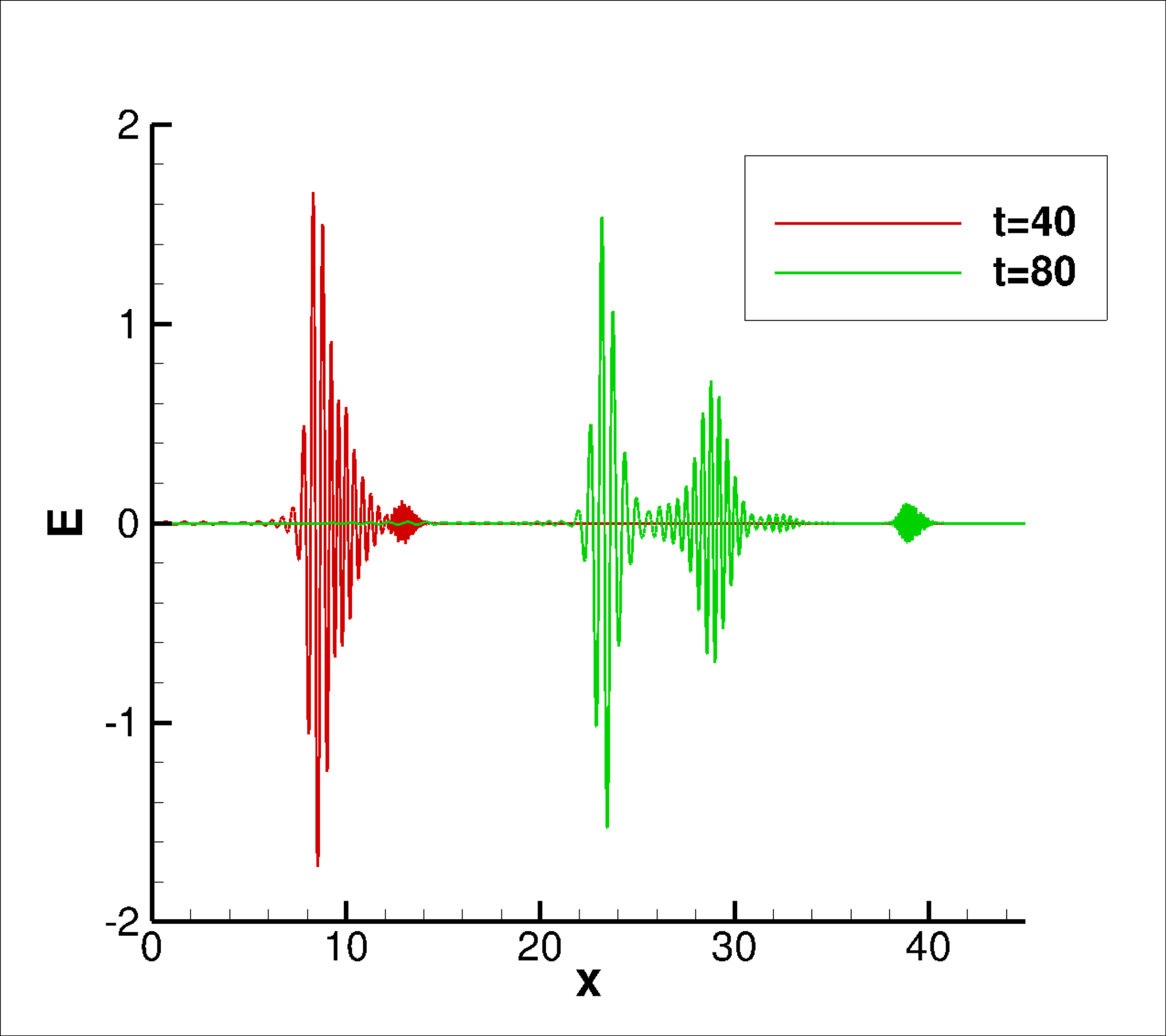}\label{Fig4.12}}
 	\caption{\em  Transient second-order ($M=2$) temporal soliton propagation with the leap-frog scheme. $N=6400$ grid points. First column: $k=1$; second column: $k=2$; third column: $k=3$. First row: upwind flux;  second row: central flux; third row: alternating flux I; fourth row: alternating flux II. }
 	\label{Fig4}
 \end{figure}

   \begin{figure}
 	\centering
 	\subfigure{
 		\includegraphics[width=0.3\textwidth]{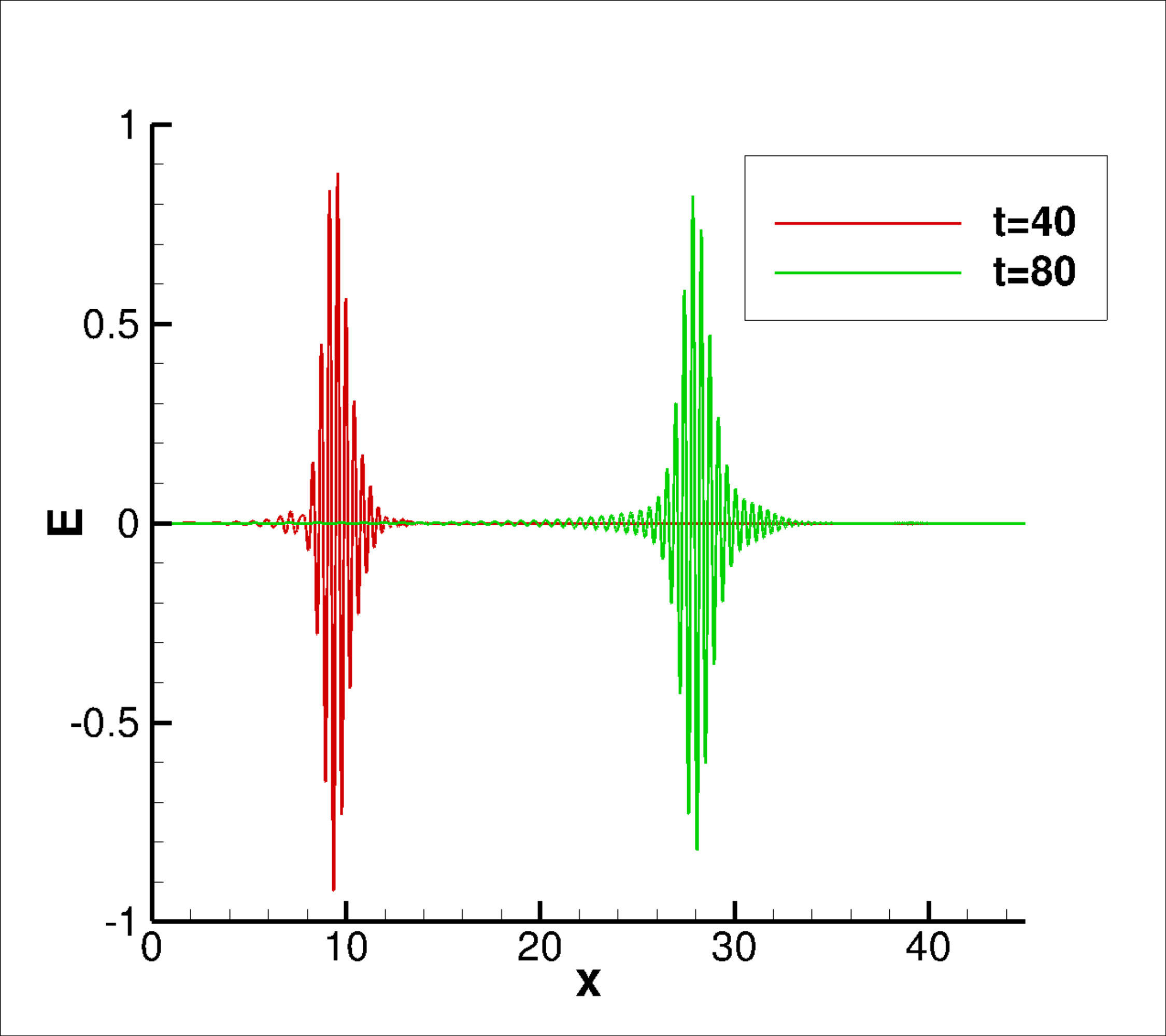}\label{Fig5.1}}
 	\subfigure{
 		\includegraphics[width=0.3\textwidth]{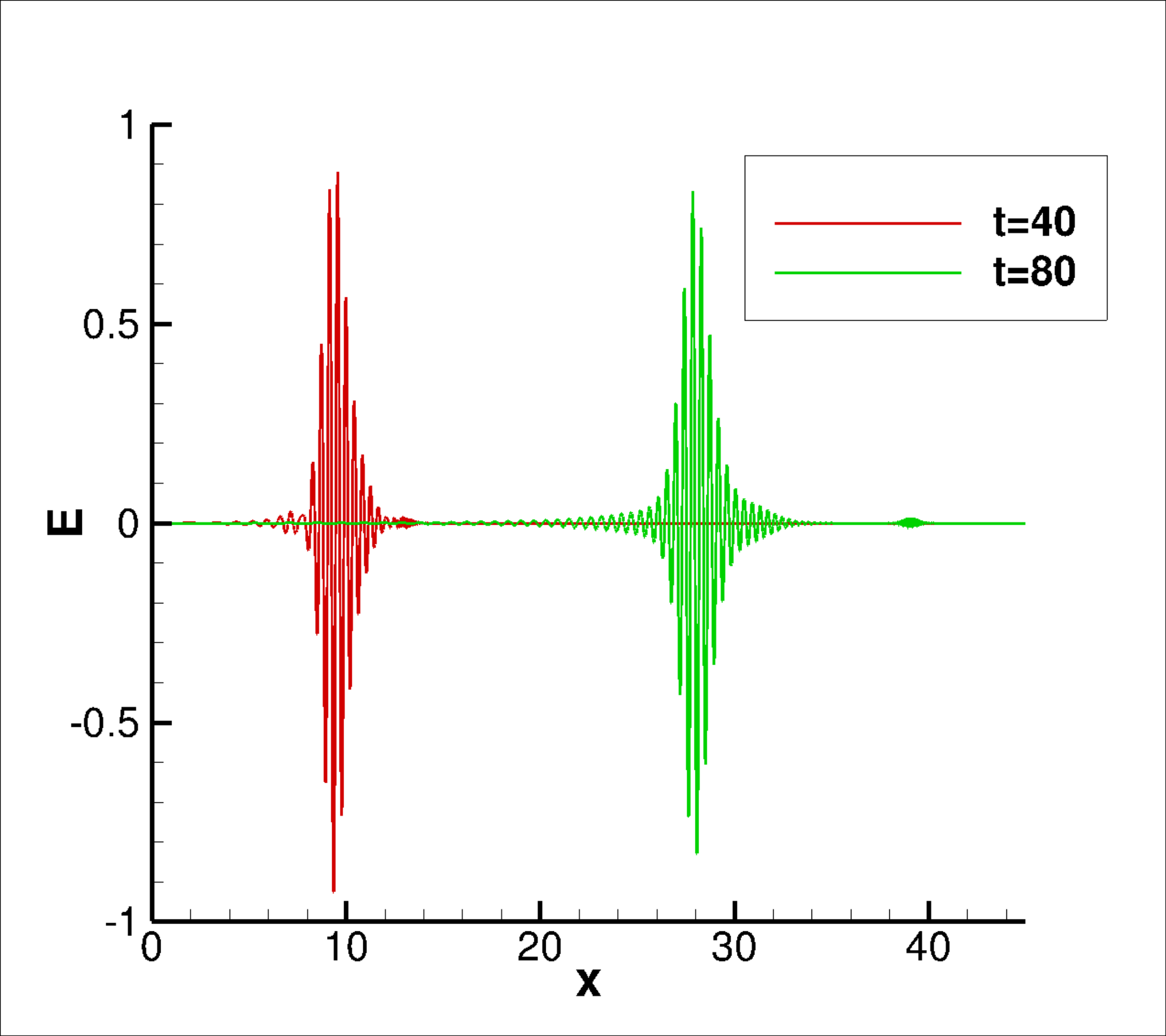}\label{Fig5.2}}
 	\subfigure{
 		\includegraphics[width=0.3\textwidth]{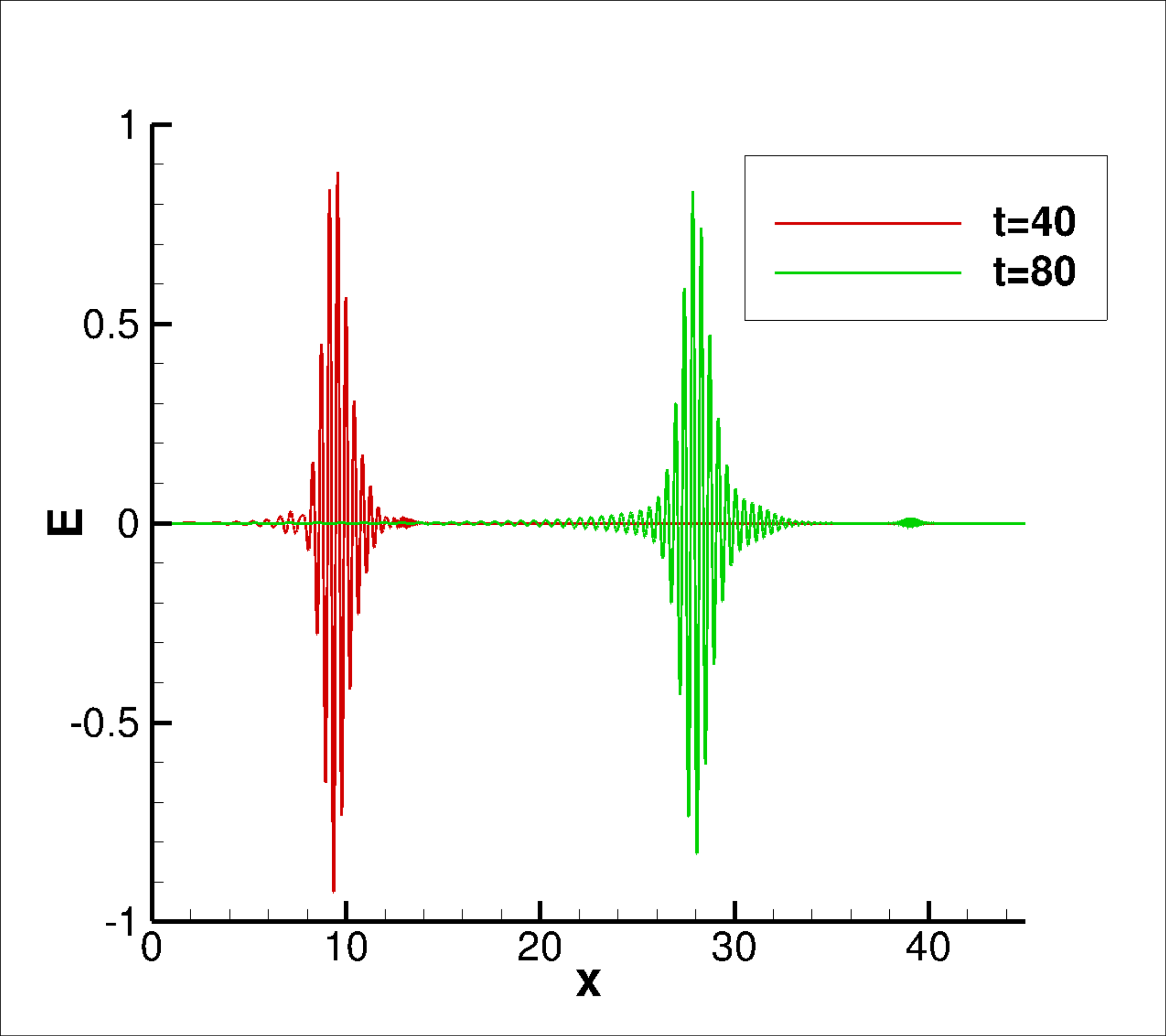}\label{Fig5.3}}
 	 \subfigure{
 		\includegraphics[width=0.3\textwidth]{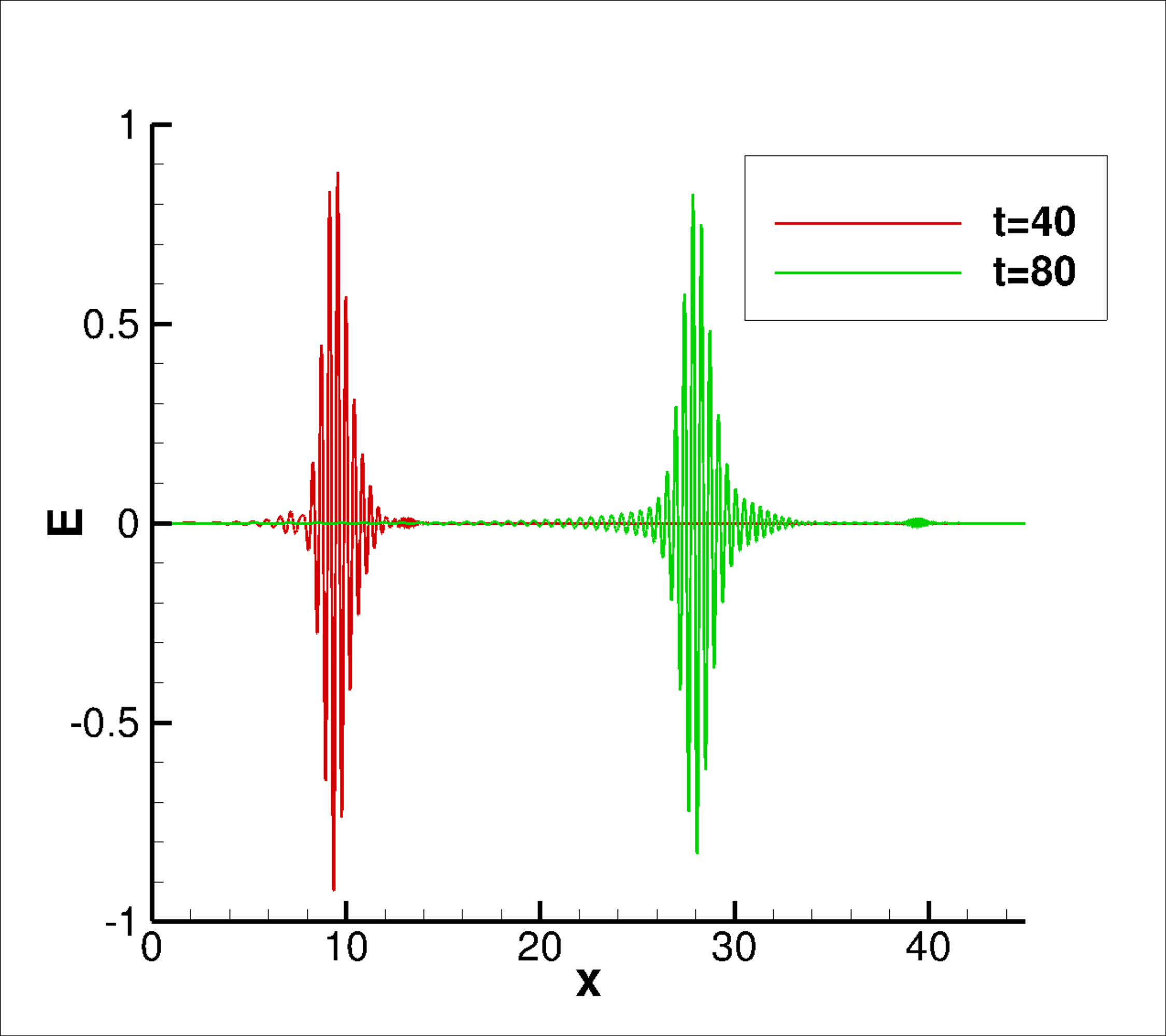}\label{Fig5.4}}
 	 \subfigure{
 	 	\includegraphics[width=0.3\textwidth]{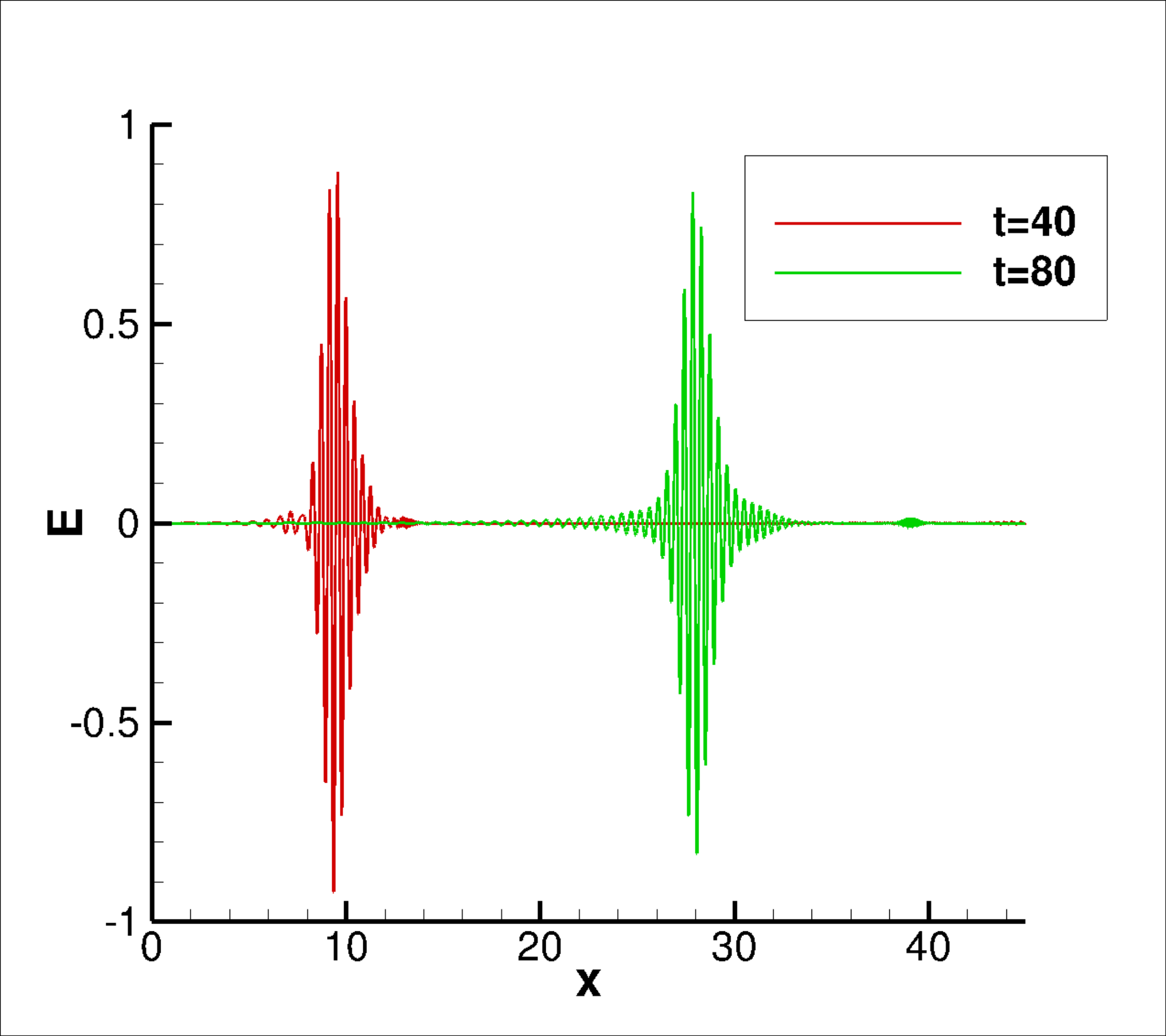}\label{Fig5.5}}
 	 \subfigure{
 	 	\includegraphics[width=0.3\textwidth]{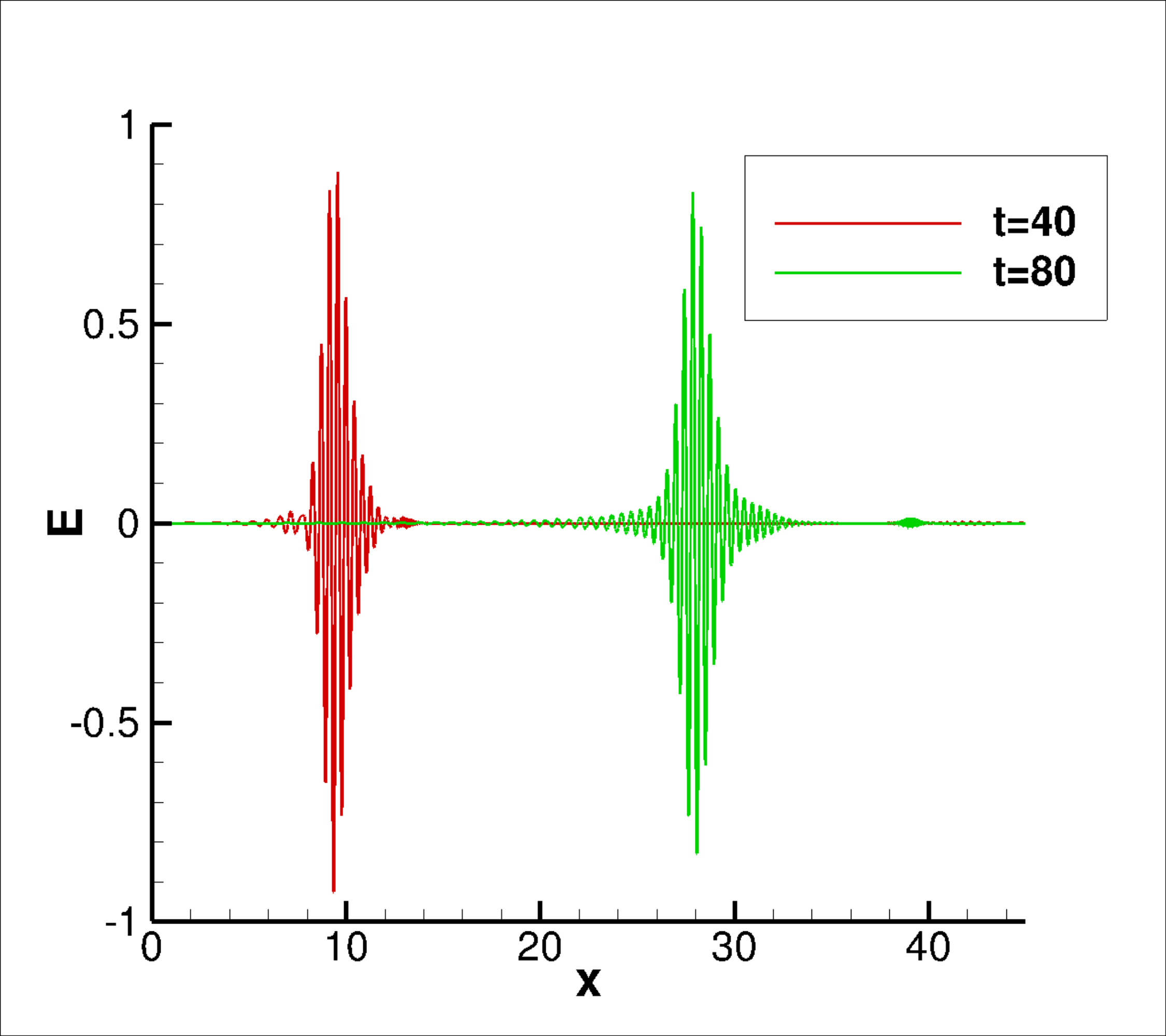}\label{Fig5.6}}
 	  \subfigure{
 	  	\includegraphics[width=0.3\textwidth]{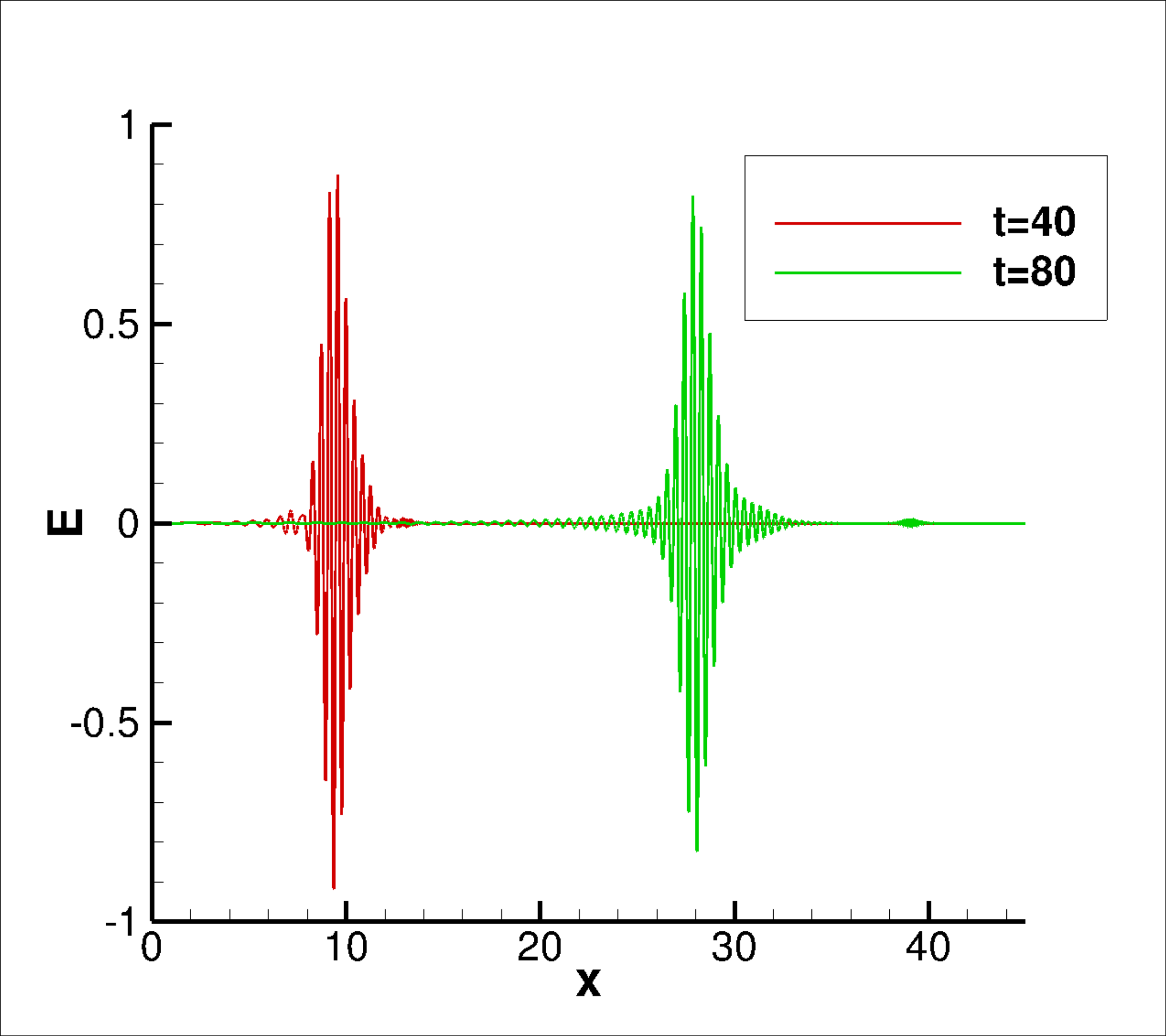}\label{Fig5.7}}
 	  \subfigure{
 	  	\includegraphics[width=0.3\textwidth]{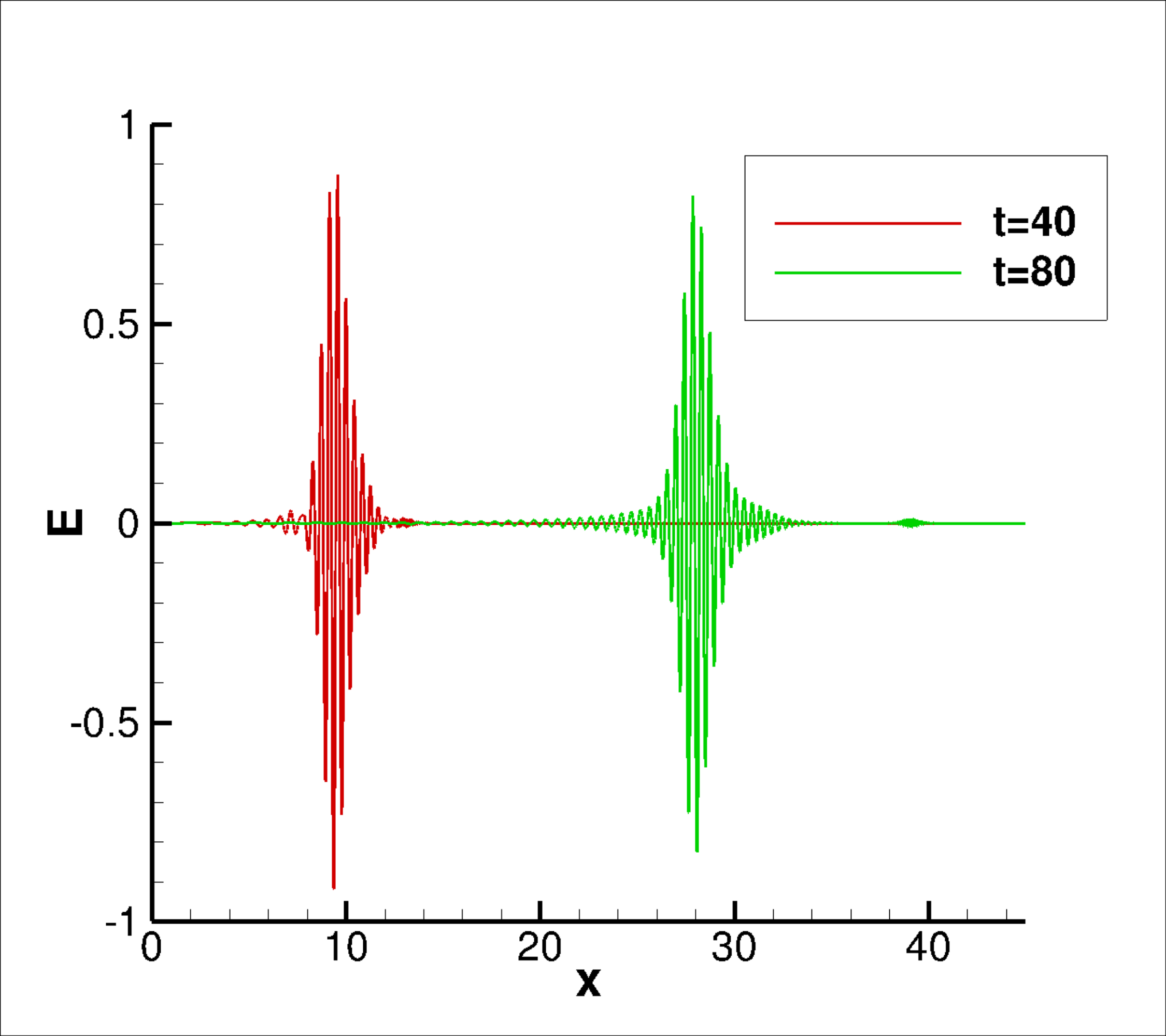}\label{Fig5.8}}
 	  \subfigure{
 	  	\includegraphics[width=0.3\textwidth]{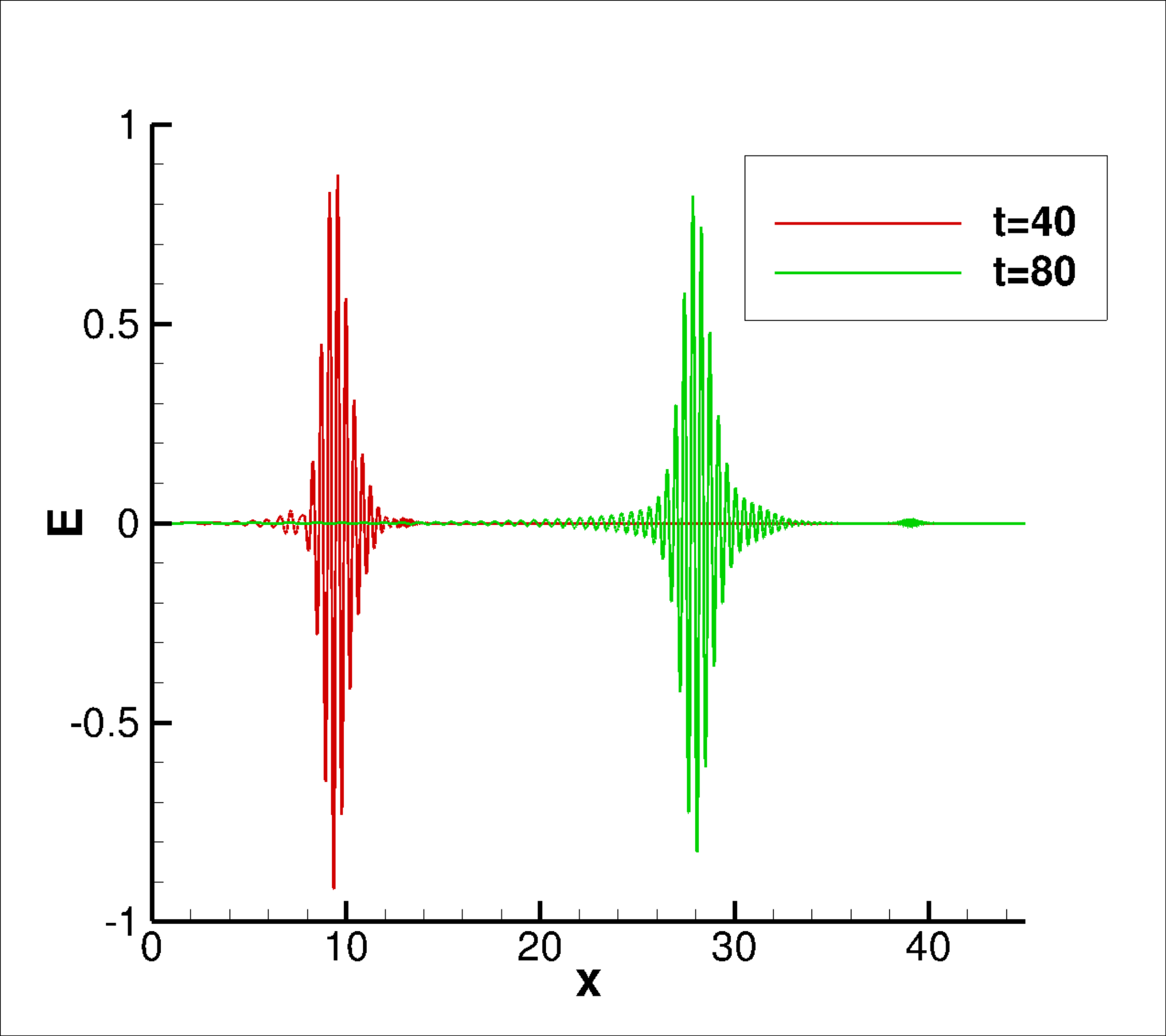}\label{Fig5.9}}
 	  	 	  \subfigure{
 	  	\includegraphics[width=0.3\textwidth]{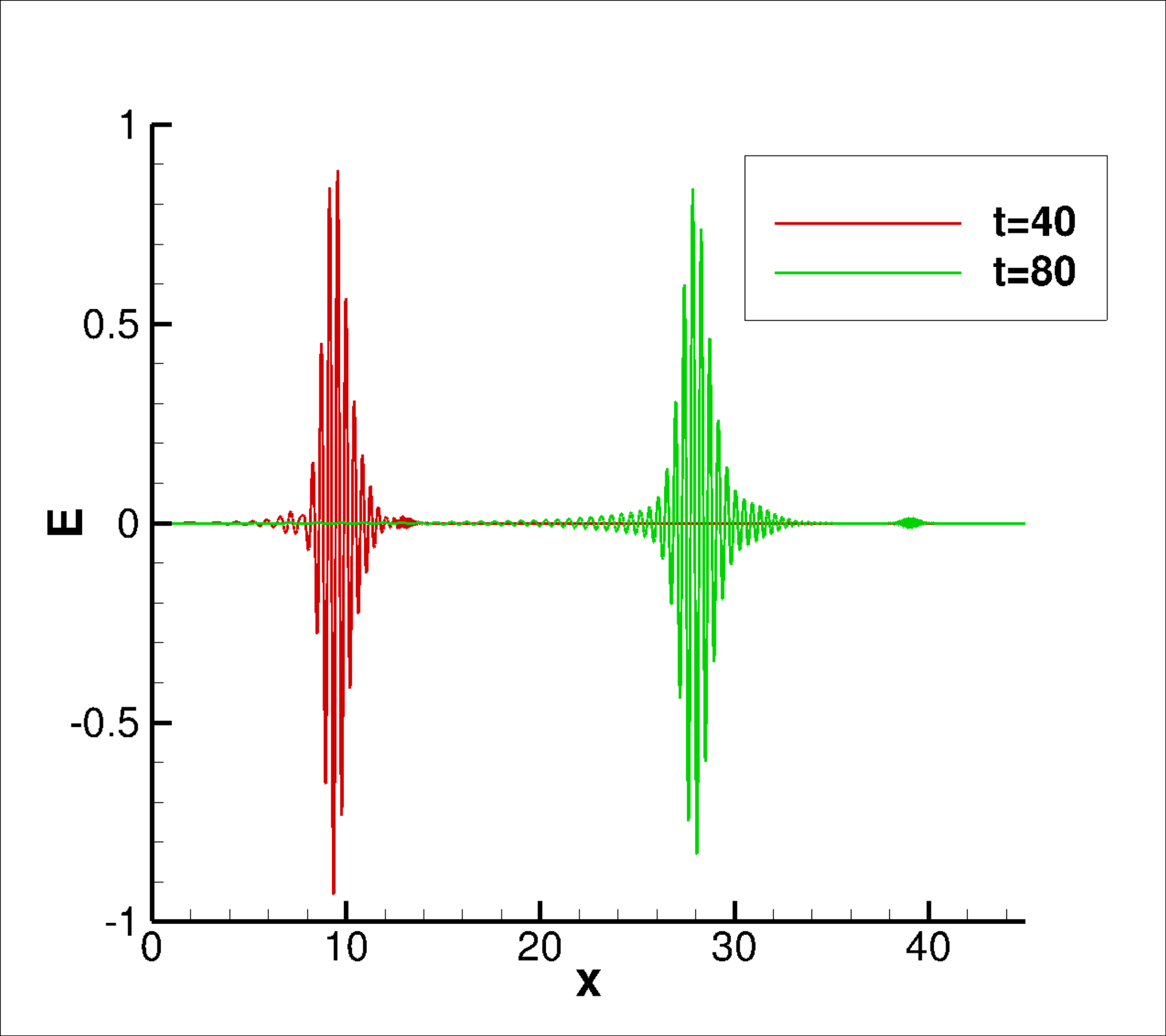}\label{Fig5.10}}
 	  \subfigure{
 	  	\includegraphics[width=0.3\textwidth]{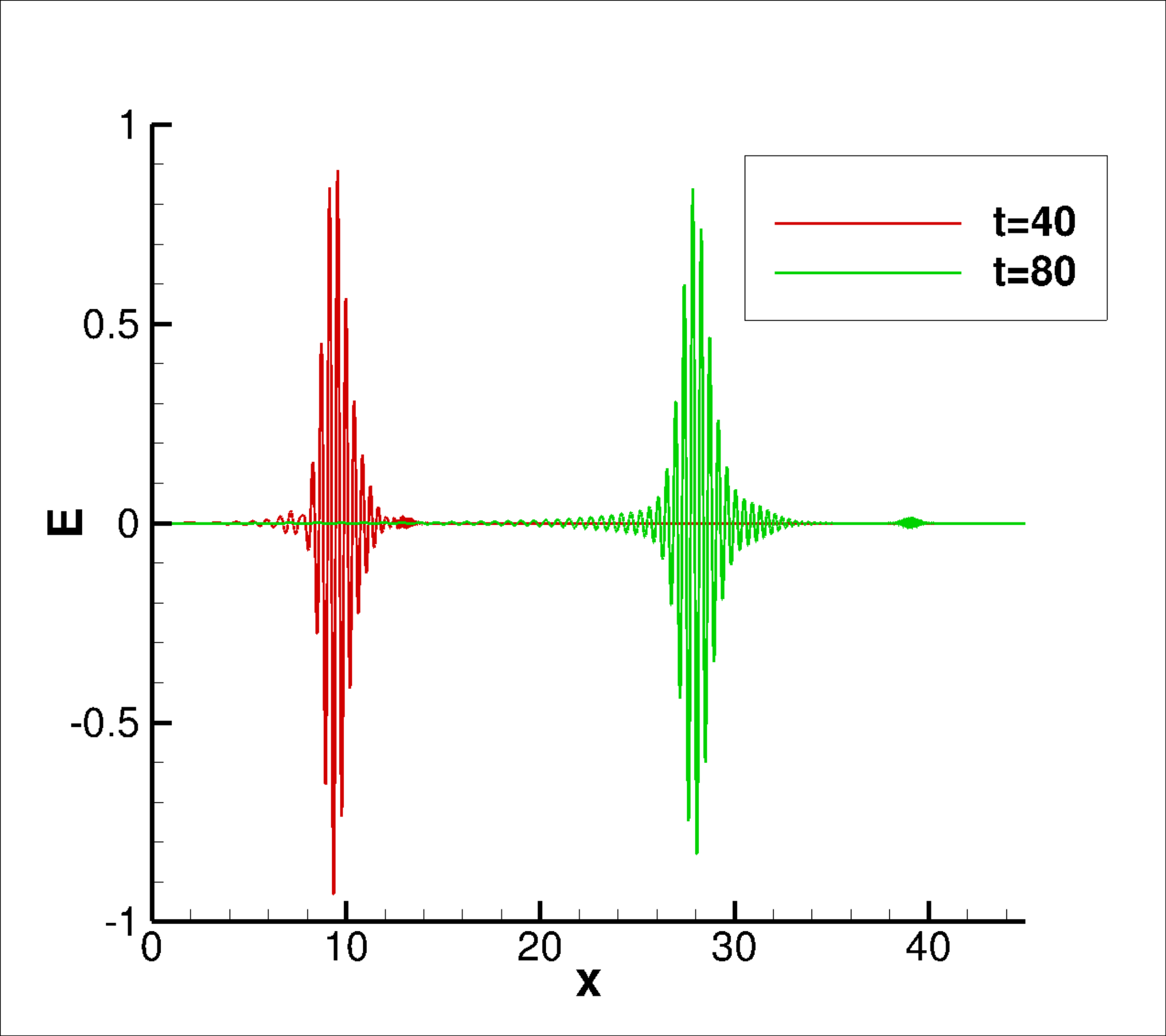}\label{Fig5.11}}
 	  \subfigure{
 	  	\includegraphics[width=0.3\textwidth]{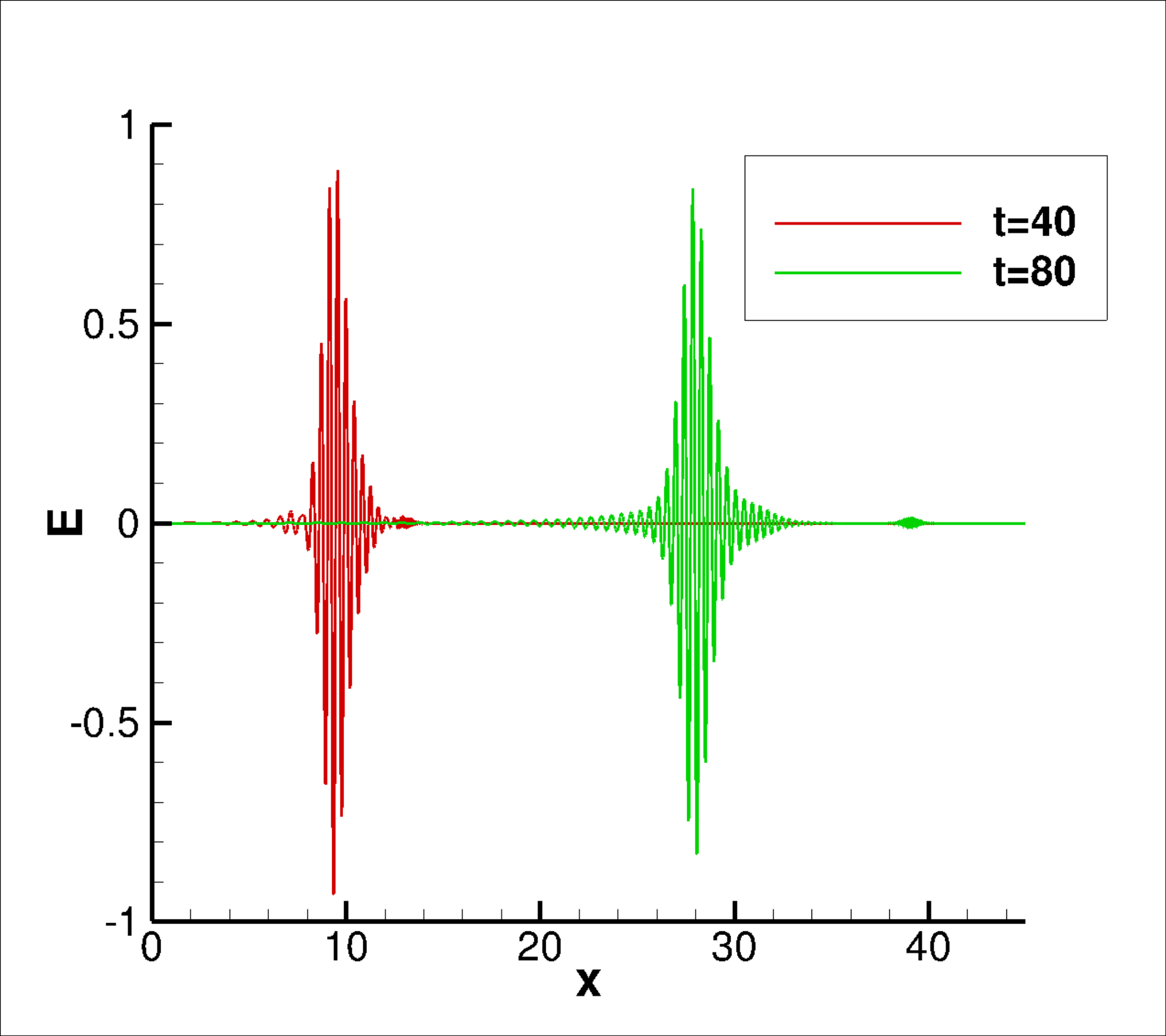}\label{Fig5.12}}
 	\caption{\em  Transient fundamental ($M=1$) temporal soliton propagation with the fully implicit scheme. $N=6400$ grid points. First column: $k=1$; second column: $k=2$; third column: $k=3$. First row: upwind flux;  second row: central flux; third row: alternating flux I; fourth row: alternating flux II. }
 	\label{Fig5}
 \end{figure}

   \begin{figure}
 	\centering
 	\subfigure{
 		\includegraphics[width=0.3\textwidth]{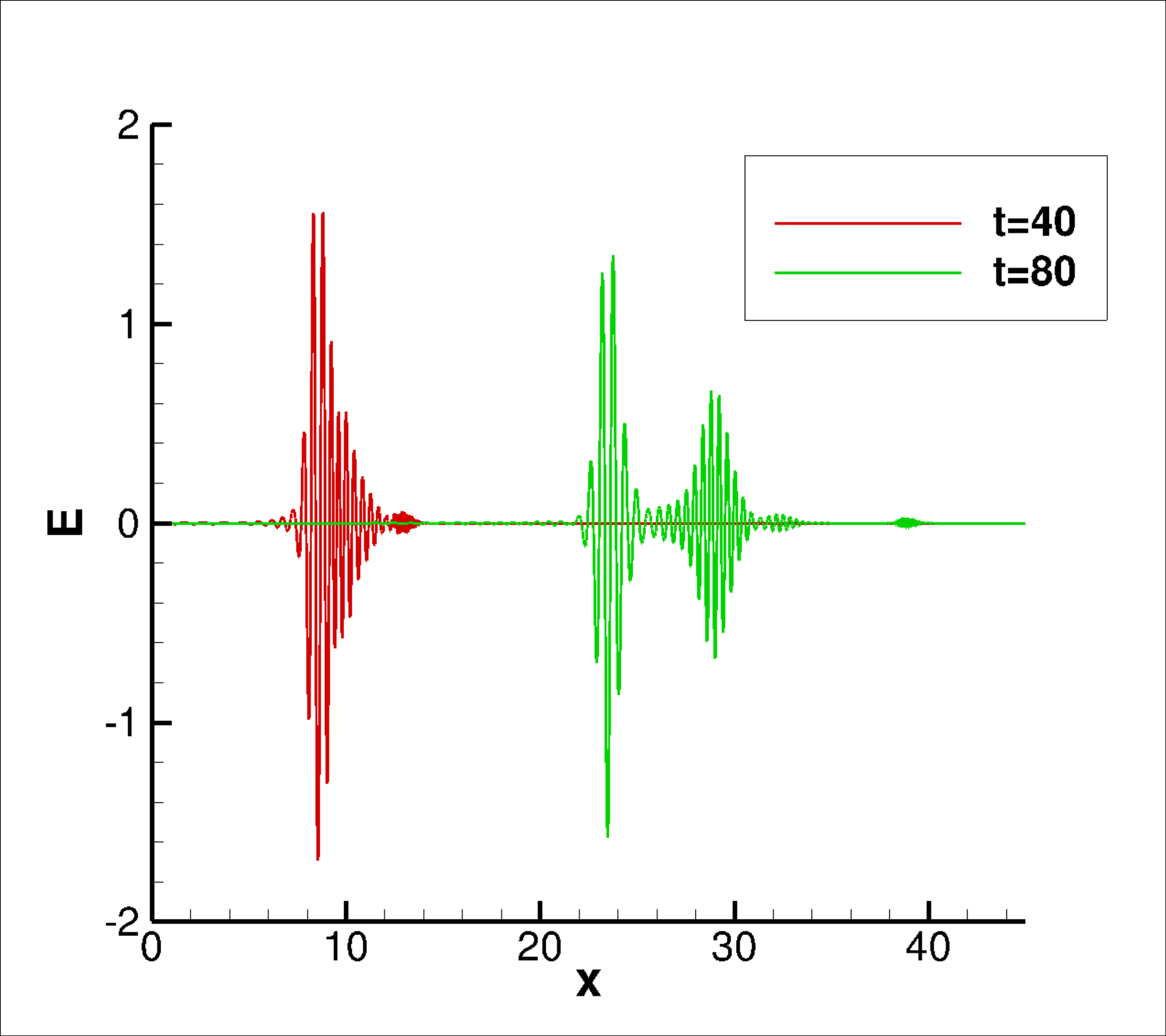}\label{Fig6.1}}
 	\subfigure{
 		\includegraphics[width=0.3\textwidth]{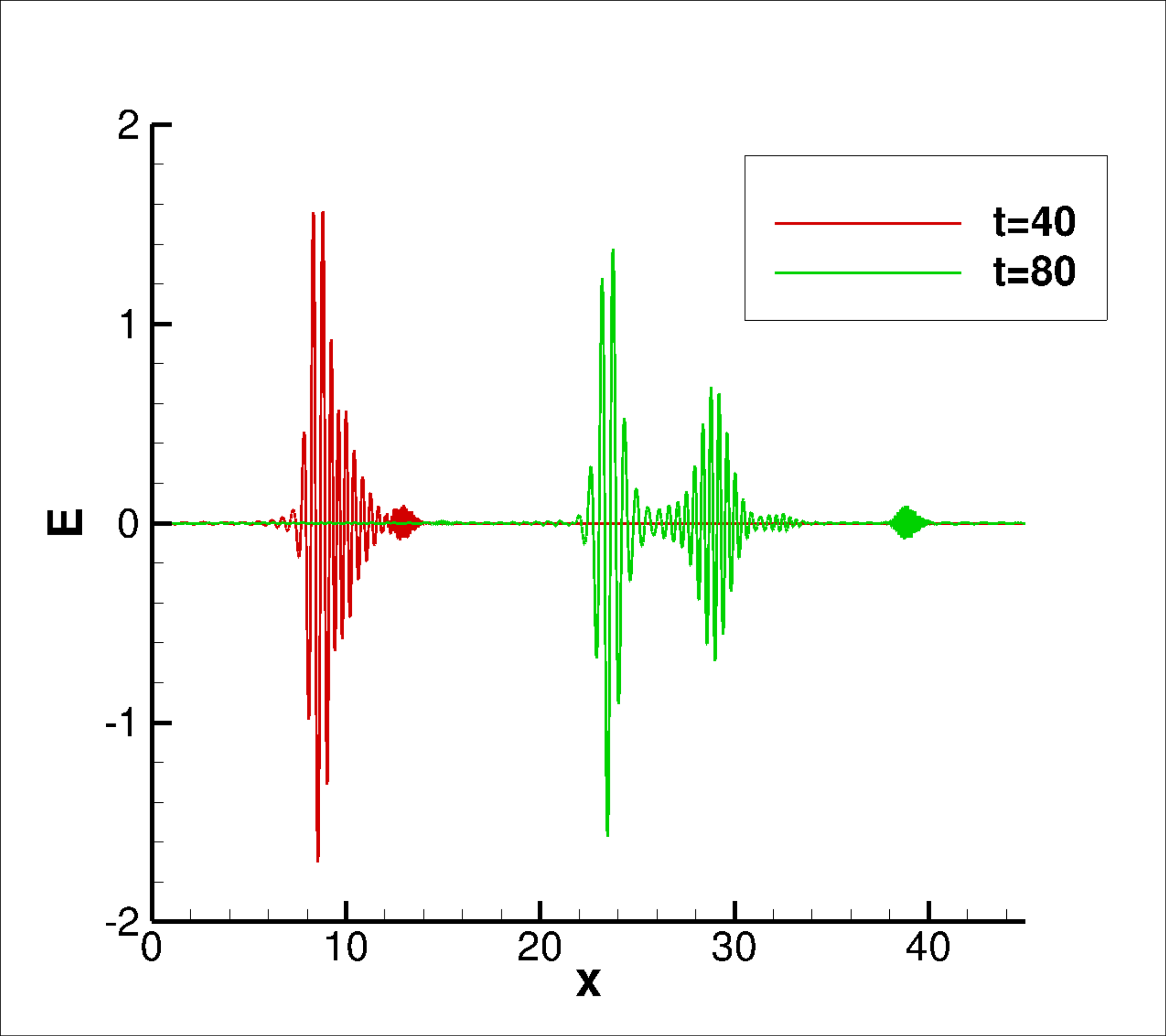}\label{Fig6.2}}
 	\subfigure{
 		\includegraphics[width=0.3\textwidth]{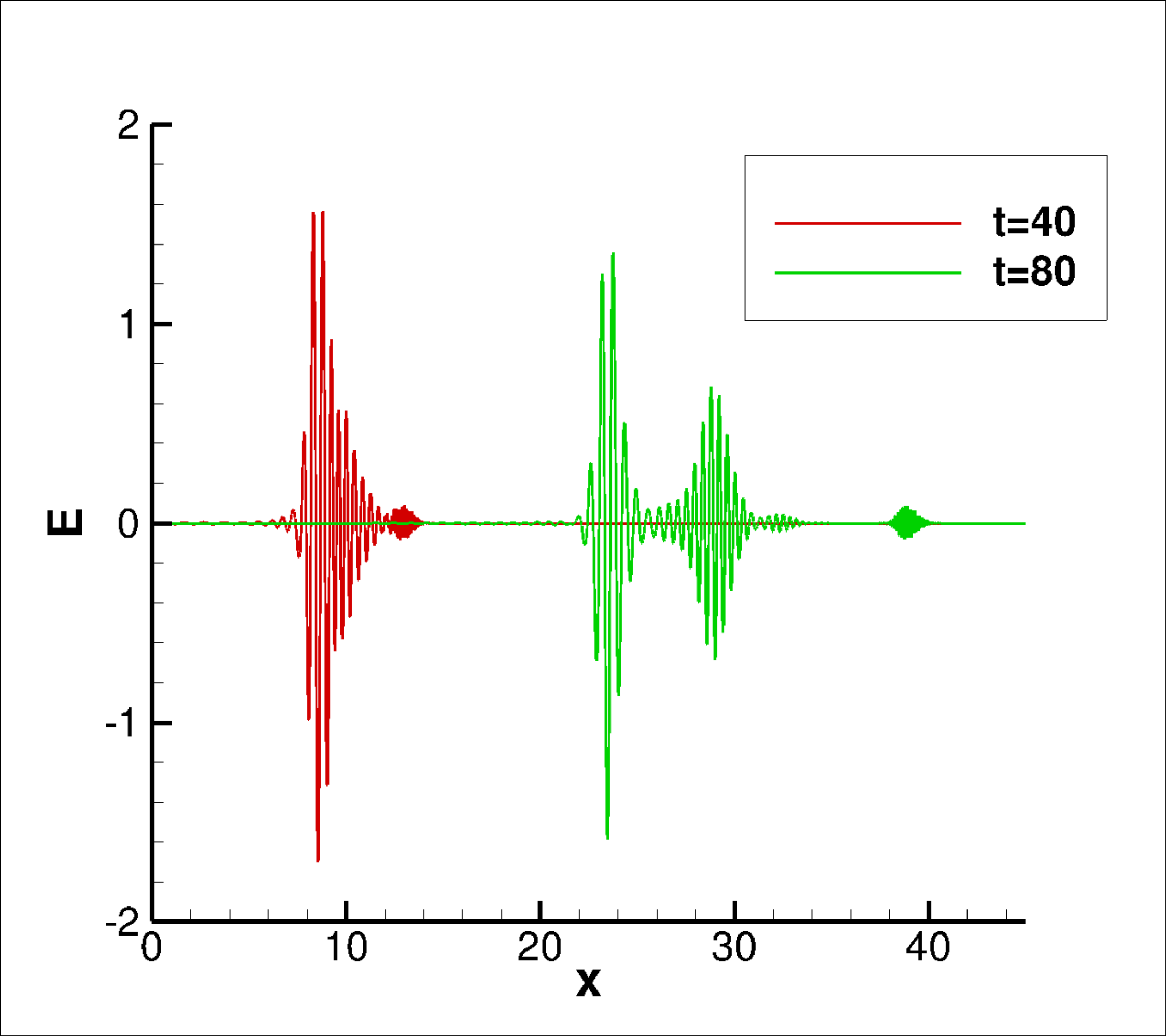}\label{Fig6.3}}
 	 \subfigure{
 		\includegraphics[width=0.3\textwidth]{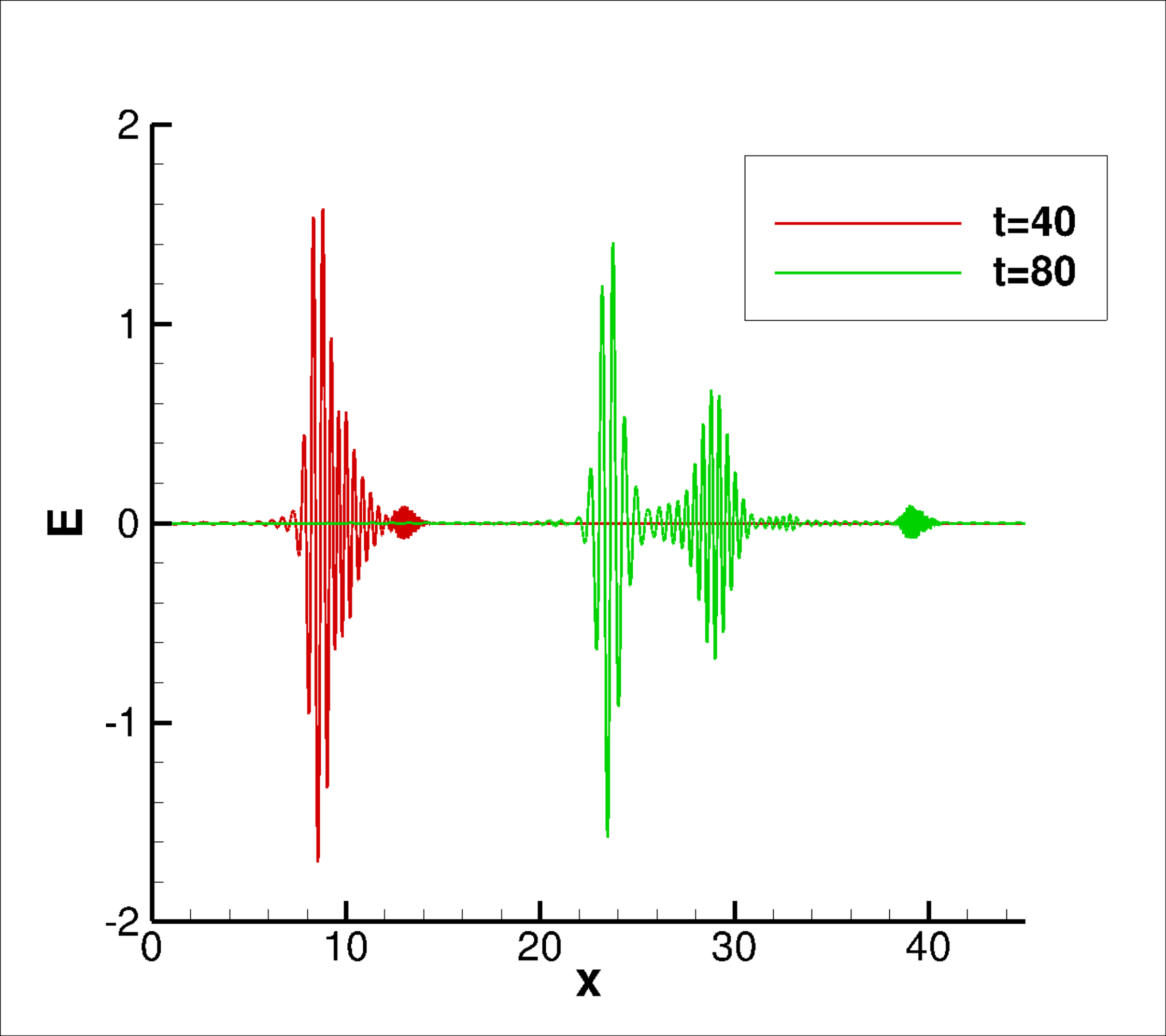}\label{Fig6.4}}
 	 \subfigure{
 	 	\includegraphics[width=0.3\textwidth]{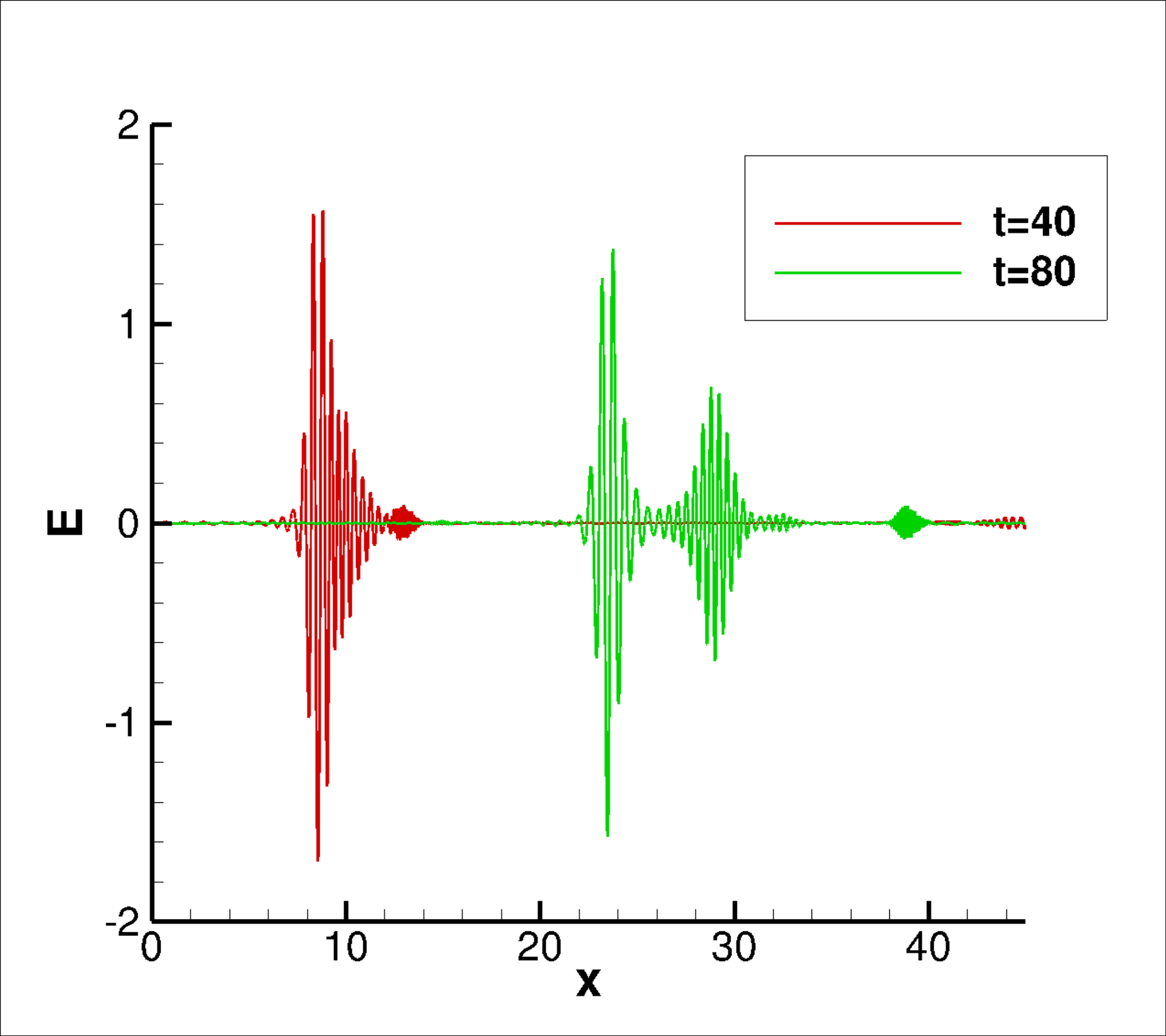}\label{Fig6.5}}
 	 \subfigure{
 	 	\includegraphics[width=0.3\textwidth]{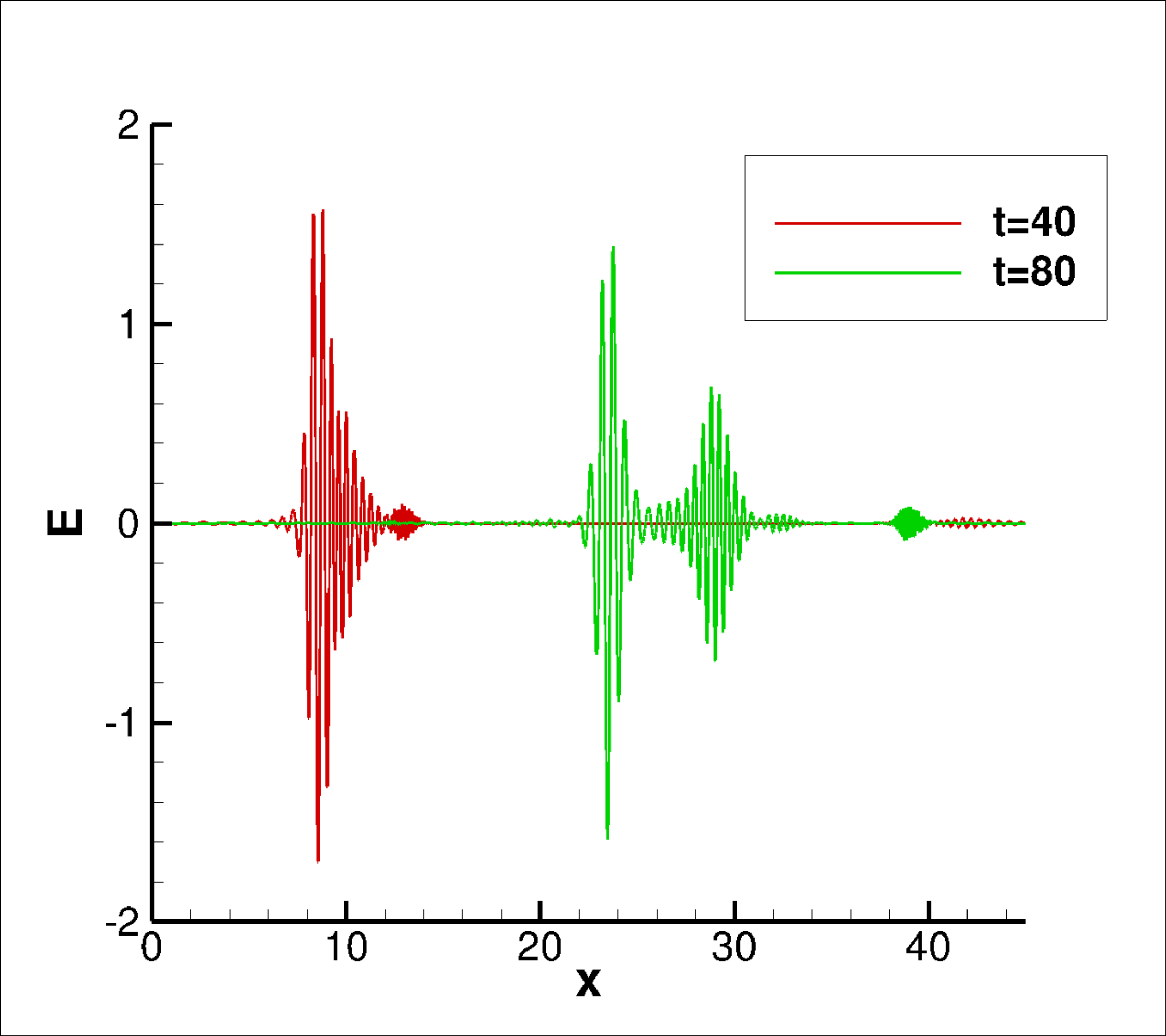}\label{Fig6.6}}
 	  \subfigure{
 	  	\includegraphics[width=0.3\textwidth]{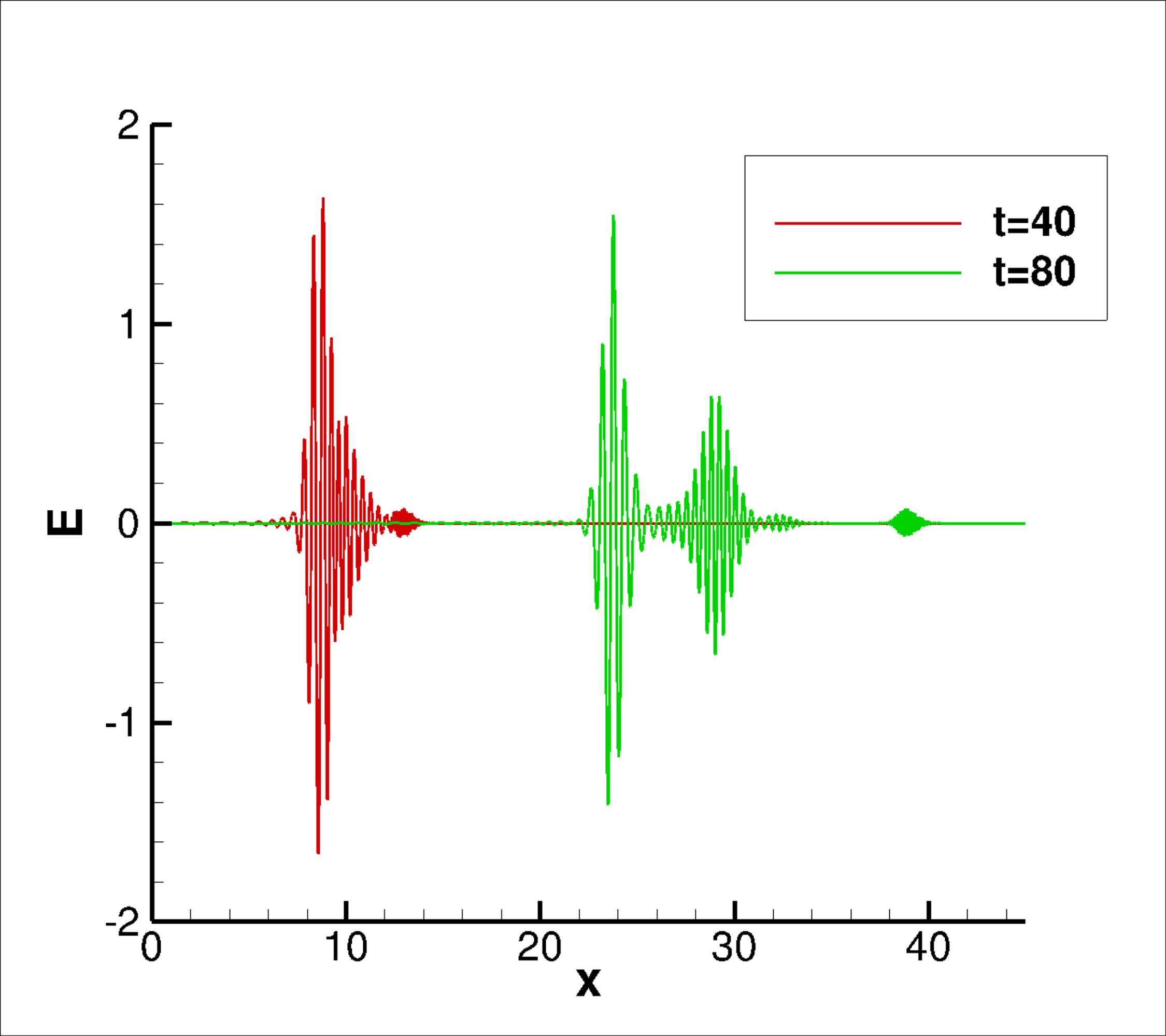}\label{Fig6.7}}
 	  \subfigure{
 	  	\includegraphics[width=0.3\textwidth]{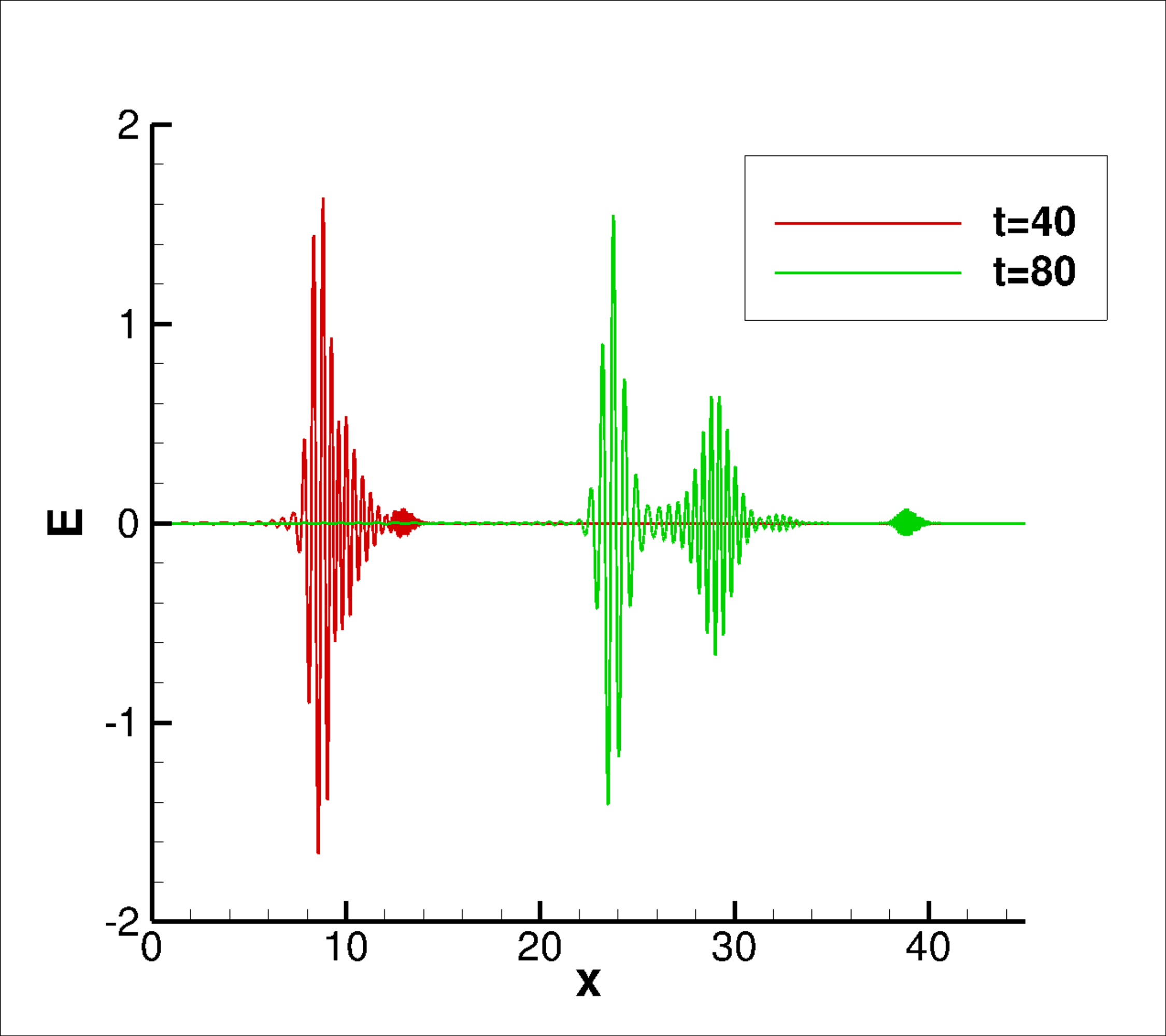}\label{Fig6.8}}
 	  \subfigure{
 	  	\includegraphics[width=0.3\textwidth]{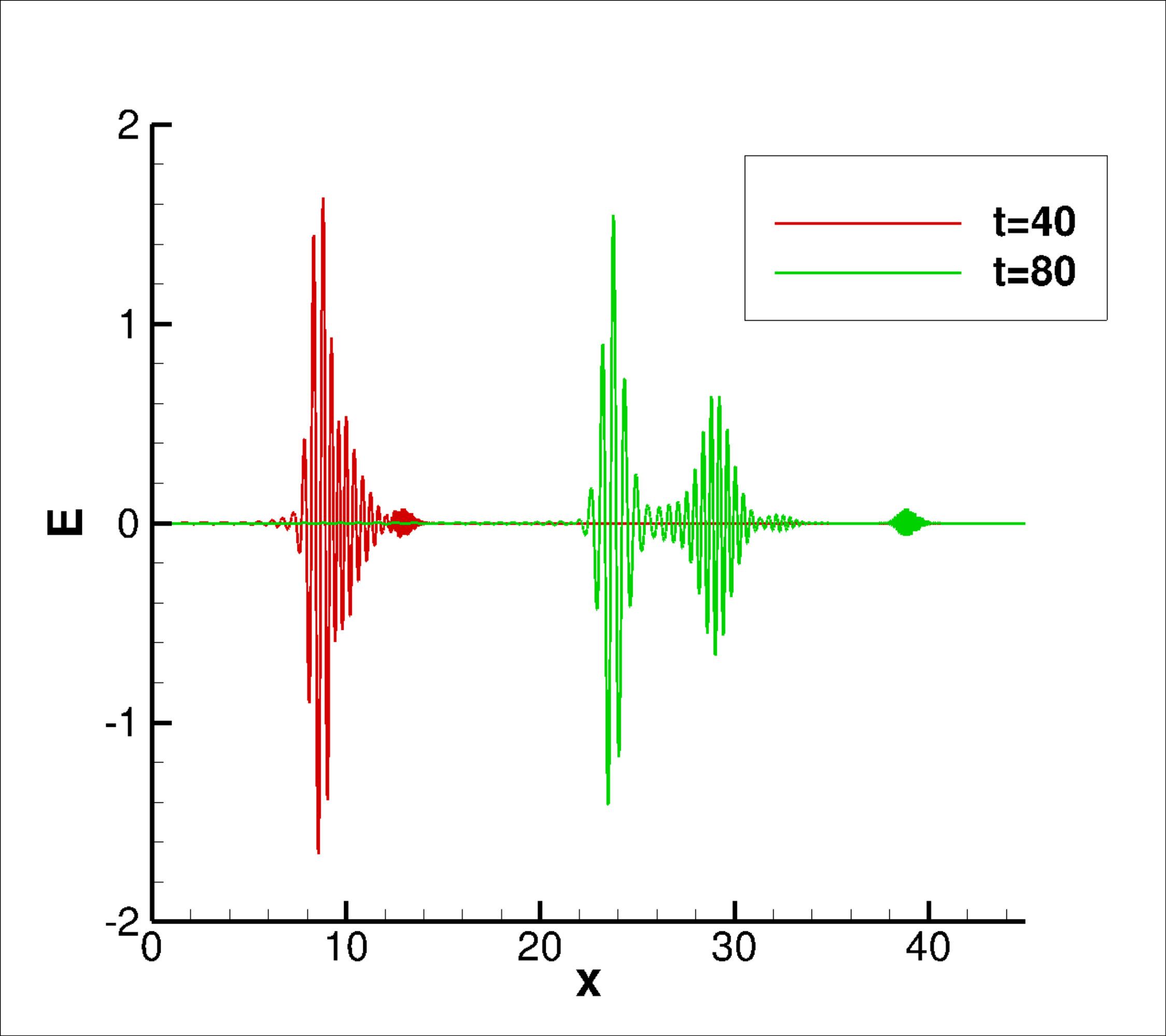}\label{Fig6.9}}
 	  \subfigure{
 	  	\includegraphics[width=0.3\textwidth]{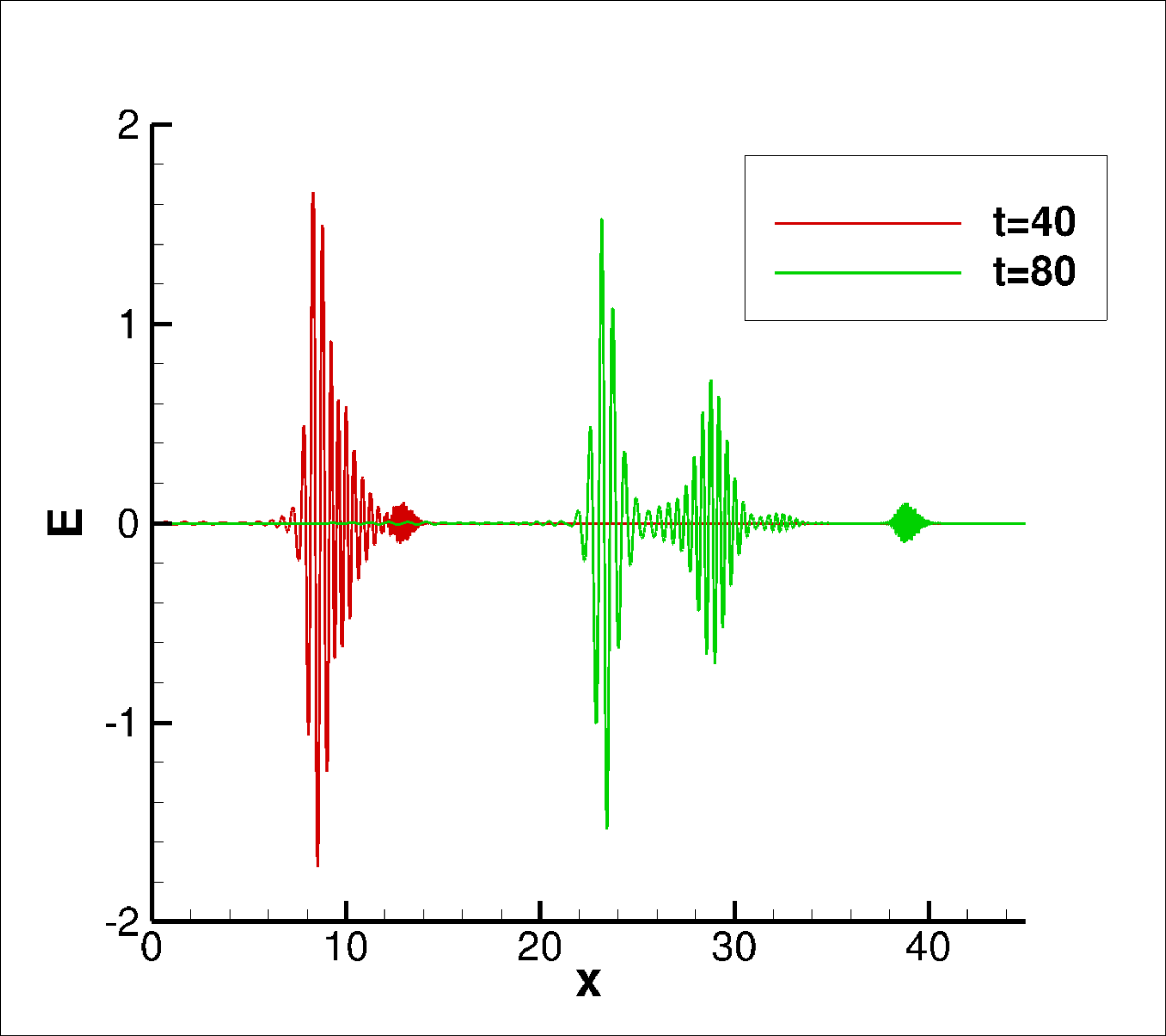}\label{Fig6.10}}
 	  \subfigure{
 	  	\includegraphics[width=0.3\textwidth]{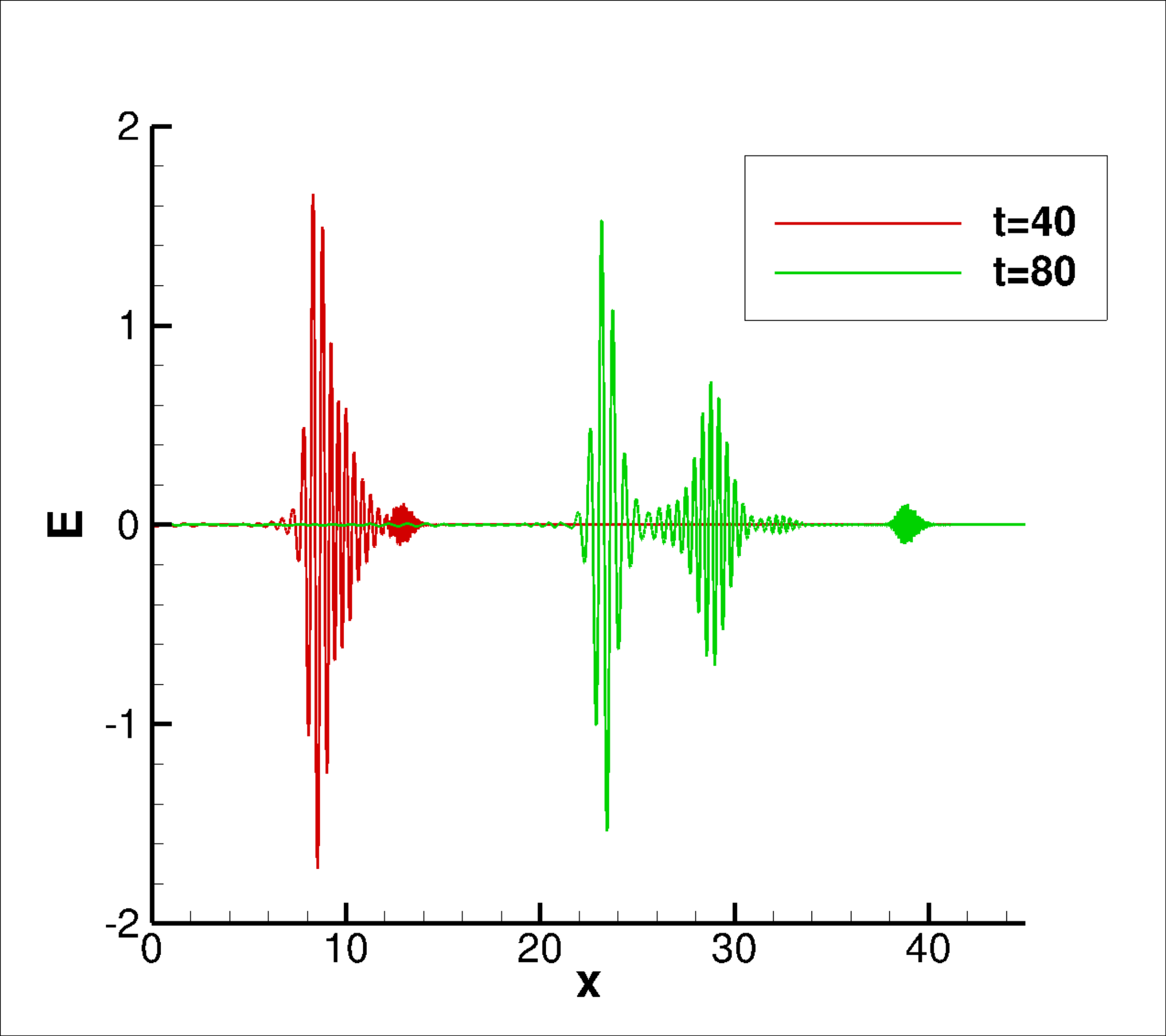}\label{Fig6.11}}
 	  \subfigure{
 	  	\includegraphics[width=0.3\textwidth]{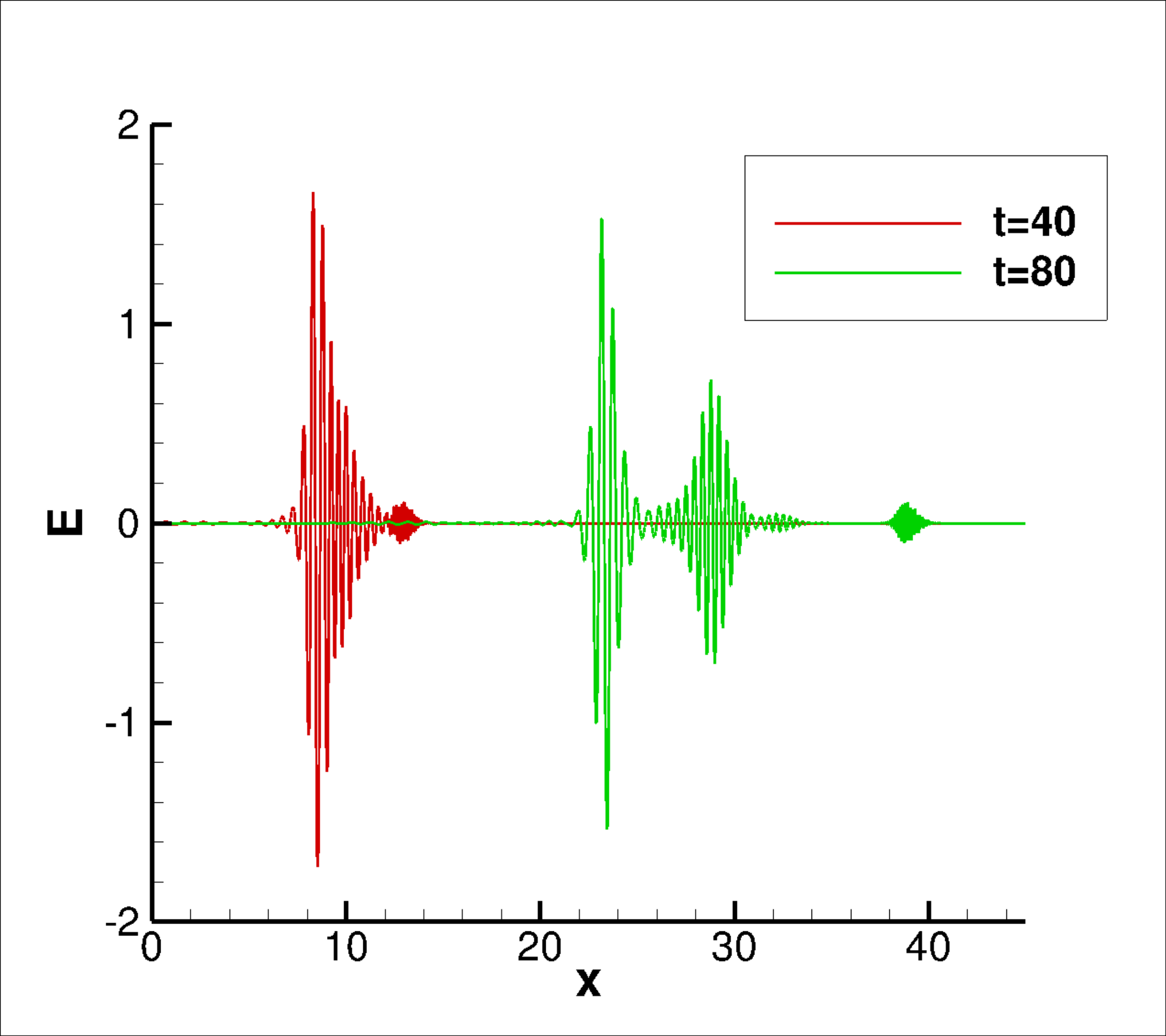}\label{Fig6.12}}
 	\caption{\em  Transient second-order ($M=2$) temporal soliton propagation with the fully implicit scheme. $N=6400$ grid points. First column: $k=1$; second column: $k=2$; third column: $k=3$. First row: upwind flux;  second row: central flux; third row: alternating flux I; fourth row: alternating flux II. }
 	\label{Fig6}
 \end{figure}

   \begin{figure}
 	\centering
 	\subfigure{
 		\includegraphics[width=0.3\textwidth]{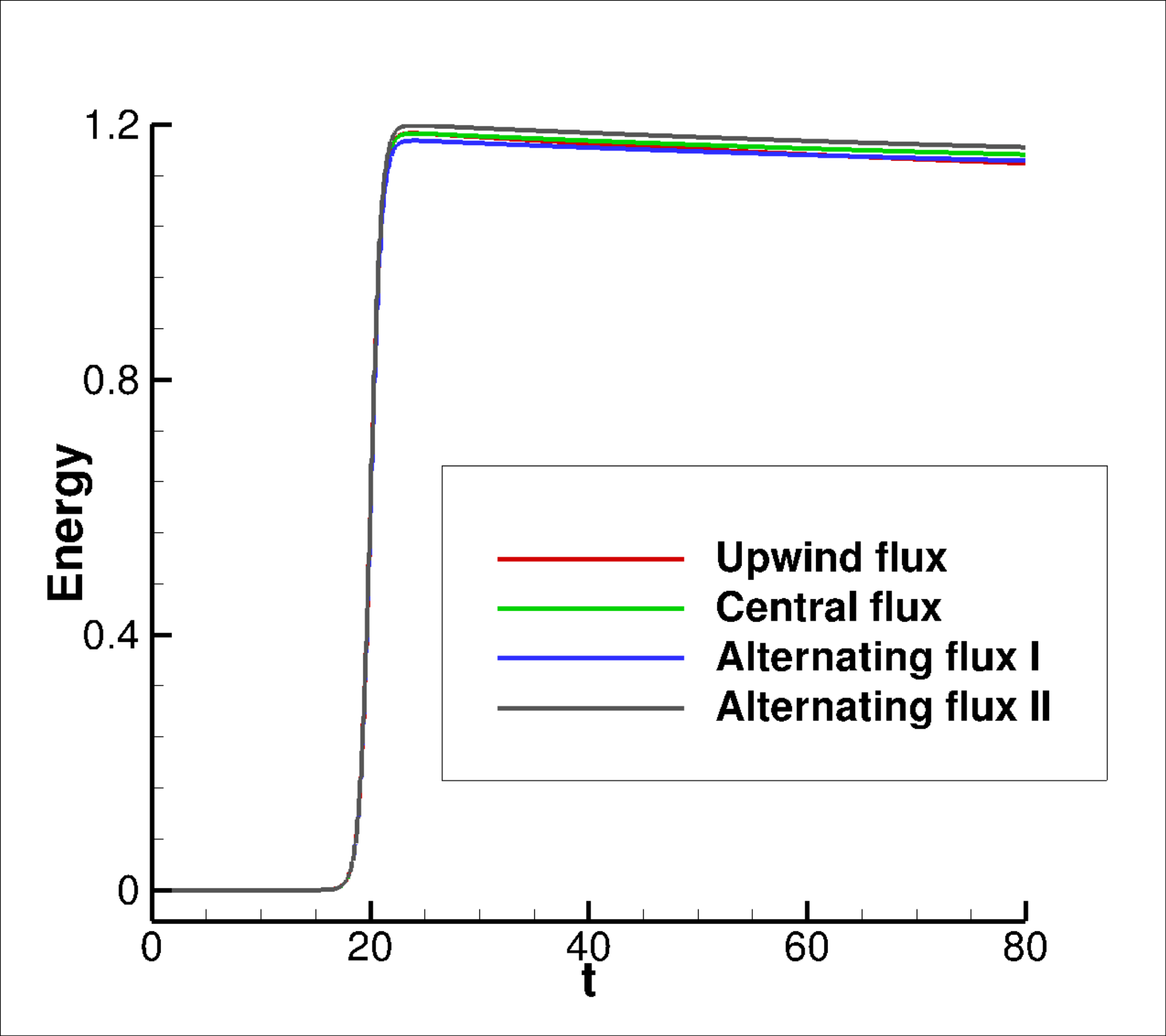}\label{Fig7.1}}
 	\subfigure{
 		\includegraphics[width=0.3\textwidth]{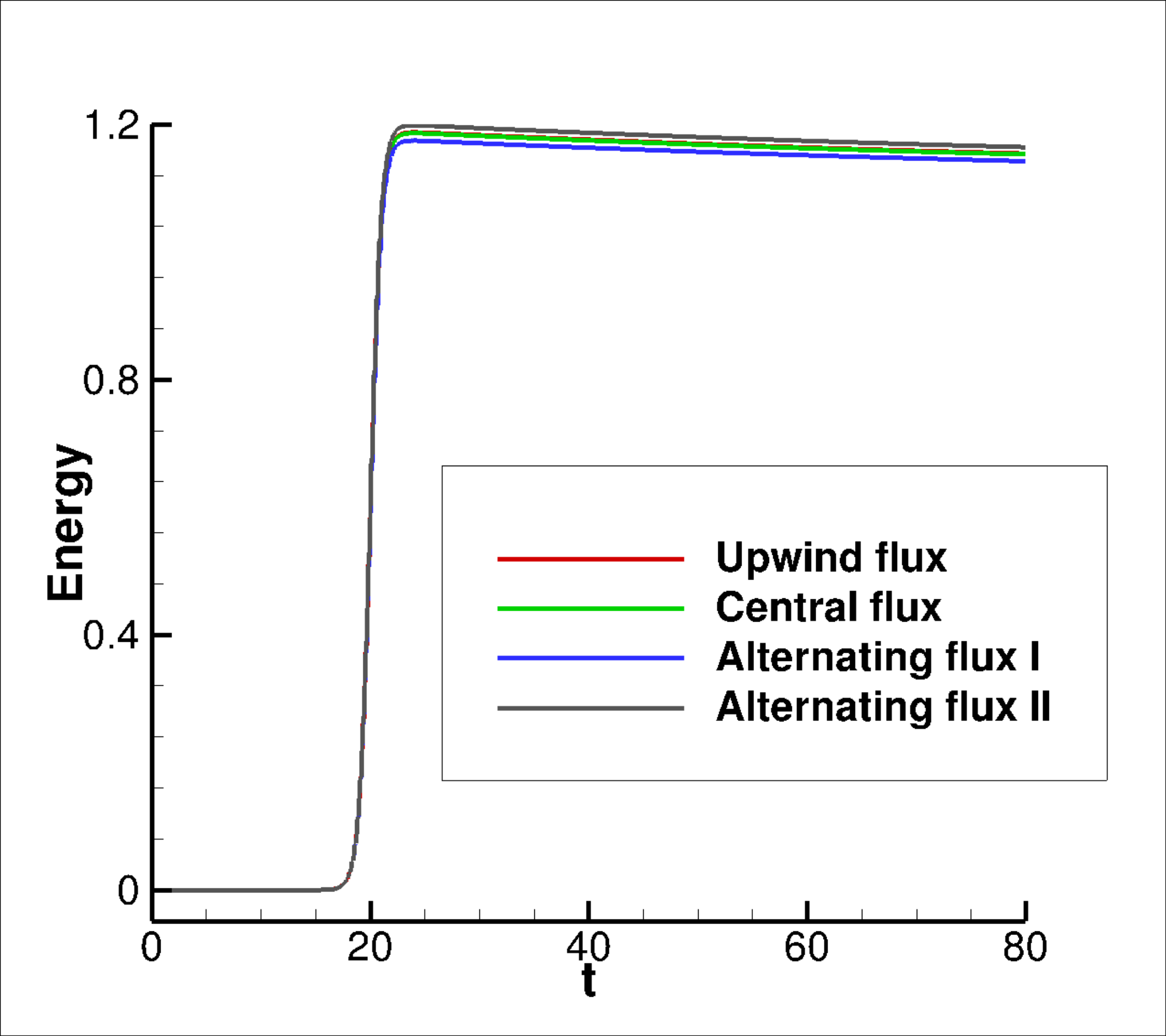}\label{Fig7.2}}
 	\subfigure{
 		\includegraphics[width=0.3\textwidth]{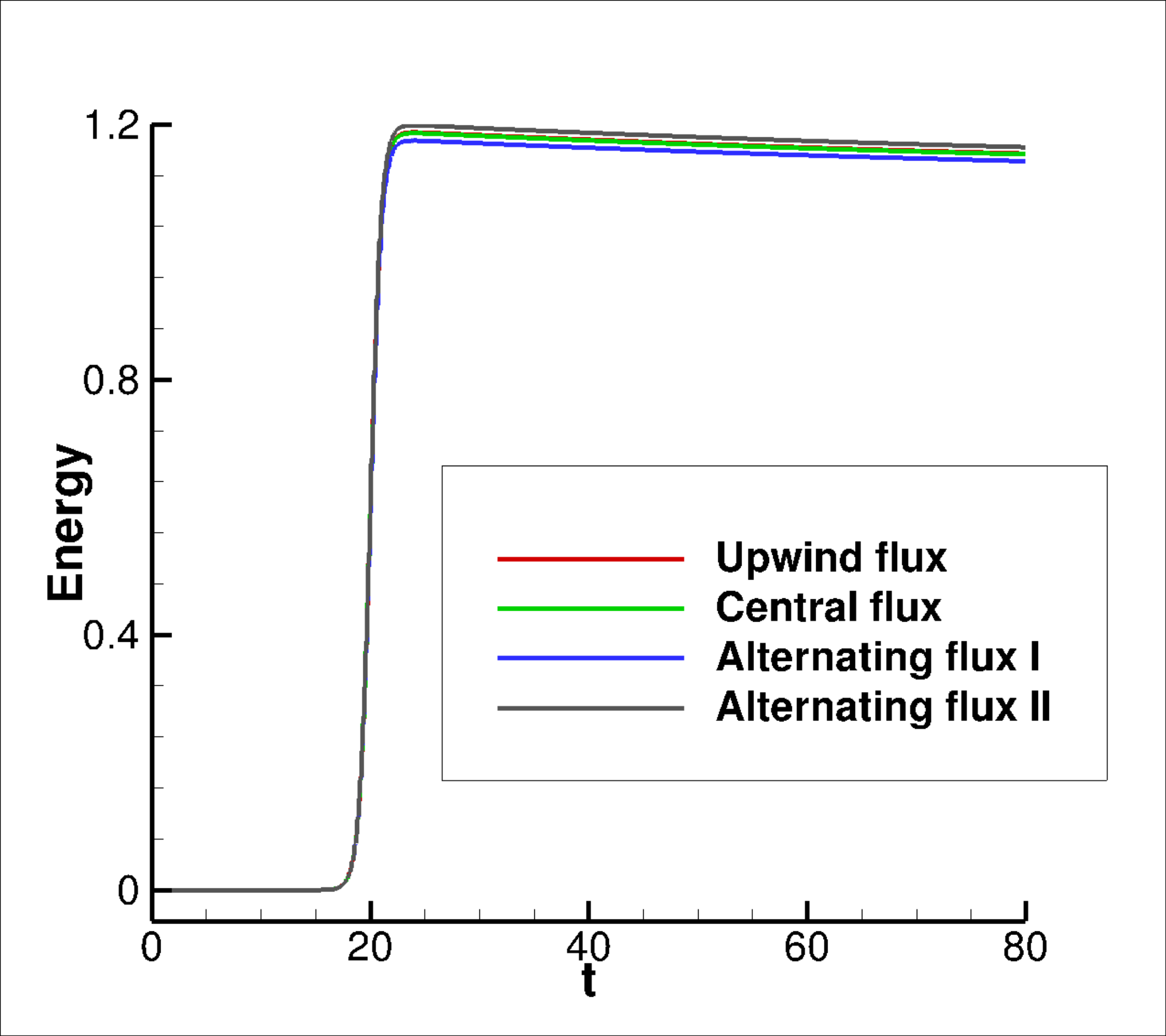}\label{Fig7.3}}
 	 \subfigure{
 		\includegraphics[width=0.3\textwidth]{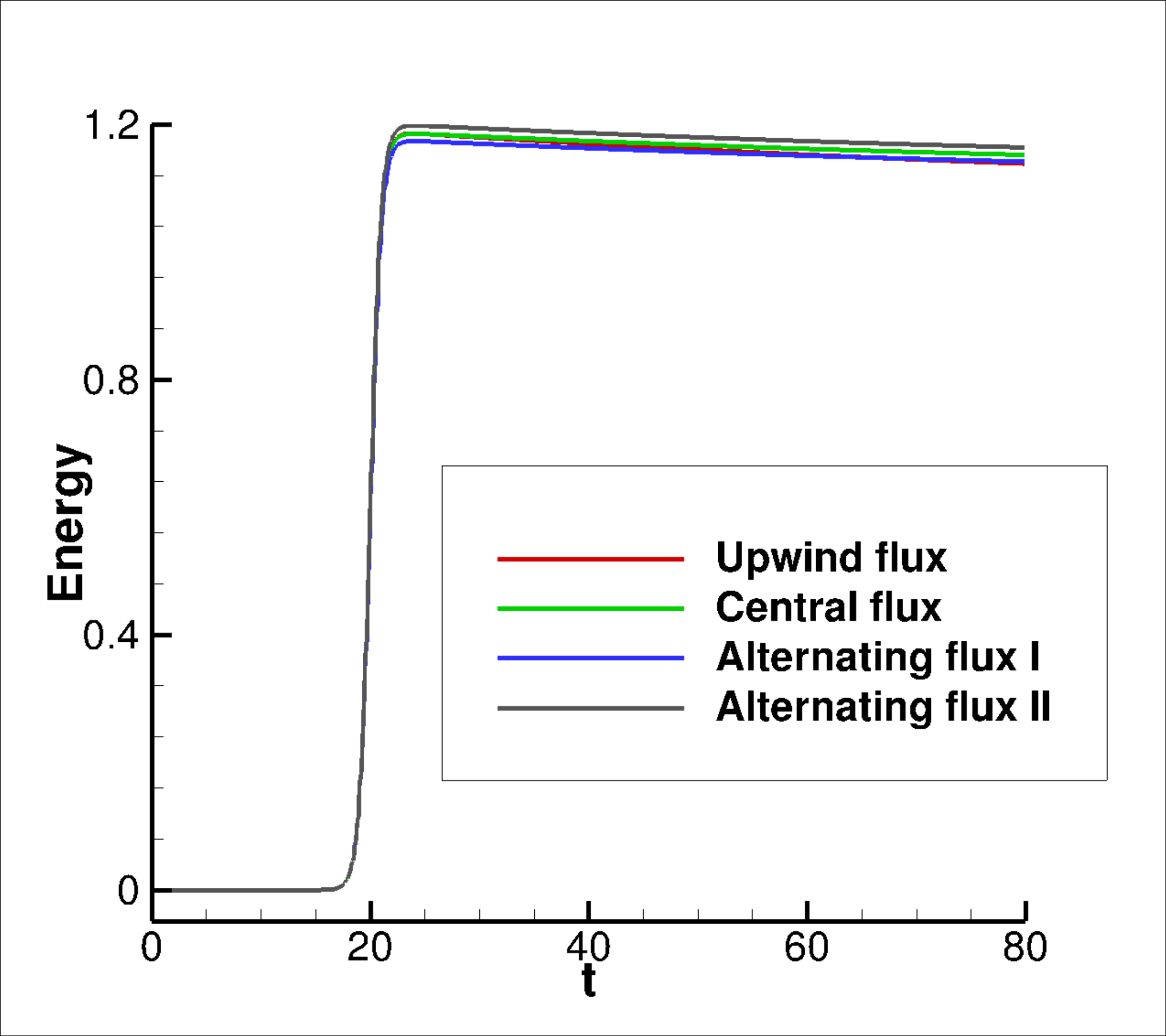}\label{Fig7.4}}
 	 \subfigure{
 	 	\includegraphics[width=0.3\textwidth]{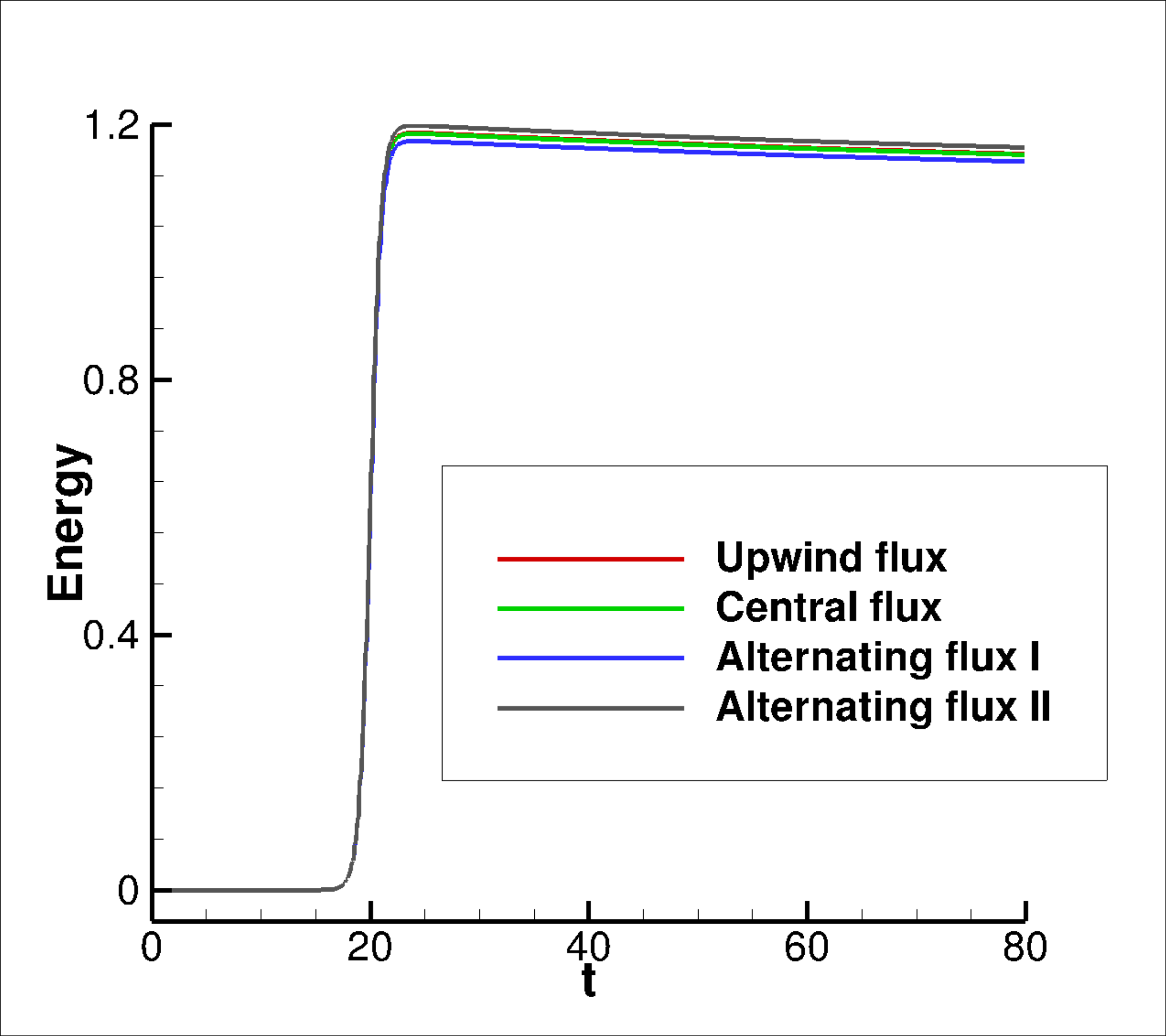}\label{Fig7.5}}
 	 \subfigure{
 	 	\includegraphics[width=0.3\textwidth]{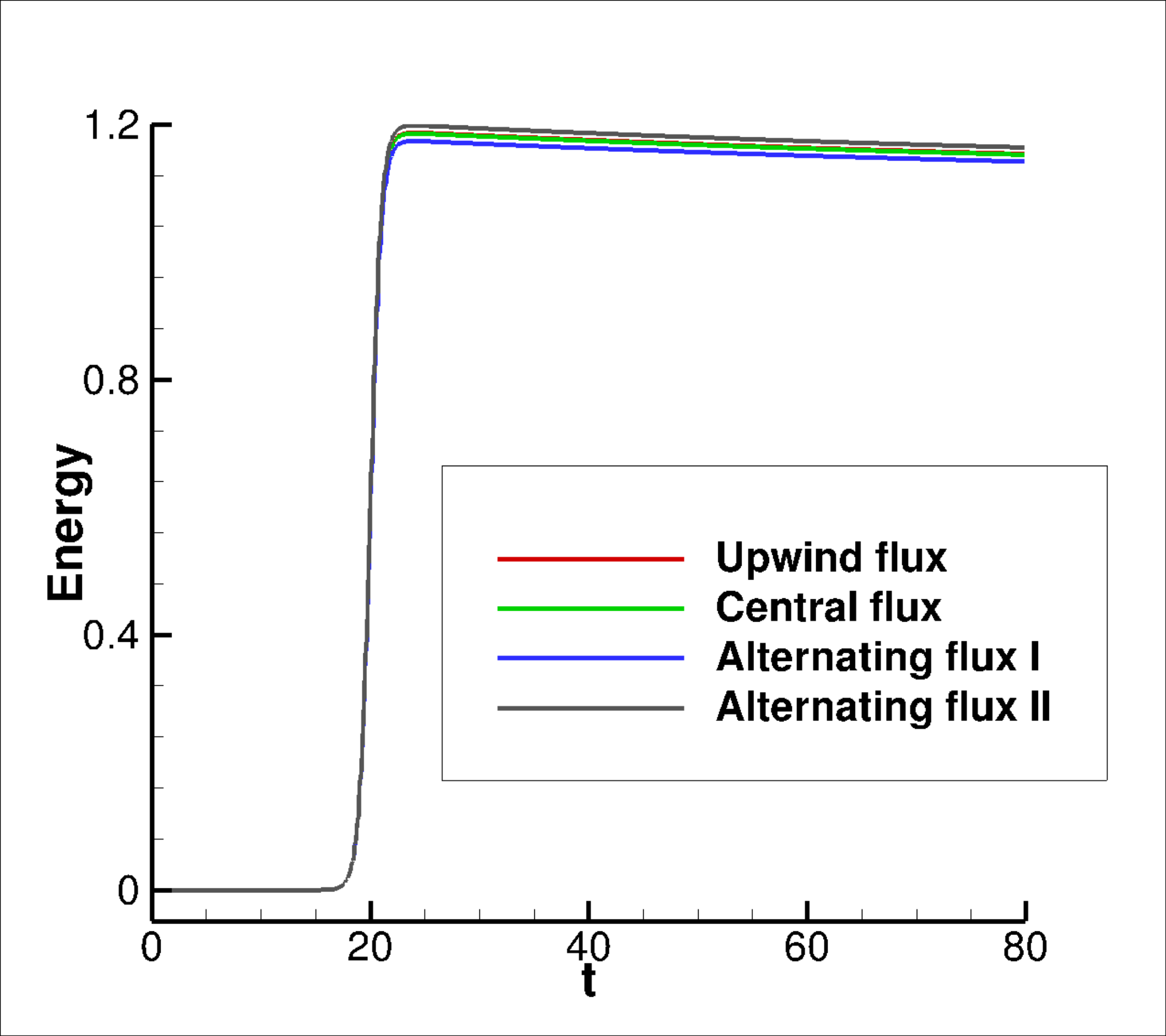}\label{Fig7.6}}
 	\caption{\em Numerical energy of transient fundamental ($M=1$) temporal soliton propagation with the leapfrog scheme and the fully implicit scheme. $N=6400$ grid points. First column: $k=1$; second column: $k=2$; third column: $k=3$. First row: the leapfrog scheme; second row: the fully implicit scheme.  }
 	\label{Fig7}
 \end{figure}

   \begin{figure}
 	\centering
 	\subfigure{
 		\includegraphics[width=0.3\textwidth]{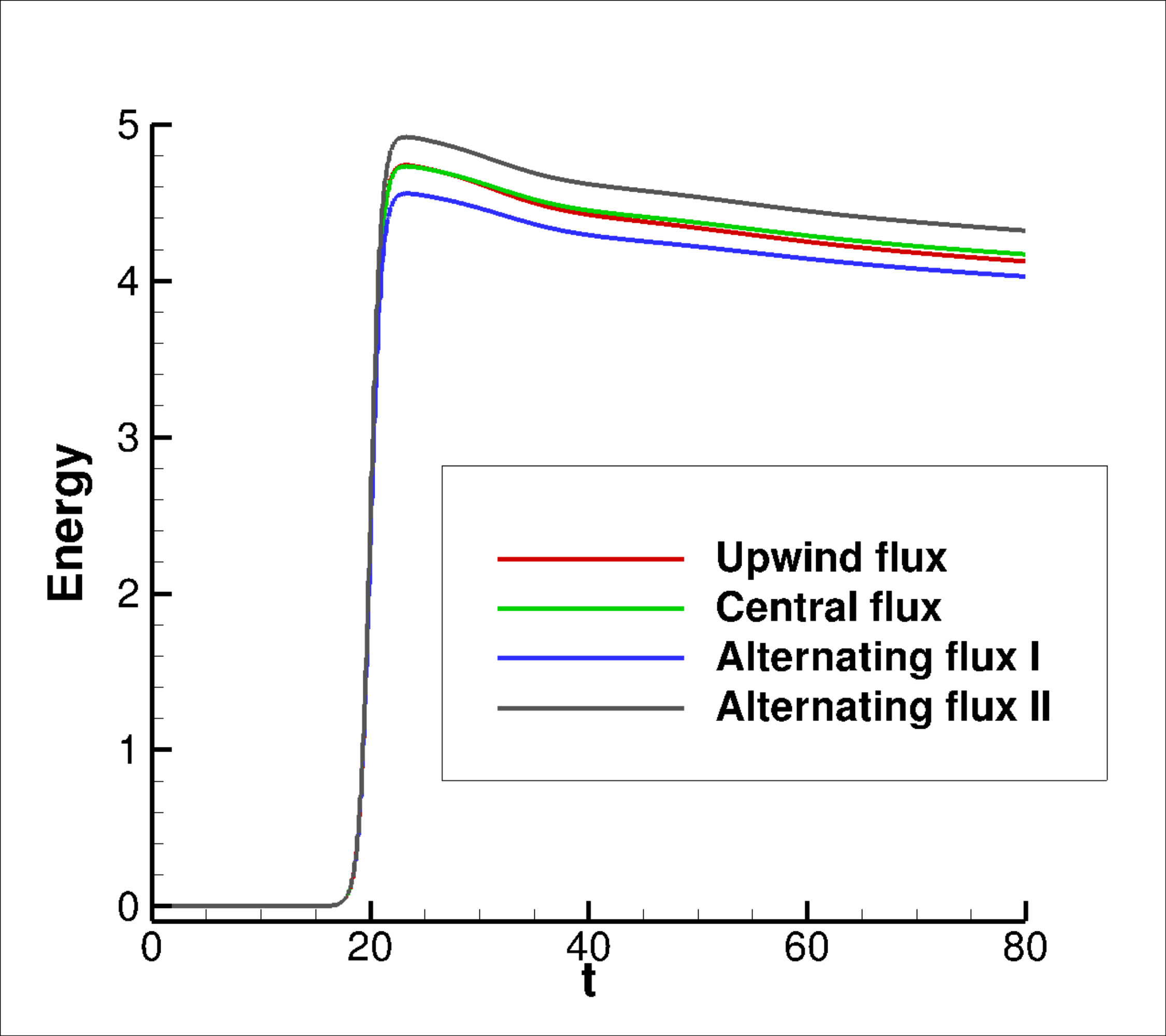}\label{Fig8.1}}
 	\subfigure{
 		\includegraphics[width=0.3\textwidth]{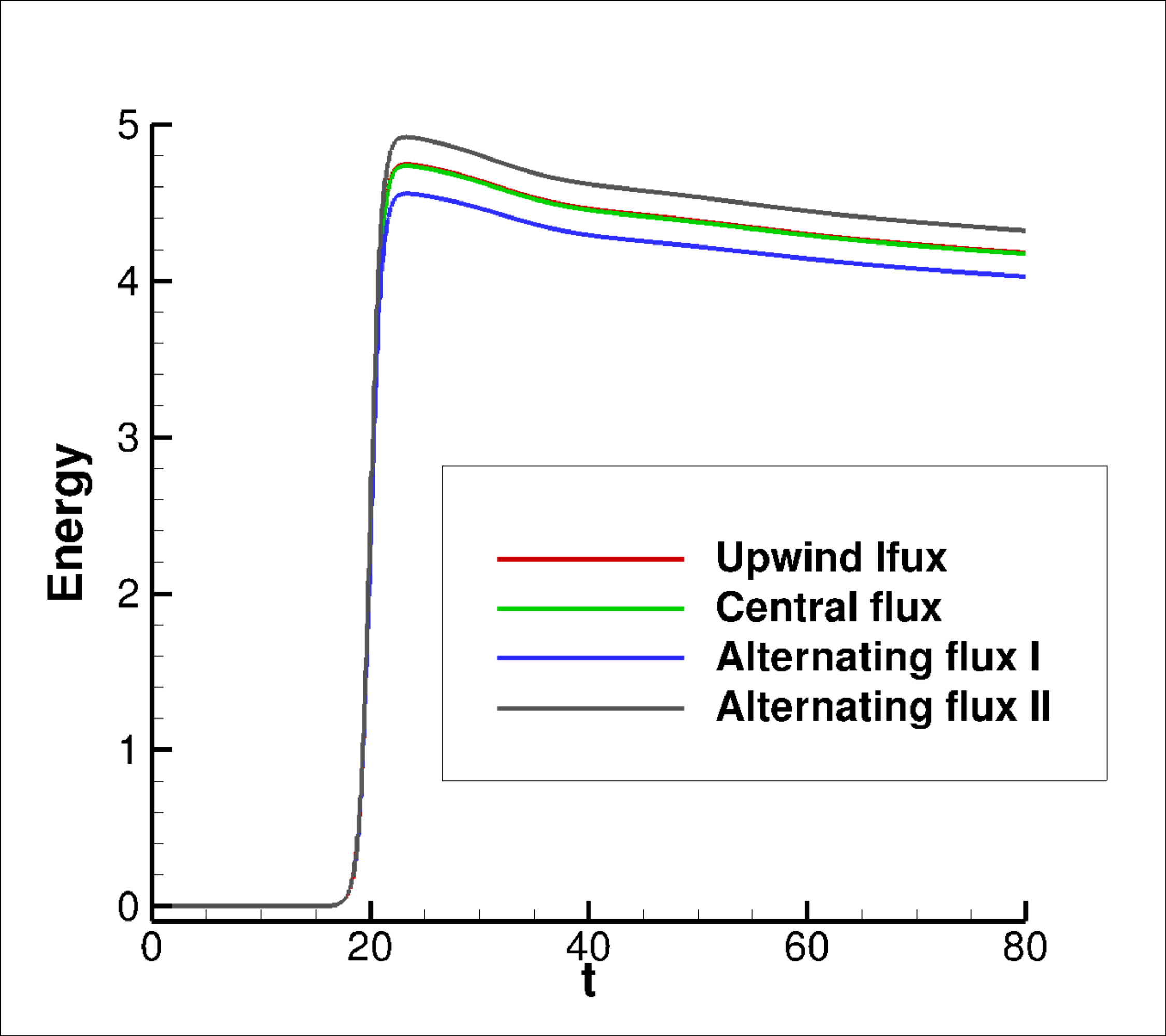}\label{Fig8.2}}
 	\subfigure{
 		\includegraphics[width=0.3\textwidth]{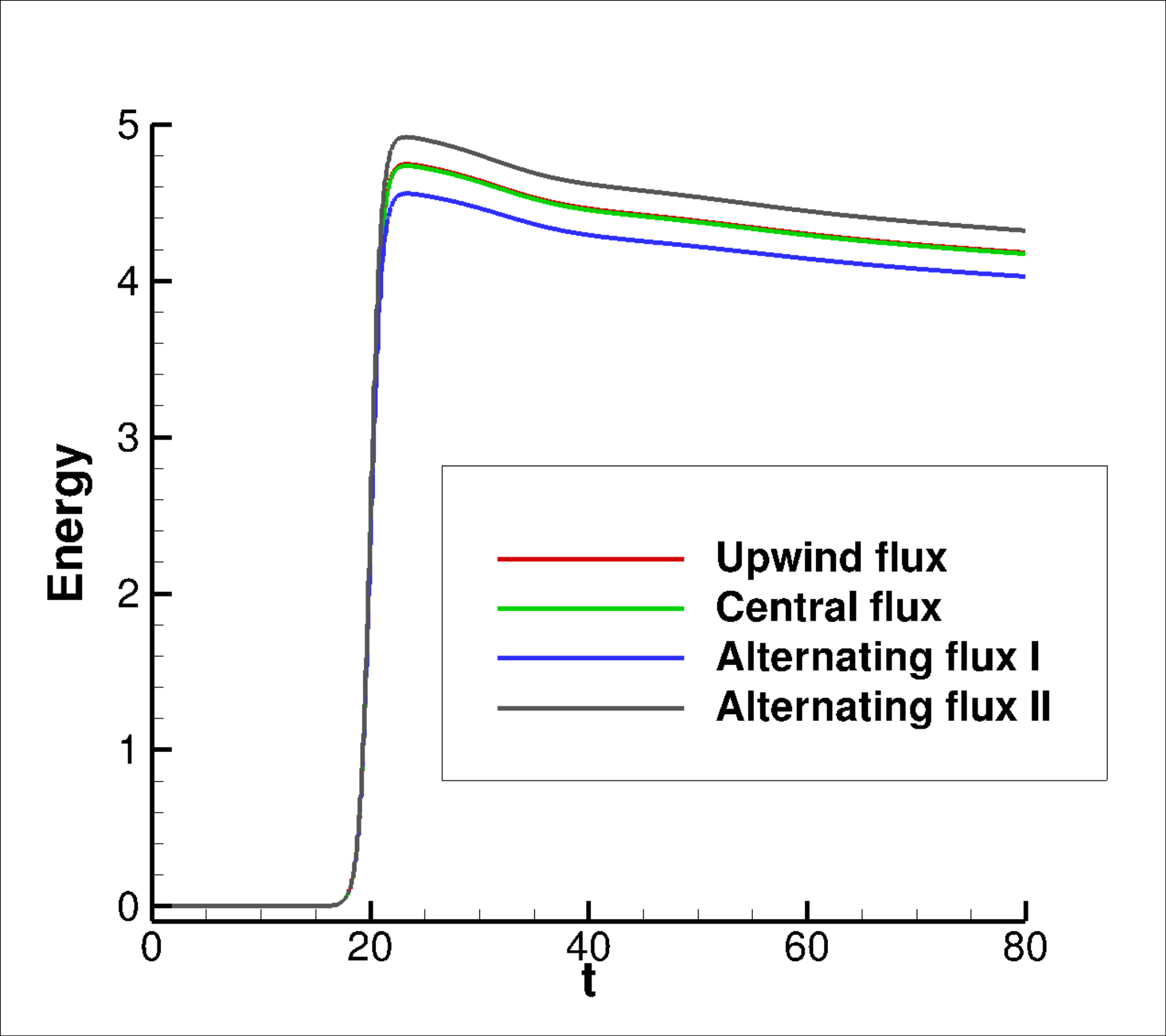}\label{Fig8.3}}
 	 \subfigure{
 		\includegraphics[width=0.3\textwidth]{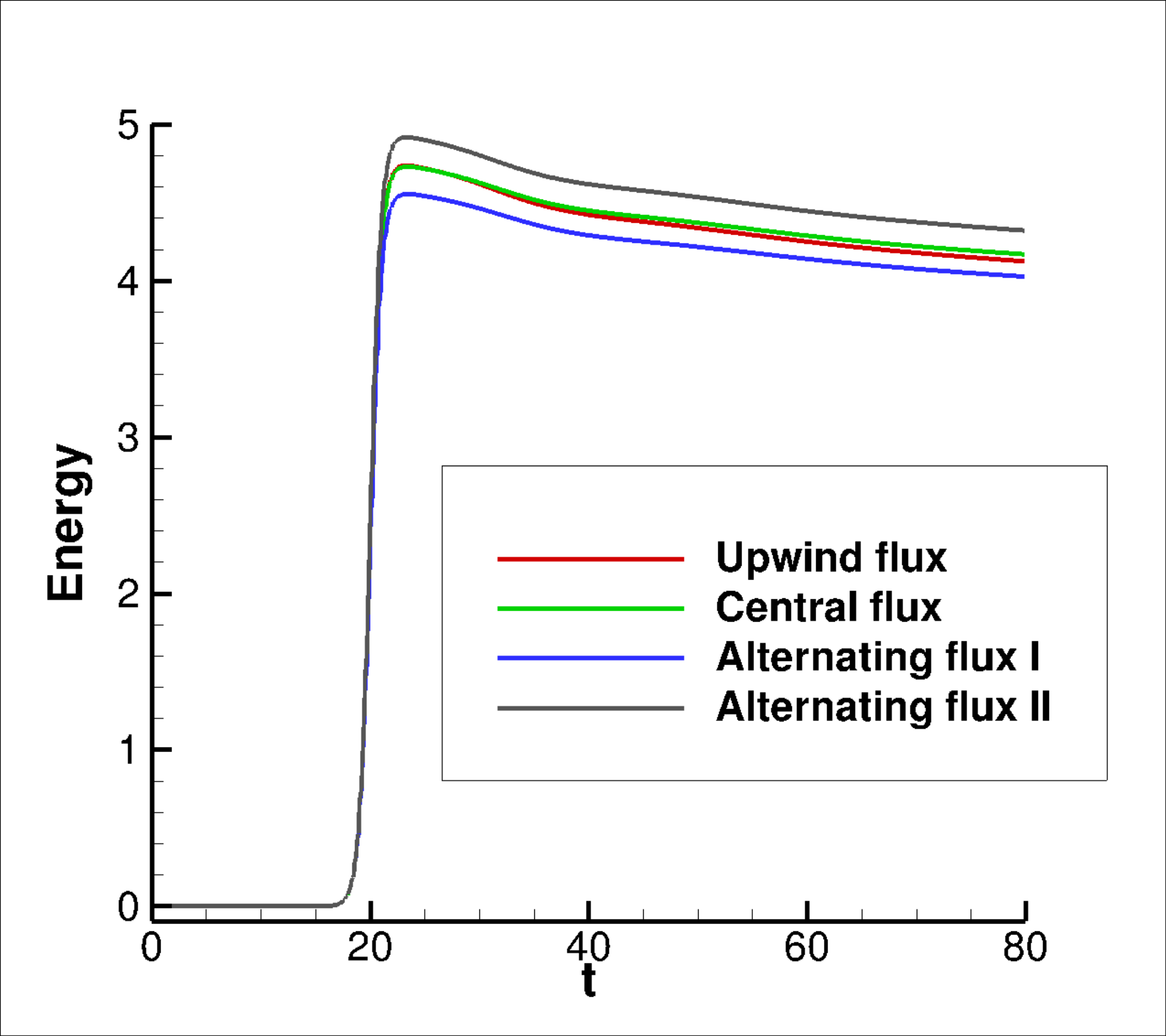}\label{Fig8.4}}
 	 \subfigure{
 	 	\includegraphics[width=0.3\textwidth]{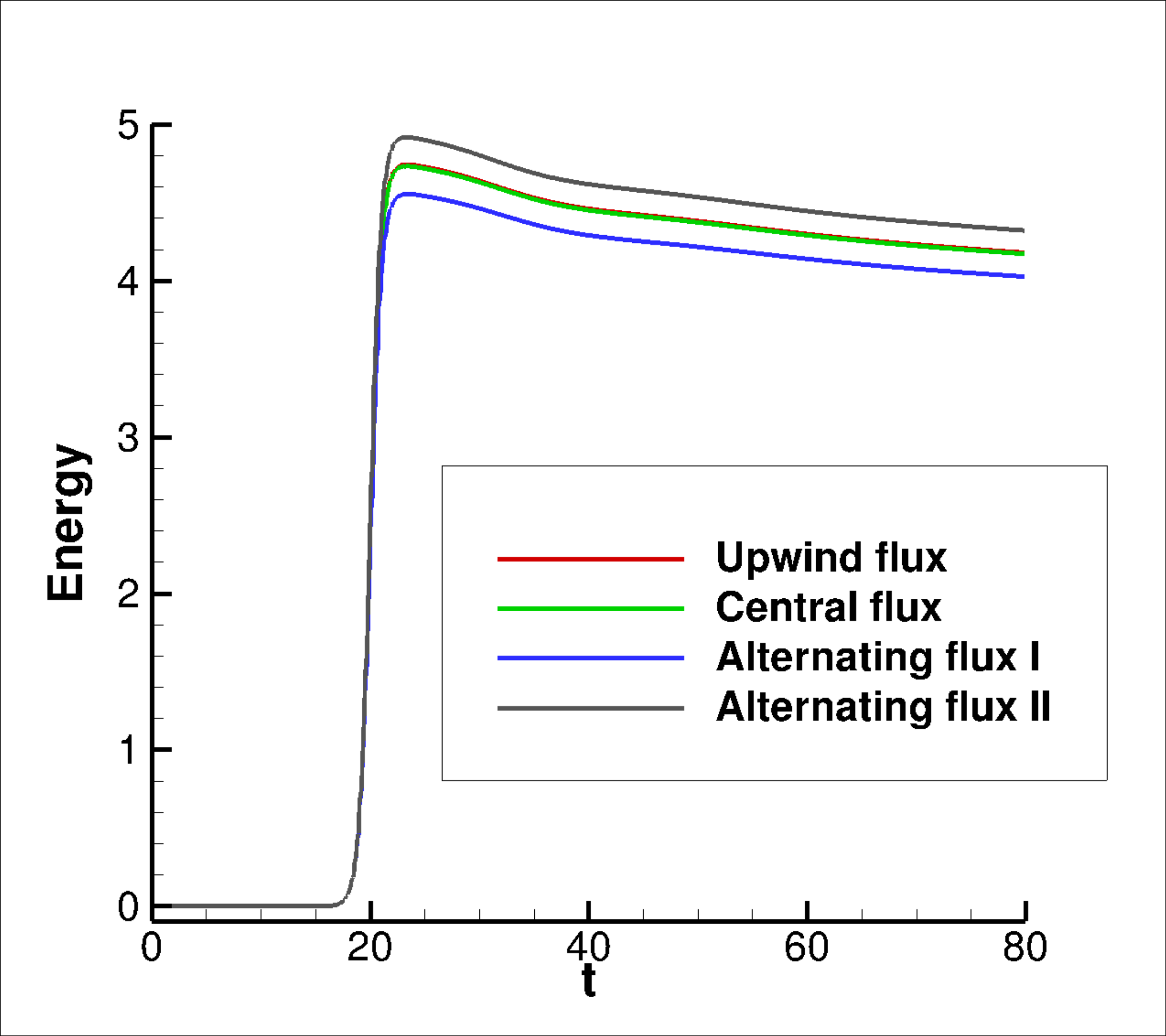}\label{Fig8.5}}
 	 \subfigure{
 	 	\includegraphics[width=0.3\textwidth]{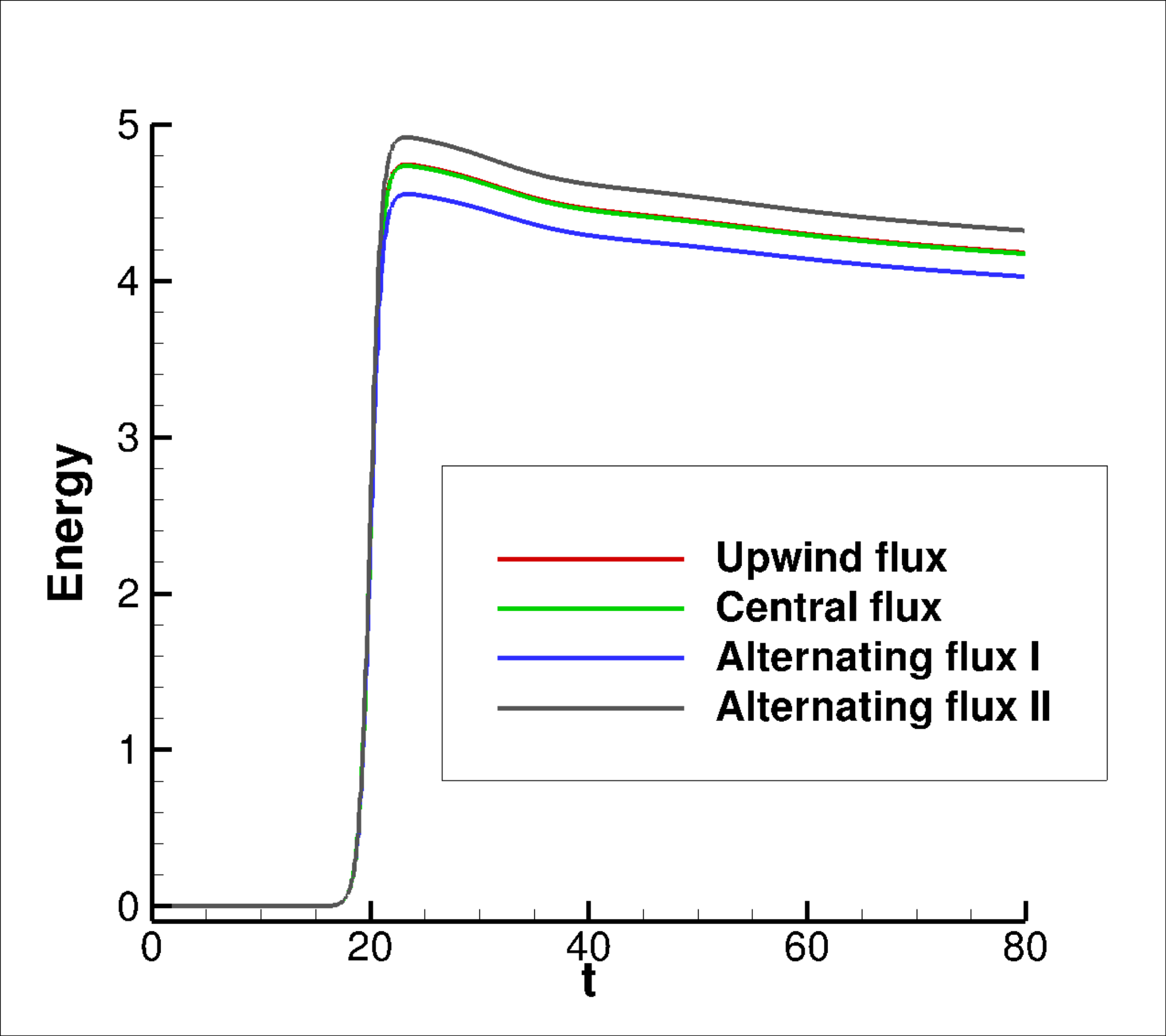}\label{Fig8.6}}
 	\caption{\em Numerical energy of transient second-order ($M=2$) temporal soliton propagation with the leapfrog scheme and the fully implicit scheme. $N=6400$ grid points. First column: $k=1$; second column: $k=2$; third column: $k=3$. First row: the leapfrog scheme; second row: the fully implicit scheme.  }
 	\label{Fig8}
 \end{figure}

   \begin{figure}
 	\centering
 	\subfigure{
 		\includegraphics[width=0.3\textwidth]{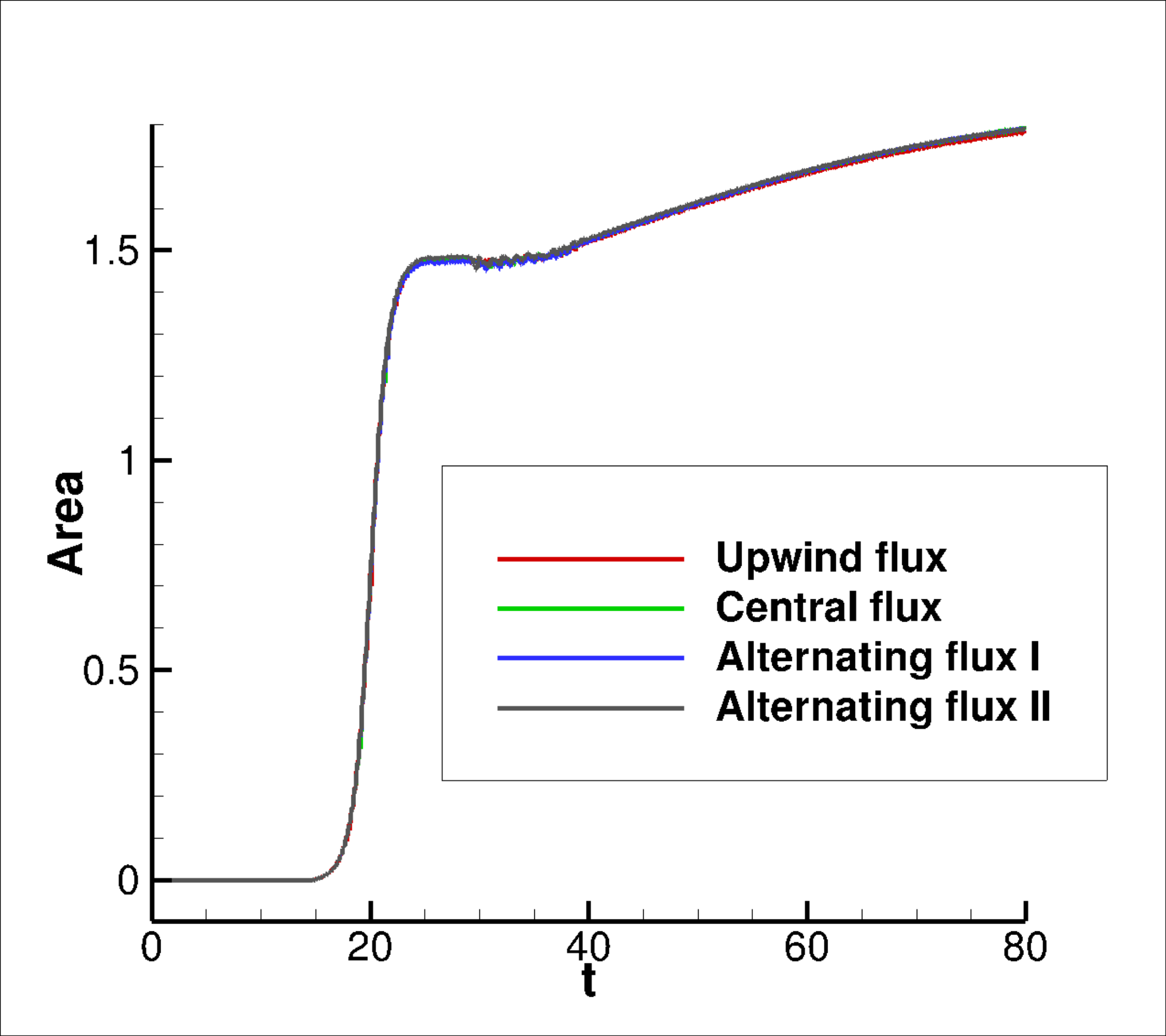}\label{Fig9.1}}
 	\subfigure{
 		\includegraphics[width=0.3\textwidth]{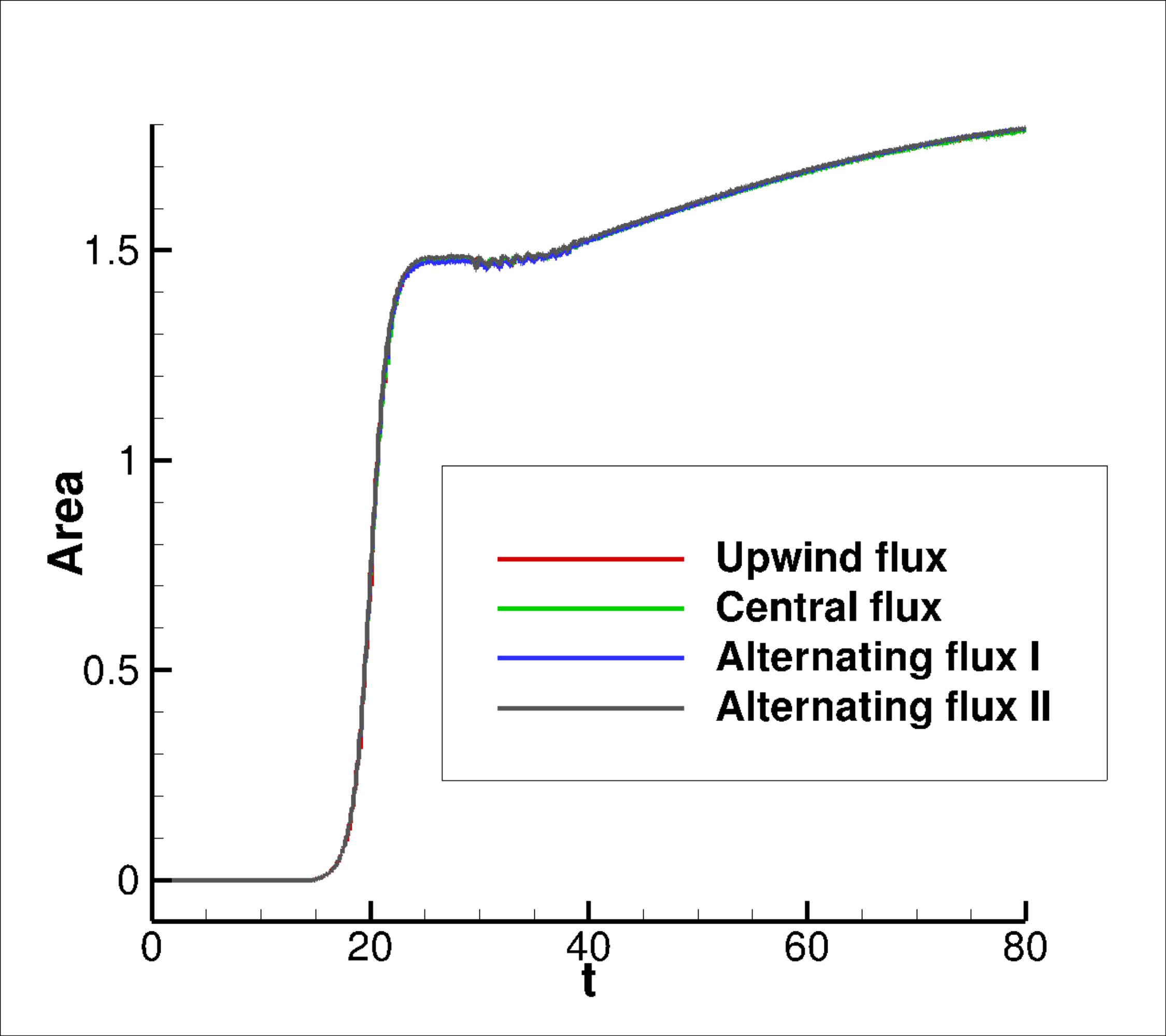}\label{Fig9.2}}
 	\subfigure{
 		\includegraphics[width=0.3\textwidth]{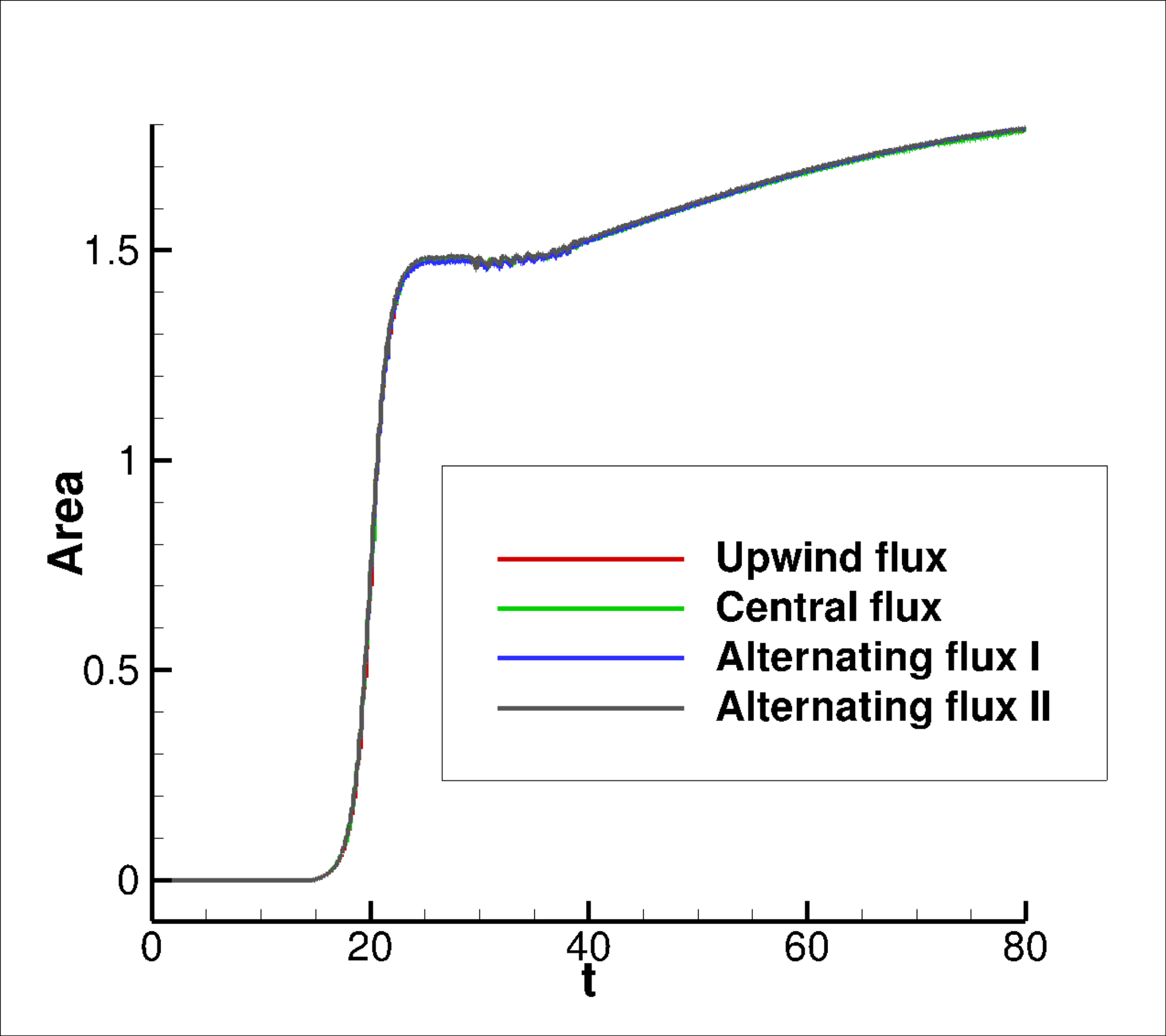}\label{Fig9.3}}
 	 \subfigure{
 		\includegraphics[width=0.3\textwidth]{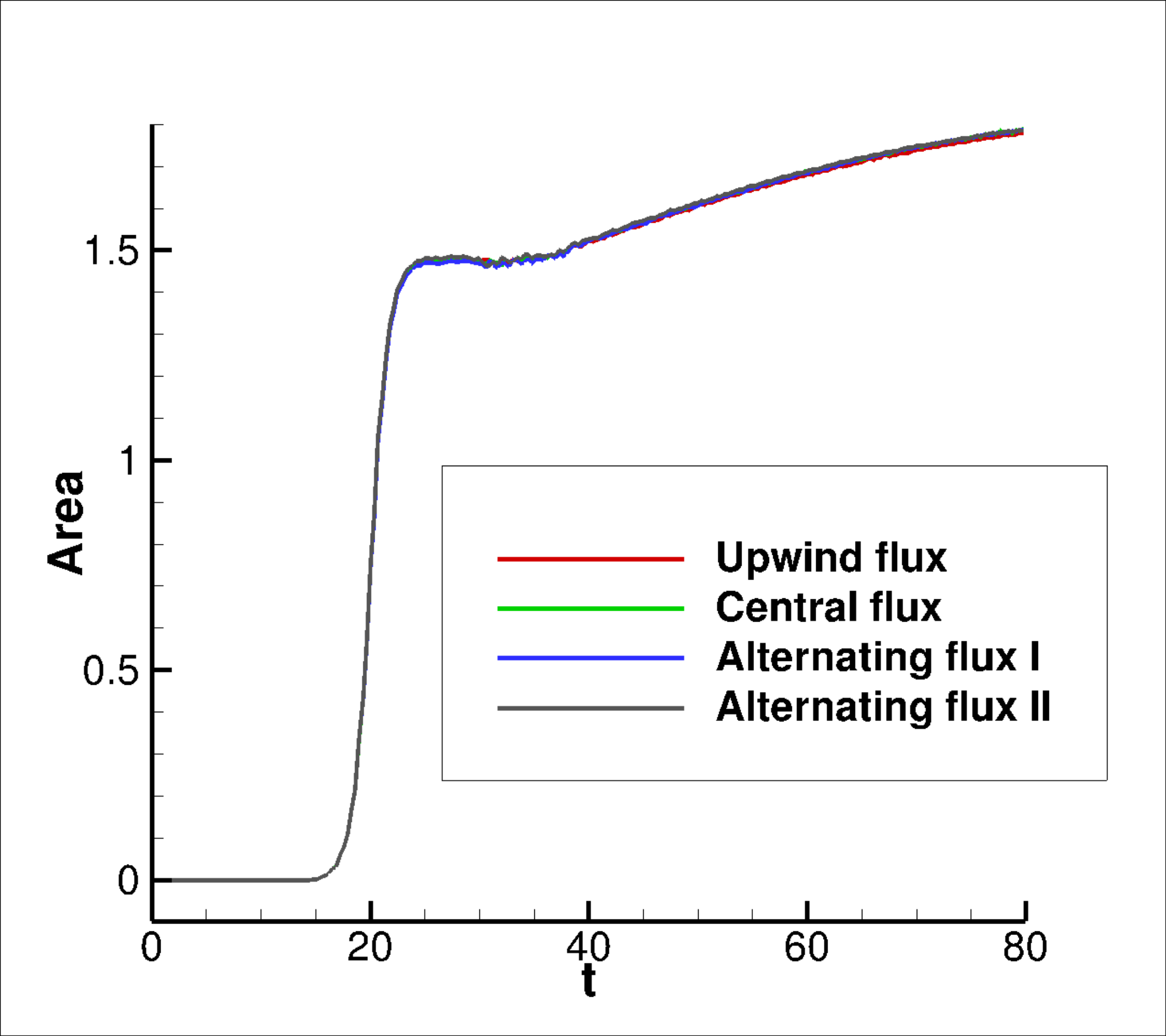}\label{Fig9.4}}
 	 \subfigure{
 	 	\includegraphics[width=0.3\textwidth]{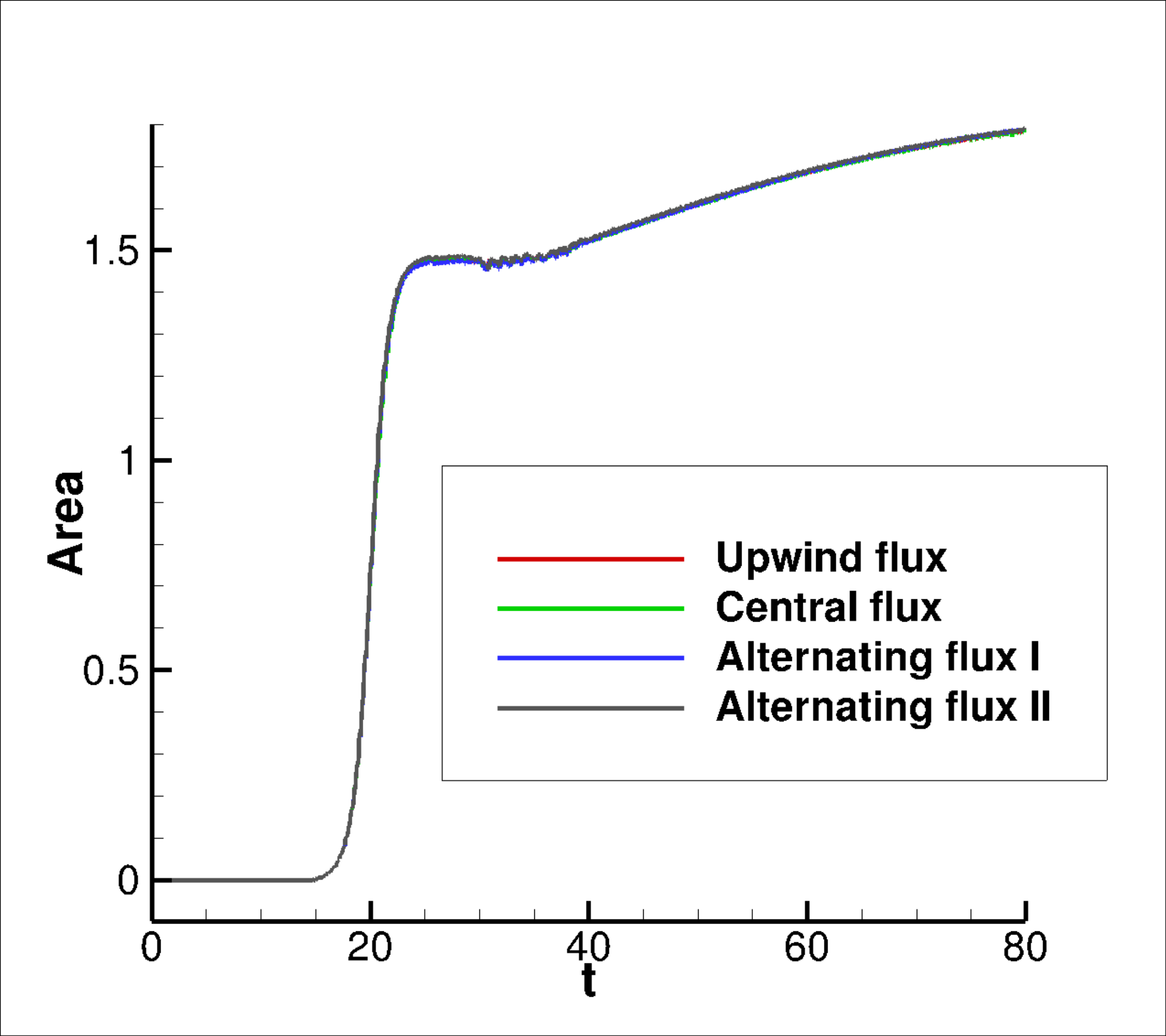}\label{Fig9.5}}
 	 \subfigure{
 	 	\includegraphics[width=0.3\textwidth]{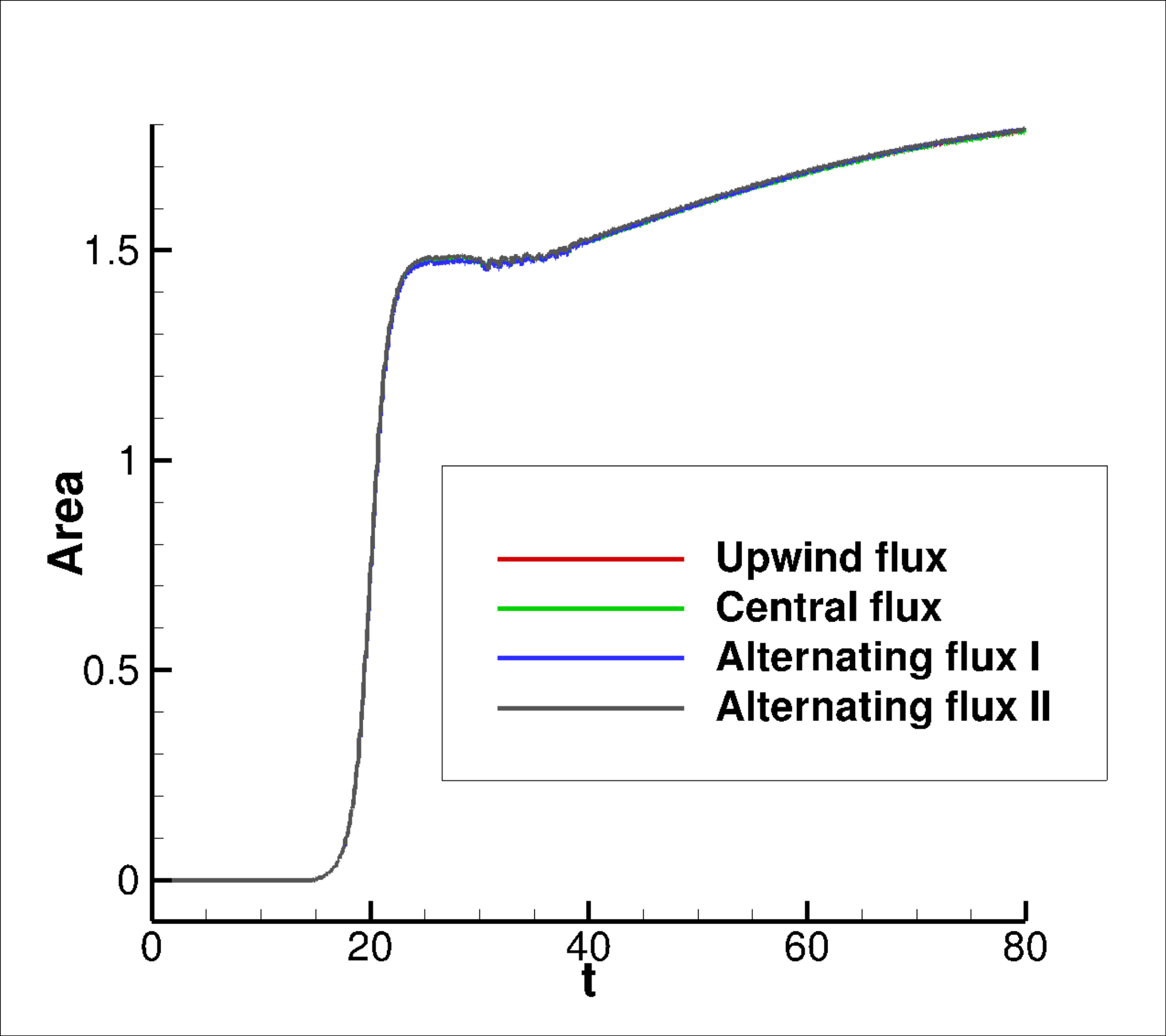}\label{Fig9.6}}
 	\caption{\em Pulse area of transient fundamental ($M=1$) temporal soliton propagation with the leapfrog scheme and the fully implicit scheme. $N=6400$ grid points. First column: $k=1$; second column: $k=2$; third column: $k=3$. First row: the leapfrog scheme; second row: the fully implicit scheme. }
 	\label{Fig9}
 \end{figure}

   \begin{figure}
 	\centering
 	\subfigure{
 		\includegraphics[width=0.3\textwidth]{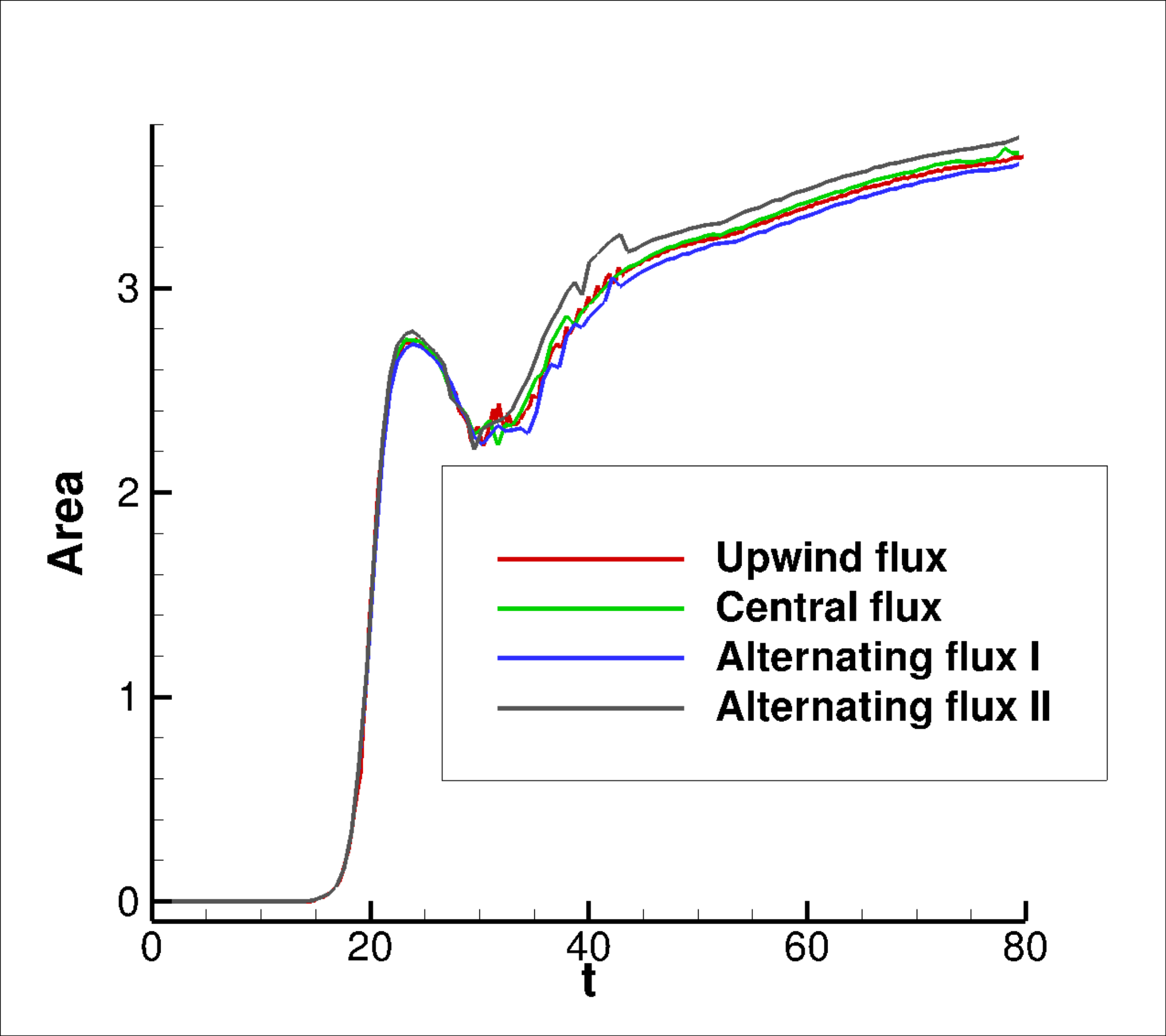}\label{Fig10.1}}
 	\subfigure{
 		\includegraphics[width=0.3\textwidth]{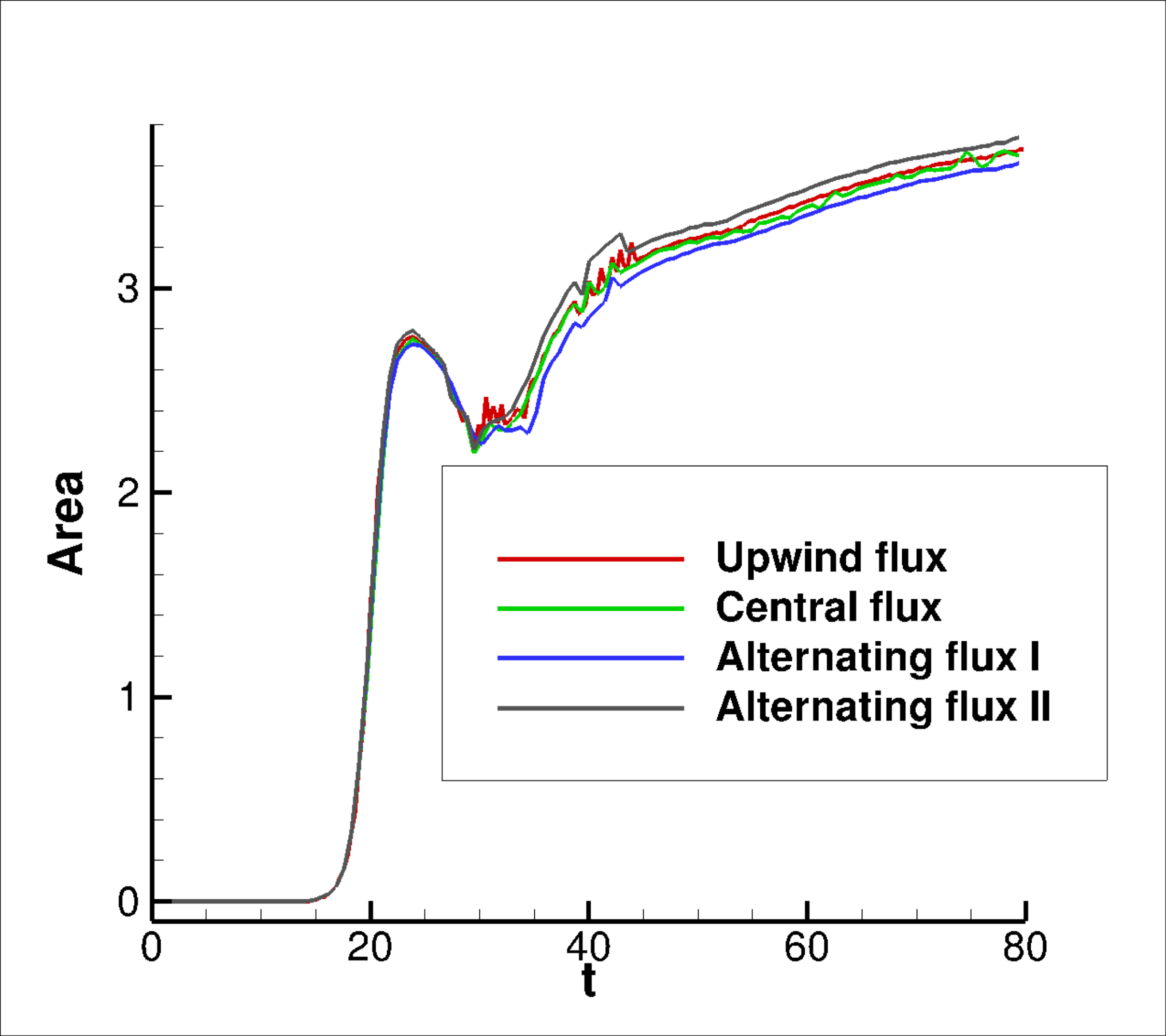}\label{Fig10.2}}
 	\subfigure{
 		\includegraphics[width=0.3\textwidth]{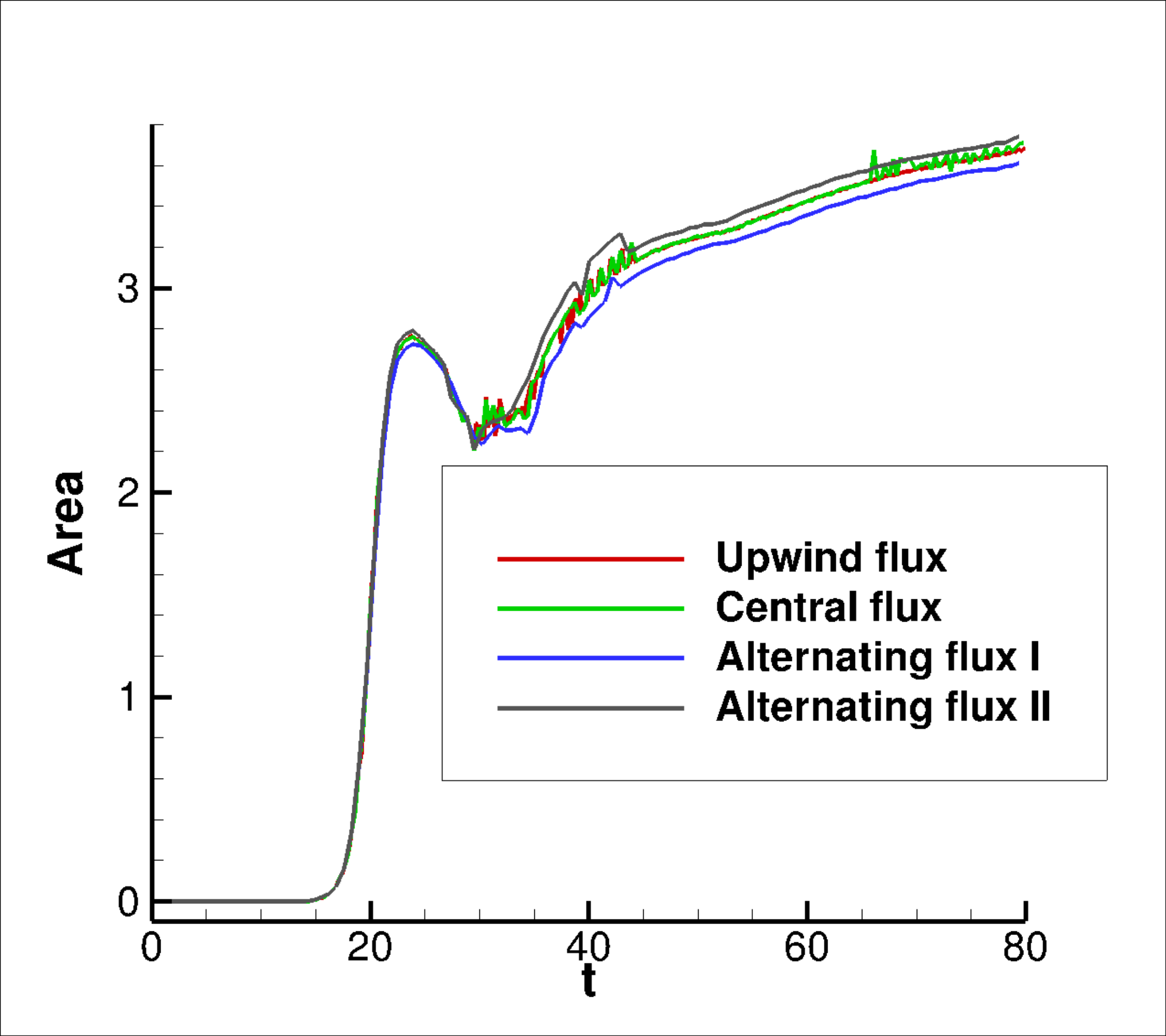}\label{Fig10.3}}
 	 \subfigure{
 		\includegraphics[width=0.3\textwidth]{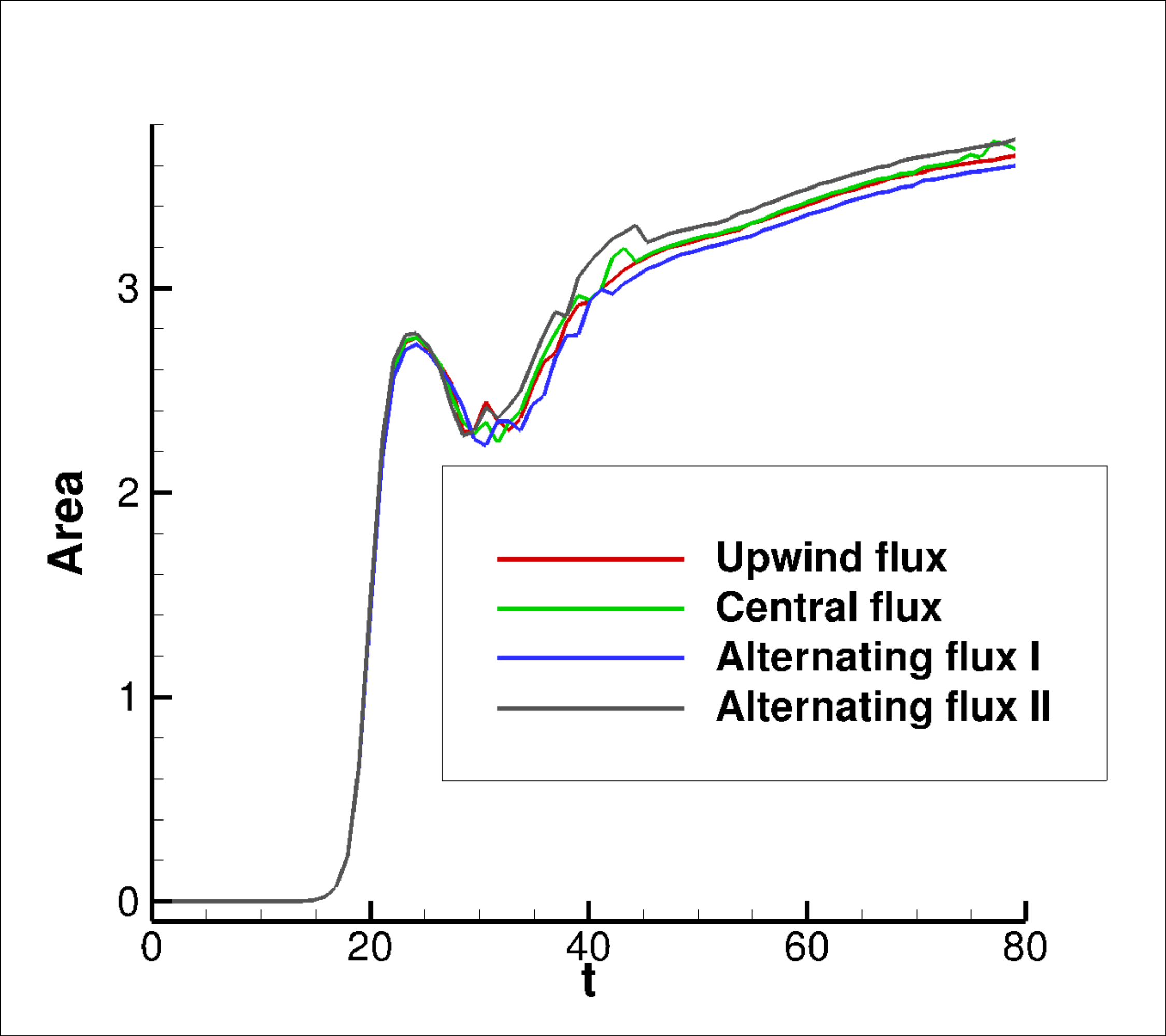}\label{Fig10.4}}
 	 \subfigure{
 	 	\includegraphics[width=0.3\textwidth]{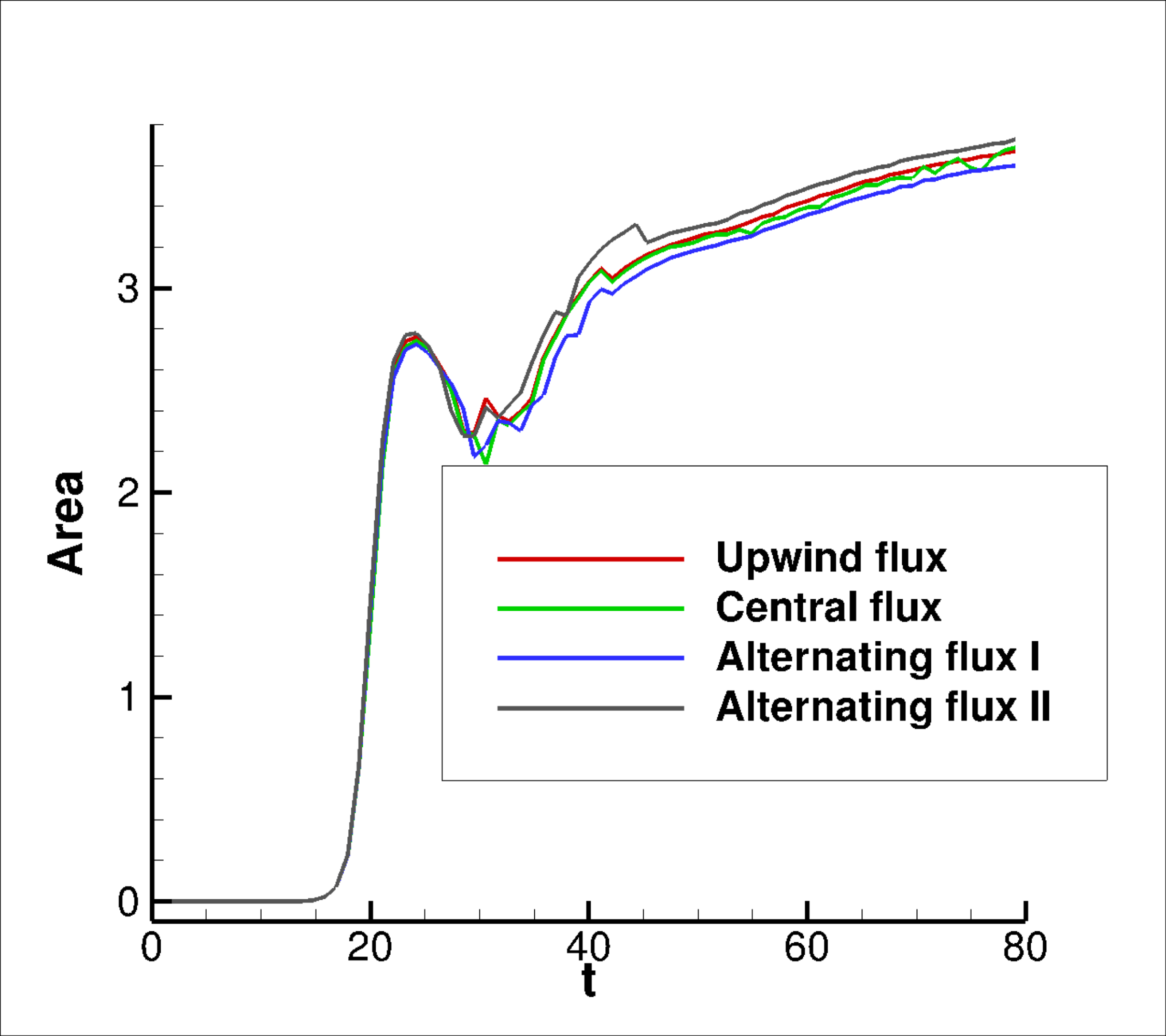}\label{Fig10.5}}
 	 \subfigure{
 	 	\includegraphics[width=0.3\textwidth]{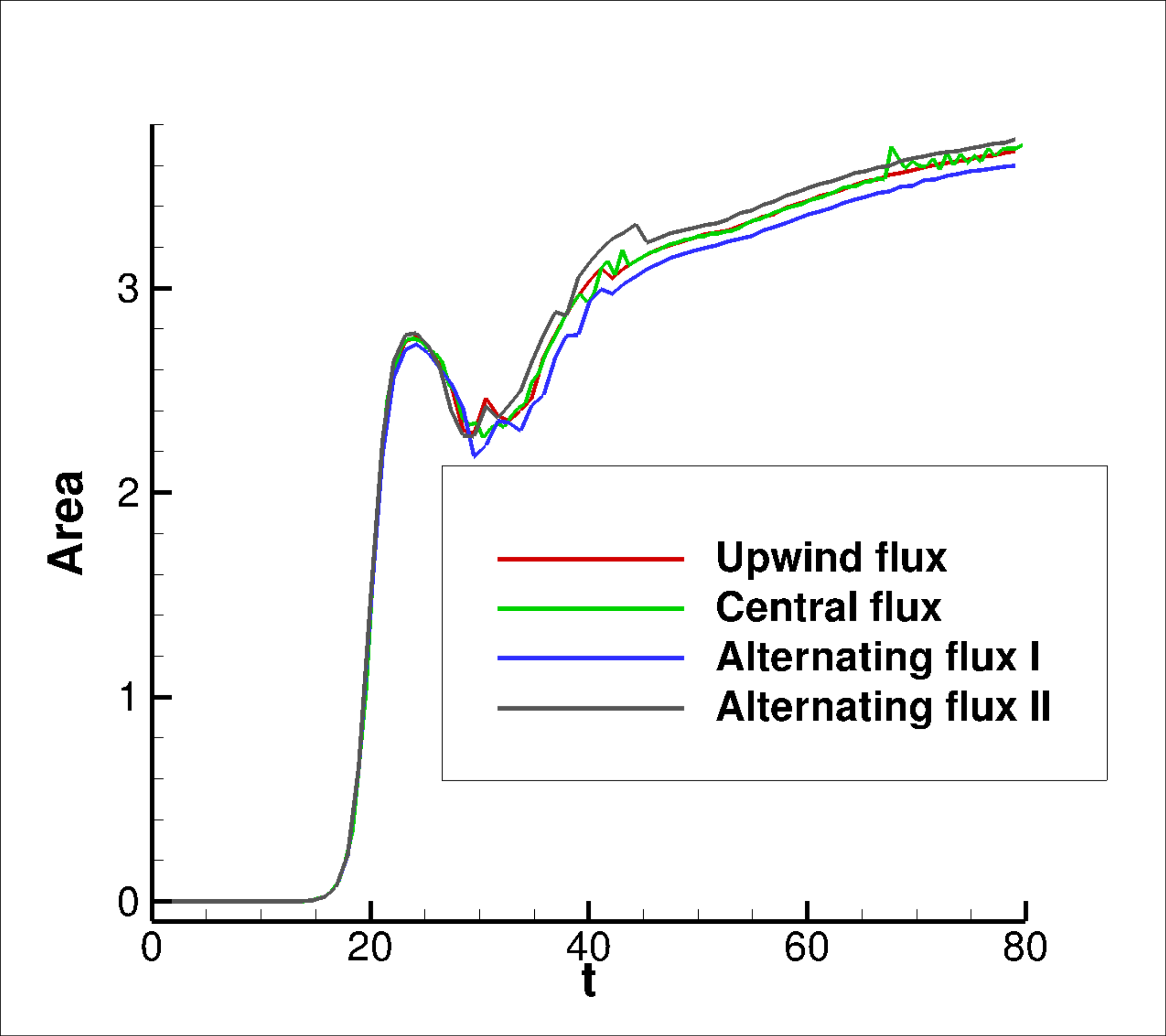}\label{Fig10.6}}
 	\caption{\em Pulse area of transient second-order ($M=2$) temporal soliton propagation with the leapfrog scheme and the fully implicit scheme. $N=6400$ grid points. First column: $k=1$; second column: $k=2$; third column: $k=3$. First row: the leapfrog scheme; second row: the fully implicit scheme.  }
 	\label{Fig10}
 \end{figure}

\section{Conclusions}
\label{sec:conclusion}

In this paper, we propose fully discrete energy stable schemes for 1D Maxwell's equations in nonlinear optics. The schemes use novel treatments in temporal discretizations and discontinous Galerkin schemes in space with various choices of fluxes. We prove semi-discrete and fully discrete energy stability of the proposed methods, and provide error estimates of the semi-discrete schemes with conditions on the strength of the nonlinearity of the system.
Numerical results validate the theoretical predictions, which show that the fully implicit scheme allow larger CFL numbers than the leap-frog schemes. The upwind flux exhibits more dissipation, which can damp the spurious oscillations from the boundary treatment, but also in the mean time affect the effective capturing of the daughter pulse for the soliton propagation example for low order polynomial spaces. From our experience, the alternating fluxes outperform the central and upwind fluxes in the numerical examples studied in terms of accuracy and resolution of the wave profiles.
Future work includes  extensions to higher dimensions, to finite difference schemes, and Fourier analysis of the semi-discrete and fully discrete DG methods for linearized Maxwell systems in dispersive media.

\appendix
\section{Energy relation for the fully discrete schemes with non-periodic boundary conditions in Section \ref{sec:num2}}

Here, we list the  energy relation for the fully discrete schemes with boundary conditions as discussed in Section \ref{sec:num2}.

 The results with fully implicit time discretizations are very similar to the semi-discrete case, i.e. we have that
 the fully implicit scheme with alternating and central fluxes  satisfies
\begin{equation}
\mathcal{E}_h^{n+1}- \mathcal{E}_h^{n}=-\frac{\Delta t}{4 \omega_p^2 \tau} \int_\Omega (J_h^{n+1} +J_h^{n})^2 dx-\frac{a\theta \Delta t}{8\omega_v^2 \tau_v}  \int_\Omega  (\sigma_h^{n+1} +\sigma_h^{n})^2  dx -\Delta t\Theta_{in}^{n} -\Delta t \Theta^{n}_{out}\le -\Delta t\Theta^{n}_{in},
\end{equation}
and that with the upwind flux satisfies
\begin{eqnarray}
 \mathcal{E}_h^{n+1}- \mathcal{E}_h^{n} &=&
-\frac{\Delta t}{4 \omega_p^2 \tau} \int_\Omega (J_h^{n+1} +J_h^{n})^2 dx-\frac{a\theta \Delta t}{8\omega_v^2 \tau_v}  \int_\Omega  (\sigma_h^{n+1}+\sigma_h^{n})^2  dx 
-\frac{\Delta t}{8\sqrt{\epsilon_\infty}} \sum_{j=1}^{N-1}[H_h^{n}+H_h^{n+1}]_{j+1/2}^2\\
&&-\frac{\Delta t\sqrt{ \epsilon_\infty}}{8}\sum_{j=1}^{N-1}[E_h^n+E_h^{n+1}]_{j+1/2}^2 -\Delta t\Theta_{in}^{n} -\Delta t \Theta^{n}_{out}\le -\Delta t\Theta^{n}_{in},\notag
\end{eqnarray}
where 
\begin{eqnarray}
\mathcal{E}_h^n&=&\int_{\Omega}  \frac{1}{2} (H_h^{n})^2 + \frac{ \epsilon_\infty}{2} (E_h^n)^2 +  \frac{1}{2\omega_p^2} (J_h^n)^2 + \frac{\omega_0^2}{2 \omega_p^2}   (P_h^n)^2+ \frac{a\theta}{4\omega_v^2} (\sigma_h^n)^2 + \frac{a\theta}{2}  Q_h^n (E_h^n)^2  \notag\\
&& + \frac{3 a (1-\theta)}{4} (E_h^n)^4+\frac{a\theta}{4}(Q_h^n)^2  dx.\notag
\end{eqnarray}
\begin{align*}
0 \le\Theta^{n}_{out}&=&\left\{\begin{array}{ll}
\frac{1}{16\sqrt{\epsilon_{\infty}}} ((H^{n+1}_{h}+H^{n}_{h})^{-}_{N+1/2}
-\sqrt{\epsilon_{\infty}}(E^{n+1}_{h}+E^{n}_{h})^{-}_{N+1/2})^2, & \mbox{ for central and alternating fluxes},\\
\frac{1}{8\sqrt{\epsilon_{\infty}}}((H^{n+1}_{h}+H^{n}_{h})^{-}_{N+1/2})^2
+\frac{\sqrt{\epsilon_{\infty}}}{8}((E^{n+1}_{h}+E^{n}_{h})^{-}_{N+1/2})^2, & \mbox{ for upwind flux},\\
\end{array}
\right.
\end{align*}

\begin{equation*}
\Theta^{n}_{in}=\left\{\begin{array}{ll}
	\frac{1}{8} \left(E(0,t^{n+1})+E(0,t^{n})\right) \left(H^{n+1}_{h}+H^{n}_{h}\right)^{+}_{1/2}
	+\frac{1}{8} \left(H(0,t^{n+1})+H(0,t^{n})\right)\\ \left(E^{n+1}_{h}+E^{n}_{h}\right)^{+}_{1/2}, 
	& \mbox{ for central flux},\\
	\frac{1}{4} \left(H(0,t^{n+1})+H(0,t^{n})\right) \left(E^{n+1}_{h}+E^{n}_{h}\right)^{+}_{1/2}, & \mbox{ for alternating flux I},\\
	\frac{1}{4} \left(E(0,t^{n+1})+E(0,t^{n})\right) \left(H^{n+1}_{h}+H^{n}_{h}\right)^{+}_{1/2}, & \mbox{ for alternating flux II},\\
	\frac{1}{8} \left(E(0,t^{n+1})+E(0,t^{n})\right) \left(H^{n+1}_{h}+H^{n}_{h}\right)^{+}_{1/2}
	+\frac{1}{8} \left(H(0,t^{n+1})+H(0,t^{n})\right)\\ \left(E^{n+1}_{h}+E^{n}_{h}\right)^{+}_{1/2}
	+\frac{1}{8\sqrt{\epsilon_\infty}}(H^{n}_{h}+H^{n+1}_{h})^{+}_{1/2}[H^{n}_{h}+H^{n+1}_{h}]_{1/2}\\
	+\frac{\sqrt{\epsilon_\infty}}{8}(E^{n+1}_{h}+E^{n}_{h})^{+}_{1/2}[E^{n+1}_{h}+E^{n}_{h}]_{1/2},
	& \mbox{ for upwind flux}.\\
\end{array}
\right.
\end{equation*}
Moreover, for the upwind flux, we have
\begin{align*}
\Theta_{in}=& \frac{1}{16\sqrt{\epsilon_\infty}}[H^{n+1}_{h}+H^{n}_{h}]^{2}_{1/2} +\frac{\sqrt{\epsilon_\infty}}{16}[E^{n+1}_{h}+E^{n}_{h}]^{2}_{1/2}\\ 
& +\frac{1}{16\sqrt{\epsilon_\infty}}\left( (H^{n+1}_{h}+H^{n}_{h})^{+}_{1/2} +\sqrt{\epsilon_\infty} \left( E(0,t^{n+1})+E(0,t^{n}) \right) \right)^2\\ &+\frac{1}{16\sqrt{\epsilon_\infty}}\left( \left( H(t^{n+1},0)+H(t^{n},0) \right) +\sqrt{\epsilon_\infty} \left( E^{n+1}_{h}+E^{n}_{h} \right) \right)^2\\
& -\frac{1}{8\sqrt{\epsilon_\infty}} \left( H(t^{n+1},0)+H(t^{n},0) \right)^2 -\frac{\sqrt{\epsilon_\infty}}{8} \left( E(t^{n+1},0)+E(t^{n},0) \right)^2.
\end{align*}
Thus, 
\begin{align*}
 \mathcal{E}_h^{n+1}- \mathcal{E}_h^{n}\leq
\frac{1}{8\sqrt{\epsilon_\infty}}\Delta t \left( H(t^{n+1},0)+H(t^{n},0) \right)^2 +\frac{\sqrt{\epsilon_\infty}}{8} \Delta t \left( E(t^{n+1},0)+E(t^{n},0) \right)^2.
\end{align*}

On the other hand, the leap-frog scheme with alternating and central fluxes  satisfies
\begin{equation}
\mathcal{E}_h^{n+1}- \mathcal{E}_h^{n}=-\frac{\Delta t}{4 \omega_p^2 \tau} \int_\Omega (J_h^{n+1} +J_h^{n})^2 dx-\frac{a\theta \Delta t}{8\omega_v^2 \tau_v}  \int_\Omega  (\sigma_h^{n+1} +\sigma_h^{n})^2  dx -\Delta t\Theta^{n}_{in} -\Delta t \Theta^{n}_{out} 
\le  -\Delta t\Theta^{n}_{in} -\Delta t \Theta^{n}_{out},
\end{equation}
where 
\begin{eqnarray}
\mathcal{E}_h^n&=&\int_{\Omega}  \frac{1}{2} H_h^{n+1/2} H_h^{n-1/2} + \frac{\epsilon_\infty}{2} (E_h^n)^2 +  \frac{1}{2\omega_p^2} (J_h^n)^2 + \frac{\omega_0^2}{2 \omega_p^2}   (P_h^n)^2 + \frac{a\theta}{4\omega_v^2} (\sigma_h^n)^2 + \frac{a\theta}{2}  Q_h^n (E_h^n)^2,  \notag \\
&& + \frac{3a (1-\theta)}{4} (E_h^n)^4+\frac{a\theta}{4}(Q_h^n)^2  dx, \notag\\
\Theta^{n}_{out}&=& \frac{\sqrt{\epsilon_\infty}}{16} \left( \left(E^{n}_{h}+E^{n+1}_{h}\right)^{-}_{N+1/2} - \frac{2}{\sqrt{\epsilon_\infty}}\left(H^{n+1/2}_{h}\right)^{-}_{N+1/2}\right)^2 +\frac{1}{16\sqrt{\epsilon_\infty}}\left(H^{n+1/2}_{h}\right)^{-}_{N+1/2}\notag\\ &&\left( H^{n-1/2}_{h}-2H^{n+1/2}_{h}+H^{n+3/2}_{h}\right)^{-}_{N+1/2}, \label{eq:nosign}
\end{eqnarray}
\begin{align*}
\Theta^{n}_{in}&=&\left\{\begin{array}{ll}
\frac{1}{4}\left(E(0,t^{n})+E(0,t^{n+1})\right) \left(H^{n+1/2}_{h}\right)^{+}_{1/2} + \frac{1}{4}H(0,t^{n+1/2}) \left(E^{n}_{h}+E^{n+1}_{h}\right)^{+}_{1/2}, & \mbox{ for central flux},\\
\frac{1}{2}H(0,t^{n+1/2}) \left(E^{n}_{h}+E^{n+1}_{h}\right)^{+}_{1/2},& \mbox{ for alternating flux I},\\
 \frac{1}{2}\left(E(0,t^{n})+E(0,t^{n+1})\right) \left(H^{n+1/2}_{h}\right)^{+}_{1/2}, & \mbox{ for alternating flux II}.\\
\end{array}
\right.
\end{align*}
Unlike the previous cases, \eqref{eq:nosign} cannot be shown as non-negative, which means some energy may be injected at the right boundary in this case.

The leap-frog scheme with the upwind flux satisfies
\begin{eqnarray}
 \mathcal{E}_h^{n+1}- \mathcal{E}_h^{n} &=&
-\frac{\Delta t}{4 \omega_p^2 \tau} \int_\Omega (J_h^{n+1} +J_h^{n})^2 dx-\frac{a\theta \Delta t}{8\omega_v^2 \tau_v}  \int_\Omega  (\sigma_h^{n+1}+\sigma_h^{n})^2  dx \\
&& -\frac{\Delta t}{8\sqrt{\epsilon_\infty}} \sum_{j=1}^{N-1}[H_h^{n-1/2}+H_h^{n+1/2}]_{j+1/2}^2
-\frac{\Delta t\sqrt{ \epsilon_\infty}}{8} \sum_{j=1}^{N-1}[E_h^n+E_h^{n+1}]_{j+1/2}^2 \notag\\
&&-\Delta t\Theta_{out}^{n} -\Delta t\Theta_{in}^{n},\notag
\end{eqnarray}
where $\Theta_{in}$, $\Theta_{out}$ and the discrete energy $\mathcal{E}_{h}^{n}$ are
\begin{eqnarray}
\mathcal{E}_h^n&=&\int_{\Omega}  \frac{1}{2} H_h^{n+1/2} H_h^{n-1/2} + \frac{\epsilon_\infty}{2} (E_h^n)^2 +  \frac{1}{2\omega_p^2} (J_h^n)^2 + \frac{\omega_0^2}{2 \omega_p^2}   (P_h^n)^2+ \frac{a\theta}{4\omega_v^2} (\sigma_h^n)^2 \notag \\
&& + \frac{a\theta}{2}  Q_h^n (E_h^n)^2  + \frac{3 a (1-\theta)}{4} (E_h^n)^4+\frac{a\theta}{4}(Q_h^n)^2  dx\notag\\
&&+\frac{\Delta t}{8\sqrt{\epsilon_\infty}}\sum_{j=1}^{N-1} ([H_h^{n-1/2}] [H_h^{n-1/2}+H_h^{n+1/2}])_{j+1/2} \notag \\
&&+\frac{\Delta t}{8\sqrt{\epsilon_\infty}} (H_h^{n-1/2})^{-}_{N+1/2} (H_h^{n-1/2}+H_h^{n+1/2})|^{-}_{N+1/2}, \notag \\
\Theta_{out}^{n}&=& \frac{1}{8\sqrt{\epsilon_\infty}}\left( (H_h^{n-1/2}+H_h^{n+1/2})^{-}_{N+1/2} \right)^2
+\frac{\sqrt{ \epsilon_\infty}}{8} \left( (E_h^n+E_h^{n+1})^{-}_{N+1/2} \right)^2,\notag\\
\Theta_{in}^{n}&=& \frac{1}{4}\left(E(0,t^{n})+E(0,t^{n+1})\right) \left(H^{n+1/2}_{h}\right)^{+}_{1/2} + \frac{1}{4}H(0,t^{n+1/2}) \left(E^{n}_{h}+E^{n+1}_{h}\right)^{+}_{1/2}\notag\\
&&+\frac{1}{8\sqrt{\epsilon_\infty}}\left(H^{n+1/2}_{h}\right)^{+}_{1/2}[H^{n-1/2}_{h}+2H^{n+1/2}_{h}+H^{n+3/2}_{h}]_{1/2}
+\frac{\sqrt{\epsilon_\infty}}{8}(E^{n}_{h}+E^{n+1}_{h})^{+}_{1/2}[E^{n}_{h}+E^{n+1}_{h}]_{1/2}\notag\\
&=&\frac{1}{4\sqrt{\epsilon_\infty}}[H^{n+1/2}_{h}]^{2}_{1/2} +\frac{\sqrt{\epsilon_\infty}}{16}[E^{n+1}_{h}+E^{n}_{h}]^{2}_{1/2} +\frac{1}{4\sqrt{\epsilon_\infty}}\left( (H^{n+1/2}_{h})^{+}_{1/2} +\sqrt{\epsilon_\infty} \frac{E(0,t^{n+1})+E(0,t^{n})}{2} \right)^2\notag\\ &&+\frac{1}{4\sqrt{\epsilon_\infty}}\left( H(t^{n+1/2},0) +\sqrt{\epsilon_\infty} \frac{ E^{n+1}_{h}+E^{n}_{h}}{2}  \right)^2
 -\frac{1}{2\sqrt{\epsilon_\infty}} \left( H(t^{n+1/2},0) \right)^2 -\frac{\sqrt{\epsilon_\infty}}{8} \left( E(t^{n+1},0)+E(t^{n},0) \right)^2\notag\\
 &&+\frac{1}{8\sqrt{\epsilon_\infty}}\left(H^{n+1/2}_{h}\right)^{+}_{1/2}[H^{n-1/2}_{h}-2H^{n+1/2}_{h}+H^{n+3/2}_{h}]_{1/2}.\label{eq:nosign2}
\end{eqnarray}

Note that at fully discrete level, we can only prove energy stability for fully implicit scheme with upwind flux.

\bibliographystyle{siam}
\bibliography{bibfile,bib_Li,Bokil,cheng,cheng_papers,bokil_papers,mimetic}

\end{document}